\documentclass[12pt]{amsart}       %was 12pt
\usepackage{amsmath,amssymb,ifthen,url,wasysym,stmaryrd,xfrac,enumerate,etoolbox,graphicx,
  xr,xypic,pict2e,mathtools,algorithm2e}

\newcommand{\heiscitation}[1]{$\llbracket$#1$\rrbracket$}
\newboolean{heispaper}
\setboolean{heispaper}{true}
\ifthenelse{\boolean{heispaper}}{%
  \externaldocument[HEIS::]{Heisenberg-v1}%
  \newcommand{\heiscite}[1]{\heiscitation{\ref{HEIS::#1}}}%
}{%
  \newcommand{\heiscite}[1]{\ref{HEIS::#1}}}  

\usepackage[utf8,latin1]{inputenc}
\usepackage[T3,T1]{fontenc}
\DeclareSymbolFont{tipa}{T3}{cmr}{m}{n}
\DeclareMathAccent{\invbreve}{\mathalpha}{tipa}{16}

\usepackage[cal=boondoxo]{mathalfa}

\DeclareMathAlphabet{\mathbbm}{U}{bbm}{m}{n}% from bbm.sty

\DeclareFontFamily{U} {MnSymbolA}{}
\DeclareFontShape{U}{MnSymbolA}{m}{n}{
  <-6> MnSymbolA5
  <6-7> MnSymbolA6
  <7-8> MnSymbolA7
  <8-9> MnSymbolA8
  <9-10> MnSymbolA9
  <10-12> MnSymbolA10
  <12-> MnSymbolA12}{}
\DeclareFontShape{U}{MnSymbolA}{b}{n}{
  <-6> MnSymbolA-Bold5
  <6-7> MnSymbolA-Bold6
  <7-8> MnSymbolA-Bold7
  <8-9> MnSymbolA-Bold8
  <9-10> MnSymbolA-Bold9
  <10-12> MnSymbolA-Bold10
  <12-> MnSymbolA-Bold12}{}
\DeclareSymbolFont{MnSyA} {U} {MnSymbolA}{m}{n}
\DeclareMathSymbol{\lcurvearrowdown}{\mathrel}{MnSyA}{195}
\DeclareMathSymbol{\rcurvearrowdown}{\mathrel}{MnSyA}{187}

\usepackage[colorlinks,final,backref=page,hyperindex]{hyperref}

\newcommand{\mma}{Mathematica$^{\scriptscriptstyle\text{\textregistered}}$}

\newcommand\blnk{%
  \mathord{\mskip1mu
    \mathchoice
      {\squarebullet{.35ex}{.35ex}}%
      {\squarebullet{.35ex}{.35ex}}%
      {\kern.15ex\squarebullet{.30ex}{.30ex}}%
      {\kern.15ex\squarebullet{.30ex}{.30ex}}
    \mskip1mu}
}
\newcommand\squarebullet[2]{\vcenter{\hbox{\rule{#1}{#2}}}}

\newcommand{\gr}{\operatorname{gr}}
\newcommand{\inif}{\operatorname{in}}
\newcommand{\A}{\mathcal{A}}
\newcommand{\B}{\mathcal{B}}
\newcommand{\C}{\mathcal{C}}
\newcommand{\D}{\mathcal{D}}
\newcommand{\E}{\mathcal{E}}
\newcommand{\F}{\mathcal{F}}
\newcommand{\K}{\mathcal{K}}
\newcommand{\G}{\mathcal{G}}

\newcommand{\Db}{\mathcal{D}}
\newcommand{\NN}{\mathbb{N}}
\newcommand{\ZZ}{\mathbb{Z}}
\newcommand{\QQ}{\mathbb{Q}}
\newcommand{\RR}{\mathbb{R}}
\newcommand{\CC}{\mathbb{C}}
\newcommand{\II}{\mathbb{I}}
\newcommand{\nnz}[1]{#1^\times}
\newcommand{\unit}[1]{#1^*}

\newcommand{\Hom}{\mathrm{Hom}}
\newcommand{\End}{\mathrm{End}}
\newcommand{\Aut}{\mathrm{Aut}}

\newcommand{\Ab}{\mathbf{Ab}}
\newcommand{\Rng}{\mathbf{Rng}}

\newcommand{\SES}{\mathbf{SES}}
\newcommand{\SESsg}{\mathbf{SES}_{\mathrm{sg}}}
\newcommand{\Mod}{\mathbf{Mod}}

\newcommand{\Hei}{\mathbf{Hei}}

\newcommand{\Alg}{\mathbf{Alg}}
\newcommand{\CAlg}{\mathbf{C\kern-0.1ex{}Alg}}

\newcommand{\Set}{\mathbf{Set}}

\newcommand{\Bi}{\mathbf{Bi}}
\newcommand{\Du}{\mathbf{Du}}
\newcommand{\Four}{\mathcal{F}}
\newcommand{\four}{\mathcal{f}}
\newcommand{\FFour}{\mathcal{G}}
\newcommand{\Genarr}[1]{\smash{\raisebox{-1.5pt}{$\overset{#1}{\kern2\mu\longrightarrow\kern2\mu}$}}}
\newcommand{\Genbiarr}[1]{\smash{\raisebox{-1.5pt}{$\overset{#1}{\kern2\mu\longleftrightarrow\kern2\mu}$}}}

\newcommand{\Par}{\mathcal{P}}
\newcommand{\ModH}[1]{\ifthenelse{\equal{#1}{}}{\mathbf{ModH}}{\mathbf{ModH}(#1)}}
\newcommand{\AlgH}[1]{\ifthenelse{\equal{#1}{}}{\mathbf{AlgH}}{\mathbf{AlgH}(#1)}}
\newcommand{\TwAlgH}[1]{\ifthenelse{\equal{#1}{}}{\mathbf{AlgH_2}}{\mathbf{AlgH_2}(#1)}}
\newcommand{\Fou}[1]{\mathbf{Fou}(#1)}

\newcommand{\inner}[2]{\langle#1|#2\rangle}

\newcommand{\funcinner}[2]{(#1|#2)}

\newcommand{\cum}{{\textstyle\varint}}
\newcommand{\ocum}{{\textstyle \varoint}}
\newcommand{\cumG}{\cum_{\!G}}

\newcommand{\im}{\operatorname{im}}

\newcommand{\tr}{\operatorname{tr}}
\newcommand{\GF}{\operatorname{GF}}

\newcommand{\dotarrow}{\mathrel{\vbox{\offinterlineskip\ialign{\hfil##\hfil\cr\scalebox{1.2}{%
        \normalfont.}\cr\noalign{\kern-.1ex}$\rightarrow$\cr}}}}
\newcommand{\der}{\partial}
\newcommand{\Cat}{\mathbf{Cat}}

\newcommand{\Schw}{\mathcal{S}}
\newcommand{\Hol}{\mathcal{H}}
\newcommand{\Gau}{\mathcal{G}}

\newcommand{\Tor}{\mathbb{T}}
\newcommand{\pont}{\varpi}
\DeclareMathOperator{\sinc}{sinc}
\newcommand{\trp}{'}
\newcommand{\evl}{\text{\scshape\texttt e}}
\newcommand{\ini}{\text{\scshape\texttt i}}
\newcommand{\QQtr}{\QQ^{\mathrm{tr}}}
\newcommand{\Ktr}{K^{\mathrm{tr}}}
\newcommand{\QQab}{\QQ^{\mathrm{ab}}}
\newcommand{\trdeg}{\mathrm{tr.deg}}
\newcommand{\freesln}{\mathfrak{F}}

\newcommand{\supp}{\operatorname{supp}}
\newcommand{\freepln}{\mathfrak{P}}
\newcommand{\freetwn}{\mathfrak{T}}
\newcommand{\Sym}{\operatorname{Sym}}
\newcommand{\zentrum}{\mathcal{Z}}

\newcommand{\ase}{\mathcal{E}}
\newcommand{\smallmat}[4]{\left(\begin{smallmatrix}#1&#2\\#3&#4\end{smallmatrix}\right)}

\newcommand{\Ext}{\operatorname{Ext}}
\newcommand{\galg}{\mathcal{F}}

\newcommand{\ev}{\mathrm{ev}}

\newcommand{\hotimes}{\mathrel{\hat{\otimes}}}
\newcommand{\heiscocy}{\langle\beta\rangle}
\newcommand{\Ann}{\operatorname{Ann}}
\newcommand{\Tordu}{{\mathsf{Tor}}}
\newcommand{\pathalg}[1]{\llparenthesis#1\rrparenthesis}
\newcommand\tdot[1][.6]{{\mathbin{\,\vcenter{\hbox{\scalebox{#1}{$\bullet$}}}}\,}}
\newcommand{\Lie}{\mathrm{Lie}}

\newcommand{\overbrk}[1]{\overbracket[0.75pt]{\smash{\kern2pt#1\kern2pt}\rule{0mm}{6pt}}}
\newcommand{\sech}{\operatorname{sech}}
\newcommand{\csch}{\operatorname{csch}}

\makeatletter
\newcommand*{\act}{\mathpalette\bigcdot@{.65}}
\newcommand*\bigcdot@[2]{\mathbin{\vcenter{\hbox{\scalebox{#2}{$\m@th#1\bullet$}}}}}
\makeatother

\newcommand{\ffact}[1]{%
  ^{\underline{#1\kern-0.75pt}}}

\def\vcum{{\setbox0=%
    \hbox{$\textstyle{\scriptscriptstyle\diagup}{\varint}$}%
    \textstyle{\vcenter{\hbox{$\scriptscriptstyle\diagup$}}\kern-.5\wd0}%
    \!\varint}}

\newcommand*\pFqskip{8mu}
\catcode`,\active
\newcommand*\pFq{\begingroup
        \catcode`\,\active
        \def ,{\mskip\pFqskip\relax}%
        \dopFq
}
\catcode`\,12
\def\dopFq#1#2#3#4#5{%
        {}_{#1}F_{#2}\bigl[\genfrac..{0pt}{}{#3}{#4}\mid#5\bigr]%
        \endgroup
}

\DeclareFontFamily{U}{mathb}{\hyphenchar\font45}
\DeclareFontShape{U}{mathb}{m}{n}{
<-6> mathb5 <6-7> mathb6 <7-8> mathb7
<8-9> mathb8 <9-10> mathb9
<10-12> mathb10 <12-> mathb12
}{}
\DeclareSymbolFont{mathb}{U}{mathb}{m}{n}
\DeclareMathSymbol{\Prec}{\mathrel}{mathb}{"CE}
\DeclareMathSymbol{\Succ}{\mathrel}{mathb}{"CF}

\makeatletter
\DeclareRobustCommand{\pto}{\mathrel{\mathpalette\p@to\to}}
\DeclareRobustCommand{\ppto}{\mathrel{\mathpalette\pp@to\to}}
\DeclareRobustCommand{\pppto}{\mathrel{\mathpalette\ppp@to\to}}
\DeclareRobustCommand{\pgets}{\mathrel{\mathpalette\p@gets\gets}}
\newcommand{\p@to}[2]{%
  \ooalign{\hidewidth$\m@th#1\mapstochar\mkern2mu$\hidewidth\cr$\m@th#1\to$\cr}%
}
\newcommand{\pp@to}[2]{%
  \ooalign{\hidewidth$\m@th#1\mkern-2mu\mapstochar\,\mapstochar\mkern2mu$\hidewidth\cr$\m@th#1\to$\cr}%
}
\newcommand{\ppp@to}[2]{%
  \ooalign{\hidewidth$\m@th#1\mkern-4mu\mapstochar\,\mapstochar\,\mapstochar\mkern2mu$\hidewidth\cr$\m@th#1\to$\cr}%
}
\newcommand{\p@gets}[2]{%
  \ooalign{\hidewidth$\m@th#1\mapstochar\mkern5mu$\hidewidth\cr$\m@th#1\gets$\cr}%
}
\makeatother

\newcommand{\hooklongrightarrow}{\lhook\joinrel\longrightarrow}

\newcommand{\isomarrow}{\xrightarrow{
   \,\smash{\raisebox{-0.65ex}{\ensuremath{\scriptstyle\sim}}}\,}}

\makeatletter
\newcommand{\oset}[3][0ex]{%
  \mathrel{\mathop{#3}\limits^{
    \vbox to#1{\kern-2\ex@
    \hbox{$\scriptstyle#2$}\vss}}}}
\makeatother

\makeatletter
\newcommand{\adjunction}[4]{%
  #1\colon #2%
  \mathrel{\vcenter{%
    \offinterlineskip\m@th
    \ialign{%
      \hfil$##$\hfil\cr
      \longrightharpoonup\cr
      \noalign{\kern-.3ex}
      \smallbot\cr
      \longleftharpoondown\cr
    }%
}}%
  #3 \noloc #4%
}
\newcommand{\longrightharpoonup}{\relbar\joinrel\rightharpoonup}
\newcommand{\longleftharpoondown}{\leftharpoondown\joinrel\relbar}
\newcommand\noloc{%
  \nobreak
  \mspace{6mu plus 1mu}
  {:}
  \nonscript\mkern-\thinmuskip
  \mathpunct{}
  \mspace{2mu}
}
\newcommand{\smallbot}{%
  \begingroup\setlength\unitlength{.15em}%
  \begin{picture}(1,1)
  \roundcap
  \polyline(0,0)(1,0)
  \polyline(0.5,0)(0.5,1)
  \end{picture}%
  \endgroup
}
\makeatother

\let\phi\varphi
\let\epsilon\varepsilon
\let\mathbb\mathbbm
\let\pomega\varpi
\def\boxedplus{{\scriptstyle\boxplus}}

\newtheorem{theorem}{Theorem}
\newtheorem{proposition}[theorem]{Proposition}
\newtheorem{lemma}[theorem]{Lemma}
\newtheorem{corollary}[theorem]{Corollary}

\theoremstyle{definition}
\newtheorem{definition}[theorem]{Definition}
\newtheorem{example}[theorem]{Example}
\newtheorem{remark}[theorem]{Remark}

\newcommand{\myexend}{\hfill$/\!/$}
\newenvironment{myexample}{\begin{example}}{\myexend\end{example}}
\newenvironment{myremark}{\begin{remark}}{\hfill$\circledcirc$\end{remark}}

\title{An Algebraic Approach to\\ Fourier Transformation}

\makeatletter
\def\@setthanks{\vspace{-0.75\baselineskip}\def\thanks##1{\@par\noindent##1\@addpunct.}\thankses}
\def\@setdate{\hspace*{-\parindent}\datename\ \@date\@addpunct.}
\makeatother

\author{Markus Rosenkranz and G{\"u}nter Landsmann$^*$}%
\address{%
  RISC,
  Johannes Kepler University, A-4040 Linz, Austria}%
\email{marcus@rosenkranz.or.at,landsmann@risc.uni-linz.ac.at}%
\date{\today}
\thanks{$^*$ Research funded by the Austrian Science Fund (FWF) under Grant P30052.}%

\begin{document}

\maketitle

\begin{abstract}
  The notion of Fourier transformation is described from an algebraic perspective that lends itself
  to applications in Symbolic Computation. We build the algebraic structures on the basis of a given
  Heisenberg group (in the general sense of nilquadratic groups enjoying a splitting property); this
  includes in particular the whole gamut of Pontryagin duality. The free objects in the
  corresponding categories are determined, and various examples are given. As a first step towards
  Symbolic Computation, we study two constructive examples in some detail---the Gaussians (with and
  without polynomial factors) and the hyperbolic secant algebra.
\end{abstract}

\newpage
\tableofcontents

% Nice general page on Fourier transforms:
% https://terrytao.wordpress.com/2009/04/06/the-fourier-transform/

% =======================
\section{Introduction}
% =======================

\subsection{Motivation from Algorithmic Analysis.}
In this paper we have developed an algebraic theory of Fourier
transforms, which is intended (but not restricted) to serve as a
convenient framework for Symbolic Computation. The goal we have in
mind is to build up a \emph{symbolic operator calculus} to determine
the Green's operators of certain boundary problems for linear partial
differential equations (LPDE). In this application, the role of the
Fourier operator is roughly analogous to the indefinite
integration operator, denoted by~$\cum$ or~$A$ in our earlier
papers~\cite{Rosenkranz2005,RosenkranzRegensburger2008a,RosenkranzRegensburgerTecBuchberger2009,RosenkranzRegensburgerTecBuchberger2012,BuchbergerRosenkranz2012}. While
it is possible to extend the operator calculus of these papers to the
multivariate case by encoding the substitution rule of integration
into a suitable relation ideal~\cite{RosenkranzGaoGuo2019}, such
operators are usually not sufficiently expressive for capturing
Green's operators for LPDE boundary problems.

In fact, even for constant coefficients, only the degenerate case of
completely reducible characteristic
polynomials~\cite{RosenkranzPhisanbut2013} is amenable to a direct
treatment via~$\cum^x, \cum^y$, \dots in conjunction with
substitutions and multiplication operators. In contrast, Fourier
transforms are well known to provide feasible tools for expressing the
solution operators in the case of \emph{LPDE with constant
  coefficients}, so they form a reasonable basis for developing
algorithms to \emph{compute Green's operators} in an algebraic
operator calculus. Of course, the full specification of such an
algorithm will also fix a suitabe class of boundary conditions with
certain constraints to ensure well-posedness.

Apart from the principal goal of setting up an operator calculus, we
would like to mention two other potential application areas: The first
would be a kind of algebraic \emph{structure theory} for Fourier
transforms, somewhat akin to differential Galois
theory~\cite{PutSinger2003}. Some first steps in this direction may be
perceived in the material of Section~\ref{sec:constructive-schwartz}.

The other major application area in Symbolic Computation would be to
build up an \emph{algorithm for Fourier transforms}, perhaps remotely
reminiscent of the Risch algorithm~\cite{Risch1969}. Guided by the
structure theory addressed earlier, such an algorithm would either
express the Fourier transform of a given function (as an element of
the ``signal space'' according to the terminology in the present
paper) in terms of a fixed target domain (here called ``spectral
space''), or otherwise report that this is not possible. In the latter
case, adjunction of new elements would lead to larger spaces/algebras
that contain the desired functions. Clearly, it will take considerable
effort to build up such a theory, but we believe that the structures
developed in this paper could provide a useful basis for such
an undertaking.

\subsection{Overview of the paper.} The material of this paper is to
some extent coupled with that of its \emph{companion
paper}~\cite{LandsmannRosenkranz2019}, where a more general view of
Heisenberg groups is developed, mainly from the perspective of
homological algebra. In the present paper, we focus on the structures
actually arising in constructive analysis as indicated in the previous
subsection. Whenever we refer to specific places in the companion
paper, we will use the shorthand~\heiscitation{\dots}
for~\cite[\dots]{LandsmannRosenkranz2019}.

The remainder of this paper is divided into three sections. Just as field theory is built up before
exploring differential fields (by introducing the notion of derivation), we first build up the
\emph{theory of Heisenberg modules}, based on a fixed Heisenberg group,
in~\S\ref{sec:heis-mod}. From there we develop the theory of \emph{Fourier doublets} (by introducing
an algebraic notion of Fourier operator in~\S\ref{sec:fourier-operators-algebra}); in this way we
can capture the various function spaces---with or without multiplication---on which various species
of Fourier operators act. Finally, we investigate in~\S\ref{sec:constructive-schwartz} some
particular instances of Fourier doublets from the viewpoint of \emph{Symbolic Computation}, giving
special attention to the important example of Gaussian functions (with or without polynomial factors
adjoined).

Let us now go through these three sections in some more detail. We start
in~\S\ref{sub:review-heisgrp} by introducing the particular \emph{notion of Heisenberg group} that
is explored in this paper, with special emphasis on the crucial examples coming from Pontryagin
duality and the so-called symplectic Heisenberg groups. Fixing some specific Heisenberg group~$H$,
the category of \emph{Heisenberg modules} over~$H$ along with the more enriched categories of
\emph{Heisenberg algebras} (featuring one multiplication) and \emph{Heisenberg twain algebras}
(featuring two multiplications) are set up in~\S\ref{sub:cat-heisalg}. Some basic categorical
properties are revealed (monoidal structure, coproduct). In~\S\ref{sub:heis-twist}, we turn to a
feature of Heisenberg groups that is more directly relevant to the goal of building algebraic
Fourier structures---the \emph{Heisenberg twist}~$J$, an involutive antihomomorphism on a Heisenberg
group. Roughly speaking (confer Remark~\ref{rem:justify-twist}), Fourier operators might be likened
to conjugate-linear maps between complex vector spaces (where the involution~$J$ is complex
conjugation). Returning to categorical considerations in~\S\ref{sub:free-heismod-heisalg}, we
explore next the \emph{free objects} in various categories of Heisenberg modules and (twain)
algebras.

In \S\ref{sub:four-doublet}, we define a \emph{Fourier doublet} as a pair consisting of two
Heisenberg modules or two Heisenberg (twain) algebras---referred to as signal/spectral spaces,
respectively---harnessed together with a Fourier operator between them. The role of the Heisenberg
twist is elucidated by way of the \emph{Heisenberg clock} (Figure~\ref{fig:heis-clock}). We set up
suitable categories and their free objects for the various sorts of Fourier doublets. The
all-important case of \emph{Pontryagin duality} is explored in~\S\ref{sub:class-pont-duality},
including various well-known Fourier operators as special cases (Fourier integral, Fourier series,
discrete-time Fourier transform, discrete Fourier transform). While this includes multivariate
functions, one may alternatively reduce them to univariate functions via the tensor product. The
closely related \emph{symplectic Fourier transform} is also mentioned. The topic of \emph{Fourier
  inversion} is initiated in~\S\ref{sub:inversion}, along with crucial examples such as an
``invertible subdoublet'' of the canonical~$L^1$ setting, an extension to $L^2$ functions, and a
Fourier doublet with measures as signals. For the above-mentioned special cases, Fourier inversion
exhibits the familiar sign change in the ``exponential'' (= character). Arguably the most convenient
setting for Fourier analysis, \emph{Schwartz functions}---along with the corresponding tempered
distributions---are introduced in~\S\ref{sub:schwartz-class} for the general case of Pontryagin
duality (where they are more properly named Schwartz-Bruhat). We sketch the setup of a \emph{ring of
  differential operators} following~\cite{Akbarov1995}. This generalizes the action of the Weyl
algebra on the \emph{classical Schwartz space} over~$\RR^n$, which we investigate
in~\S\ref{sub:alg-schwartz-pol} from the algebraic perspective. This includes in particular the
crucial \emph{differentiation law} underlying applications for differential equations
(``differentiation is multiplication by the symbol''). The relation beteen Fourier and Laplace
transformation---for general Pontryagin duality---is briefly mentioned.

We start the consideration of Symbolic Computation aspects in~\S\ref{sub:min-schwartz-algebra} by
extracting from Schwartz space all the \emph{Gaussian functions}. While this is not yet fully
constructive, it allows to come up with explicit formulae for the twain algebra operations and the
Fourier operator. Linear independence of the Gaussians and related functions in ensured by employing
a general asymptotic scale (we develop this approach in some detail since it may come in handy later
when considering more general cases). The resulting Heisenberg algebra is viewed as a semigroup
extension similar to the group extensions studied in group cohomology (again this may be of value
for extensions to be contemplated at a later stage). For making the Gaussian doublet fully
algorithmic, we set up an effective number field in~\S\ref{sub:gelfond-field}. We call it the
\emph{Gelfond field} since it hinges crucially on Gelfond's famous transcendency result
for~$e^\pi$. Restricting to this field in~\S\ref{sub:rational-gauss}, we trim down the Gaussian
doublet to a fully algorithmic domain that is shown to be generated within Schwartz space by a
\emph{single Gaussian} (essentially the probability density of the normal distribution). It is then
shown to be a quotient of a free twain algebra modulo certain identities relating convolution and
pointwise multiplication of Gaussians. For going beyond the Gaussian doublet in Symbolic
Computation, we look at the standard approach via \emph{holonomic functions} (and holonomic
distributions) in~\S\ref{sub:hol-four-extns}. This has the decisive advantage that the Weyl action
mentioned above is also available. Some specific examples are given, in particular the important
case of adjoining the Gaussian doublet by polynomial factors. The Fourier integral is given in terms
of \emph{Hermite functions}, which are analyzed by umbral algebra. As another example,
in~\S\ref{sub:hyp-four-doublet} we set up a Fourier doublet generated by the \emph{hyperbolic
  secent} (which is its own Fourier transform). They are viewed as subdoublets of Schwartz space,
and explicit formulae are available to some extent. More work will be needed at this point to
obtain a more explicit description.

\subsection{Terminological conventions and notation.}
The category of \emph{abelian groups} is denoted by~$\Ab$, the category of \emph{left and right
  modules} over a ring~$R$ by~$_R \Mod$ and~$\Mod_R$, respectively. \emph{Algebras} generally
assumed commutative but nonunital; the corresponding categories are written~$_R \Alg$
and~$\Alg_R$. By an \emph{LCA group} we mean a locally compact abelian group (where locally compact
includes Hausdorff). An \emph{involution} is an automorphism---of groups or algebras---that is
inverse to itself. The \emph{nonzero elements} of~$R$ are denoted by~$\nnz{R}$, its \emph{group of
  units} by~$\unit{R}$. The \emph{path algebra} for a quiver~$Q$ over~$R$ is denoted
by~$R\pathalg{Q}$. If~$Q$ is not endowed with a quiver structure, it is taken as a discrete quiver
(no arrows). This is to be contrasted to the \emph{group algebra}~$R[G]$ over the ring~$R$, for any
given group~$G$.

A nilpotent group of nilpotency class at most two will be called \emph{nilquadratic}. The
\emph{center of a group}~$G$ is denoted by~$\zentrum G$. If~$\Gamma$ and~$G$ are any abelian groups,
we write their \emph{tensor product} as~$\Gamma \otimes G := \Gamma \otimes_{\ZZ} G$. A
\emph{bilinear form} on these groups is a $\ZZ$-bilinear map~$\beta\colon \Gamma \times G \to T$,
viewing abelian groups as $\ZZ$-modules (hence a $\ZZ$-linear map~$\Gamma \otimes G \to T$). We will
often write the action of such a form as~$\inner{\xi}{x}_\beta = \beta(\xi, x)$, suppressing the
subscript~$\beta$ when it is obvious from the context. We call~$\beta$ or a \emph{duality}
if~$\beta(\xi,-)$ and~$\beta(-,x)$ are injective for all~$\xi \in \Gamma$ and~$x \in G$. The term
\emph{bicharacter} is usually taken as synonymous with the notion of bilinear form (with the
exception of Example~\ref{ex:vecgrp}).

Given a $R$-module~$M$, an $n$-form is a bilinear map~$M^n \to R$; for~$n=2$ one either
suppresses~$n$ or speaks of a \emph{bilinear form}. An \emph{alternating form} is a bilinear
map~$\omega\colon M \oplus M \to R$ such that~$\omega(x,x) = 0$ for all~$x \in M$; it is called a
\emph{symplectic form} if it is moreoever, again in the sense that~$\omega(x,-)\colon M \to M$ is
injective for all~$x \in M$. One refers to the structure~$(M, \omega)$ as an alternating or
symplectic module, respectively. The \emph{symmetric algebra} over~$M$ is denoted
by~$\Sym(M) \cong T(M)/I(M)$, where~$I(M) \trianglelefteq T(M)$ is the ideal generated
by~$\{x \otimes y - y \otimes x \mid x,y \in M\}$. In case of ambiguity, the base ring~$R$ may be
explicated in writing~$\Sym_R(M) \cong T_R(M)/I_R(M)$. Note that for a set~$X$ one
has~$R[X] = \Sym_R(RX)$, where~$RX$ is the free $R$-module over the basis~$X$. Given a
homomorphism~$\sigma\colon R' \to R$, an $R$-module~$M$ becomes an $R'$-module~$M[\sigma]$ via
scalar restriction along~$\sigma$; one also calls~$M[\sigma]$ a \emph{twisted module}.

If~$K$ is a field, we write the vector space of \emph{column
  vectors} as~$K^n$, that of \emph{row vectors} as~$K_n$. The $n \times n$ \emph{identity matrix} is
denoted by~$I_n$. We shall sometimes employ the abbreviation~$x^- := x^{-1}$ for the
\emph{reciprocal} (when it exists) of an element~$x$ in a multiplicatively written monoid.

The \emph{unit interval} is denoted by~$\II = [0,1]$. Given numbers~$\delta_1, \delta_2 \in \RR$, we
use~$x \pm_\delta y$ as an abbreviation for their \emph{weighted sum or difference}
$(\delta_1 x \pm \delta_2 y)/(\delta_1 + \delta_2)$; note that~$+_\delta$ is in general not
commutative. We introduce also the harmonic sum~$\boxedplus$ on~$\RR_{>0}$
as~$x \boxedplus y := xy/(x+y)$ to create the \emph{harmonic semigroup}
$(\RR_{>0}, \boxedplus) \cong (\RR_{>0}, +)$, with the iso\-morphism given
by~$i\colon \RR_{>0} \to \RR_{>0}, x \mapsto x^{-1}$.

% =====================================================================
\section{Heisenberg Modules and Algebras}
\label{sec:heis-mod}

\subsection{Review of Heisenberg Groups.}\label{sub:review-heisgrp}
As outlined in the Introduction, our treatment of Fourier operators is based on Heisenberg modules
and Heisenberg algebras. These in turn are built on the central notion of \emph{Heisenberg
  group}. In the literature~\cite{BonattoDikranjan2017} \cite{LuefManin2009}
\cite{PrasadShapiroVemuri2010} \cite{ParasadVemuri2008} \cite{Vemuri2008} one finds several versions
of this idea, exhibiting various degrees of generality. In the present paper, we shall use a
definition close to~\cite[Def.~3.1]{BonattoDikranjan2017}, since this definition remains entirely
within algebraic confines, admitting a level of generality sufficient for our current goal: to forge
algebraic and algorithmic tools that support calculations with Fourier operators. In
Definition~\heiscite{def:heis-grp}, we provide a somewhat wider notion of Heisenberg group, with a
view towards homological questions.

For preparing our definition, recall that a Heisenberg group---even
under more general conceptions---is always a nilquadratic extension,
meaning a \emph{central extension}~$H$ of an abelian group~$P$ by
another abelian group~$T$. Whereas the \emph{phase space}~$P$ is
written additively, the \emph{torus}~$T$ is conventionally presented
in multiplicative notation (in view of its most prominent
specimen~$T = \mathbb{S}^1 \subset \nnz{\CC}$). The corresponding
short exact
sequence~$E\colon T \oset{\iota}{\rightarrowtail} H
\oset{\pi}{\twoheadrightarrow} P$ gives rise to a symplectic structure
by way of the commutator form~$\omega_E$; we refer the reader to
\S\heiscite{sub:nilquadratic-groups} for some basic terminology in
this area, in particular the notions of (co)symplectic subgroups and
Lagrangian subgroups.

At this point we shall only review the notions immediately entering
the definition of Heisenberg group given below. Recall that
\emph{maximal abelian subgroups}~$\tilde{G} \le H$ are in bijective
correspondence with \emph{Lagrangian subgroups}~$G \le P$ via the
projection~$\pi\colon H \to P$. A pair~$(\tilde{G}, \tilde{\Gamma})$
of maximal abelian subgroups is called an \emph{abelian bisection}
of~$H$ over~$T$ if~$\tilde{G} \tilde{\Gamma} = H$ and
$\tilde{G} \cap \tilde{\Gamma} = \hat{T} := \iota(T)$; it corresponds
under~$\pi$ to a \emph{Lagrangian bisection} of~$H$, meaning a direct
decomposition of~$P$ into Lagrangian subgroups. It is clear that~$\pi$
then descends to isomorphisms~$H/\tilde{G} \cong \Gamma$
and~$H/\tilde{\Gamma} \cong G$.

Note that~$\tilde{G}$ and~$\tilde{\Gamma}$ both contain the embedded
torus~$\hat{T} := \iota(T)$. If the latter splits off as a direct
summand in~$\tilde{G}$ as well as~$\tilde{\Gamma}$, we
call~$(\tilde{G}, \tilde{\Gamma})$ an \emph{abelian splitting}
and~$(G, \Gamma)$ a \emph{Lagrangian splitting}. As a consequence,
$\hat{T}$ has direct complements in~$\tilde{G}$ and~$\tilde{\Gamma}$,
which are easily seen to be isomorphic to~$G$ and~$\Gamma$,
respectively. Hence we shall take the liberty of designating those
direct complements by~$G, \Gamma \le H$ as well. We can now introduce
Heisenberg groups in the form they are used in this paper.

\begin{definition}
  \label{def:abs-heisgrp}
  A \emph{Heisenberg group}~$(H; \tilde{G}, \tilde{\Gamma})$ is a nilquadratic group~$H$ with an
  abelian splitting~$(\tilde{G}, \tilde{\Gamma})$.
\end{definition}

As usual in such cases, the abelian splitting~$(\tilde{G}, \tilde{\Gamma})$ is sometimes suppressed
when speaking of ``the Heisenberg group~$H$'', but~$(\tilde{G}, \tilde{\Gamma})$ must nevertheless
be viewed as part of the data: As we shall see below (Example~\ref{ex:fail-ex-uniq}),
\emph{existence and uniqueness of abelian splittings may fail} for a general nilquadratic group. In
contrast, we shall soon learn that the (implicit) choice of complements~$G, \Gamma \le H$
of~$\hat{T}$ has no influence on the Heisenberg group.

The Heisenberg groups form a \emph{category}~$\Hei$ with objects short exact
sequences~$\mathfrak{H}\colon T \oset{\iota}{\rightarrowtail} H \oset{\pi}{\twoheadrightarrow} G
\oplus \Gamma$.
If~$\mathfrak{H}'\colon T' \oset{\iota'}{\rightarrowtail} H' \oset{\pi'}{\twoheadrightarrow} G'
\oplus \Gamma'$
is an\-other Heisenberg object, a \emph{Heisenberg
  morphism}~$(t, h, g \times \gamma)\colon \mathfrak{H} \to \mathfrak{H}'$ consists of group
homomorphisms~$t\colon T \to T'$ and $h \colon H \to H'$ as well as
$g \oplus \gamma\colon G \oplus \Gamma \to G' \oplus \Gamma'$ such that the diagram
\begin{equation}
  \label{eq:heis-extn}
  \xymatrix @M=0.5pc @R=1pc @C=2pc%
  { 1 \ar[r] & T \ar[r] \ar[d]^t & H \ar[r] \ar[d]^h & G \oplus \Gamma \ar[r] \ar[d]^{g \oplus
      \gamma} & 0\\
    1 \ar[r] & T' \ar[r] & H' \ar[r] & G' \oplus \Gamma' \ar[r] & 0 }
\end{equation}
commutes. Thus a Heisenberg morphism is just a morphism of short exact sequences that respects the
Lagrangian splittings.

The Heisenberg groups~\eqref{eq:heis-extn} with fixed Lagrangian
splitting~$(G, \Gamma)$ form a subclass within all nilquadratic
extensions, corresponding to a certain
subgroup~$H_{G, \Gamma}^2(P,T) \le H^2(P,T)$ under cohomology. As for
all nilquadratic extensions, the Universal Coefficient Theorem induces
a short extact sequence
\begin{equation}
  \label{eq:H2-seq}
  \xymatrix @M=0.5pc @R=1pc @C=2pc%
  { 1 \ar[r] & \Ext^1(P,T) \ar[r]^-{j} & H^2(P, T) \ar[r]^-{q} & \Omega^2(P,T) \ar[r] & 0 }  
\end{equation}
which splits, albeit not naturally. Here~$j$ is the natural embedding of abelian (= symmetric)
cocycles while~$q$ is the skewing map~$[\gamma] \mapsto \omega$
with~$\omega(z, w) = \gamma(z,w)/\gamma(w,z)$. In the (rare) event that~$T$ is uniquely
$2$-divisible~\cite[Thm.~5.4]{Warfield1976} in the sense that for every~$c \in T$ there exists
a unique~$d \in T$ with~$d^2 = c$, one may choose a canonical section~$\Omega^2(P,T) \to H^2(P, T)$
of~$q$ as follows.  Writing~$\sqrt{\;}$ for the inverse of the squaring homomorphism, one
sends~$\omega \in \Omega^2(P,T)$ to~$[\gamma] \in H^2(P, T)$ defined
by~$\gamma(z, w) = \sqrt{\omega(z, w)}$. In other words, one can recover the equivalence class of
the group extension~$H$ of~\eqref{eq:heis-extn} from its commutator form~$\omega$
via~$H \cong [H]_\omega := T \times_{\sqrt{\omega}} P$; see Lemma~\ref{lem:sympl-heisgrp} below for
the case of Heisenberg groups.

There is an alternative perspective on Heisenberg groups that
emphasizes their constructive nature. Given any bilinear
form~$\beta\colon \Gamma \times G \to T$ from abelian groups~$G$
and~$\Gamma$ to another abelian group~$T$, one may construct the
\emph{phase action}
\begin{equation}
  \label{eq:phase-action}
  \vartriangleleft\colon \Gamma \to \Aut(T G),\quad
  \vartriangleleft_\xi \, cx = c \, \inner{x}{\xi} \, x,
\end{equation}
where~$TG$ denotes the direct product of~$T$ with~$G$, whose elements
we choose to write as~$cx \in TG$ with~$c \in T$ and~$x \in G$. It is
clear that the phase action~\eqref{eq:phase-action} is faithful
iff~$\inner{}{}_\beta$ is a duality.

\begin{definition}
  \label{def:assoc-heisgrp}
  Let~$\beta\colon \Gamma \times G \to T$ be a bilinear form. Then the \emph{Heisenberg group}
  assciated to~$\beta$ is the semidirect product~$H(\beta) := TG \rtimes \Gamma$ with respect to the
  phase action~\eqref{eq:phase-action}.
\end{definition}

As has been shown in~\cite[Thm.~3.4]{BonattoDikranjan2017} for a
slightly restricted setting (with torus~$T = \zentrum(H)$ and
with~$\beta$ being a duality) and in
Definition~\heiscite{def:heis-grp} for somewhat more general cases
(requiring only bisections rather than splittings), the Heisenberg
groups are exactly those of the form~$H(\beta)$. More precisely, for
each Heisenberg group~$(H; \tilde{G}, \tilde{\Gamma})$ there is a
\emph{unique bilinear form}~$\beta\colon \Gamma \times G \to T$ such
that~$H \cong H(\beta)$ in~$\Hei$, where~$H(\beta)$ is endowed with
the Lagrangian splitting~$(\tilde{G}, \tilde{\Gamma})$
with~$\tilde{G} := T \times G$
and~$\tilde{\Gamma} := T \times \Gamma$. This also reiterates our
observation above that the abelian splitting is part of the data for a
Heisenberg group as it corresponds to the choice of bilinear form. On
the other hand, the choice of complements of~$\hat{T}$ in
both~$\tilde{G}$ and~$\tilde{\Gamma}$ is immaterial; it merely
determines the particular isomorphism between~$H$ and~$H(\beta)$.

In this paper, we represent Heisenberg groups in the form~$H(\beta)$, with a bilinear
form~$\beta\colon \Gamma \times G \to T$. Applied to arguments~$(\xi, x) \in \Gamma \times G$, the
former is written as~$\inner{\xi}{x}_\beta$, with the index~$\beta$ omitted when the context makes
it clear. Motivated by the fundamental example of symplectic duality (see
Remark~\ref{rem:physint-heisgrp} below), the elements of~$G$ will be called \emph{positions}, those
of~$\Gamma$ as \emph{momenta}. If~$\beta$ is nondegenerate, we can identify positions~$x \in G$ with
their \emph{position characters} \mbox{$\inner{-}{x}\colon \Gamma \to T$} given
by~$\xi \mapsto \inner{\xi}{x}$, and momenta~$\xi \in \Gamma$ with their \emph{momentum characters}
$\inner{\xi}{-}\colon G \to T$, $x \mapsto \inner{\xi}{x}$.

The isomorphism~$H \cong [H]_\omega = T \times_{\sqrt{\omega}} P$ mentioned above takes on the
following form in case of a Heisenberg group~$H = H(\beta)$. Note that~$[H]_\omega$ may be viewed as
a canonical representation of all Heisenberg extensions~\eqref{eq:heis-extn} with fixed commutator
form~$\omega\colon P \times P \to T$, but $[H]_\omega$ is \emph{not} itself a Heisenberg group in
our sense of Definition~\ref{def:abs-heisgrp}. Due to its most important instances (see
Example~\ref{ex:heisgrp-sympl} below), we shall call it the \emph{symplectic Heisenberg group}.

\begin{lemma}
  \label{lem:sympl-heisgrp}
  If~$\beta\colon \Gamma \times G \to T$ is a bilinear form over a uniquely $2$-divisible torus~$T$
  with square root~$\sqrt{\;}$, the Heisenberg extension~\eqref{eq:heis-extn} is equivalent to the
  extension~$T \rightarrowtail [H]_\omega \twoheadrightarrow P$ via the \emph{polarization map}
  $\mathfrak{P}_{G, \Gamma}\colon [H]_\omega \isomarrow H(\beta)$ given
  by~$(c; x, \xi) \mapsto c \, \sqrt{\inner{\xi}{x}} \, (x, \xi)$.
\end{lemma}
\begin{proof}
  Clearly, $\Phi$ is inverted
  by~$c \, (x, \xi) \mapsto (c \, \sqrt{\inner{\xi}{x}^{-1}}; x,
  \xi)$, and it is a group homomorphism as one checks
  immediately. Hence~$(1_T, \Phi, 1_P)$ is an equivalence from the
  extension~$T \rightarrowtail T \times_{\sqrt{\omega}}\, P
  \twoheadrightarrow P$ to the Heisenberg
  extension~\eqref{eq:heis-extn}.
\end{proof}

As noted above, there is no canonical section of the \emph{entire} skewing
map~$q\colon H^2(P, T) \to \Omega^2(P, T)$ if the torus~$T$ is not uniquely $2$-divisible. In the
presence of a Lagrangian splitting~$P = G \oplus \Gamma$, however, one can still recover the
Heisenberg cocycles~$[\gamma] \in H^2(P, T)$ from their commutator forms~$\omega \in \Omega^2(P, T)$.
In fact, the bilinear cocycles~$\gamma \in Z^2(P, T)$ with~$\gamma(z,w) = 0$ for~$z \in G$
or~$w \in \Gamma$ form a subgroup~$Z^2_{G,\Gamma}(P, T)$ holding canonical representatives
for~$H^2_{G,\Gamma}$; see Proposition~\heiscite{prop:H2-GGamma-Subgroup}. Similarly, one may
introduce the subgroup~$\Omega^2_{G,\Gamma}(P,T)$ of skew-symmetric forms~$\omega \in \Omega^2(P,T)$
that vanish on~$\Gamma \times \Gamma$ and~$G \times G$. Let
\begin{equation*}
  \xymatrix @M=0.5pc @R=1pc @C=2pc%
  { \Gamma \ar@<0.5ex>[r]^{\iota_\Gamma} & P \ar@<0.5ex>[l]^{\pi_\Gamma} 
    \ar@<-0.5ex>[r]_{\pi_G} & G \ar@<-0.5ex>[l]_{\iota_G} }
\end{equation*}
be the corresponding direct sum diagram with associated projectors
$p_\Gamma = \iota_\Gamma \pi_\Gamma, p_G = \iota_G \pi_G\colon P \to P$. Then we have the following
commutative diagram of group isomorphisms:
\begin{equation*}
  \xymatrix @M=0.5pc @R=2pc @C=2pc%
  { & \Omega^2_{G,\Gamma}(P, T) \ar@/_1.5pc/[dl]^-{(p_\Gamma \times p_G)^*} 
    \ar@/^1.5pc/[dr]_-{(\iota_\Gamma \times \iota_G)^*}\\
    Z^2_{G,\Gamma}(P, T) \ar@<0.5ex>[rr]^-{(\iota_\Gamma \times \iota_G)^*}
    \ar@<1ex>@/^1.5pc/[ur]^-{q} && (\Gamma \otimes G)^* \ar@<-1ex>@/_1.5pc/[ul]_-{\tilde{q}}
    \ar@<0.5ex>[ll]^-{(\pi_\Gamma \times \pi_G)^*}}
\end{equation*}
Here~$\tilde{q}$ maps a bilinear map~$\beta\colon \Gamma \times G \to T$ to the skew-symmetric form
$\omega = (\pi_\Gamma \times \pi_G)^*(\beta) - (\pi_G \times \pi_\Gamma)^*(\beta\trp)$.  In other
words, one applies antisymmetrization in the
form~$\omega(x, \xi; y, \eta) = \beta(\xi, y) - \beta(\eta, x)$. According to the above diagram,
\emph{Heisenberg cocycles}~$\gamma \in Z^2_{G,\Gamma}(P, T)$, their \emph{commutator
  forms}~$\omega \in \Omega^2_{G,\Gamma}(P, T)$ and the associated
\emph{dualities} $\beta\colon \Gamma \times G \to T$ all contain the same information once a
fixed Lagrangian splitting~$P = G \oplus \Gamma$ is known.

Before considering some examples of Heisenberg groups, let us first
have a closer look at the group operations. According to
Definition~\ref{def:assoc-heisgrp}, the \emph{multiplication} in a
Heisenberg group is given by
\begin{equation}
  \label{eq:heis-grouplaw}
  c \, (x, \xi) \cdot c' \, (x', \xi') = cc' \, \inner{\xi}{x'} \,
  (x+x', \xi+\xi'),
\end{equation}
the \emph{unit element} by~$1\,(0,0)$ and the \emph{inversion map} by
\begin{equation}
  \label{eq:heis-groupinv}
  \big( c \, (x, \xi) \big)^{-1} = c^{-1} \inner{\xi}{x} \, (-x, -\xi)
\end{equation}
for any $c \, (x, \xi) \in H(\beta)$ and
$c' \, (x', \xi') \in H(\beta)$. The group
law~\eqref{eq:heis-grouplaw} may also be expressed as a kind of
schematic matrix multiplication if one identifies the triples with
$3 \times 3$ matrices such that their products correspond via
\begin{equation}
  \label{eq:heis-matrix}
  c \, (x, \xi) \cdot
  c' \, (x', \xi')\qquad\leftrightarrow\qquad
  \begin{pmatrix}1 & \xi & c\\ 0 & 1 & x\\ 0 & 0 & 1\end{pmatrix} \cdot
  \begin{pmatrix}1 & \xi' & c'\\ 0 & 1 & x'\\ 0 & 0 & 1\end{pmatrix} .
\end{equation}
Here one should keep in mind that the addition of the upper right
matrix elements corresponds to the (multiplicatively written) group
operation in~$T$.  Furthermore, one agrees that a left matrix
element~$\xi \in \Gamma$ multiplies with a right matrix
element~$x' \in G$ to yield~$\inner{\xi}{x'} \in T$; all other
products are trivial since they involve~$0$ or~$1$.

For closer analysis we introduce the \emph{bilinear category} as the comma
category~$\Bi := \mathord{\otimes} \downarrow 1_{\Ab}$ where~$\otimes\colon \Ab \times \Ab \to \Ab$
is the tensor product. By the characteristic property of the latter, we may view the
objects~$\Gamma \otimes G \smash{\overset{\beta}{\to}} T$ of~$\Bi$ as bilinear forms (also known as
\emph{bicharacters}). For brevity, we shall also write these as~$\inner{\Gamma}{G}_\beta$, omitting
the subscript~$\beta$ when it is clear from the context. A morphism
$(\gamma \times g, t)\colon \beta \to \beta'$ between $\beta\colon \Gamma \times G \to T$
and~$\beta'\colon \Gamma' \times G' \to T'$ is given by homomorphisms
$\gamma \times g \colon \Gamma \times G \to \Gamma' \times G'$ and $t \colon T \to T'$ such that
$t \, \inner{\xi}{x}_\beta = \inner{\gamma\xi}{gx}_{\beta'}$ for all $(\xi,x) \in \Gamma \times G$.
One may check that~$\Bi$ is a bicomplete abelian category. Fixing a torus~$T$, we write~$\Bi(T)$ for
the subcategory of~$\Bi$ consisting of bilinear forms~$\Gamma \otimes G \to T$ with
arbitrary~$\Gamma, G \in \Ab$, having morphisms~$(\gamma \times g, 1_T)$ in between them. We shall
subsequently suppress the identity~$1_T$, referring to~$\gamma \times g$ as a morphism in~$\Bi(T)$.

There is a \emph{tensor product} on~$\Bi(T)$ that sends a pair~$\beta\colon \Gamma \times G \to T$
and~$\beta'\colon \Gamma' \times G' \to T$ of bilinear forms to the bilinear form~$\beta \otimes
\beta'\colon (\Gamma \oplus \Gamma') \times (G \oplus G') \to T$ with
\begin{equation*}
  \inner{\xi, \xi'}{x, x'}_{\beta\otimes\beta'} := \inner{\xi}{x}_\beta \, \inner{\xi'}{x'}_{\beta'}
\end{equation*}
and a pair~$\gamma \times g$ and~$\gamma' \times g'$ of morphisms in~$\Bi(T)$
to~$(\gamma \oplus \gamma') \times (g \oplus g')$. As one sees immediately---associators and unitors
just shuffling parentheses---this gives a symmetric monoidal structure on~$\Bi(T)$ with unit object
the trivial bilinear form~$0 \oplus 0 \to T$.

We write $\Du$ for the full subcategory of $\Bi$ consisting of \emph{dualities} (i.e.\@
nondegenerate bilinear forms). It is easy to see that $\Du(T)$ is a symmetric monoidal subcategory
of~$\Bi$. It is sometimes useful to think of~$\Du$ as fibered over~$\Ab$ by sending a
duality~$\beta\colon \Gamma \times G \to T$ to its torus~$T$.

As detailed in Proposition~\heiscite{prop:heis-functor}, the association~$\beta \mapsto H(\beta)$ is
a functor~$H\colon \Du \to \Hei$, which is adjoint to the functor~$B\colon \Hei \to \Du$ extracting
the restriction~$\beta\colon \Gamma \times G \to T$ of the commutator form~$\omega_E$ of a given
Heisenberg extension. In fact, more is true\cite[p.~94]{MacLane1998}: The functor~$B$ is
\emph{left-adjoint left-inverse} to the functor~$H$, so the latter is an isomorphism of~$\Du$ to the
reflective subcategory~$\{ H(\beta) \mid \beta \in \Du\}$ of~$\Hei$, taking license to use curly
braces for class formation. In other words, we have an adjoint equivalence whose counit is trivial
and whose unit may be viewed as a canonical simplifier~\cite{BuchbergerLoos1982}: Every Heisenberg
group~$H$ is isomorphic to exactly one~$H(\beta) \cong H$ that acts as the canonical representative
of its equivalence class. Thinking of~$B\colon \Hei \to \Du$ as a very special fibration (each fiber
consisting of equivalent objects), the unit creates a cross-section through the fibers.

This is why in the following we usually write \emph{Heisenberg groups in canonical
  form}~$H(\beta) = TG \rtimes \Gamma$. The full subcategory of~$\Hei$ consisting of Heisenberg
groups over~$\beta \in \Du$ is denoted by~$\Hei_\beta$. Moreover, we shall identify~$G$ and~$\Gamma$
as subgroups of~$H(\beta)$. Similarly, the torus~$T$ is identified with its embedding~$\hat{T}$
in~$H(\beta)$.

\begin{myexample}
  \label{ex:pont-duality-first}
  The most important examples of Heisenberg groups are based on \emph{Pontryagin duality}. Given a
  locally compact abelian group~$G$, its \emph{characters} in the classical sense are the continuous
  homomorphisms from~$G$ into the complex torus~$\Tor \subseteq \CC$. The collection~$\hat{G}$ of
  all characters is then again an abelian group known as the \emph{dual group} of~$G$, and the
  famous Pontryagin duality theorem implies that the natural
  pairing~$\pont\colon \hat{G} \times G \to \Tor$ with~$\inner{\xi}{x}_\pont := \xi(x)$ is a
  nondegenerate bilinear form; confer for example Theorem~1.7.2 of~\cite{Rudin2017}
  or~\cite[\S4]{Morris1977}. In this way, the Heisenberg group~$H(\pont)$ may be associated to every
  locally compact group~$G$, as in Andr{\'e} Weil's definition~(4) of~\cite[p.~149]{Weil1964}, where
  the Heisenberg group~$H(\pont)$ is denoted by~$\mathbf{A}(G)$.
  
  Note that this also includes plain groups~$G$ \emph{without
    topology} since one may always endow them with the discrete
  topology. In this case, the character group $\hat{G}$ is the group
  of \emph{all} homomorphisms~$G \to T$, no matter what the topology
  on~$T$ may be. Thus one may view the algebraic setting as contained
  in a topological frame (where~$\hat{G}$ is an LCA group). Note that
  a bilinear form~$\beta\colon \Gamma \times G \to T$ is nondegenerate
  iff one has embeddings $G \hookrightarrow \hat{\Gamma}$ and
  $\Gamma \hookrightarrow \hat{G}$ for encoding positions and momenta
  by their characters.
\end{myexample}

\begin{myexample}
  \label{ex:classical-torus-group}
  One of the most famous examples of Pontryagin duality is given by the
  map~$\Tor \times \ZZ \to \Tor$ defined by~$\inner{\xi}{x} := \xi^x$. It is clear that this may be
  extended to~$\Tor^n \times \ZZ^n \to \Tor$
  with~$\inner{\xi}{x} = \xi^x = \xi_1^{x_1} \cdots \xi_n^{x_n}$. We refer to this example as the
  \emph{torus duality}~$\Tordu = \inner{\Tor^n}{\ZZ^n}$. Interchanging positions and momenta, we
  obtain the \emph{conjugate torus duality}~$\Tordu\trp = \inner{\ZZ^n}{\Tor^n}$. Their basic
  difference will be seen in Example~\ref{ex:pont-doublet}\ref{it:four-dtft}, their close
  relationship at the end of Example~\ref{ex:pont-doublet-inverse}.
\end{myexample}

Another basic example of Pontryagin duality is finite-dimensional
\emph{normed vector spaces} (Euclidean vector spaces with norm
topology). In this case, all linear functionals are automatically
continuous, so one may in fact ignore the topology altogether and
consider the general case of vector spaces with bilinear forms.

\begin{myexample}
  \label{ex:vecgrp}
  More precisely, we study now bilinear forms~$\Gamma \times G \to T$
  where~$G = \Gamma$ is an $n$-dimensional vector space~$V$ over the
  common scalar field~$F$. In this setting it is more natural to study
  \emph{additive bilinear forms}
  $\funcinner{}{}\colon V \times V \to (F, +)$ rather than
  multiplicative bilinear
  forms~$\inner{}{}\colon V \times V \to (T, \cdot)$ for a
  torus~$T$. The latter kind is our exclusive object of study in the
  present paper, but for contrast we shall refer to them in this
  context by their alternative name \emph{bicharacter}. Fixing any
  \emph{character} $\chi\colon (F, +) \to (T, \cdot)$, one obtains the
  \emph{associated bicharacter} by
  setting~$\inner{\xi}{x} := \chi \funcinner{\xi}{x}$
  for~$(x,\xi) \in V \times V$. By a character we mean here a group
  homomorphism that we may take to be an epimorphism (shrinking the
  torus~$T$ if necessary).

  The bicharacter~$\inner{}{}$ induced by the additive bilinear form~$\funcinner{}{}$ is clearly
  degenerate if the latter is. As noted above, the converse is true if we restrict ourselves to
  \emph{standard characters}. We define these as epimorphisms~$\chi\colon (F, +) \to (T, \cdot)$
  such that~$\ker \chi = \ZZ_*$, where~$\ZZ_*$ denotes the ``prime ring'' of~$F$, meaning the
  smallest nontrivial subring of~$F$. We have~$\ZZ_* = \ZZ_p$ if the field~$F$ has positive
  characteristic~$p$, otherwise~$\ZZ_* = \ZZ$. The resulting duality~$\inner{V}{V}$ will be called
  an \emph{abstract vector duality}.
\end{myexample}

Using standard characters, nondegeneracy may be viewed equivalently as a property of bilinear forms
or their bicharacters. In the sequel, we may therefore identify nondegenerate bilinear forms with
their associated bicharacters.

\begin{lemma}
  \label{lem:bichar-nondeg}
  For a bilinear form $\funcinner{}{}\colon V \times V \to F$ on a vector space,
  let~$\inner{}{}\colon V \times V \to T$ be its associated bicharacter with respect to any fixed
  standard character~$\chi\colon (F, +) \to (T, \cdot)$. Then~$\inner{}{}$ is nondegenerate
  iff~$\funcinner{}{}$ is nondegenerate.
\end{lemma}
\begin{proof}
  See Lemma~\heiscite{lem:bichar-nondeg}.
\end{proof}

\begin{myexample}
  \label{ex:classical-vector-group}
  Typically, the bilinear form on~$V=F^n$ is the standard inner
  product~$\inner{\xi}{x} = \xi \cdot x = \xi_1 x_1 + \cdots + \xi_n
  x_n$, yielding the \emph{$n$-dimensional vector duality}
  over~$F$. Since the inner product is also equivalent to the natural
  paring~$F^n \times F_n \to F$, one may use the dot notation for both.

  The most important special case of Example~\ref{ex:vecgrp} is given
  by the complex field~$K = \CC$ with the classical torus~$T = \Tor$
  and vector spaces~$V = \RR^n$ having their canonical Euclidean
  structure. Via the standard character $\chi(t) = e^{i\tau t}$, one
  obtains the bicharacter~$\inner{\xi}{x} = e^{i \tau x\cdot\xi}$. We
  shall refer to this example as the $n$-dimensional \emph{standard
    vector duality}~$\inner{\RR^n}{\RR^n}$, which may also be
  presented as~$\inner{\RR_n}{\RR^n}$ under the natural pairing.
\end{myexample}

\begin{myexample}
  \label{ex:modular-vector-group}
  If~$V = F^n$ is a vector space over a Galois field~$F = \GF(q)$
  with~$q = p^m \; (m \in \NN)$ elements, we may again use the usual
  dot product as a bilinear form.  Using the multiplicative cyclic
  group~$\langle\zeta_p\rangle \subset \QQ(\zeta_p)$ generated by a
  $p$-th unit root~$\zeta_p$ as in the proof of
  Lemma~\ref{lem:bichar-nondeg}, one gets a standard
  character~$\chi_a\colon \GF(q) \to \langle\zeta_p\rangle$,
  $c \mapsto \zeta_p^{\tr ac}$, for arbitrary $a \in \nnz{\ZZ_p}$ and
  trace map~$\tr\colon \GF(q) \to \ZZ_p$,
  $c \mapsto c + c^p + \cdots + \smash{c^{p^{m-1}}}$; for details
  regarding characters on Galois fields
  see~\cite{BensonRatcliff2008}. With the induced
  duality~$\inner{\xi}{x} = \chi_a(x \cdot \xi)$, we call this example
  the \emph{modular vector duality}~$\inner{\GF(q)^n}{\GF(q)^n}$.
\end{myexample}

As is well known, the torus duality gives rise to Fourier series (see
Example~\ref{ex:pont-doublet}\ref{it:four-ser}) and the vector duality to Fourier integrals (see
Example~\ref{ex:pont-doublet}\ref{it:four-int}). There is another well-known breed of Fourier
transforms, known as the \emph{discrete Fourier transform} (see
Example~\ref{ex:pont-doublet}\ref{it:four-dft}); it is based on the following important duality.

\begin{myexample}
  \label{ex:finite-group}
  Given any $N \in \NN$, let us introduce two concrete realizations of the \emph{cyclic group} of
  order~$N$. The first is the common representation~$\ZZ_N := \{ 0, 1, \dots, N-1 \}$ with addition
  modulo~$N$, the second
  is~$\Tor_N := N^{-1} \, \ZZ_N = \{ 0, \tfrac{1}{N}, \dots, \tfrac{N-1}{N} \}$ with addition
  modulo~$1$. The duality~$\Tor_N^n \times \ZZ_N^n \to \Tor$ is now defined
  by~$\inner{\xi}{x} := e^{i \tau x \cdot \xi}$.  This can be obtained from the torus duality in two
  steps: First one restricts to the bilinear form~$\inner{}{}\colon\Tor_N^n \times \ZZ^n \to \Tor$
  via the embedding~$\Tor_N^n \hookrightarrow \Tor^n, \xi_i \mapsto e^{i\tau \xi_i}$. Then one
  applies the usual universal construction~\cite[\S I.9]{Lang2002} for making bilinear forms into
  dualities by taking the quotient on the right modulo the right kernel~$(N\ZZ)^n$; nothing is
  needed on the left since the left kernel is trivial. The resulting duality will be called the
  \emph{cyclic duality}~$\inner{\Tor_N^n}{\,\ZZ_N^n}$.

  Of course, one may also form the \emph{conjugate cyclic duality}~$\inner{\ZZ_N^n}{\Tor_N^n}$. But
  unlike their infinite relatives, the dualities~$\inner{\Tor_N^n}{\,\ZZ_N^n}$
  and~$\inner{\ZZ_N^n}{\Tor_N^n}$ are not only similar but actually the \emph{same} (i.e.\@
  isomorphic): Since the cyclic groups~$\Tor_N$ and~$\ZZ_N$ are the same abstract group~$\ZZ/N$,
  both are one and the same duality~$\inner{(\ZZ/N)^n}{\,(\ZZ/N)^n}$,
  given~\cite[Thm.~4.5d]{Folland1994} by
  \begin{equation}
    \label{eq:finite-duality}
    \inner{k + N \ZZ^n}{\,l + N \ZZ^n} = e^{i\tau (k \cdot l)/N}
  \end{equation}
  for~$k, l \in \ZZ^n$. Nevertheless, it can be worthwhile to
  distinguish the two realizations of this duality as we shall see
  when introducing the corresponding Fourier operators in
  Example~\ref{ex:pont-doublet}\ref{it:four-dft}.

  The different nature of~$\Tor_N^n$ and~$\ZZ_N^n$ can also be seen
  in the context of \emph{normalizing Haar measure}~$\mu$. While such
  a choice is \emph{per se} immaterial, it must be consistent between
  the position and momentum group for the inversion theorem to
  hold~\cite[\S1.5.1]{Rudin2017}. For compact groups~$G$, the
  canonical choice is to set~$\mu(G) = 1$, while for discrete
  groups~$G$ one sets~$\mu(\{x\}) = 1$ for all points~$x \in G$. But
  since~$(\ZZ/N)^n$ happens to be both discrete and compact, one must
  decide whether to impose the discrete or the compact normalization
  on the position group~$(\ZZ/N)^n$, so that the other normalization is
  then conferred onto its dual, which is again~$(\ZZ/N)^n$. From the
  above construction, it is clear that the natural choice is to
  endow~$\Tor_N^n$ with the compact and~$\ZZ_N^n$ with the discrete
  normalization. In other words, we have~$\mu(\Tor_N^n) = 1$
  and~$\mu(\ZZ_N^n) = N^n$, thus also~$\mu(\{x\}) = 1/N^n$
  for~$x \in \Tor_N^n$ but~$\mu(\{x\}) = 1$ for~$x \in \ZZ_N^n$.
\end{myexample}

\begin{myexample}
  \label{ex:nich-duality}
  Following the nice approach layed out in~\cite{Nicholson1971}, one
  can generalize the cyclic duality from the group~$\ZZ_N$ to an
  arbitrary finite group~$G$. If the latter has order~$n$ and
  exponent~$m$, one may take for the torus any integral domain~$R$
  that contains~$1/n$ and a primitive $m$-th root of unity. These two
  conditions are shown to be sufficient and necessary
  for~$R[G] \cong R^n$, which will be relevant for the corresponding
  Fourier transform sketched in Example~\ref{ex:four-nich} below. Note
  that one may always take~$R = \Tor$, as is done for the cyclic
  duality.

  Writing $R_*$ for the multiplicative group of~$R$, the above
  conditions imply that~$|\hat{G}| = n$ for the ``dual
  group''~$\hat{G} := \Hom(G, R_*)$ and that the bilinear
  form~$\nu\colon \hat{G} \times G \to R_*$
  with~$\inner{\xi}{x}_\nu := \xi(x)$ is a duality;
  confer~\cite[(3.10)]{Nicholson1971}. We will refer to~$\nu$ as the
  \emph{Nicholson duality} of the group~$G$ over the ring~$R$.
\end{myexample}

Before concluding with the last crucial class of examples for
Heisenberg groups, let us briefly mention a rather degenerate example
that is nevertheless occasionally useful when studying
(counter)examples of various properties that may be ascribed to
Heisenberg groups.

\begin{myexample}
  \label{ex:sym-duality}
  If~$R$ is a ring, its multiplication map may be viewed as a bilinear
  map~$\cdot\colon R_+ \times R_+ \to R_+$, where~$R_+$ is the
  additive group of~$R$. As a bilinear form,
  $\cdot\colon R_+ \times R_+ \to R_+$ is degenerate
  iff~$\Ann_R(R) \neq 0$. In that case we call~$\inner{R_+}{R_+}$ the
  \emph{symmetric duality} of the ring~$R$, with associated Heisenberg
  group~$H(R) := R_+ R_+ \rtimes R_+$. If one is willing to generalize
  the setting of Example~\ref{ex:vecgrp} from $F$-vector spaces to
  $R$-modules, the symmetric duality of~$R$ is the ``vector duality''
  on the one-dimensional module~$R^1$ under the trivial (non-standard)
  character $\chi\colon R_+ \to R_+, x \mapsto x$.
\end{myexample}

The one-dimensional cyclic duality~$\Tor_N \times \ZZ_N \to \Tor$ is the multiplication
map~$\ZZ_N \times \ZZ_N \to \ZZ_N$, transported from the additive value
group~$\ZZ_N \cong N^{-1} \ZZ_N = \{ 0, \frac{1}{N}, \dots, \frac{N-1}{N} \} \subset \RR$ into the
multiplicative group~$\Tor$ via the standard
character~$\chi\colon \RR \to \Tor, t \mapsto e^{i\tau t}$ of
Example~\ref{ex:classical-vector-group}. It is interesting to see what one gets for the plain
multiplication map~$\ZZ_N \times \ZZ_N \to \ZZ_N$. In the following example, we explore this
question for the special case~$N=2$.

\begin{myexample}
  \label{ex:dihedral-group}
  So we appply the construction~$R_+ R_+ \rtimes R_+$ to the
  ring $R = \ZZ_2$, corresponding to the
  duality~$\beta\colon \ZZ_2 \times \ZZ_2 \to \ZZ_2$ defined
  by $\beta(m,n) = mn$.  True to our conventions, we write the
  torus multiplicatively via
  $\ZZ_2 \isomarrow \ZZ^\times, [c] \mapsto (-1)^c$,
  meaning~$[0] \leftrightarrow +1$, $[1] \leftrightarrow -1$. As
  usual, we often express the values~$\pm 1$ just by the sign. The
  duality is then given by~$\beta(m,n) = (-1)^{mn}$. We will show
  that~$H(\beta)$ is the dihedral group~$D_4$, the symmetry group of
  the square (which we assume centered in the origin with
  axis-paraellel sides). If~$t$ denotes the counter-clockwise
  $90^\circ$ turn and~$r$ the reflection in the vertical axis, we
  obtain the presentation
  \begin{equation*}
    D_4 = \langle t, r \mid t^4 = r^2 = 1, rt = t^3 r \rangle
    = \{ 1, t, t^2, t^3, r, tr, t^2r, t^3r \},
  \end{equation*}
  Where~$tr, t^2r, t^3r$ may be respectively interpreted as reflections in the
  anti-diagonal~$x+y=0$, horizontal~$y=0$, and diagonal~$x-y=0$. We choose~$Z(D_4) = \{ 1, t^2 \}$
  as our torus~$T = \ZZ_2$, which enforces the identification~$1 \leftrightarrow +1$,
  $t^2 \leftrightarrow -1$. For the position group~$G = \ZZ_2$ we take~$\{ 1, r \}$, leading to the
  identification~$1 \leftrightarrow [0]$, $r \leftrightarrow [1]$; for the momentum
  group~$\Gamma = \ZZ_2$ we use~$\{ 1, tr \}$ with identification~$1 \leftrightarrow [0]$,
  $tr \leftrightarrow [1]$. In these terms, the
  duality~$\beta\colon \ZZ_2 \times \ZZ_2 \to \ZZ_2, (m,n) \mapsto mn$ sends~$(r, tr)$ to~$-1$ and
  all other pairs to~$+1$.

  \medskip
  \noindent\begin{tabular}[h]{||l|l||l|l||}
    \hline
    $D_4$ & $H(\beta)$ & $D_4$ & $H(\beta)$\\\hline
    $1$ & $+00$ & $r$ & $+10$\\
    $t$ & $-11$ & $tr$ & $+01$\\
    $t^2$ & $-00$ & $t^2r$ & $-10$\\
    $t^3$ & $+11$ & $t^3r$ & $-01$\\\hline
  \end{tabular}
  \medskip

  \noindent For defining a group isomorphism~$D_4 \isomarrow H(\beta)$, we construct first the
  unique homomorphism on the free group with~$1 \mapsto +00$, $t \mapsto -11$, $r \mapsto +10$. As
  one sees immediately, this homomorphism annihilates the relators~$t^4, r^2, trtr$ and thus yields
  a homomorphism~$\iota\colon D_4 \to H(\beta)$. Computing all other elements in terms of the
  generators~$t$ and~$r$, one will verify that~$\iota$ is given as in the table above. Since this is
  obviously a bijection, it provides us with the required isomorphism~$D_4 \isomarrow H(\beta)$.
\end{myexample}

For the standard vector duality~$\inner{\RR^n}{\RR^n}$ of
Example~\ref{ex:classical-vector-group}, the link to the matrix
group~\eqref{eq:heis-matrix} can be made more precise by reframing
the Heisenberg group in terms of a \emph{symplectic vector
  space}~$V$. This is the formulation commonly used in more advanced
treatments of the Heisenberg group~\cite[\S5.1]{BinzPods2008},
\cite{Folland2016}, \cite{LibermannMarle2012}.

\begin{myexample}
  \label{ex:heisgrp-sympl}
  As we have seen above (Lemma~\ref{lem:sympl-heisgrp}), all Heisenberg extensions associated with a
  fixed commutator form can be represented by the so-called symplectic Heisenberg group, provided
  the torus is uniquely $2$-divisible. The most important instance arises for a finite-dimensional
  symplectic vector space~$(Z, \Omega)$ over a field~$F$. Here one takes the additive group~$(F, +)$
  as torus and~$\Omega\colon Z \times Z \to F$ as commutator form to yield the \emph{symplectic
    Heisenberg group}~$[H]_{\Omega/2}$ via the nilquadratic
  extension~$F \rightarrowtail [H]_{\Omega/2} \twoheadrightarrow Z$. Note that we write
  here~$\Omega/2$ for the ``square root'' of the commutator form since the torus is written
  additively.  Taking~$Z = \RR^{2n}$ with its canonical symplectic form, this agrees
  with~\cite[p.~19]{Folland2016}, where~$[H]_\Omega$ is written~$\mathbf{H}_n$. Dissociated from the
  underyling symplectic structure, the factor~$1/2$ in~$[H]_{\Omega/2}$ may look arbitrary (it can
  obviously be eliminated by rescaling), but it also ensures that exponentiation yields the
  canonical commutators of Hamiltonian mechanics~\cite[(1.15)]{Folland2016}.

  We have already mentioned earlier that the symplectic Heisenberg
  group is not a Heisenberg group \emph{stricto sensu} (as per
  Definition~\ref{def:abs-heisgrp}). But it is closely related to a
  whole bunch of them via \emph{polarization}. For a symplectic vector
  space~$(Z, \Omega)$, a polarization by itself is just a choice of
  Lagrangian subspace~$G \le Z$, so that a Lagrangian
  splitting~$(G, \Gamma)$ of~$Z$ could be viewed as a ``duplex
  polarization'' (noting that any Lagrangian bisection is a splitting
  for vector spaces). The polarizations~$G$ make up the so-called
  Lagrange-Grassmann manifold, where each~$G$ has plenty of Lagrangian
  complements~$\Gamma$;
  see~\cite[Prop.~A6.1.6]{LibermannMarle2012}. Any choice of
  Lagrangian splitting~$(G, \Gamma)$ induces a natural symplectic
  isomorphism~$\iota_{G,\Gamma}\colon Z \isomarrow G \oplus G^*$
  fixing~$G$ and sending~$\xi \in \Gamma$ to~$\Omega(\xi, -) \in G^*$;
  here injectivity follows from the Lagrangian nature of~$G$ while
  surjectivity needs the finite dimensionality of~$Z$.
  Here~$G \oplus G^*$ has the \emph{canonical symplectic structure}
  given by
  $\Omega_G(x, \xi; \tilde x, \tilde\xi) = \funcinner{\xi}{\tilde x} -
  \funcinner{\tilde\xi}{x}$, where~$\funcinner{-}{-}$ denotes the
  natural pairing on~$G^* \times G$. Via~$\iota_{G,\Gamma}$,
  symplectic bases of~$Z$ are in bijective correspondence with bases
  of~$G$.

  As mentioned after Definition~\ref{def:assoc-heisgrp}, each Lagrangian splitting~$(G, \Gamma)$
  of~$Z$ yields a unique \emph{Heisenberg group}~$H(\beta_\Omega)$; its associated bilinear
  form~$\beta_\Omega\colon \Gamma \times G \to F$ is given by restriction
  of~$\Omega\colon Z \times Z \to F$. In the standard case $Z = \RR^{2n} = \RR^n \oplus \RR^n$ one
  has~$\beta_\Omega(\xi, x) = \xi \cdot x$, and the Heisenberg group~$H(\beta_\Omega)$ agrees
  with~$\mathbf{H}_n^{\mathrm{pol}}$ in~\cite[p.~19]{Folland2016}. Hence the matrix multiplication
  law~\eqref{eq:heis-matrix} applies literally, yielding the familiar matrix Lie group. By
  Lemma~\ref{lem:sympl-heisgrp}, it is iso\-morphic to the symplectic Heisenberg group via the
  polarization map $\mathfrak{P}_{G,\Gamma}\colon [H]_\Omega \isomarrow H(\beta_\Omega)$
  with~$(\lambda; x, \xi) \mapsto (\lambda + \xi \cdot x/2) \, (x, \xi)$.

  If a standard character~$\chi\colon (F, +) \to (T, \cdot)$ is
  available, one can construct yet another version of
  both~$[H]_\Omega$ and~$H(\beta_\Omega)$ by pushing forward the
  $F$-valued commutator form~$\Omega$ to its $T$-valued
  cousin~$\omega := \chi \circ \Omega$. Even if~$T$ is not (uniquely)
  $2$-divisible, the commutator form~$\omega$ has a \emph{standard
    square root},
  namely~$\sqrt{\omega} = \chi \circ \tfrac{1}{2} \circ \Omega$,
  where~$\tfrac{1}{2}\colon F \to F$ is division by two (so the
  torus~$T$ is supposed to have characteristic different from~$2$).
  In the classical setting~$F = \CC$, the
  torus~$T = \Tor \subset \CC$ is $2$-divisible but not uniquely so;
  one may here use the standard character $\chi(t) = e^{i\tau t}$,
  which is tantamount to choosing the principal branch of the
  multivalued square-root function.

  In fact, the strategy of forging square roots from the halving
  isomorphism~$\tfrac{1}{2}\colon F \to F$ may be generalized as follows
  by~\heiscite{prop:skewing-section-2div}. If one has an (additively written) uniquely $2$-divisible
  torus~$(\hat{T}, +)$ above the torus~$(T, \cdot)$ via~$\chi\colon \hat{T} \twoheadrightarrow T$,
  the pushforward $\chi_*\colon \Omega^2(P,\hat{T}) \to \Omega^2(P,T)$ induces a \emph{section of
    the skewing map~$q$} in the short exact sequence~\eqref{eq:H2-seq}
  over~$\Omega^2_{\chi}(P,T) := \im(\chi_*) \le \Omega^2(P,T)$, namely the
  map~$\chi_*(\Omega) \mapsto \chi_*(\Omega/2)$. Lemma~\ref{lem:sympl-heisgrp} can be generalized to
  yield an isomorphism $[H]_\omega \isomarrow H(\beta)$ for~$\omega \in \Omega^2_{\chi}(P,T)$.

  Back to vector spaces, we can now construct the \emph{little
    symplectic Heisenberg
    group}~$[H]_\omega = T \times_{\sqrt{\omega}} Z$, along with the
  corresponding \emph{little Heisenberg group}~$H(\beta_\omega)$. The
  latter has the bilinear form~$\beta_\omega = \chi \circ \Omega$,
  which reads~$\inner{x}{\xi} = \chi \funcinner{x}{\xi}$ in our usual
  bracket notation. In terms of the
  projections~$\pi_\Gamma\colon Z \to \Gamma$
  and~$\pi_G\colon Z \to G$, we
  have~$\beta_\omega = \omega \circ (\pi_\Gamma \times \pi_G)$ as well
  as~$\beta_\Omega = \Omega \circ (\pi_\Gamma \times \pi_G)$ for the
  big symplectic Heisenberg group
  $[H]_\Omega = F \times_{\Omega/2} Z$. From the generalized
  Lemma~\ref{lem:sympl-heisgrp} we obtain the little polarization
  map~$\mathfrak{p}_{G,\Gamma}\colon [H]_\omega \isomarrow
  H(\beta_\omega)$ with
  $\mathfrak{p}_{G,\Gamma}(c; x, \xi) = c \,
  \sqrt{\inner{\xi}{x}}$. The canonical
  epimorphism~$\pi_Z := \chi \times 1_Z\colon [H]_\Omega
  \twoheadrightarrow [H]_\omega$ has kernel
  $\ker(\chi) \times 0 \cong \ZZ_*$, so the little symplectic
  Heisenberg group is the big one modulo the prime ring; a similar
  statement holds for the Heisenberg groups via the
  (set-theoretically) same
  map~$\pi_Z\colon H(\beta_\Omega) \to H(\beta_\omega)$.
  \begin{figure}[h]
    \label{fig:heisgrp-sympl}
    \begin{equation*}
      \xymatrix @M=0.5pc @R=1pc @C=2pc%
      {   [H]_\Omega \ar[r]^{\mathfrak{P}_{G,\Gamma}} \ar[d]_{\pi_Z} & H(\beta_\Omega) \ar[d]^{\pi_Z}\\
        [H]_\omega \ar[r]^{\mathfrak{p}_{G,\Gamma}} & H(\beta_\omega) }
    \end{equation*}
    \caption{Symplectic Heisenberg Group}
    \label{fig:syml-heisgrp}
  \end{figure}

  The ``Heisenberg group'' given
  in~\cite[Exc.~5.1-4]{AbrahamMarsdenRatiu1983} is
  essentially~$[H]_\omega$ for the symplectic space~$Z = \RR^{2n}$
  with standard character~$\chi\colon \RR \to \Tor$ as above, except
  that they shun the factor~$1/2$ so effectively rescale~$\Omega$ by
  the factor~$2$. As a consequence of this, they incur the factor~$2$
  in the canonical
  commutator~\cite[Exc.~5.1-4b]{AbrahamMarsdenRatiu1983}.

  We have given references for all but one of the Heisenberg flavors
  in Figure~\ref{fig:syml-heisgrp}. The remaining
  group~$H(\beta_\omega)$ coincides with the Heisenberg group of the
  abstract vector duality~$\inner{V}{V}$ for~$V = G = \Gamma$ and the
  chosen standard character~$\chi\colon F \to T$, where as
  before~$\funcinner{\xi}{x} = \beta_\Omega(\xi, x)$
  with~$\beta_\Omega = \Omega \circ (\pi_\Gamma \times \pi_G)$. The
  general case of vector dualities is treated in
  Example~\ref{ex:vecgrp}. In the important special case~$V = \RR^n$,
  this yields the standard vector duality~$\inner{\RR^n}{\RR^n}$ or
  isomorphically~$\inner{\Gamma}{G} = \inner{\RR_n}{\RR^n}$; see
  Example~\ref{ex:classical-vector-group}. It is this case which leads
  to the famous classical Fourier integral, as we shall see
  in~Example~\ref{ex:pont-doublet}\ref{it:four-int} below.
\end{myexample}

\begin{myremark}
  \label{rem:physint-heisgrp}
  Let us add a few remarks on the \emph{physical interpretation} of
  Example~\ref{ex:heisgrp-sympl}. The symplectic vector space~$T^*V = V \oplus V^*$ is nothing but
  the Hamiltonian phase space~\cite[\S1.1]{MarsdenRatiu1994}, whereas the abstract vector
  duality~$\inner{V}{V}$ of Example~\ref{ex:vecgrp} is linked to the Lagrangian phase
  space~$TV = V \oplus V$. In both cases, the elements of~$V$ denote positions while the
  tangent/cotangent vectors are the \emph{velocities/momenta}. This generalizes to the nonlinear
  case where the configuration space is a manifold~$M$ rather than a vector space~$V$, with
  Lagrangian phase space~$TM$ and Hamiltonian phase space~$T^*M$.

  The Hamiltonian case is important since it leads to
  \emph{quantization} via replacing the commutative algebra of
  classical observables~$C^\infty(T^*V)$ by the noncommutative
  algebra~$\mathcal{H}\big(L^2(V)\big)$ of self-adjoint Hermitian
  operators on the Hilbert space~$L^2(V)$. The observables position
  and momentum are quantized~\cite[\S3.5]{Hall2013} as position
  operator~$f(x) \mapsto x \, f(x)$ and canonically conjugate momentum
  operator~$f(x) \mapsto \der f/\der x$). This is intimately linked to
  the unitary irrep (= irreducible representation) of the Heisenberg
  group (see Remark~\ref{rem:pmech} below).
\end{myremark}

In closing this short investigation of Heisenberg groups, let us corroborate our earlier claim about
the \emph{failure of existence/uniqueness of abelian splittings} for general nilquadratic
groups. Regarding existence, we refer to Example~\heiscite{ex:free-nil}, which makes it clear that
the free nilquadratic group~$N_3$ fails to have an abelian splitting.

\begin{myexample}
  \label{ex:fail-ex-uniq}
  For seeing that one and the same nilquadratic group may be equipped with \emph{different abelian
    splittings}, take the standard vector duality~$\beta = \inner{\RR^1}{\RR^1}$ of
  Example~\ref{ex:classical-vector-group}. The corresponding Heisenberg
  group~$H(\beta) = \Tor \RR \rtimes \RR$ with the classical torus~$\Tor \subset \nnz{\CC}$ and
  phase space~$\RR^2$ comes endowed with the standard Lagrangian
  splitting~$(\RR \times 0, 0 \times \RR)$. But it is easy to see that any other direct
  decomposition~$G \dotplus \Gamma = \RR^2$ with one-dimensional subspaces~$G, \Gamma \le \RR^2$
  also yields a Lagrangian splitting~$(G, \Gamma)$, which is however distinct from the standard
  one. Since Lagrangian splittings of~$P$ are in bijective correspondence with abelian splittings
  of~$H$ by Theorem~\heiscite{thm:sympl-corr}, this establishes that the former are not uniquely
  determined for the given Heisenberg
  extension~$\Tor \oset{\iota}{\rightarrowtail} H \oset{\pi}{\twoheadrightarrow} \RR^2$.
\end{myexample}

\subsection{The Category of Heisenberg Algebras.}\label{sub:cat-heisalg}
Analysis deals with Fourier operators on real or complex function spaces, which are endowed with
certain \emph{Heisenberg actions} (which we define here as linear representations of Heisenberg
groups). Our task in this subsection is to capture this idea in a suitable algebraic setting.

A Heisenberg group~$H(\beta)$ typically comes with an action on some function space on which the
torus~$T \le H(\beta)$ ``acts naturally via scalars'' (see Example~\ref{ex:pont-doublet} below for
classical cases). For making this precise, we view the function spaces as modules or algebras over a
fixed scalar ring~$K$, and the latter equipped with a \emph{torus
  action}~$\ast\colon T \to \Aut_K(K)$. Thus we have the action laws $1 \ast \lambda = \lambda$ and
$(cd) \ast \lambda = c \ast (d \ast \lambda)$ as well as
linearity~$c \ast (\lambda + \mu) = c \ast \lambda + c \ast \mu$ and
$c \ast (\lambda \mu) = (c \ast \lambda) \mu = \lambda (c \ast \mu)$, for all~$c, d \in T$
and~$\lambda, \mu \in K$. Equivalently, the torus action~$\ast$ may be described by the
map~$\epsilon_T\colon (T, \cdot) \to (\nnz{K}, \cdot)$ with~$\epsilon_T(c) := c \ast 1_K$ since we
have~$c \ast \lambda = \epsilon_T(c) \, \lambda$. In the sequel, we shall refer to both~$\ast$
and~$\epsilon_T$ as a torus action.

Any $K$-module~$S$ is naturally a $T$-module since,
if~$\Delta\colon K \to \End_\ZZ(S)$ is the structure map of~$S$ with
induced action~$\nnz{\Delta}\colon \nnz{K} \to \Aut_K(S)$, we obtain
an action~$\Delta_T$ of~$T$ on the abelian group~$(S, +)$
by~$\Delta_T = \nnz{\Delta} \circ \epsilon_T$. We can now give a
precise meaning to the informal phrase used in the previous paragraph:
If~$\eta\colon H(\beta) \to \Aut_K(S)$ is any $K$-linear action, we
say the \emph{torus acts naturally via~$\epsilon_T$} if~$\eta$ factors
through~$\Delta_T$. In other words, we require the following diagram
to commute:

\begin{equation*}
  \xymatrix @M=0.75pc @H=1pc @R=1pc @C=2pc%
  { T \ar[r]^{\epsilon_T} \ar@{^{(}->}@<-0.5ex>[d] & \nnz{K} \ar[d]^-{\nnz{\Delta}}\\
    H(\beta) \ar[r]_{\eta} & \Aut_K(S) }
\end{equation*}

If we have a Heisenberg action on a $K$-algebra~$(S, +, \diamond)$, we shall always designate the
action of~$u \in H(\beta)$ on an element~$s \in S$ by~$u \act s$, thus avoiding confusion with the
pointwise product~$\cdot$ to be introcued later. In relation to Fourier structures, we encounter two
crucially distinct flavors of elements in~$H(\beta)$ in relation to the multiplicative
structure\footnote{Later on, we will consider mainly two multiplicative structures given defined by
  a Pontryagin duality---the convolution~$\star$ and the pointwise product~$\cdot$ just mentioned.}
of~$S$:
\begin{itemize}
\item We call~$u \in H(\beta)$ a \emph{Heisenberg scalar}
  if~$u \act (s \diamond \tilde{s}) = (u \act s) \diamond \tilde{s}$ for
  all~$s, \tilde{s} \in S$.
\item We call~$u \in H(\beta)$ a \emph{Heisenberg operator}
  if~$u \act (s \diamond \tilde{s}) = (u \act s) \diamond (u \act \tilde{s})$ for
  all~$s, \tilde{s} \in S$.
\end{itemize}
When the torus acts naturally via~$\epsilon_T$, the torus elements are
clearly Heisenberg scalars.

We are now in a position to give a concise algebraic description of
the function spaces underlying Fourier transforms. Following the
terminology of~\cite{Walters2004} and~\cite{Rieffel1988}, we call such
spaces \emph{Heisenberg modules} since they are modules over some
Heisenberg group~$H(\beta)$. In a natural---though less
conventional---extension, we shall speak of a \emph{Heisenberg
  algebra} if the module is additionally equipped with a compatible
multiplication.\footnote{This should not be confused with the
  classical Lie algebra of~$H_n(\RR)$,
  properly~\cite[p.~18]{Folland2016} called ``Heisenberg Lie
  algebra'', but occasionally~\cite[p.~18, 21]{Folland2016} shortened
  to ``Heisenberg algebra''.}

\begin{definition}
  \label{def:heis-alg}
  Let~$\beta\colon \Gamma \times G \to T$ be a duality, and let~$K$ be
  a ring with torus action~$\epsilon_T$.
  \begin{itemize}
  \item A \emph{Heisenberg module} over~$\beta$ is a $K$-module~$S$
    with a linear action~$H(\beta) \times S \to S$ where the torus
    acts naturally via~$\epsilon_T$.
  \item A \emph{Heisenberg algebra} over~$\beta$ is a $K$-algebra and a Heisenberg module where all
    elements of $TG \le H(\beta)$ are Heisenberg scalars while all elements of $\Gamma \le H(\beta)$
    are Heisenberg operators.
  \end{itemize}
  A \emph{Heisenberg morphism} from a Heisenberg module/algebra~$S$
  over~$\beta$ to another Heisenberg module/algebra $S'$ over~$\beta$
  is defined to be an equivariant $K$-module/algebra
  homomorphism~$S \to S'$.
\end{definition}

For avoiding tedious repetitions, we shall from now on regard
\emph{Heisenberg modules as degenerate Heisenberg algebras}, in the
sense of having trivial multiplication. We extend also the
scalar/operator terminology to Heisenberg groups~$H(\beta)$ per se,
regardless of any Heisenbeg algebras on which they might act (dropping
the qualification ``Heisenberg'' where it is obvious): Elements
of~$TG \le H(\beta)$ will be called \emph{scalars}, elements
of~$\Gamma \le H(\beta)$ \emph{operators}, and generic elements
\emph{actors}. (This will not lead to any confusion since we shall
never encounter algebras that are Heisenberg modules but not
Heisenberg algebras.)

Let us spell out the definition of Heisenberg algebras in more detail,
splitting the Heisenberg action into the three
actions~$G \times S \to S$ and $\Gamma \times S \to S$ and
$T \times S \to S$, while writing multiplication in the algebra as
juxtaposition. In addition to the $K$-algebra axioms, for
arbitrary~$s, s' \in S$ and~$c \in T$ and~$x, x' \in G$
and~$\xi, \xi' \in \Gamma$, we require the \emph{following
  identities}:

\medskip
\begingroup
  \centering
  \renewcommand{\baselinestretch}{1.1}
  \def\leqcol#1{\begin{minipage}{0.49\textwidth}\begin{equation}#1\end{equation}\end{minipage}}
  \def\reqcol#1{\begin{minipage}{0.49\textwidth}\begin{equation}#1\end{equation}\end{minipage}}
  \def\beqcol#1{\multicolumn{2}{@{\hspace{0.255\textwidth}}p{.49\textwidth}}{%
      \begin{minipage}{0.49\textwidth}%
        \begin{equation}#1\end{equation}\end{minipage}}}
  \makeatletter\setbool{@fleqn}{true}\makeatother
  \begin{tabular}{@{\hspace{0.05\textwidth}}p{.46\textwidth}p{.52\textwidth}}
    \leqcol{1_G \act s = s\tag{H$_1$}\label{eq:left-unit}} &
    \reqcol{1_\Gamma \act s = s\tag{H$_2$}\label{eq:right-unit}}\\
    \leqcol{(xx') \act s = x \act (x' \act s)\tag{H$_3$}\label{eq:left-assoc}} &
    \reqcol{(\xi \xi') \act s = \xi \act (\xi' \act s)\tag{H$_4$}\label{eq:right-assoc}}\\
    \leqcol{c \act s = \epsilon_T(c) \, s\tag{H$_5$}\label{eq:torus-act}} &
    \reqcol{\xi \act (x \act s) = \inner{\xi}{x} \; x \act (\xi \act s)\tag{H$_6$}\label{eq:twisted-bimod}}\\
    \leqcol{x \act (ss') = (x \act s) s'\tag{H$_7$}\label{eq:left-scal}} &
    \reqcol{\xi \act (ss') = (\xi \act s) (\xi \act s')\tag{H$_8$}\label{eq:right-op}}\smallskip
  \end{tabular}
\endgroup
\bigskip

Dropping the last two axioms (and requiring~$S$ to be a $K$-module instead of a $K$-algebra), one
obtains a Heisenberg module instead of a Heisenberg algebra over~$\beta$. It is easy to check that
the above identities are \emph{equivalent} to the requirements of Definition~\ref{def:heis-alg}. For
example, the \emph{twist axiom}~\eqref{eq:twisted-bimod} is a consequence of the composition law
in~$H(\beta)$. The Heisenberg action may be decomposed into the ``scalar action'' of~$TG$ and the
``operator action'' of~$\Gamma$; taking this view, $S$ is a twisted bimodule
under~\eqref{eq:twisted-bimod}. On another view, Heisenberg modules over~$\beta$ are linear
representations of~$H(\beta)$, with Heisenberg morphisms as intertwiners.

At any rate, the \emph{category of Heisenberg modules} over~$\beta$ is denoted by~$\ModH\beta$, and
the \emph{category of Heisenberg algebras} by~$\AlgH\beta$. By our convention of regarding
Heisenberg modules as degnerate Heisenberg algebras, $\ModH{\beta}$ is a full subcategory
of~$\AlgH{\beta}$. The assignments $\beta \mapsto \ModH\beta$ and~$\beta \mapsto \AlgH\beta$ may be
seen as contravariant functors $\ModH{}\colon \Du \to \Cat$ and~$\AlgH{}\colon \Du \to \Cat$ in the
following way: If $(g \times \gamma, t)$ is a morphism between dualities
$\beta\colon G \times \Gamma \to T$ and $\beta'\colon G' \times \Gamma' \to T'$, the functor
$\AlgH{}(g \times \gamma, t)\colon \AlgH{\beta'} \to \AlgH\beta$ is defined as follows.
\begin{itemize}
\item On objects, it acts by sending $S' \in \AlgH{\beta'}$ to the
  twisted module~$S := S'[\sigma] \in \AlgH\beta$, with the twist map
  given
  by~$\sigma := H(g \times \gamma, t)\colon H(\beta) \to H(\beta')$.
\item On morphisms, the functor~$\AlgH{}(g \times \gamma, t)$ acts trivially: A
  morphism~$\phi'\colon S_1' \to S_2'$ in~$\AlgH{\beta'}$ stays the same since~$\phi=\phi'$ respects
  the $\sigma$-twisted action of~$H(\beta)$.
\end{itemize}
It is easy to see that $\AlgH{}\colon \Du \to \Cat$ is then indeed a contravariant functor, and the
same holds for~$\ModH{}\colon \Du \to \Cat$.

By the well-known Grothendieck construction~\cite[\S12.2]{BarrWells1995}, the functor
$\AlgH{}\colon \Du \to \Cat$ yields a \emph{split
  fibration}~$\pi\colon \AlgH{} \rtimes \Du \to \Du$, where the category~$\AlgH{} \rtimes \Du$ has
the objects~$(\beta, S)$ with~$\beta \in \Du$, $S \in \AlgH\beta$, and
morphisms~$(\phi, \Phi)\colon (\beta, S) \to (\beta', S')$ with~$\phi\colon \beta \to \beta'$
in~$\Du$ and $\Phi\colon S \to \AlgH{}(\phi) \, S'$ in~$\AlgH{\beta}$.
Writing~$\phi = (g \times \gamma, t)$ for suitable
morphisms~$g \times \gamma\colon G \times \Gamma \to G' \times \Gamma'$ and~$t\colon T \to T'$, this
requires the conditions~$\Phi(x \act s) = g(x) \act \Phi(s)$
and~$\Phi(\xi \act s) = \gamma(\xi) \act \Phi(s)$ as well
as~$\Phi(c \act s) = t(c) \act \Phi(s)$ for $s \in S$ and~$c \, (x,\xi) \in H(\beta)$. Composition
of morphisms is defined by
\begin{equation*}
  (\phi_2, \Phi_2) \circ (\phi_1, \Phi_1) = (\phi_2 \circ \phi_1, \AlgH{}(\phi_1) \, \Phi_2 \circ
  \Phi_1),
\end{equation*}
as usual for semidirect products. As a result, the \emph{fibered category}
$\AlgH{\blnk} := \AlgH{} \rtimes \Du$ of Heisenberg algebras is a disjoint union
\begin{equation*}
  \AlgH{\blnk} = \biguplus_{\beta \in \Du} \AlgH\beta
\end{equation*}
similar to the well-known fibration~$\mathbf{Alg}_{\blnk} = \biguplus\limits_{R \in \Rng} \Alg_R$.

\smallskip

As noted above, $\AlgH\beta$ includes~$\ModH\beta$. It is clear that the
category~$\ModH{\blnk} := \ModH{} \rtimes \Du$ of \emph{Heisenberg modules} has an analogous
fibration over~$\Du$, similar to the fibration of~$\Mod{}$ over~$\Rng$.

The category~$\AlgH\beta$ has \emph{products}, namely the direct product of $K$-algebras with
componentwise Heisenberg action. The commutative diagram for products carries over from~$\Alg_K$
to~$\AlgH\beta$ since the projection maps~$\pi_1\colon S_1\times S_2 \to S_1$
and~$\pi_2\colon S_1\times S_2 \to S_2$ are Heisenberg morphisms.

The category~$\AlgH\beta$ also has a \emph{tensor product}. Given Heisenberg
algebras~$S, S' \in \AlgH\beta$, we view them as $KG$-algebras and endow
$S \otimes_{KG} S' \in \Alg_{KG}$ with an action of~$H(\beta) = TG \rtimes \Gamma$ as follows: The
action of~$T$ is via~$\epsilon_T$, that of~$G$ via the $KG$-module structure, while Heisenberg
operators~$\xi \in \Gamma$ act via
\begin{equation*}
  \xi \act s \otimes s' := (\xi \act s) \otimes (\xi \act s') .
\end{equation*}
One checks immediately that we have in fact~$S \otimes_{KG} S' \in \AlgH\beta$. For brevity, we
refer to this Heisenberg algebra as $S \otimes S'$.

\begin{theorem}
  \label{thm:algh-monoidal}
  Let~$\beta$ be any duality. Then~$\AlgH\beta$ is a symmetric monoidal category.
\end{theorem}
\begin{proof}
  It is clear that~$\AlgH\beta$ is a monoidal category since the usual
  associators~$\alpha_{ABC}\colon (A \otimes B) \otimes C \to A \otimes (B \otimes C)$ as well as
  the unitors $\lambda_A\colon (KG) \otimes A \to A$ and~$\rho_A\colon A \otimes (KG) \to A$ of the
  category~$\Alg_{KG}$ are easily seen to be Heisenberg isomorphisms. For the unitors, it should be
  noted that~$KG$ is a Heisenberg algebra with action defined via~$\epsilon_T$
  and~$\xi \act y := \inner{y}{\xi} y$. Finally, $\AlgH\beta$ is a symmetric monoidal
  category since the braiding~$\gamma_{AB}\colon A \otimes B \to B \otimes A$ of~$\Alg_{KG}$ is also
  a Heisenberg isomorphism.
\end{proof}

Unlike in $\Alg_K$, the tensor product in~$\AlgH\beta$ is \emph{not a coproduct}, however. The
problem is that Heisenberg algebras are typically without unit element so that the tensor product
generally lacks the injection maps~$S \to S \otimes S'$ and~$S' \to S \otimes S'$ required of
coproducts.

This may be remedied by taking recourse to the \emph{coproduct of commutative nonunital
  algebras}. The construction is straightforward but difficult to locate in the literature. If~$A$
and~$B$ are commutative nonunital algebras over a commutative unital ring~$R$, their coproduct is
defined by~$A \hotimes_R B := A \oplus B \oplus (A \otimes_R B)$ with multiplication given by
setting~$(a_1, b_1, a_1'\otimes b_1') \cdot (a_2, b_2, a_2'\otimes b_2')$ equal to
\begin{equation}
  \label{eq:heisalg-coprod}
  \begin{aligned}
    (a_1a_2, b_1b_2, a_1\otimes b_2 & + a_2\otimes b_1 + a_1a_2'\otimes b_2' + a_1'a_2\otimes b_1'\\
    & + a_2'\otimes b_1b_2' + a_1'\otimes b_1'b_2 + a_1'a_2'\otimes b_1'b_2') .
  \end{aligned}
\end{equation}
Then we have the coproduct injections $\iota_1\colon A \to A \hotimes_R B$, $a\mapsto(a,0,0)$ and
$\iota_2\colon B \to A \hotimes_R B$, $b\mapsto(0,b,0)$.

Using this structure, it is straightforward to define the coproduct of~$S, S' \in \AlgH\beta$.
Viewing~$S$ and~$S'$ again as $KG$-algebras, we
equip~$S \hotimes_{KG} S' = S \oplus S' \oplus (S \otimes_{KG} S')$ with the componentwise
Heisenberg action induced by those on~$S$, $S'$ and~$S \otimes_{KG} S'$, where the latter is the
Heisenberg module~$S \otimes S'$ introduced above. It is easy to see
that~$\iota_1\colon S \to S \hotimes_{KG} S'$ and~$\iota_2\colon S' \to S \hotimes_{KG} S'$ are
Heisenberg morphisms, hence~$S \hotimes_{KG} S'$ is indeed a coproduct in~$\AlgH\beta$, which we
also abbreviate by~$S \hotimes S'$.

A Heisenberg algebra is a Heisenberg module whose structure is
enriched by tacking on a nontrivial product (replacing the trivial
default product). In Fourier theory (see
\S\ref{sec:constructive-schwartz}), one encounters modules with
an even richer structure---modules that carry \emph{two} products. To
make things precise, let us call~$(A, \star, \cdot)$ a \emph{twain
  algebra} over the commutative and unital ring~$K$ if~$A$ is a
$K$-module with two bilinear products~$\star$ and~$\cdot$. The
corresponding algebras~$(A, \star)$ and~$(A, \cdot)$ will be denoted
by~$A_\star$ and~$A_{\tdot}$. If one of the products is trivial, we
have a \emph{plain algebra}---if both are trivial, we retain the naked
$K$-module, which we might then call a \emph{slain algebra}. Clearly,
there are corresponding notions of twain/plain/slain homomorphisms,
depending on how many multiplication maps need to be preserved.

A twain algebra over~$K$ may be described by giving two $K$-algebras $(A, \star)$ and~$(B, \cdot)$
together with a $K$-linear isomorphism~$\iota\colon A \isomarrow B$. This yields the twain
algebra~$(A, \star, \cdot)$, where we write its transferred product
$x \cdot y := \iota^{-1}(\iota x \cdot \iota y)$ for~$x, y \in A$ with the same symbol. In such a
case we shall use the notation~$A \divideontimes_\iota B$ for the corresponding twain algebra
(suppressing the subscript~$\iota$ when the isomorphism is understood), which we call the
\emph{overlay of~$A$ on top of~$B$}.

We shall also employ the casual twain/plain/slain jargon in conjunction with Heisenberg actions
of~$H(\beta) = TG \rtimes \Gamma$. To this end, let us recall that a Heisenberg module over~$\beta$
was called a Heisenberg algebra over~$\beta$ if it is endowed with a bilinear multiplication such
that the positions~$x \in G \le H(\beta)$ act as scalars while the
momenta~$\xi \in \Gamma \le H(\beta)$ act as operators (of course the~$c \in T \le H(\beta)$ always
act as scalars). Borrowing typographic terminology, we shall also call such a structure a
\emph{recto Heisenberg algebra}. In contrast, a Heisenberg module over~$\beta$ will be called a
\emph{verso Heisenberg algebra} if it has a bilinear multiplication where the positions are
operators and the moments scalars. (In the sequel, the qualification ``recto'' will only be used
for emphasis or symmetry.)

Now a twain algebra~$A$ is called a \emph{Heisenberg twain algebra}
if~$A_\star$ is a recto Heisenberg algebra while~$A_{\tdot}$ is a
verso Heisenberg algebra. Of course, the terms \emph{Heisenberg plain
  algebra} and \emph{Heisenberg slain algebra} are just synonyms for
``Heisenberg module'' and ``Heisenberg algebra'', respectively. This
somewhat flippant terminology will come in handy when dealing with
Fourier structures. Note also that a Heisenberg twain morphism is just
a Heisenberg morphism that is at the same time a twain
homomorphism. We denote the category of Heisenberg twain algebras
over~$\beta$ by~$\TwAlgH\beta$.

Let us now look at a rather \emph{simple specimen} of a Heisenberg twain algebra, here formulated in
entirely algebraic terms. Its analytic significance as a Fourier structure will be recognized in
Example~\ref{ex:pont-doublet-inverse}\ref{itt:four-dtft} below.

\begin{myexample}
  \label{ex:tor-twainalg}
  The \emph{toroidal twain algebra} is defined as~$\CC[\ZZ^n] \divideontimes \CC\pathalg{\ZZ^n}$,
  where~$\big( \CC[\ZZ^n], \star \big)$ denotes the group algebra (with~$\ZZ^n$ the free abelian
  group on $n$ generators) while~$\big( \CC\pathalg{\ZZ^n}, \cdot \big)$ is the path algebra
  (with~$\ZZ^n$ viewed as a discrete quiver). In other words, the two products are defined on the
  generators~$z^a \; (a \in \ZZ^n)$ as
  \begin{equation*}
    z^a \star z^{a'} = z^{a+a'},\quad
    z^a \cdot z^{a'} = \delta_{a,a'} \, z^a .
  \end{equation*}
  Thus~$\CC[\ZZ^n]$ is the Laurent polynomial algebra while~$\CC\pathalg{\ZZ^n}$ is the complex
  algebra generated by the orthogonal idempotents~$z^a$. Viewing Laurent
  polynomials~$\sum_a c_a z^a$ as multivariate sequences~$(c_a)_{a \in \ZZ^n}$, their
  product~$\star$ is the discrete convolution (see also
  Example~\ref{ex:pont-doublet}\ref{it:four-dtft} below).

  Let~$\inner{}{}_\Tordu\colon \Tor^n \times \ZZ^n \to \Tor$ be the torus
  duality~$\inner{\xi}{x}_\Tordu = \xi^x$ introduced in Example~\ref{ex:classical-torus-group} with
  Heisenberg group~$H(\Tordu) = \Tor \ZZ^n \rtimes \Tor^n$. We define the Heisenberg
  action~$H(\Tordu) \times \CC[\ZZ^n] \to \CC[\ZZ^n]$ via
  \begin{equation*}
    x \cdot z^a = z^{a+x},\qquad
    \xi \cdot z^a = \xi^a \, z^a
  \end{equation*}
  for~$(x, \xi) \in \ZZ^n \times \Tor^n$, and where the
  torus~$\Tor \subset \CC$ acts trivially. It is easy to see
  that~$(\CC[\ZZ^n], \star, \cdot)$ is then a Heisenberg twain
  algebra. For later reference, let us also note the \emph{interchange
    law}
  \begin{equation}
    \label{eq:intch-law}
    (z^a \cdot z^b) \star (z^c \cdot z^d) = \delta_{a+d,b+c} \, (z^a \star z^c) \cdot (z^b \star z^d),
  \end{equation}
  which is an immediate consequence of the composition laws given
  above. Used from left to right, its effect is similar to the
  distributivity axiom. The induced normal form of $(\star, \cdot)$
  terms is then a complex linear combination of pointwise products of
  convolutions.
\end{myexample}

\subsection{The Heisenberg Twist.}
\label{sub:heis-twist}

We now turn our attention to an interesting feature of Heisenberg groups that is also important for
understanding the nature of Fourier operators, in particular when iterated.

Up to now we have been speaking of \emph{left} Heisenberg algebras,
omitting the qualification ``left'' since we have not yet considered
their \emph{right} counterparts. For introducing Fourier operators,
though, the distinction between left and right Heisenberg algebras
turns out to be crucial. Referring to Definition~\ref{def:heis-alg}, a
\emph{right Heisenberg algebra}~$S$ is exactly the same except that
one has a linear right action~$S \times H(\beta) \to S$. The latter
induces a map~$\eta\colon H(\beta)^o \to \Aut_K(S)$, which is again
required to factor through the restricted torus action~$\Delta_T$.

In terms of axioms, a right Heisenberg action is also characterized
by~\eqref{eq:left-unit}--\eqref{eq:right-op}, except that the phase
factor~$\inner{\xi}{x}$ in~\eqref{eq:twisted-bimod} changes
sides. Thus Heisenberg scalars~$u \in TG \le H(\beta)$ act as
$(s \star \tilde{s}) \act u = s \star (\tilde{s} \act u)$,
Heisenberg operators~$u \in \Gamma$ as
$(s \star \tilde{s}) \act u = (s \act u) \star (\tilde{s} \act u)$.
In the sequel, the unqualified term ``Heisenberg algebra'' is meant to
refer to left Heisenberg algebras, which we continue to denote
by~$\AlgH\beta$.

We recall the \emph{restriction of scalars} for modules: Given a ring
homomorphism~$\phi\colon R \to S$, any left $S$-module~$N$ can be turned into a left
$R$-module~$\phi^*(N)$ by precomposing its scalar action~$S \to \Aut_\ZZ(N)$ with~$\phi$. If~$S$ is
instead a right $S$-module with scalar action~$S^o \to \Aut_\ZZ(N)$, one gets a right
$R$-module~$\phi^*(N)$ by precomposing. We use the same definition and notation when~$\phi$ is an
antihomomorphism, but now the sides are swapped: If~$N$ is a left/right $S$-module, $\phi^*(N)$ is a
right/left $R$-module.

Let~$M$ be a left $R$-module, $N$ a left/right $S$-module and~$\phi\colon R \to S$ a
homomorphism/antihomomorphism of rings. Then we call~$f\colon M \to N$ a \emph{homomorphism
  over}~$\phi$ if~$f\colon M \to \phi^*(N)$ is a homomorphism of $R$-modules. Thus we
have~$f(\lambda \act x) = \phi(\lambda) \act f(x)$ in the homomorphic
and~$f(\lambda \act x) = f(x) \act \phi(\lambda)$ in the antihomomorphic case (adopting $\act$ for
the scalar action).

The rings in question may be group rings~$\ZZ H$ over arbitrary groups~$H$. In that case, the
modules are called $H$-modules in the terminology of representation theory, namely abelian groups on
which~$H$ acts via automorphisms~\cite[p.~95]{Cohn2003a}. For such group modules, one often uses the
inversion in~$H$ as a preferred antihomomorphism~$\ZZ H \to \ZZ H$ for switching between left and
right $H$-modules. For \emph{Heisenberg groups}~$H=H(\beta)$, however, two other antihomorphisms are
more important---at least in the context of Fourier analysis---and both are involutions (``twists'')
in the sense of \emph{involutive antihomomorphism}. In addition, there is also an \emph{involutive
  homomorphism} (``flip''), which plays an important role in Fourier operators. Before investigating
their signficance in some detail, let us first list their definitions:
\begin{align}
  \label{eq:fwd-twist}
  & \text{Forward Twist} & \hat{J}\colon H(\beta) &\to H(\beta)^o, & 
  c\,(x,\xi) &\mapsto \tfrac{c}{\inner{\xi}{x}} \, (x, -\xi)\\
  \label{eq:bwd-twist}
  & \text{Backward Twist} & \check{J}\colon H(\beta) &\to H(\beta)^o, &
  c\,(x,\xi) &\mapsto \tfrac{c}{\inner{\xi}{x}} \, (-x, \xi)\\
  \label{eq:par-flip}
  & \text{Parity Flip} & \bar{J}\colon H(\beta) &\to H(\beta), &
  c\,(x,\xi) &\mapsto c\,(-x, -\xi)
\end{align}
Note that the twists~$\hat{J}$ and~$\check{J}$ can be squared since each of them may also be viewed
as an antihomomorphism~$H(\beta)^o \to H(\beta)$, and they combine to the flip
as~$\hat{J} \check{J} = \bar{J} = \check{J} \hat{J}$. One may picture the twist-and-flip actions as
a cyclic process of periodicity four, intimately linked with the action of Fourier operators (see
Figure~\ref{fig:heis-clock} below and \S\ref{sub:inversion}).

For understanding the significance of those involutions, we remark
that each duality~$\beta\colon \Gamma \times G \to T$ may be
\emph{transposed} to yield its mirror
image~$\beta\trp\colon G \times \Gamma \to T$ such
that~$\beta\trp(x,\xi) = \beta(\xi,x)$. From the definition of the
Heisenberg group~$H(\beta) = TG \rtimes \Gamma$ one sees
that~$H(\beta\trp) = T\Gamma \rtimes G$ yields the \emph{dual
  Heisenberg group}, which reverses the roles of positions and
momenta. The crucial fact to observe now is
that~$H(\beta\trp) \cong H(\beta)^o$ via the
isomorphism~$c\,(\xi,x) \leftrightarrow c\,(x,\xi)$. For this reason
we shall denote the category of right Heisenberg algebras
by~$\AlgH{\beta\trp}$.

Under the identification~$H(\beta\trp) \cong H(\beta)^o$, the
anti-isomorphism~$\hat{J}$ becomes an
\emph{isomorphism}~$J\colon H(\beta) \to H(\beta\trp)$
with~$c\,(x,\xi) \mapsto \tfrac{c}{\inner{\xi}{x}} \, (-\xi,
x)$. Projected onto the phase spaces $P = G \times \Gamma$
and~$P\trp = \Gamma \times G$, this yields the
map~$j\colon P \to P\trp$ with~$(x,\xi) \mapsto (-\xi, x)$. This map
is geometrically signficant in at least two important examples:
\begin{enumerate}[(a)]
\item If~$G = \Gamma$ is a finite-dimensional vector space over~$F$, the phase
  space~$P = V \times V$ may be viewed as its complexification with~$j$ as its canonical complex
  structure. Taking~$\funcinner{}{}\colon V \times V \to F$ to be an additive bilinear form on~$V$,
  the map~$j$ is like a rotation (it is literally so if~$\beta$ is symmetric and positive definite
  over~$F = \RR$). Taking~$V = \RR^n$, one recognizes~$j$ as the usual block
  matrix~$\smallmat{0}{-I_n}{I_n}{0}$, so for~$n=1$ the action of~$j$ is multiplication
  by~$\sqrt{-1}$. Taking a standard character (Example~\ref{ex:vecgrp}), all this is seen to be an
  instance of \emph{abstract vector duality}.
\item In the case of the \emph{symplectic Heisenberg
    group}~$H(\beta_\Omega)$ of Example~\ref{ex:heisgrp-sympl}, we
  have a finite-dimensional vector space~$G$ over~$F$ with
  dual~$\Gamma = G^*$. The canonical symplectic form~$\Omega_G$ is
  then a skew-symmetric bilinear form on the phase
  space~$P = G \times G^*$. Interpreted as a linear map into its dual,
  $\Omega_G$ is just~$j\colon P \to P^* \cong P\trp$, and its matrix is
  again~$\smallmat{0}{-I_n}{I_n}{0}$.
\end{enumerate}
The second example can be generalized to arbitrary Heisenberg groups $H(\beta) = TG \rtimes
\Gamma$. It points to the right way of understanding the \emph{significance of the tilt
  map}~$j\colon P \to P\trp$, namely as an alternative encoding of the symplectic structure
expressed by~$\omega$. Indeed, multiplying~$\beta\colon \Gamma \times G \to T$
and~$\beta j\colon G \times \Gamma \to T$, we obtain the
duality~$\omega\trp = \beta \times \beta j\colon P\trp \times P \to T$, so~$j$ describes how the
symplectic structure arises from the given duality~$\beta$.

Once the tilt map~$j\colon P \to P\trp$ is fixed, its \emph{extension
  to a homomorphism}~$J\colon H(\beta) \to H(\beta\trp)$ is unique
provided~$J$ is required to leave~$T$ invariant: Using the fact
that~$J$ should agree with~$j$ on~$G$ and~$\Gamma$ one immediataly
obtains
$J\big(c \, (x,\xi)\big) = \tfrac{c}{\inner{\xi}{x}} \, (x, -\xi)$
from the homomorphism property. In the same way, the inverse
tilt~$j^*\colon P\trp \to P$ with~$(\xi, x) \mapsto (-x, \xi)$ can be
extended to a unique homomorphism
$J^*\colon H(\beta\trp) \to H(\beta)$, which is \emph{not} the inverse
of $J\colon H(\beta) \to H(\beta\trp)$. Under the above-mentioned
identification~$H(\beta\trp) \cong H(\beta)^o$, the homomorphism~$J^*$
corresponds to $\check{J}\colon H(\beta)^o \to H(\beta)$, which is the
same as~$\check{J}\colon H(\beta) \to H(\beta)^o$. Finally, the
inversion map~$(x,\xi) \mapsto (-x,-\xi)$ is a homomorphism on~$P$
since the latter is an abelian group. Similar to~$j$ and~$j^*$, it has
a unique \emph{extension to a
  homomorphism}~$\bar{J}\colon H(\beta) \to H(\beta)$ that leaves~$T$
invariant, unlike the inverion map on~$H(\beta)$.

Summing up, there are three natural (anti)homomorphisms for changing a left $H(\beta)$-module~$S$
into a left/right $H(\beta)$-module, which we designate by the same terminology as the maps
themselves:
\begin{itemize}\setlength\itemsep{0.25em}
\item The module~$S^{\land} := \hat{J}^*(S)$ is called the \emph{forward twist} of~$S$.
\item The module~$S^{\lor} := \check{J}^*(S)$ is called the \emph{backward twist} of~$S$.
\item The module~$S^- := \bar{J}^*(S)\vphantom{\check{J}^*(S)}$ is called the \emph{parity flip} of~$S$.
\end{itemize}
If~$S$ is moreover a \emph{Heisenberg algebra}, the torus acts
naturally via~$\epsilon_T$, so the three derived modules above will
also have the torus acting naturally via~$\epsilon_T$. Thus (left and
right) Heisenberg algebras can be flipped as well as twisted forward
and backward amongst themselves.

It is well known~\cite[p.~95]{Cohn2003a} that the \emph{category of $P$-modules} over a group~$P$ is
isomorphic to the \emph{category of left modules over the group ring~$\ZZ[P]$}, and this fact is
much exploited in areas such as group homology~\cite[\S6.1]{Weibel1994}. The same works for
identifying $K$-modules (instead of plain abelian groups) having a linear $P$-action with left
modules over the group algebra~$K[P]$. We shall exploit a similar isomorphism for the Heisenberg
category~$\ModH\beta$, especially in the next subsection for the construction of various free
objects. But in our case we have to account for one slight complication that requires Heisenberg
modules to be more than just modules over the group ring~$K[H(\beta)]$, namely the naturality of the
torus action.

A straightforward way to incorporate this requirement is a slight
generalization of~$K[P]$, namely the so-called \emph{$\gamma$-twisted
  group algebra}~$K^\gamma[P]$ as described
in~\cite[\S3.6]{Kelarev2002}. Given any cocycle~$\gamma \in Z^2(P, \unit{K})$
with~$K$ a trivial $P$-module, the definition of~$K^\gamma[P]$ is the
same as for~$K[P]$ as $K$-module, thus consisting of all
maps~$P \to K$ with finite support. But the multiplication is given by
setting~$h_z \, h_{z'} = \gamma(z, z') \, h_{zz'}$ for
generators~$h_z, h_{z'} \: (z, z' \in P)$ and extending by
$K$-linearity. The cocycle condition~\cite[(3.7)]{Kelarev2002}
for~$\gamma$ is equivalent to the associativity of~$K^\gamma[P]$.

We shall now apply this construction to the phase space~$P = G \oplus \Gamma$ of the Heisenberg
group~$H(\beta)$, endowed by the \emph{Heisenberg cocycle}
\begin{equation*}
  \heiscocy(x, \xi; x', \xi') := \epsilon_T\inner{\xi}{x'}.
\end{equation*}
One may view~$\heiscocy\colon P \times P \to K$ as an extension
of~$\beta\colon \Gamma \times G \to T$ via the natural embeddings~$\Gamma, G \hookrightarrow P$
and~$\epsilon_T\colon T \to K$. The cocycle condition may either be checked by a routine calculation
or inferred from general facts about central group extensions~\cite[Cor.~5.2]{MacLane1995}
since~$\heiscocy$ is just the $\epsilon_T$-image of the factor
set~$\beta\circ (\pi_\Gamma \times \pi_G) \in H^2(P, T)$ that describes Heisenberg
extensions~$T \rightarrowtail H \twoheadrightarrow P$ by Lemma~\heiscite{lem:factorset-from-beta}.

\begin{definition}
  Let~$H(\beta) = TG \rtimes \Gamma$ be a Heisenberg group over a duality~$\beta$. The
  \emph{Heisenberg group algebra} is the twisted group algebra~$K^{\heiscocy}[G \oplus \Gamma]$,
  which we denote by~$H_K(\beta)$.
\end{definition}

The Heisenberg group algebra is really almost the same as the plain group algebra, except that it
works to \emph{identify the torus with the scalar ring}
via~$\epsilon_T\colon (T, \cdot) \to (\nnz{K}, \cdot)$.

\begin{lemma}
  \label{lem:char-heisgrpalg}
  For any duality~$\beta$, we have $H_K(\beta) \cong K[H(\beta)]/\mathfrak{I}_T$ as an isomorphism
  of $K$-algebras, where the ideal~$\mathfrak{I}_T$ is generated by all
  $h_{c(x,\xi)} - \epsilon_T(c) \, h_{1(x,\xi)}$ with~$c(x,\xi) \in H(\beta)$.
\end{lemma}
\begin{proof}
  We define the evident group homomorphism~$\epsilon\colon H(\beta) \to \nnz{H_K(\beta)}$ by
  $c\,(x,\xi) \mapsto \epsilon_T(c) \, h_{x,\xi}$ and extend it, via the universal property of the
  group algebra, to a $K$-algebra homorphism~$\epsilon\colon K[H(\beta)] \to
  H_K(\beta)$.
  Since~$\epsilon$ is surjective, we obtain $H_K(\beta) \cong K[H(\beta)]/\mathfrak{K}_T$ as
  $K$-algebras, with~$\mathfrak{K}_T := \ker(\epsilon)$ being an ideal of~$K[H(\beta)]$.  It is
  clear that~$\mathfrak{I}_T \subseteq \mathfrak{K}_T$, so it only remains to show the reverse
  inclusion.

  We introduce the $T$-\emph{degree} of~$U = \sum_{u \in H(\beta)} \lambda_u h_u \in K[H(\beta)]$ as
  the number of nontrivial occurrences of~$c \in T$ for each~$u = c(x,\xi)$ involved in~$U$. More
  precisely, we set
  \begin{equation*}
    \deg_T(U) := \# \{ c(x,\xi) \in \supp(U) \mid c \neq 1 \},
  \end{equation*}
  where~$\supp(U)$ denotes the support of~$U\colon H(\beta) \to K$, namely the set of
  those~$u \in H(\beta)$ for which~$\lambda_u = U(u) \neq 0$. We
  prove~$\mathfrak{K}_T \subseteq \mathfrak{I}_T$ by induction over $T$-degree.

  For the base case, let us take~$U \in \mathfrak{K}_T$ with~$\deg_T(U) = 0$. In that case we
  have~$U = \sum_{x,\xi} \lambda_{1(x,\xi)} \, h_{1(x,\xi)}$,
  hence~$\epsilon(U) = \sum_{x,\xi} \lambda_{1(x,\xi)} \, h_{x,\xi} = 0$ and
  $\lambda_{1(x,\xi)} = 0$ for all $(x,\xi) \in G \oplus \Gamma$. Now
  assume~$\mathfrak{K}_T \subseteq \mathfrak{I}_T$ for all~$U$ having $T$-degree below a
  fixed~$n > 0$. Taking any~$U \in \mathfrak{K}_T$ with $\deg_T(U) = n$ we must show
  that~$U \in \mathfrak{I}_T$. By the $T$-degree hypothesis, we may
  write~$U = \lambda \, h_{c(x,\xi)} + U'$ with~$\lambda, c \neq 0$ and~$\deg_T(U') < n$.  Denoting
  the generators of~$\mathfrak{I}_T$
  by~$i_{c,x,\xi} := h_{c(x,\xi)} - \epsilon_T(c) \, h_{1(x,\xi)} \in \mathfrak{K}_T$, it is clear
  that we have
  also~$U'' := U - \lambda \, i_{c,x,\xi} = U' + \lambda \, \epsilon_T(c) \, h_{1(x,\xi)} \in
  \mathfrak{K}_T$.
  Since~$\deg_T(U'') < n$, the induction hypothesis implies~$U'' \in \mathfrak{I}_T$ and hence
  also~$U = U'' + \lambda \, i_{c,x,\xi} \in \mathfrak{I}_T$; this completes the induction step.
\end{proof}

As announced, the crucial fact is that we may identify Heisenberg modules with \emph{modules over
  the twisted group algebra}. (This does \emph{not} generalize to Heisenberg algebras since in an
algebra over~$H_K(\beta)$, any element~$h_{1,\xi} \in H_K(\beta)$ acts as scalar instead of an
operator.)

\begin{lemma}
  \label{lem:grp-ring-isom}
  For any duality~$\beta$, we have~$\ModH\beta \cong {}_{H_K(\beta)}\Mod$ as an isomorphism of
  categories.
\end{lemma}
\begin{proof}
  We write Heisenberg modules as~$(S, \Delta, \eta)$, where~$S$ is the underlying abelian group, the
  ring homomorphism $\Delta\colon K \to \End_\ZZ(S)$ is the structure map,
  and~$\eta\colon H(\beta) \to \Aut_K(S)$ is the Heisenberg action. We
  have~$\nnz{\Delta} \circ \epsilon_T = \eta \circ \iota_T$ by the definition of Heisenberg modules
  with~$\iota_T\colon T \hookrightarrow H(\beta)$ the natural embedding. Similarly,
  $H_K(\beta)$-modules can be represented by~$(S, \tilde\Delta)$,
  where~$\tilde\Delta\colon H_K(\beta) \to \End(S)$ is the corresponding structure map. The
  isomorphism~$\ModH\beta \to {}_{H_K(\beta)}\Mod$ sends a Heisenberg module~$(S, \Delta, \eta)$
  to~$(S, \tilde\Delta)$ with
  \begin{equation*}
    \tilde\Delta\Big(\sum_{(x,\xi)\in G \oplus \Gamma} \lambda_{x,\xi} \, h_{x, \xi}\Big) := 
    \sum_{(x,\xi)\in G \oplus \Gamma} \Delta(\lambda_{x,\xi}) \circ \eta\big(1\,(x,\xi)\big),
  \end{equation*}
  where~$h_{x,\xi}$ are the generators of~$H_K(\beta) = K^{\heiscocy}[G \oplus \Gamma]$. Its inverse
  maps the $H_K(\beta)$-module~$(S, \tilde\Delta)$ to the Heisenberg module~$(S, \Delta, \eta)$ with
  structure map~$\Delta := \tilde\Delta \circ \iota_K$ and Heisenberg
  action~$\eta := \tilde\Delta \circ \epsilon$, where~$\epsilon\colon H(\beta) \to H_K(\beta)$ is
  defined by~$c\,(x,\xi) \mapsto \epsilon_T(c) \, h_{x,\xi}$ and
  where $\iota_K\colon K \hookrightarrow H_K(\beta)$ is the
  embedding~$\lambda \mapsto \lambda \, h_{0,0}$.

  Let us first check that the
  map~$\Phi\colon \ModH\beta \to {}_{H_K(\beta)}\Mod$ is
  well-defined. It is clear that~$\tilde\Delta$ is additive and even
  $K$-linear in the
  sense~$\tilde\Delta(\lambda \tilde{h}) = \Delta(\lambda) \circ
  \tilde\Delta(\tilde{h}) = \tilde\Delta(\tilde{h}) \circ
  \Delta(\lambda)$ for all~$\lambda \in K$
  and~$\tilde{h} \in H_K(\beta)$. For ensuring that~$\tilde\Delta$ is
  a homomorphism of rings, it is thus sufficient to
  check~$\tilde\Delta(h_{x,\xi} \, h_{x',\xi'}) =
  \tilde\Delta(h_{x,\xi}) \, \tilde\Delta(h_{x',\xi'})$.  The
  left-hand side comes out
  as~$\Delta(\epsilon_T \inner{\xi}{x'}) \circ \eta\big(1 (x+x',
  \xi+\xi') \big)$ by the definition of the $\heiscocy$-twisted
  multiplication in~$H_K(\beta)$. The right-hand side
  is~$\eta\big(\inner{\xi}{x'} \, (x+x', \xi+\xi') \big) =
  \eta(\iota_T \inner{x'}{\xi}) \circ \eta \big( 1 (x+x', \xi+\xi')
  \big)$, and this coincides with the left-hand side
  by~$\Delta \circ \epsilon_T = \eta \circ \iota_T$.  We conclude
  that~$\Phi$ is well-defined.

  Next we consider the
  map~$\Psi\colon {}_{H_K(\beta)}\Mod \to \ModH\beta$. Since~$\iota_K$
  is clearly a homomorphism of rings, the same is true
  of~$\tilde\Delta$. One may also check that~$\epsilon$ and
  thus~$\eta$ is a group homomorphism. We have now that~$S$ is a
  $K$-module with an action~$\eta\colon H(\beta) \to \Aut_K(S)$, and
  it remains to
  show~$\nnz{\Delta} \circ \epsilon_T = \eta \circ \iota_T$. But this
  follows from~$\iota_K \circ \epsilon_T = \epsilon \circ \iota_T$,
  which is evident. Thus~$\Psi$ is also well-defined.

  Finally, we prove that~$\Phi \circ \Psi = 1_{{}_{H_K(\beta)}\Mod}$
  and~$\Psi \circ \Phi = 1_{\ModH\beta}$. The former follows
  from~$\Delta(\lambda) \circ \eta\big( 1 (x,\xi) \big) =
  \tilde\Delta(\lambda \, h_{0,0}) \circ \tilde\Delta(h_{x,\xi}) =
  \tilde\Delta(\lambda \, h_{x,\xi})$ and the fact that~$\tilde\Delta$
  is additive. The other identity is true because we
  have~$\tilde{\Delta}(\lambda \, h_{0,0}) = \Delta(\lambda) \circ
  \eta\big( 1 (0,0) \big) = \Delta(\lambda)$ for all
  scalars~$\lambda \in K$ and
  also~$\tilde\Delta( \epsilon_T(c) \, h_{x,\xi} ) =
  \Delta\big(\epsilon_T(c)\big) \circ \eta\big( 1(x,\xi) \big) =
  \eta\big( c(x,\xi) \big)$ for all Heisenberg
  operators~$c (c,\xi) \in H(\beta)$, where the last equality
  uses~$\Delta \circ \epsilon_T = \eta \circ \iota_T$.

  It is easy to check that this yields the desired pair of isomorphic functors, which actually leave
  the morphisms---viewed as set-theoretic maps---invariant, since a map is $H_K(\beta)$-linear
  precisely when it is a Heisenberg morphism.
\end{proof}

The reason why we have introduced the Heisenberg group algebra in this subsection is that it
exhibits the \emph{forward and backward twists} as well as the \emph{parity flip} in a transparent
manner. Indeed, it is easy to check that~\eqref{eq:fwd-twist}, \eqref{eq:bwd-twist}
and~\eqref{eq:par-flip} all stabilize the ideal~$\mathfrak{I}_T$ of Lemma~\ref{lem:char-heisgrpalg};
thus one may pass to the quotient to obtain the following maps:
\begin{align*}
  & \text{Forward Twist} & \hat{J}\colon H_K(\beta) &\to H_K(\beta)^o, & 
  h_{x,\xi} &\mapsto \epsilon_T\inner{\xi}{x}^{-1} \, h_{x, -\xi}\\
  & \text{Backward Twist} & \check{J}\colon H_K(\beta) &\to H_K(\beta)^o, &
  h_{x,\xi} &\mapsto \epsilon_T\inner{\xi}{x}^{-1} \, h_{-x, \xi}\\
  & \text{Parity Flip} & \bar{J}\colon H_K(\beta) &\to H_K(\beta), &
  h_{x,\xi} &\mapsto h_{-x, -\xi}
\end{align*}
Now~$\hat{J}, \check{J}$ are involutions (involutive anti-automorphisms) of the
$K$-algebra~$H_K(\beta)$ while~$\bar{J} = \hat{J} \check{J} = \check{J} \hat{J}$ is an involutive
automorphism. Endowing~$K$ with the trivial involution, we see
that~$\smash{\big( H_K(\beta), \hat{J} \big)}$ as well
as~$\smash{\big( H_K(\beta), \check{J} \big)}$ is an \emph{involutive algebra} (also known as a
\emph{$^*$-algebra}) over~$K$. One may refer
to~$\smash{\big( H_K(\beta), \smash{\hat{J}}, \check{J} \big)}$ as a bi-involutive $K$-algebra.

Let us conclude this subsection on the Heisenberg twist with the
remark that, qua extensions, the Heisenberg groups~$H(\beta)$
and~$H(\beta\trp)$ are the same: Indeed, using the notation from above,
we have the equivalence
\begin{equation}
  \label{eq:twist-eqv}
  \xymatrix @M=0.75pc @H=1pc @R=1.5pc @C=2pc%
  { 1 \ar[r] & T \ar@{^{(}->}[r] \ar@2{-}[d] & TG \rtimes \Gamma \ar@{->>}[r] \ar[d]^J & G \times \Gamma \ar[r] \ar[d]^j & 0\\
    1 \ar[r] & T \ar@{^{(}->}[r] & T\Gamma \rtimes G \ar@{->>}[r] & \Gamma \times G \ar[r] & 0}
\end{equation}
of central extensions, where both projections are omission of the torus component. (If one prefers
to have identity for both marginal morphisms, one may alter the second extension by using the
projection $T\Gamma \rtimes G \twoheadrightarrow G \times \Gamma$ with~$c\,(\xi,x) \mapsto (x,\xi)$
instead.) But note that~\eqref{eq:twist-eqv} is not a Heisenberg morphism in the sense
of~\eqref{eq:heis-extn}. So, $H(\beta)$ and~$H(\beta\trp)$ are indeed \emph{distinct as Heisenberg
  groups}, despite forming equivalent central extensions.

\subsection{Free Heisenberg Modules and Algebras.}
\label{sub:free-heismod-heisalg}
It is always good to have at one's disposal various free objects, meaning \emph{left adjoints} to
various kinds of forgetful functor. Throughout this section, we fix a commutative and unital scalar
ring~$K$.

Let us start with the \emph{free Heisenberg module}~$\freesln_\beta(F)$ over a set~$F$ and the free
Heisenberg module~$\freesln_\beta(M)$ over a $K$-module~$M$. In other words (overloading the same
notation), we ask for the left adjoints~$\freesln_\beta$ of the forgetful
functors~$\ModH\beta \to \Set$ and~$\ModH\beta \to {}_K\Mod$. The latter
functor~$\ModH\beta \to {}_{H_K(\beta)}\Mod$ is scalar restriction along
$\iota\colon K \hookrightarrow H_K(\beta)$, so its left adjoint is scalar extension along~$\iota$
and we have~$\freesln_\beta(M) \cong H_K(\beta) \otimes_K M$. Since adjoints respect composition of
functors~\cite[Thm.~IV.8.1]{MacLane1998}, we have~$\freesln_\beta(F) \cong \freesln_\beta(K^{(F)})$
with~$K^{(F)}$ the free $K$-module over the set~$F$.

For the sake of later reference, we state these free objects \emph{in
  explicit terms}. Here and subsequently we shall not write out the
action of the torus~$T \le H(\beta) = TG \rtimes \Gamma$ since it is
always required to act via~$\epsilon_T$; we need only the
range~$(x,\xi) \in P = G \oplus \Gamma \le H(\beta)$. But note that
here~$P$ is also written multiplicatively since we view it as a
subgroup of~$H(\beta)$.

\begin{proposition}
  \label{prop:free-heismod}
  Let~$\beta$ be a duality, $M$ a $K$-module and~$F$ a set.
  \begin{enumerate}
  \item We
    have~$$\freesln_\beta(M) \cong \bigoplus\limits_{(y,\eta) \in P} yM^\eta\qquad \Big( yM^\eta
    \equiv \{y\} \times M \times \{ \eta \} \Big)$$
    with $(x, \xi) \act yf^\eta = \epsilon_T\inner{\xi}{y} \, (xy)f^{\xi\eta}$
    for~$yf^\eta \equiv (y, f, \eta) \in yM^\eta$.
  \item We have~$\freesln_\beta(F) \cong K^{(G \times F \times \Gamma)}$
    with the same action law but for $K$-basis
    elements~$yf^\eta \equiv h_{y, f, \eta} \in \freesln_\beta(F)$.
  \end{enumerate}
  The embeddings~$\iota\colon M \to \freesln_\beta(M)$
  and~$\iota\colon F \to \freesln_\beta(F)$ are~$a \mapsto 1a^1$.
\end{proposition}
\begin{proof}
  The notation $yf^\eta$ for elements in the free Heisenberg modules is meant to convey the
  intuition of ground functions~$f$ modified by the action of a scalar~$y$ and an
  operator~$\eta$. Naturally, scalars are written multiplicatively and operators additively.
  \begin{enumerate}
  \item\label{it:free-heismod-over-module} Note that~$\freesln_\beta(M)$ is a $K$-module since each
    $\{y\} \times M \times \{ \eta \} \cong M$ is a distinct copy of the $K$-module~$M$ indexed by a
    particular phase point~$(y,\eta) \in P$.

    For proving~$\freesln_\beta(M) \cong H_K(\beta) \otimes_K M$, consider the bilinear map
    $H_K(\beta) \times M \to \freesln_\beta(M)$ that
    sends~$\sum_{(y,\eta)\in P} \lambda_{y,\eta} \, h_{y,\eta} \in H_K(\beta)$ and~$f \in M$
    to~$\sum_{(y,\eta)\in P} \lambda_{y,\eta} \,\, y\!f^\eta$. It descends to a $K$-linear map
    $j\colon H_K(\beta) \otimes_K M \to \freesln_\beta(M)$, which is clearly surjective. For seeing
    that~$j$ is injective, assume~$\sum_{(y,\eta)\in P} \lambda_{y,\eta} \,\, y\!f^\eta = 0$. By the
    definition of the direct sum, $\lambda_{y,\eta} = 0$ for all~$(y,\eta) \in P$, and
    $\big(\sum_{(y,\eta)\in P} \lambda_{y,\eta} \, h_{y,\eta}\big) \otimes f = 0$. We conclude
    that~$j$ is a $K$-linear isomorphism. The canonical Heisenberg action
    of~$H_K(\beta) \otimes_k M$ migrates to~$\freesln_\beta(M)$ via~$j$, yielding the action law as
    stated.
  \item\label{it:free-heismod-over-set} Writing the generators by~$h_f$, we
    have~$K^{(F)} \cong \oplus_{f \in F} \; [h_f]$, hence the previous item implies
    $$\freesln_\beta(F) \cong \bigoplus\limits_{(y,\eta)\in P} \bigoplus\limits_{f \in F} \; \{ y \}
    \times [h_f] \times \{\eta\},$$
    where the $K$-module on the right-hand side is free on the generators~$(y, h_f, \eta)$. Mapping
    these to~$h_{y,f,\eta} \in K^{(G \times F \times \Gamma)} = \freesln_\beta(F)$ yields the
    desired $K$-isomorphism and the corresponding action law. (The notational convention is the same
    as before but restricted to the generators~$h_f$, which are shortened to~$f$.)
  \end{enumerate}
  There is a simple alternative proof of item~\eqref{it:free-heismod-over-set}, which is not based
  on item~\eqref{it:free-heismod-over-module}. Identifying Heisenberg module with modules over the
  Heisenberg group algebra, the free Heisenberg module over~$F$ is clearly given
  by~$H_K(\beta)^{(F)}$. But as a $K$-module, we have~$H_K(\beta) = K^{(G \times \Gamma)}$, hence
  the $K$-isomorphism~$H_K(\beta)^{(F)} \cong K^{(G \times F \times \Gamma)}$
  with~$h_{y,\eta} \, h_f \leftrightarrow h_{y,f,\eta}$. It is easy to see that the natural action
  of~$H_K(\beta)^{(F)}$ then induces the action law stated.
\end{proof}

The action laws in Proposition~\ref{prop:free-heismod} may be
summarized by characterizing Heisenberg scalars~$x \in G$ as
\emph{formal multipliers} and Heisenberg operators~$\xi \in \Gamma$ as
\emph{formal exponents}.

Now we take the next step, constructing the \emph{free Heisenberg
  algebra} $\freepln(S)$ over a given Heisenberg
module~$S \in \ModH\beta$.  Note that this is \emph{not} simply the
free $H_K(\beta)$-algebra over the $H_K(\beta)$-module~$S$, due to the
distinct roles of Heisenberg scalars and operators. Indeed, one must
take the symmetric algebra~$\Sym_{KG}(S) \equiv T_{KG}(S)/I_{KG}(S)$
with~$S$ considered as~$KG$-module. By analogy with
Propostion~\ref{prop:free-heismod}, we write the given Heisenberg
action of~$S$ as exponents.

\begin{proposition}
  \label{prop:free-heisalg}
  Given a Heisenberg module~$S \in \ModH\beta$, the induced free Heisenberg
  algebra is~$S \hookrightarrow \freepln(S) := \Sym_{KG}(S)$, with action
  \begin{equation}
    \label{eq:free-heisalg-action}
    (x,\xi) \,\act\, s_1 \cdots s_k = x \; (s_1^{\xi}) \cdots ( s_k^{\xi})
  \end{equation}
  for~$s_1, \dots, s_k \in S$.
\end{proposition}
\begin{proof}
  Since a Heisenberg module~$S$ may be viewed as a module over the Heisenberg group
  algebra~$H_K(\beta)$, the inclusion~$\iota_G\colon KG \hookrightarrow H_K(\beta)$ induces the
  structure of $KG$-module on~$S$. Hence we may form the symmetric
  algebra~$\freepln(S) = \Sym_{KG}(S)$, and it is easy to see that~\eqref{eq:free-heisalg-action}
  yields a $K$-linear action on~$\freepln(S)$. By definition, the torus~$T \le H(\beta)$ acts
  naturally via~$\epsilon_T$. Indeed, $\freepln(S)$ is a Heisenberg algebra since the action
  of~$G \le H(\beta)$ coincides with the $KG$-scalar action on~$\freepln(S)$ while
  each~$\xi \in \Gamma \le H(\beta)$ acts as $K$-algebra endomorphism.

  For seeing that~$\freepln(S)$ is free over~$S$, we take a arbitrary Heisenberg algebra~$U$ and
  Heisenberg morphism~$\phi\colon S \to U$, where~$U$ is viewed as a Heisenberg module. We must show
  that~$\phi$ extends to a unique $\AlgH\beta$-morphism~$\tilde\phi\colon \freepln(S) \to U$. Being
  an algebra homomorphism, it must
  satisify~$\tilde\phi(s_1 \cdots s_k) = \phi(s_1) \cdots \phi(s_k)$, which
  fixes~$\tilde\phi$ uniquely. It remains to show that~$\tilde\phi\colon \freepln(S) \to U$ defined
  in this way is actually an $\AlgH\beta$-morphism. Using~\eqref{eq:left-scal} for the Heisenberg
  algebra~$U$, it is easy to see that~$\tilde{\phi}$ is a $KG$-algebra homomorphism. For seeing that
  it respects the action of~$\Gamma \le H(\beta)$, one similarly applies~\eqref{eq:right-op}
  for~$U$.
\end{proof}

Note that the elements of the free Heisenberg algebra~$\freepln(S)$ are essentially polynomials,
hence the notation~$\freepln$. If~$S$ has generators~$\bar{S}$, the symmetric algebra $\Sym_{KG}(S)$
is generated by~$\bar{s}_1 \cdots \bar{s}_k \: (\bar{s}_i \in \bar{S})$ as $KG$-module, and the
action law may be expanded to
\begin{equation*}
      (x,\xi) \act x_0 \: \bar{s}_1 \cdots \bar{s}_k = \epsilon_T
    \inner{x_0}{\xi} \: x_0 x \; (\bar{s}_1^{\xi}) \cdots (\bar{s}_k^{\xi}),
\end{equation*}
which follows from~\eqref{eq:free-heisalg-action} via~\eqref{eq:twisted-bimod}.

We have proved that~$\freepln\colon \ModH\beta \to \AlgH\beta$ is left adjoint to the forgetful
functor~$\AlgH\beta \to \ModH\beta$. We can combine the latter with the left adjoints of
Proposition~\ref{prop:free-heismod} to obtain the \emph{free Heisenberg algebra~$\freepln_\beta(F)$
  over a set}~$F$ and the \emph{free Heisenberg algebra~$\freepln_\beta(M)$ over a $K$-module~$M$},
which we describe now in more explicit terms.

\begin{corollary}
  \label{cor:free-heisalg}
  Let~$\beta$ be a duality, $M$ a $K$-module and~$F$ a set.
  \begin{enumerate}
  \item\label{it:free-heisalg-mod} We
    have~$\freepln_\beta(M) \cong \Sym_{KG}\Big(\bigoplus_{\eta \in \Gamma} M^\eta\Big)$
    with action
    $$(x, \xi) \act f_1^{\eta_1} \cdots f_k^{\eta_k} = x \: f_1^{\xi\eta_1}
    \cdots f_k^{\xi\eta_k}$$ for~$(f_1^{\eta_1}, \dots, f_k^{\eta_k}) \in
    M^{\eta_1} \times \cdots \times M^{\eta_k}$.
  \item\label{it:free-heisalg-set} We
    have~$\freepln_\beta(F) \cong KG[F \times \Gamma]$ with action
      $$(x,\xi) \act (f_1^{\eta_1})^{\nu_1} \cdots (f_k^{\eta_k})^{\nu_k} =
      x \; (f_1^{\xi\eta_1})^{\nu_1} \cdots (f_k^{\xi\eta_k})^{\nu_k}$$
      for~$f_i^{\eta_i} \equiv (f_i, \eta_i) \in F \times \Gamma \; (i = 1, \dots, k)$
      and~$\nu_1, \dots, \nu_k \in \NN$.
    \end{enumerate}
    The embeddings~$M \to \freesln_\beta(M)$
    and~$F \to \freesln_\beta(F)$ are~$a \mapsto a^1$.
\end{corollary}
\begin{proof}
  Let~$\beta$, $M$ and~$F$ be as stated.
  \begin{enumerate}
  \item Writing~$M^\eta := 1M^\eta$
    and~$M^\Gamma := \bigoplus_{\eta \in \Gamma} M^\eta \le \freesln_\beta(M)$ as abbreviations, it
    is clear that~$M^\Gamma$ generates~$\freesln_\beta(M)$ as $KG$-module, hence
    the~$f_1^{\eta_1} \cdots f_k^{\eta_k}$ generate~$\freepln_\beta(M)$ as $KG$-module; the
    isomorphism consists in multiplying out. The action law follows by combining that
    of~$\freesln_\beta(M)$ with~\eqref{eq:free-heisalg-action}.
  \item The isomorphism~$\Phi\colon \freepln_\beta(F) \isomarrow KG[F \times \Gamma]$ is constructed
    by noting
    that~$M := \freesln_\beta(F)= K^{(G \times F \times \Gamma)} \cong (KG)^{(F \times \Gamma)}$ is
    a free $KG$-module with basis~$F \times \Gamma$ so that~$\freepln_\beta(F) = S_{KG}(M)$ is a
    polynomial ring over~$KG$ with indeterminates~$F \times \Gamma$. Again, the action law follows
    by combining that of~$\freesln_\beta(F)$ with~\eqref{eq:free-heisalg-action}.\qedhere
  \end{enumerate}
\end{proof}

The action law of Item~\ref{it:free-heisalg-set} may be assimilated
further to that of Item~\ref{it:free-heisalg-mod} in
Corollary~\ref{cor:free-heisalg} by using \emph{formal
  products}~$\eta_i \nu_i \in \Gamma \times \NN$ in the exponents: In
that case, Heisenberg operators act on their left components while
iterated multiplication acts on their right components.

Up to now, we have now obtained the following five functors for
generating the free slain and plain Heisenberg algebras:
\begin{align*}
  & \freesln_\beta\colon \Set \to \ModH\beta,\quad
    \freesln_\beta\colon \Mod_K \to \ModH\beta,\text{ and}\\
  & \freepln\colon \ModH\beta \to \AlgH\beta,\text{ and}\\
  & \freepln_\beta = \freepln\freesln_\beta\colon \Set \to \AlgH\beta,\quad
    \freepln_\beta = \freepln\freesln_\beta\colon \Mod_K \to \AlgH\beta.
\end{align*}
On morphisms, these functors act in the usual way of free functors. For example, given any set
map~$\zeta\colon F \to \Phi$, the \emph{induced morphism}
$\freepln_\beta(\zeta)\colon \freepln_\beta(F) \to \freepln_\beta(\Phi)$ in~$\AlgH\beta$ is obtained
by sending the $KG$-basis
element~$(f_1^{\xi_1})^{\nu_1} \cdots (f_k^{\xi_k})^{\nu_k} \in \freepln_\beta(F)$
to~$(\phi_1^{\xi_1})^{\nu_1} \cdots (\phi_k^{\xi_k})^{\nu_k} \in \freepln_\beta(\Phi)$, where
~$\phi_1 := \zeta(f_1), \dots, \phi_k := \zeta(f_k)$ are the assigned image elements.

The description of the corresponding free Heisenberg twain algebras is somewhat more cumbersome. It
is easier to describe first the general situation: Given an $R$-module~$M$, we want to construct the
\emph{free twain algebra}~$F$ over~$M$, which is the smallest twain algebra~$(F, +, \star, \cdot)$
such that~$(M, +)$ is a submodule of~$(F, +)$. We define a sequence of product
algebras~$(C_k \times D_k)_{k>0}$ starting with $C_0 \times D_0 := M \oplus M$ and constructing
recursively~$C_{k+1} \times D_{k+1} := \Sym_R(D_k)_\star \times \Sym_R(C_k)_{\tdot}$, where the
subscripts on the symmetric algebras serve to keep the products apart. Using inducion, one checks
immediately that~$C_k \subset C_{k+1} \land D_k \subset D_{k+1}$, so the~$(C_k \times D_k)_{k>0}$
form an ascending sequence of $R$-algebras whose direct limit we denote by
\begin{equation}
  \label{eq:dirlim}
  C_\infty \times D_\infty := \bigcup_{k>0} C_k \times D_k .
\end{equation}
Obviously, $(C_\infty, \star)$ and~$(D_\infty, \cdot)$ are both algebras over~$R$, and their
carriers coincide since~$C_k \subset \Sym_R(C_k) = D_{k+1} \subset D_\infty$ and
conversely $D_k \subset C_\infty$.  Thus~$F:= C_\infty = D_\infty$ is a twain
algebra~$(F, \star, \cdot)$ over~$R$, and it is easy to see that it satisfies the required universal
property for the free twain algebra over the $R$-module~$M$.

The construction of the \emph{free Heisenberg twain algebra} follows
similar lines, but the role of the ring~$R$ in the intertwined
recursion steps for~$C_{k+1} \times D_{k+1}$ is more subtle, reflecting
the alternating scalar/operator roles for the Heisenberg action on the
recto/verso algebras. We shall make use of the free functor~$\freepln$
of Proposition~\ref{prop:free-heisalg}, using an overbar when
referring to its verso variant (where~$G$ act as scalars and~$\Gamma$
as operators).

\begin{proposition}
  \label{prop:free-heis-twn}
  Given a Heisenberg module~$S \in \ModH\beta$, the free Heisenberg
  twain algebra $S \hookrightarrow \freetwn(S)$ is defined as the
  direct limit~\eqref{eq:dirlim} for the ascending sequence
  \begin{equation}
    \label{eq:twn-heis-asc}
    C_0 \times D_0 := S \oplus S,\qquad
    C_{k+1} \times D_{k+1} := \freepln(D_k)_\star \times \bar{\freepln}(C_k)_{\tdot}
  \end{equation}
  with induced Heisenberg action.
\end{proposition}
\begin{proof}
  Let us first reassure ourselves that~\eqref{eq:twn-heis-asc} is
  indeed an ascending sequence: It is clear that each
  $C_{k+1} = \Sym_{KG}(D_k)$ is a recto plain Heisenberg algebra
  over~$\beta$ since Proposition~\ref{prop:free-heisalg} is applicable
  to~$D_k$ viewed as a Heisenberg module. Similarly,
  each~$D_{k+1} = \Sym_{K\Gamma}(C_i)$ is a verso plain Heisenberg
  algebra over~$\beta$. It follows by joint induction
  that~$C_k \le C_{k+1}$ in~$\Alg_{KG}$ and~$D_k \le D_{k+1}$
  in~$\Alg_{K\Gamma}$. Moreover, the operator actions for each are
  compatible, so the $(C_k)_{k>0}$ and~$(D_k)_{k>0}$ are ascending
  sequences of (recto and verso) plain Heisenberg algebras. Hence we
  obtain in the direct limit~\eqref{eq:dirlim} a recto plain
  Heisenberg algebra~$C_\infty$ and a verso plain Heisenberg
  algebra~$D_\infty$.

  For making these into one twain Heisenberg algebra, we use the
  overlay~$C_\infty \divideontimes_\iota D_\infty$ introduced before Example~\ref{ex:tor-twainalg},
  with the transfer isomorphism~$\iota\colon C_\infty \isomarrow D_\infty$ defined as follows. We
  keep~$C_0 = M = D_0$ invariant. Then let~$s \in C_\infty$ be an element of rank~$k>0$ so that~$s$
  is contained in~$C_k$ but in no earlier stage. We can write such elements as $K$-linear
  combinations of~$s = [s_1] \star \cdots \star [s_m]$ for~$s_1, \dots, s_m \in D_{k-1}$, with~$[-]$
  denoting the generator embedding. We set~$\iota(s) = [s] \in D_{k+1} \le D_\infty$ if~$m>1$ and
  $\iota(s) = s_1 \in D_{k-1} \le D_\infty$ if~$m=1$. It is easy to see that~$\iota$ is bijective
  with~$\iota^{-1}$ having an analogous description.\footnote{This presupposes that embedded
    elements~$[s_1] \in C_k$ embed into~$D_{k+1}$ as their original, thus
    identifying~$[[s_1]] = s_1 \in D_{k-1}$; the same is assumded for embeddings into~$C_{k+1}$. We
    have implicitly used this for establishing~$C_k \le C_{k+1}$ and~$D_k \le D_{k+1}$.} The
  resulting product works in the expected manner, for example
  \begin{equation*}
    ([s_1] \star [s_2]) \cdot ([t_1] \star [t_2]) = [[[s_1] \star [s_2]] \cdot{} [[t_1] \star [t_2]]],
  \end{equation*}
  expressed without embeddings by the apparently vacuous statement (similar to the
  corresponding statement for multiplying polynomials): The pointwise product of~$s_1 \star s_2$
  and~$t_1 \star t_2$ yields~$(s_1 \star s_2) \cdot (t_1 \star t_2)$.

  For checking that~$\freetwn(S)$ is the free Heisenberg twain algebra over~$S$,
  let~$\phi\colon S \to T$ be a Heisenberg morphism to an arbitrary Heisenberg twain algebra~$T$. We
  must determine a unique map~$\tilde\phi\colon \freetwn(S) \to T$ that factors throught the
  embedding~$S \hookrightarrow \freetwn(S)$. Obviously, we must set~$\tilde\phi(s) = \phi(s)$ for
  all~$s \in S$, and this in fact determines~$\tilde\phi$ on all of~$\freetwn(S)$ since it must be a
  homomorphism (with respect to both multiplication maps). In effect, one recursively replaces
  brackets by~$\tilde\phi$ until hitting on elements of~$S$. We have thus established uniqueness,
  and we know that~$\tilde\phi$ is a twain homomorphism. Finally, one checks that~$\tilde\phi$
  respects the Heisenberg action using induction on rank: While the action of~$cx \in TG$ involves
  tracking down one path to a single leaf~$s \in S$, the action of~$\xi \in \Gamma$ precipitates
  down to all the leaves.
\end{proof}
\begin{proof}
  Let us also sketch an \emph{alternative proof via universal algebra}. To this end, consider the
  signature~$[\TwAlgH\beta]$ defined over the signature of abelian groups (binary plus, unary minus,
  nullary zero), together with the two binary products~$\star$ and~$\cdot$, the unary
  homotheties~$u \cdot$ for all Heisenberg actors~$u \in H(\beta)$, and the unary
  homotheties~$\lambda \cdot$ for the scalars~$\lambda \in K$. The signature~$[{\TwAlgH\beta}_S]$ is
  obtained by adding as constants (nullary operations) all elements of~$S$ to capture the embedding
  of~$S$. The variety of Heisenberg twain algebras~$\TwAlgH\beta$ is defined over the
  signature~$[\TwAlgH\beta]$ subject to the $K$-algebra axioms (distributivity doubled for~$\star$
  and~$\cdot$), the laws~\eqref{eq:left-unit}--\eqref{eq:twisted-bimod}, and two copies (one for
  each product) of~\eqref{eq:left-scal} and~\eqref{eq:right-op}. We add to this the Heisenberg
  action for each element of~$S$, to obtain the variety~${\TwAlgH\beta}_S$ over the
  signature~$[{\TwAlgH\beta}_S]$.

  Note that we do not have generators since the embedding of~$S$
  serves this purpose. Our task now is to establish an
  ${\TwAlgH\beta}_S$ isomorphism~$i$ from the term algebra~$T$ of the
  variety~${\TwAlgH\beta}_S$ to the free Heisenberg twain
  algebra~$\freetwn(S)$ as defined in the previous proof. We do this
  by orienting the laws of~${\TwAlgH\beta}_S$ in such a way that the
  corresponding normal forms can be identified with elements
  of~$\freetwn(S)$.

  Representing $0$ by the empty sum, it is clear that distributivity over the products~$\star$
  and~$\cdot$ as well as the homotheties of~$H(\beta)$ and~$K$ may be oriented in the usual manner
  to reduce each element in~$T$ to a sum of \emph{terms} (meaning elements of~$T$ that do not
  contain $+, -, 0$). Moreover, the linearity of the Heisenberg action allows us to write the
  $K$-homotheties on the very front, so everything is reduced to $K$-linear combinations of
  \emph{monomials} (terms that do not contain scalars from~$K$). Using~\eqref{eq:torus-act}, we can
  furthermore eliminate Heisenberg actors involving~$c \in T$. Applying both copies
  of~\eqref{eq:left-scal} and~\eqref{eq:right-op} allows us to move the remaining Heisenberg
  actors~$x \in G$ and~$\xi \in \Gamma$ all the way to the embedded elements~$s \in S$, where we can
  apply the imported Heisenberg action.

  Thus it remains to show how to identify Heisenberg-free monomials
  (elements of~$T$ involving only~$\star$ and~$\cdot$ as well as the
  embedded elements~$s \in S$). But it should be clear how to do this
  from the informal example mentioned in our earlier proof above. It
  is also straightforward to check that the resulting map~$i$ is
  indeed an isomorphism in the variety~${\TwAlgH\beta}_S$.
\end{proof}

The reader will agree that the above construction if full of tedious
technicalities but utterly simple from the structural perspective: As
usual, the elements of the free Heisenberg twain algebra are built up
by \emph{schematically applying} the available operations.

\section{Fourier Operators in Algebra}
\label{sec:fourier-operators-algebra}

\subsection{The Notion of Fourier Doublets.}\label{sub:four-doublet}
%Given a duality~$\beta$, we have introduced forward and backward
\emph{Heisenberg twists}~$\hat{J}\colon H(\beta) \to H(\beta)^o$
and~$\check{J}\colon H(\beta) \to H(\beta)^o$ and the parity
flip~$\bar{J}\colon H(\beta) \to H(\beta)$ in
\S\ref{sub:heis-twist}. They induce
functors~$S \mapsto S^{\land}$ and~$S \mapsto S^{\lor}$ from the
category of left/right to the category of right/left Heisenberg
algebras, and an endofunctor~$S \mapsto S^-$ on the category of
left/right Heisenberg algebras. Note that while the latter functor may
be interpreted as~$\AlgH{-_\Gamma \times -_G, 1_T}$
with~$-_\Gamma\colon \Gamma \to \Gamma$ and~$-_G\colon G \to G$ being
the negation maps, this does not work for the former functors since
the twists are not induced by morphisms of~$\Du$.

We will introduce \emph{forward/backward Fourier operators} as Heisenberg morphisms over the
forward/backward twists and \emph{reversal operators} as Heisenberg morphisms over the parity
flip. We can express this via the corresponding modules as follows (reversal is only mentioned for
right modules~$\Sigma$ but is defined in the same way for left modules~$S$).

\begin{definition}
  \label{def:four-op}
  For a fixed duality~$\beta$, let $S$ be a left and~$\Sigma$ a right
  Heisenberg algebra (slain or plain or twain).
  \begin{itemize}
  \item A Heisenberg morphism
    $\Four^{\land}\colon S \to \Sigma^{\land}$ is called a
    \emph{forward Fourier operator} from~$S$ to~$\Sigma$,
  \item and a Heisenberg morphism
    $\Four^{\lor}\colon S \to \Sigma^{\lor}$ a
    \emph{backward Fourier operator} from~$S$ to~$\Sigma$,
  \item a Heisenberg morphism~$\Par\colon \Sigma \to \Sigma^-$ a \emph{reversal
      operator} on~$\Sigma$.
  \end{itemize}
  By default, the term \emph{Fourier operator} refers to the forward kind.
\end{definition}

\begin{figure}[h]
  \includegraphics[width=\textwidth]{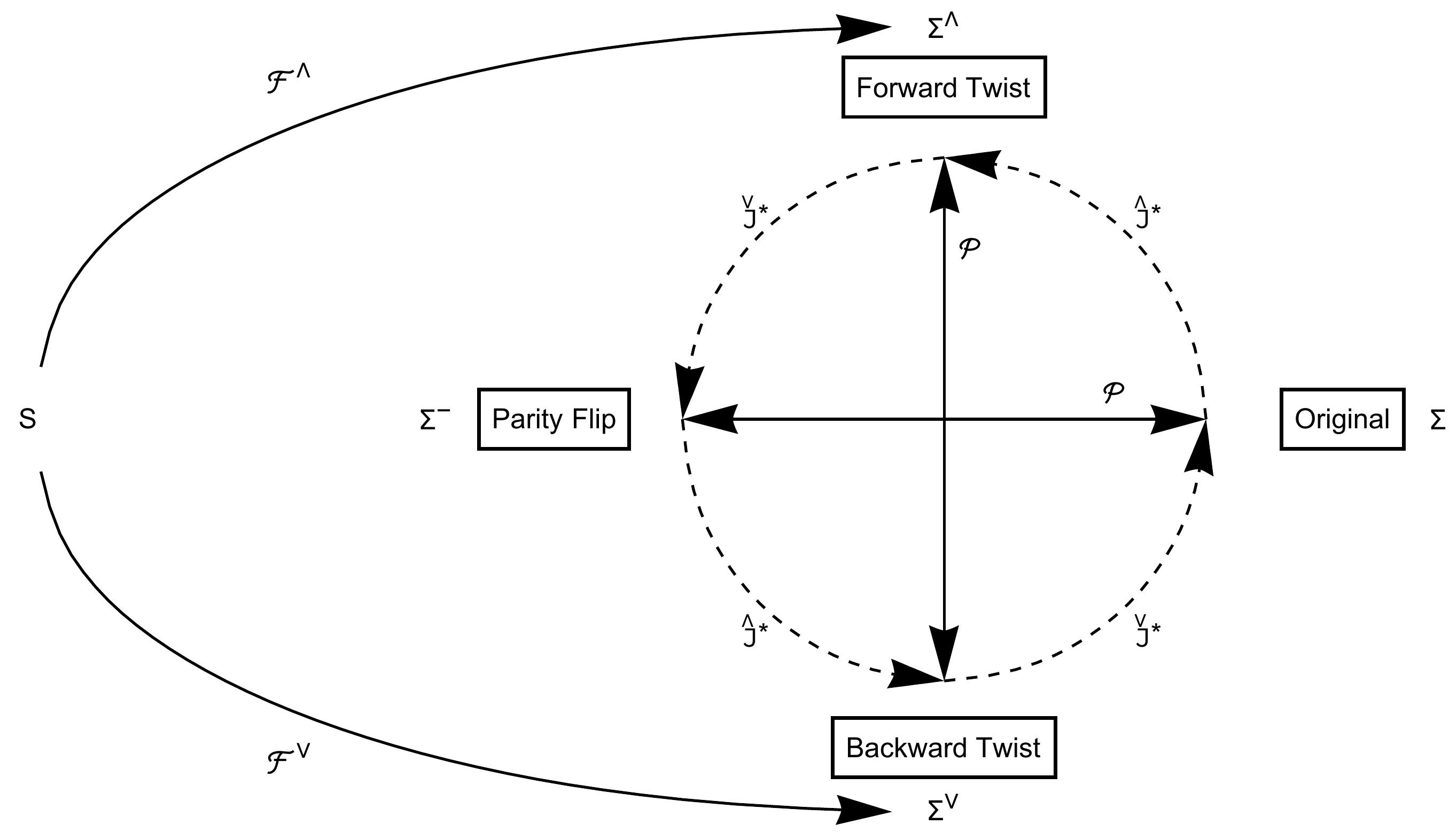}
  \caption{The Heisenberg Clock}
  \label{fig:heis-clock}
\end{figure}

Since the twists concur to cycles of periodicity four, their action may be visualized on the
\emph{Heisenberg clock} of Figure~\ref{fig:heis-clock}. One may think of this as a period-four
(``complexified'') analog of duality in finite-dimensional vector spaces. Topologically more
accurate, we may put a M\"obius strip around a clock face (left
half~$\lcurvearrowdown\!\!{12\atop6}$ standing for left modules, right
half~${12\atop6}\!\!\rcurvearrowdown$ for right modules), moving in six-hour steps: One starts
with~$S$ at~$9h$ to reach~$S^\land$ at~$15h$, moving on to~$S^-$ opposite of~$S$ at~$21h$, stopping
at~$S^\lor$ opposite of~$S^\land$ at~$3h$, finally returning to~$S$ next day at~$9h$.

\begin{myremark}
  \label{rem:justify-twist}
  Though the complexified duality structure \emph{seems to be an unnecessary complication}, it will
  be crucial for understanding Fourier inversion (\S\ref{sub:inversion}). Indeed, a single
  (forward) Fourier operator may be characterized in a much simpler way (taking all modules left):

  A \emph{conjugate-linear map} between modules~$M$ and~$M'$ over an involutive algebra~$A$ is a
  group homomorphisms~$\phi\colon (M,+) \to (M', +)$ such that~$\phi(a \cdot m) = a^* \, \phi(m)$
  for all~$a \in A$ and~$m \in M$, where~$a \mapsto a^*$ denotes the involution. Viewing Heisenberg
  modules~$S \in \ModH\beta$ as modules over the involutive algebra~$\big(H_K(\beta), \hat{J}\big)$,
  the Fourier operators from~$S$ to~$\Sigma$ are precisely the conjugate-linear maps. (Of course,
  one obtains backward Fourier operators by taking the backward twist~$\check{J}$ instead of the
  forward twist~$\hat{J}$.)

  One might be tempted to simplify matters even more by taking the codomain to be
  just~$S' := \Sigma^\land$; then a (forward) Fourier operator is an $H_K(\beta)$-linear
  map~$S \to S'$ \emph{simpliciter}. While this may be done on an adhoc basis for isolated examples,
  it does not mesh smoothly when considering forward and backward Fourier operators in tandem: The
  inverse operator is necessarily defined on a module with modified Heisenberg action (see
  Definition~\ref{def:singlet}). Moroever, in the all-important Pontryagin setting
  (\S\ref{sub:class-pont-duality}), the definition of the
  actions~\eqref{eq:pont-actions-G}--\eqref{eq:pont-actions-Gamma} would incur spurious signs
  concealing the twists.
\end{myremark}

We will sometimes write~$\hat{s} := \Four^\land(s) \in \Sigma^\land$
and~$\check{s} := \Four^\lor(s) \in \Sigma^\lor$ for the forward/backward Fourier transform
of~$s \in S$. While the \emph{hat notation}~$\hat{s}$ for the (forward) Fourier transform is very
commonly used in engineering practice, one may also find the \emph{check notation}~$\check{s}$ for
the backward transform in some places like~\cite[Def.~31.16]{HewittRoss1997}, where it is
specifically introduced in the $L^2$ setting (confer Proposition~\ref{prop:L2-doublet} below). For
the sake of uniformity, we write also~$\bar{s} := \Par(s)$ and~$\bar{\sigma} := \Par(\sigma)$ for
the reversal of~$s \in S$ and~$\sigma \in \Sigma$.

Writing~$\hat{h} = \hat{J}(h)$, $\check{h} = \check{J}(h)$ for the \emph{twists}
and~$\bar{h} = \bar{J}(h)$ for the \emph{flips} of \emph{Heisenberg actors}~$h \in H(\beta)$, we can
characterize forward/backward Fourier operators and reversal operators as follows:
\begin{equation}
  \label{eq:fourop-par-id}
  \Four^{\land} (h \act s) = \hat{s} \act \hat{h},
  \quad
  \Par(h \act s) = \bar{h} \act \bar{s}
  \qquad
  \Four^{\lor} (h \act s) = \check{s} \act \check{h},
\end{equation}
Using the hat/check notation on Heisenberg actors, the basic cycle in the Heisenberg clock
(Figure~\ref{fig:heis-clock}) reads $h \mapsto \hat{h} \mapsto h^- \mapsto \check{h} \mapsto h$.  To
avoid confusion with honest Heisenberg morphisms, we use cycle markers in
\begin{equation*}
  \Four^\land\colon S \pto \Sigma,\quad
  \Par\colon S \ppto S,\quad
  \Four^\lor\colon S \pppto \Sigma
\end{equation*}
for expressing the morphisms of Definition~\ref{def:four-op}. In fact, we will mostly need the
$\pto$ notation in the sequel.

In terms of the \emph{bimodule structure} mentioned after~\eqref{eq:left-unit}--\eqref{eq:right-op}
above, the conditions~\eqref{eq:fourop-par-id} decompose into the requirement of respecting the
torus action together with
\begin{align}
  \Four^\land (x \act s) &= \Four^\land(s) \act x,&
  \Four^\land (\xi \act s) &=  \Four^\land(s) \act \xi^-,\label{eq:cofourop}\\
  \Par (x \act s) &= x^- \act \Par(s),&
  \Par (\xi \act s) &= \xi^- \act \Par(s),\label{eq:parop}\\
  \Four^\lor (x \act s) &= \Four^\lor(s) \act x^-,&
  \Four^\lor (\xi \act s) &= \Four^\lor(s) \act \xi\label{eq:contfourop}
\end{align}
for~$(x, \xi) \in G \times \Gamma$ and~$s \in S$. Here~$x \mapsto x^-$ and~$\xi \mapsto \xi^-$
denote the negation maps of~$G$ and~$\Gamma$, respectively.

The parallel treatment of forward and backward Fourier operators, while appealing from an aesthetic
viewpoint, is not economic for algorithmic purposes. In the classical scenario described in
\S\ref{sub:class-pont-duality}, the distinction between~$\Four^{\land}$ and~$\Four^{\lor}$
hinges on the sign in the exponential. One can generate one from the other by applying a sign
change, which is incorporated in a \emph{distinguished reversal operator}.

\begin{definition}
  \label{def:par-op}
  Let~$\beta$ be a duality. We call~$S \in \AlgH\beta$ \emph{symmetric} if it is endowed with an
  involutive reversal operator~$\Par\colon S \to S$.
\end{definition}

In this case $\Four \Par$ is a backward/forward Fourier operator iff~$\Four$ is a forward/backward
Fourier operator. It is then preferrable to distinguish, say, some \emph{forward Fourier
  operator}~$\Four$ and retain~$\Four^{\lor} := \Four\Par$ as an abbreviation for the derived
\emph{backward Fourier operator}. For emphasizing the underlying symmetry, one may still
employ~$\Four^{\land} := \Four$ as a notational variant for the given Fourier operator.

A Fourier operator~$\Four$ from a symmetric Heisenberg algebra~$S$ to
another symmetric Heisenberg algebra~$\Sigma$ is called
\emph{symmetric} if it commutes with the reversal operators in the
sense that~$\Four \Par_S = \Par_\Sigma \Four$. We shall henceforth
suppress the domain of the reversal operators, writing again~$\Par$
when no confusion arises. Moreover, we assume all Heisenberg algebras
and Fourier operators as symmetric (see
Definitions~\ref{def:fourier-doublet} below), because all natural
examples appear to be like this.

Symmetric Heisenberg algebras also suggest the following convenient
jargon. We call an operator~$\FFour\colon \Sigma \to S$ \emph{sign
  inverse} to~$\Four$ if~$\FFour \Four = \Par_S$
and~$\Four \FFour = \Par_\Sigma$. In terms of forward/backward Fourier
operators: The inverse of~$\Four^\land$ is the sign inverse
of~$\Four^\lor$, and the inverse of~$\Four^\lor$ the sign inverse
of~$\Four^\land$.

\begin{myremark}
  \label{rem:Fourier-is-not-endo}
  Before we now introduce the central object of our algebraic approach to Fourier analysis, let us
  make a brief comparison with differential algebra~\cite{Kolchin1973,Ritt1966}. Faithful to its
  name as a discipline, its central algebraic objects are differential algebras, viz.\@ algebras
  with distinguished derivations. Likewise, in our case we will introduce \emph{Heisenberg algebras}
  with distinguished \emph{Fourier operators}. There are, however, two noteworthy differences:
  \begin{enumerate}
  \item In general we cannot expect Fourier operators to have the \emph{same domain and codomain},
    even if the signal and spectral spaces coincide: In the presence of an algebra structure,
    Fourier operators are only $K$-linear endomorphisms but not as algebra endomorphisms (see
    Definition~\ref{def:singlet}).
  \item While one typically has a great variety of derivations on any given ring (they form a Lie
    algebra!), the variety of \emph{Fourier operators between fixed Heisenberg algebras} appears to
    be rather restricted, at least under the usual topological constraints.
  \end{enumerate}
  In typical cases (see Remark~\ref{rem:four-op-from-sigspc}), the Heisenberg structure of the
  domain spawns the Fourier operator with its codomain---but this is something that the algebra
  ``does not see''.
\end{myremark}

With these qualifications in mind, we can now proceed to defining an appropriate algebraic notion of
Fourier structures. Since it harbors a pair of Heisenberg algebras, we will call such an object a
\emph{Fourier doublet}. As we shall see in the sequel, a Fourier doublet may occasionally coalesce
into a Fourier singlet (Definition~\ref{def:singlet}). Considering the great role of spectroscopy
as an early motivation for classical Fourier analysis, the doublet/singlet metaphor does not seem to
be out of place (see also Remark~\ref{rem:spectrum}).

\begin{definition}
  \label{def:fourier-doublet}
  Let~$\beta$ be a duality. Then $(S, \Sigma, \Four)$ is a \emph{Fourier doublet} over~$\beta$ if
  $\Four\colon S \pto \Sigma$ is a Fourier operator between the left Heisenberg
  algebra~$S \in \AlgH\beta$ and the right Heisenberg algebra~$\Sigma \in \AlgH{\beta}$.
\end{definition}

Going back to the definition of (plain) Heisenberg algebras, it will be seen that the essential
property required in a Fourier doublet is the so-called \emph{convolution
  theorem}~$\Four(s \star s') = \Four s \cdot \Four s'$. The choice of axioms for the axiomatization
in this paper is vindicated by the results derived in~\cite[Thm.~2.1]{Lavanya2016a}. It is shown
there that Fourier operators are essentially characterized uniquely by the (forward and backward)
convolution theorems, at least in the important case of the Schwartz-Bruhat functions to be treated
below (Theorem~\ref{thm:Schwartz-Bruhat}).

In the sequel, Fourier doublets will be written $\Db = [\Four\colon S \pto \Sigma]$. When referring
to \emph{elements} of the doublet~$\Db$, we identify the latter with their graphs. In other words,
every $d \in \Db$ is a pair~$d = (s, \sigma) \in S \times \Sigma$ such that~$\sigma = \Four s$.
Following classical usage~\cite[\S2; Prob.~6.30]{Bracewell1986}, we call~$d$ a \emph{Fourier
  pair}. Unlike Bracewell (who writes them~$s \subset \sigma$), we prefer the suggestive
notation~$d = [s \mapsto \sigma]$ for Fourier pairs in~$\Db = [\Four\colon S \pto \Sigma]$. Building
on widespread conventions in signal theory~\cite{BeerendsMorscheBergVrie2003,Bracewell1986}, we
refer to the~$s \in S$ as \emph{signals} and to the~$\sigma \in \Sigma$ as
\emph{spectra}. Accordingly, we call $S$ the \emph{signal space} and~$\Sigma$ the \emph{spectral
  space} of the Fourier doublet~$\Db$.

\begin{myremark}
  \label{rem:spectrum}
  While the ``spectrum of a signal'' has an immediate physical
  interpretation (in optics and acoustics), there is also a deep
  relation between Fourier analysis in the Pontryagin setting
  (Theorem~\ref{thm:pont-doublet}) and \emph{classical spectral
    theory}: As detailed in~Proposition~1.15 and Theorem~1.30
  of~\cite{Folland1994}, the Fourier
  transformation~$\Four\colon L^1(G) \to C_0(\Gamma)$ extends
  canonically to the group $C^*$-algebra~$\mathcal{A} = C^*(G)$ and
  its unitalization, yielding the Gelfand
  transformation~$\mathcal{A} \to C_0\big(\sigma(\mathcal{A})\big), a
  \mapsto \hat{a}$.  Here~$\sigma(\mathcal A)$ is the algebra spectrum
  of~$\mathcal A$, meaning its maximal ideal space, and $\hat a$ the
  homeomorphism from~$\sigma(\mathcal A)$ to the classical operator
  spectrum~$\sigma(a) = \{ \lambda \in \CC \mid \lambda 1_{\mathcal A}
  - a \text{ is singular} \}$.
\end{myremark}

The \emph{category of Fourier doublets} over a duality~$\beta$, denoted by~$\Fou\beta$, is defined
as the full subcategory of the arrow category~$\AlgH{\blnk}^\to$ generated by Fourier doublets
over~$\beta$; the corresponding morphisms are called \emph{Fourier morphisms over $\beta$}. In
detail, given two Fourier doublets $\Db = [\Four\colon S \pto \Sigma]$
and~$\Db' = [\Four'\colon S' \pto \Sigma']$, a Fourier morphism from~$\Db$ to~$\Db'$ has the
form~$(a, \alpha)$ with a left Heisenberg morphism~$a\colon S \to S'$ and a right Heisenberg
morphism~$\alpha\colon \Sigma \to \Sigma'$ such that~$\alpha \Four = \Four' a$. We refer to~$a$
and~$\alpha$, respectively, as the \emph{signal map} and \emph{spectral map} of the Fourier
morphism. Following a similar procedure as for the category~$\AlgH{\blnk}$, we have the fibration
\begin{equation}
  \label{eq:all-dbl}
  \Fou\blnk = \biguplus_{\beta \in \Du} \Fou\beta
\end{equation}
making up the category of \emph{all} Fourier doublets. Note that the Heisenberg algebras~$S$
and~$\Sigma$ in a Fourier doublet~$[\Four\colon S \pto \Sigma]$ may be slain, plain or
twain---depending on how many nontrivial multiplications they come with. Naturally, we shall also
write~$\Fou{T}$ for the full subcategory of~$\Fou\blnk$ obtained by restricting the disjoint union
in~\eqref{eq:all-dbl} to~$\beta \in \Du(T)$.

As noted in \S\ref{sub:cat-heisalg}, the category~$\AlgH{\beta}$ has products. It is then easy
to see that the same is true of~$\Fou\beta$. Indeed, given doublets~$\Db$ and~$\Db'$ as above, it is
easy to see that~$\Four \times \Four'\colon S \times S' \to \Sigma \times \Sigma'$ is a Fourier
operator so that the \emph{product doublet}~$\Db \times \Db'$ is the Fourier
doublet~$[\Four \times \Four'\colon S \times S' \pto \Sigma \times \Sigma']$.

In \S\ref{sub:free-heismod-heisalg} we have constructed the free Heisenberg
module/algebra. We shall now package them to create \emph{free Fourier doublets}. To this end, we
use the following basic result in category theory whose proof is routine.

\begin{lemma}
  \label{lem:free-arrowcat}
  Let~$\C$ be a concrete category with free functor~$\mathfrak{Z}\colon \Set \to \C$. Then the
  functor~$\mathfrak{Z}^\to\colon \Set^\to \to \C^\to$ with
  \begin{equation*}
    \mathfrak{Z}^\to(X_1 \overset{x}{\to} X_2) := \Big( \mathfrak{Z}(X_1)
    \overset{\mathfrak{Z}(x)}{\longrightarrow} \mathfrak{Z}(X_2) \Big)
  \end{equation*}
  is free, being left adjoint to the forgetful functor~$\mathfrak{U}^\to\colon \C^\to \to \Set^\to$
  that sends~$C_1 \overset{c}{\to} C_2$ to the set
  map~$\mathfrak{U}(C_1) \overset{\mathfrak{U}(c)}{\longrightarrow} \mathfrak{U}(C_2)$.
\end{lemma}

With this lemma, one can establish the \emph{free doublets} generated
by an arrow in~$\Set$ or in~$\Mod_K$.

\begin{proposition}
  \label{prop:free-doublets}
  Let~$\beta$ be a duality and~$f\colon L \to \Lambda$ a set map or
  $K$-module homomorphism.
  \begin{enumerate}
  \item Setting~$\tilde{f} := \freesln_\beta(f)$, the free slain
    doublet over~$f\colon L \to \Lambda$ is
    $[\tilde{f}\colon \freesln_\beta(L) \pto
    \freesln_\beta(\Lambda)^\land]$.
  \item Setting~$\tilde{f} := \freepln_\beta(f)$, the free plain
    doublet over~$f\colon L \to \Lambda$ is
    $[\tilde{f}\colon \freepln_\beta(L) \pto
    \freepln_\beta(\Lambda)^\land]$.
  \item Setting~$\tilde{f} := \freetwn_\beta(f)$, the free twain
    doublet over~$f\colon L \to \Lambda$ is
    $[\tilde{f}\colon \freetwn_\beta(L) \pto
    \freetwn_\beta(\Lambda)^\land]$.
  \end{enumerate}
\end{proposition}

\subsection{Classical Pontryagin Duality.}\label{sub:class-pont-duality}%
Now is a good time to contemplate the most crucial example of a Fourier doublet in classical Fourier
analysis---the \emph{Fourier transform} on LCA groups. This is in fact a very extensive class of
examples since one may start from an arbitrary Pontryagin duality.

Recall from Example~\ref{ex:pont-duality-first} that for any LCA group~$G$ there is an LCA
group~$\Gamma = \hat{G}$, called the group dual to~$G$, such that the natural pairing
\begin{equation}
  \label{eq:pont-duality}
  \pomega_G\colon G \times \Gamma \to \Tor,\quad
  (x, \xi) \mapsto \inner{x}{\xi} \equiv \xi(x)
\end{equation}
is a duality, known as the \emph{Pontryagin duality} for~$G$ and~$\Gamma$. As long as no
confusion is likely, we shall suppress the index and just write~$\pomega$ for the Pontryagin duality
in question.

Both LCA groups~$G$ and~$\Gamma$ give rise to natural algebras, which
are connected by the Fourier transform. In detail, we have the complex
vector space~$L^1(G)$ consisting of all functions on the topological
group~$G$ that are absolutely integrable with respect to Haar
measure. It is well known that they form a complex algebra under
convolution~$\star$, which we call the \emph{convolution
  algebra}. Some sources~\cite[p.~vi]{Rudin2017} dub it the
(topological) group algebra\footnote{Clearly, this reduces to the
  plain group algebra~\cite[\S II.3]{Lang2002} when~$G$ is considered
  from a purely algebraic viewpoint, i.e.\@ given the discrete
  topology).}  of~$\pomega$.

We define a left \emph{Heisenberg action}~$H(\pomega) \times L^1(G) \to L^1(G)$ by letting the
torus~$\Tor \hookrightarrow \nnz{\CC}$ act naturally via the embedding while setting
\begin{align}
  \label{eq:pont-actions-G}
    (x \act s)(y) &= s(y-x),& (\xi \act s)(y) &= \inner{\xi}{y} \, s(y)
\end{align}
for all~$x \in G$, $\xi \in \Gamma$ and~$s \in L^1(G)$.
Here~$G \act L^1(G) \subseteq L^1(G)$ follows from translation
invariance of Haar measure
while~$\Gamma \act L^1(G) \subseteq L^1(G)$ is clear
because~$|\inner{\xi}{y}| \le 1$. We refer to the two actions
of~\eqref{eq:pont-actions-G}, respectively, as \emph{translation and
  modulation} on~$G$ because of their most important instantiation
(Example~\ref{ex:pont-doublet}\ref{it:four-int}).

On the dual group~$\Gamma$, we set up the space~$C_0(\Gamma)$ of
bounded continuous functions~$\Gamma \to \CC$ vanishing at infinity as
in~\cite[A11]{Rudin2017}. Clearly, this is a complex algebra
$\big(C_0(\Gamma), \cdot\big)$ under pointwise multiplication, which
we call the \emph{pointwise algebra} of~$\pomega$. Setting up the
right Heisenberg action $C_0(\Gamma) \times H(\pomega) \to C_0(\Gamma)$
in complete analogy to the convolution algebra, we define translation
and modulation on~$\Gamma$ by
\begin{align}
  \label{eq:pont-actions-Gamma}
    (\sigma \act \xi)(\eta) &= \sigma(\eta-\xi), & (\sigma \act x)(\eta) &= \inner{\eta}{x} \, \sigma(\eta).
\end{align}
The closure properties are again evident: We
have~$\Gamma \act C_0(\Gamma) \subseteq C_0(\Gamma)$ by the
continuity of~$\eta \mapsto \xi+\eta$ and
$G \act C_0(\Gamma) \subseteq C_0(\Gamma)$ by that
of~$\inner{x}{}$. The torus action is of course again via the
embedding~$\Tor \hookrightarrow \nnz{\CC}$.

With the algebras~$L^1(G)$ and~$C_0(\Gamma)$ in place, we can now
define the \emph{forward and backward Fourier transform}
\begin{equation}
  \label{eq:fourier-transform}
  \left\{
  \begin{array}{r@{\hspace{0.75ex}}c@{\hspace{0.75ex}}lr@{\hspace{0.75ex}}c@{\hspace{0.75ex}}l}
  \Four^{\land}\colon L^1(G) &\pto& C_0(\Gamma),& \Four\! s(\xi) &:=& \cumG \, \inner{\xi}{\mathord+y} \, s(y)
  \, dy,\\[1.5ex]
  \Four^{\lor}\colon L^1(G) &\pppto& C_0(\Gamma),& \Four\! s(\xi) &:=& \cumG \, \inner{\xi}{\mathord-y} \, s(y)
  \, dy.
  \end{array}
  \right.
\end{equation}
Of course, we have to ensure that they are indeed Fourier operators
and thus deserve their name (Theorem~\ref{thm:pont-doublet} below).

\begin{myremark}
  \label{rem:four-signs}
  Fourier transforms are plagued with a multitude of arbitrary
  conventions (overall factors, signs and factors in the exponent,
  signs in the Heisenberg actions, etc.), and there seems to be no
  compelling \emph{a priori} reason as to which transform
  in~\eqref{eq:fourier-transform} should be chosen forward and which
  backward. We may refer to the two possibilities as the
  \emph{forward-positive} and the \emph{forward-negative} sign
  conventions. Using this jargon, we have thus adopted the
  forward-positive convention in this paper. In the Chapter ``A Plus
  or Minus Sign in the Fourier Transform?\@'' of the applied
  monograph~\cite{VoelklAllardJoy1999} on electron holography, the
  authors have made the same choice:

  \begin{quotation}
    The sign of the exponential in the Fourier transform is something
    that we have been concerned with for many years. Of course, there
    are two conventions that have been used with almost equal
    frequency [...] we have used the convention of the positive sign
    in the exponential for the forward transform which represents the
    Fraunhofer diffraction pattern for a real-space object [...]
    If the other convention is to be used for the Fourier transform
    exponent sign, then all authors should be advised of all these
    other implications, which are not immediately obvious. Otherwise,
    we might find ourselves producing a treatment of positron
    holography!
  \end{quotation}

  While there may be physical reasons for preferring one or the other convention, there is little
  ground for preference outside applications. In \emph{classical Fourier analysis}
  (Example~\ref{ex:pont-doublet}\ref{it:four-int}), both sign conventions are to be found---see for
  example~\cite{Bracewell1986} versus~\cite{Strichartz1994}, and note the \emph{Warning} on page~29
  of~\cite{Strichartz1994}. In \emph{abstract Fourier analysis}, however, the forward-minus
  convention appears to be more common~\cite{Folland1994}, \cite{Rudin2017}, \cite{Loomis2013}.

  We have picked the forward-plus
  convention~\eqref{eq:fourier-transform} since it meshes nicely with
  our \emph{abstract approach} (see the Heisenberg clock in
  Figure~\ref{fig:heis-clock}): After fixing the Heisenberg group in
  the form~$H(\beta) = TG \rtimes \Gamma$, the tilt
  map~$j\colon P \to P\trp$ is the natural choice to march ``forward''
  (example~(b) in \S\ref{sub:heis-twist}), which induces the
  forward twist in the form~\eqref{eq:fwd-twist}. But of course this
  does not mean our setup is written in stone: In different
  circumstances, other combinations of the various
  conventions---Heisenberg group, Heisenberg action, Heisenberg
  twists---may prove to be better suited.
\end{myremark}

As explained after Definition~\ref{def:par-op}, we can also
define~$\Four^\lor$ from~$\Four := \Four^\lor$ since we have the
natural \emph{reversal operators}~$\Par\colon L^1(G) \ppto L^1(G)$ as
well as $\Par\colon C_0(\Gamma) \ppto C_0(\Gamma)$
with~$(\Par s)(y) = s(-y)$ and~$(\Par \sigma)(\eta) =
\sigma(-\eta)$. Together with these, the convolution algebra~$L^1(G)$
and the pointwise algebra~$C_0(\Gamma)$ make up the prototypical
example of a \emph{Fourier doublet}.

\begin{theorem}
  \label{thm:pont-doublet}
  Let~$G$ and~$\Gamma$ be LCA groups under Pontryagin duality
  $\pomega\colon \Gamma \times G \to \Tor$.
  Then~$[\Four\colon L^1(G) \pto C_0(\Gamma)]$ is a Fourier doublet
  with reversal operators~$\Par\colon L^1(G) \ppto L^1(G)$
  and~$\Par\colon C_0(\Gamma) \ppto C_0(\Gamma)$.
\end{theorem}
\begin{proof}
  All these facts are very simple or otherwise well-known, so for the most part it will suffice to
  provide suitable pointers to the literature. We have to check the following facts:
  \begin{enumerate}
  \item\emph{Convolution algebra:} We refer to Theorems~1.1.6/1.1.7
    in~\cite{Rudin2017} for the well-known fact
    that~$\big(L^1(G), \star\big)$ is a Banach
    algebra. Since~\eqref{eq:pont-actions-G} are group actions, the
    laws~\eqref{eq:left-unit}--\eqref{eq:right-assoc} are satisfied,
    and it suffices to verify
    conditions~\eqref{eq:torus-act}--\eqref{eq:right-op}. The
    relation~\eqref{eq:left-scal} between convolution and translation
    is well-known~\cite[p.~51]{Folland1994}. For
    checking~\eqref{eq:right-op}, we evaluate the right-hand
    side~$(s \act \xi) \star (\tilde{s} \act \xi)$ at a
    point~$y \in G$ to obtain
    \begin{align*}
      & \qquad \cumG \, \inner{\xi}{z-y} \, s(y-z) \, \inner{\xi}{-z} \, \tilde{s}(z) \, dz\\
      & \quad\qquad = \inner{\xi}{-y} \, \cumG \, s(y-z) \, \tilde{s}(z) \, dy
        = \inner{\xi}{-y} \, (s \star \tilde{s})(y),
    \end{align*}
    which is the left-hand side~$(s \star \tilde{s}) \act \xi$
    evaluated at~$y$. The torus action law~\eqref{eq:torus-act} holds
    trivially for the embedding~$\Tor \hookrightarrow \nnz{\CC}$.
    Finally, \eqref{eq:twisted-bimod} follows from the fact
    that~$\inner{\xi}{}$ is a homomorphism. We have thus verified
    that~$L^1(G) \in \AlgH{\pomega}$.
  \item\emph{Pointwise algebra:} Again, it is well known
    that~$\big(C_0(\Gamma), \cdot\big)$ is a Banach algebra; see for
    example Appendix A12 in~\cite{Rudin2017}. It is again clear
    that~\eqref{eq:pont-actions-Gamma} constitute group actions, so it
    suffices to check~\eqref{eq:torus-act}--\eqref{eq:right-op}. This
    time, \eqref{eq:left-scal} follows directly from the associativity
    of~$(\CC, \cdot)$ while~\eqref{eq:right-op} is the statement that
    translation is a homomorphism and~\eqref{eq:torus-act} is again
    trivial. It remains to show the transposed version
    of~\eqref{eq:twisted-bimod},
    namely~$(\sigma \act \xi) \act x = \inner{\xi}{x} \, (\sigma
    \act x) \act \xi$, which now follows from~$\inner{}{x}$ being a
    homomorphism. This
    establishes~$C_0(\Gamma) \in \AlgH{\pomega\trp}$.
  \item\emph{Fourier transform:} It is well-known that the Fourier
    transform $\Four = \Four^\land$ of~\eqref{eq:fourier-transform} is
    a
    homomorphism~$\big( L^1(G), \star \big) \to \big( C_0(\Gamma),
    \cdot \big)$ of $\CC$-algebras; see for example Theorems~1.2.2
    and~1.2.4(b) of~\cite{Rudin2017}. Hence $\Four$ respects also the
    trivial torus action, and it remains to show the two
    relations~\eqref{eq:cofourop}. They follow from the homomorphism
    property, respectively, of~$\inner{\eta}{}$ and~$\inner{}{y}$,
    employing a linear substitution in the integral and appealing to
    translation invariance of Haar measure on~$G$. The corresponding
    relations~\eqref{eq:contfourop} for~$\Four^\lor$ are automatic
    since~$\Four$ is symmetric (Item~\ref{it:rev-op} below), but they
    can be established directly in an analogous manner.
  \item\label{it:rev-op}\emph{Reversal operators:} It is easy to see
    that~$\Par\colon L^1(G) \to L^1(G)$ and~$\Par\colon C_0(\Gamma) \to C_0(\Gamma)$ are involutive
    reversal operators, so both~$L^1(G)$ and~$C_0(\Gamma)$ are symmetric Heisenberg algebras. For
    seeing that~$\Four$ commutes with the reversal operators, one uses the
    substitution~$y \mapsto -y$ in the integral.
  \end{enumerate}
  It is easy to see that pre- or postcomposing by the reversal
  operators~$\Par$ exchanges~$\Four^\land$ and~$\Four^\lor$.
\end{proof}

\begin{myremark}
  \label{rem:four-op-from-sigspc}
  Using the operator~$\evl$ of evaluation at~$0 \in G$, just as in Example~1
  of~\cite{RosenkranzSerwa2019} where~$G = \RR$, the definition of the Fourier
  transformation~\eqref{eq:fourier-transform} may be written in the concise form
  $\Four^\land \! s (\xi) = \evl(s \star \Par\xi)$ and~$\Four^\lor \! s (\xi) = \evl(s \star
  \xi)$.
  But note the following provisos: In general, the characters $\xi \in \Gamma$ are not elements in
  $L^1(G)$; they are only when $G$ is compact. Nevertheless, they are always in $L^\infty(G)$ so
  that the convolution $s \star \xi$ is continuous by Proposition (2.39d)
  of~\cite{Folland1994}. This is why one may apply the evaluation operator $\evl$, which is not
  normally possible on functions in~$L^1(G)$.  In fact, we follow here the \emph{purely algebraic
    setting} as in differential algebra (where functions are viewed as elements in a ring carrying a
  derivation), so evaluation is not available: neither for the continuous elements of the signal
  space~$S = L^1(G)$ nor for those of the spectral space~$\Sigma = C_0(\Gamma)$, so also the
  left-hand side~$\Four s (\xi)$ of the above definition is not feasible in our present setting.
\end{myremark}

\begin{myremark}
  It should also be mentioned that Fourier operators are sometimes used like \emph{quantifiers}. So
  if~$\mathfrak{T}$ is a term containing the free variable~$x$ such that~$x \mapsto \mathfrak{T}$
  constitutes a signal in~$L^1(G)$, we shall write~$\Four_x \mathfrak{T}$ for the
  spectrum~$\Four(x \mapsto \mathfrak{T})$. This may be somewhat pedantic (in practical applications
  the subscript~$x$ is often suppressed), but it may prevent ambiguities. 
\end{myremark}

We refer to~$[\Four\colon L^1(G) \pto C_0(\Gamma)]$ as the \emph{classical Fourier doublet} of the
Pontryagin dualty~$\pomega\colon \Gamma \times G \to \Tor$. As a shorthand, we shall also write this
doublet as~$L^1\inner{\Gamma}{G}_\pomega$ or briefly~$L^1\inner{\Gamma}{G}$ when the Pontryagin
duality is clear from the context. It is easy to see
that~$\pomega \mapsto L^1\inner{\Gamma}{G}_\pomega$ is a functor~$\Du(\Tor) \to \Fou{\Tor}$. We
transfer the symmetric monoidal structure of~$\Du(\Tor)$ to its essential image (generally defined
as the full subcategory generated by the image objects). Up to isomorphism, we thus define the
\emph{tensor product doublet}
\begin{equation*}
  L^1\inner{\Gamma}{G}_\pomega \otimes L^1\inner{\Gamma'}{G'}_{\pomega'} := L^1\inner{\Gamma
  \oplus \Gamma'}{G \oplus G'}_{\pomega\otimes\pomega'},
\end{equation*}
with Fourier
operator~$\Four \otimes \Four' \colon L^1(G \times G') \pto C_0(\Gamma \times \Gamma')$. Obviously,
we may extend this to tensor products with finitely many factors.

Using Fubini's theorem~\cite[Prop.~I.46]{Nachbin1976}, it is easy to see that \emph{Fourier operators
  are multiplicative}. To make this precise, let~$\pomega_i\colon \Gamma_i \times G_i \to \Tor$
for~$i \in [n] := \{1, \dots, n\}$ be Pontryagin dualities with corresponding Fourier
operators~$\Four_i\colon L^1(G_i) \to C_0(\Gamma_i)$. Writing the product duality
as~$\pomega\colon \Gamma \times G \to \Tor$ with~$G := G_1 \oplus \cdots \oplus G_n$
and~$\Gamma := \Gamma_1 \oplus \cdots \oplus \Gamma_n$, the induced Fourier
operator~$\Four_1 \otimes \cdots \otimes \Four_n$ on the tensor product is denoted
by~$\Four\colon L^1(G) \to C_0(\Gamma)$. For any~$a \subseteq [n]$ with complement
$a' \subseteq [n]$ and any~$i \in [n]$, we consider the \emph{hybrid groups}
\begin{equation*}
  F_a := \bigoplus_{j \in [n]} F_j,\quad
  F_a(i) := \bigoplus_{j \in [n] \setminus \{ i \}} F_j\quad\text{with}\quad
  F_j = \begin{cases} G_j & \text{if $j \in a$,}\\\Gamma_j & \text{if $j \in a'$,} \end{cases}
\end{equation*}
which clearly satisfy~$F_a \cong F_a(i) \oplus F_i$. This
yields~$\CC^{F_a} \isomarrow (\CC^{F_i})^{F_a(i)}$ as currying isomorphism, which we
write~$f \mapsto f_i$. Then we define the \emph{hybrid function spaces}~$\mathrm{LC}_a$ as the set
of all~$f \colon F_a \to \CC$ such that $f_i(z) \in L^1(G_i)$ for all $i \in a$, $z \in F_a(i)$
and~$f_i(z) \in C_0(\Gamma_i)$ for all $i \in a'$, $z \in F_a(i)$. Note
that~$\mathrm{LC}_{[n]} = L^1(G)$ and~$\mathrm{LC}_{\emptyset} = C_0(\Gamma)$. Given~$i \in a$, we
set~$\Four_i'\colon \mathrm{LC}_a \to \mathrm{LC}_{a \setminus \{i\}}$
by~$\Four_i'(f)(z) := \Four(f_i(z))$. Then we have
\begin{equation}
  \label{eq:fourop-mult}
  \Four = \Four_{\sigma_1}' \circ \ldots \circ \Four_{\sigma_n}
\end{equation}
for all permutations~$\sigma \in S_n$, by Fubini's theorem as quoted above. In practice, one often
selects special variables~$\mathsf{x}_1, \dots, \mathsf{x}_n$ ranging over the positions
groups~$G_1, \dots, G_n$ and~$\mathsf{x} = (\mathsf{x}_1, \dots, \mathsf{x}_n)$ ranging over~$G$;
then one can use~$\Four_{\mathsf{x}_i} \mathsf{T} := \Four_i (\mathsf{x_i} \mapsto \mathsf{T})$
and~$\Four_{\mathsf{x}} \mathsf{T} := \Four (\mathsf{x} \mapsto \mathsf{T})$ like quantifiers on
terms~$\mathsf{T}$ containing free occurrences of~$\mathsf{x}_1, \dots, \mathsf{x}_n$. These
conventions are similar to those for differential and integral operators.

\begin{myexample}
  \label{ex:pont-doublet}
  At this point, it may be useful to review the four most important incarnations of Pontryagin
  duality---the standard Fourier operators of analysis (note that~\eqref{eq:fourop-mult} is
  applicable in each of these cases):
  \begin{enumerate}[(a)]
  \item\label{it:four-int} The classical \emph{Fourier integral} (FI) arises when considering the
    duality given by the \emph{standard vector duality}~$\inner{G}{\Gamma} = \inner{\RR^n}{\RR_n}$
    of Example~\ref{ex:classical-vector-group}, where~$\inner{\xi}{x} = e^{i \tau x \cdot \xi}$. In
    this case, we have
    \begin{equation}
      \label{eq:four-int}
      \Four s \, (\xi) = \int_{\RR^n} e^{i \tau x \cdot \xi} \, s(x) \, dx
    \end{equation}
    for the Fourier transform. (See the remarks in
    Example~\ref{ex:pont-doublet-inverse}\ref{itt:four-int} on the topic of alternative
    normalizations.)

    In this context, the mapping property~$\Four\colon L^1(G) \to C_0(\Gamma)$ is known as the
    Riemann-Lebesgue lemma~\cite[Prop.~6.6.1]{Stade2011}, in particular the fact
    that~$\hat{s}(\xi) \to 0$ as~$|\xi| \to \infty$. Moreover, the homomorphism
    property~$\Four(s \star s') = \Four s \cdot \Four s'$ is called the \emph{convolution theorem}.
    Writing~$s_a (x) := s(x+a)$ for the translates of a signal~$s$ by an offset $a \in \RR^n$, the
    two equivariance properties
    \begin{equation}
      \label{eq:mod-shift-thm}
      \quad
      \Four_x\big( s(x+a) \big) = e^{i\tau a \cdot \xi} \, \hat{s}(\xi)
      \quad\text{and}\quad
      \Four_x\big( e^{i\tau x \cdot \alpha} s(x) \big) = \hat{s}(\xi-\alpha)
    \end{equation}
    are known as the \emph{shift theorem} and the \emph{modulation theorem},
    respectively~\cite[\S6]{Bracewell1986} since translations are obviously also called shifts while
    multiplying with exponentials~$e^{i\tau a \cdot \xi}$ and~$e^{i\tau x \cdot \alpha}$ is known as
    modulating in engineering parlance. More precisely, this would be frequency modulation (FM):
    Taking\footnote{While~$s_\nu \not\in L^1(\RR^n)$ is technically not a signal in our present
      setting, it can be approximated by $L^1$ signals.} a sinusoidal
    signal~$s_\nu(x) = e^{i\tau x\cdot \nu}$ of frequency~$\nu$, one obtains the modulated
    signal~$e^{i \tau x \cdot \alpha} s_\nu(x) = s_{\nu+\alpha}(x)$ with altered
    frequency~$\nu+\alpha$.
  \item\label{it:four-ser} Taking the \emph{conjugate torus
      duality}~$\inner{\Gamma}{G} = \inner{\ZZ^n}{\Tor^n}$ of Example~\ref{ex:classical-torus-group}
    for the Pontryagin duality, we obtain \emph{Fourier series} (FS). More precisely, the
    multivariate sequence~$\Four s \, (\xi)_{\xi \in \ZZ^n}$ formed by the so-called Fourier
    coefficients
    \begin{equation}
      \label{eq:four-ser}
      \Four s \, (\xi) = \int_{\II^n} e^{i \tau x \cdot \xi} \, s(x) \, dx
    \end{equation}
    will be seen to constitute the Fourier series of~$s$; see
    Example~\ref{ex:pont-doublet-inverse}\ref{itt:four-ser}. Here we identify
    signals~$s \in L^1(\Tor^n)$ with periodic functions defined on~$\II^n$ rather than the more
    usual~$[0,\tau]^n$, where the change of variables~$y = \tau x$ yields an additional
    factor~$\tau^{-n}$. The advantage of this choice is to achieve a more uniform expression for the
    Fourier transformation: One sees immediately that~\eqref{eq:four-int} and~\eqref{eq:four-ser}
    differ only in their integration bounds. Correlated to these signals, their spectral
    space~$C_0(\Gamma)$ is the space~$c_0(\ZZ^n)$ of multivariate null sequences~$\ZZ^n \to \CC$;
    this is again an instance of the Riemann-Lebesgue lemma~\cite[Cor.~6.45]{Knapp2005b}. Note,
    however, that~$\Four$ is injective but not surjective~\cite[p.~547]{Knapp2005b}. There are again
    convolution, shift and modulation theorems.
  \item\label{it:four-dtft} Interchanging the roles of position and momenta, we obtain the
    \emph{torus duality}~$\inner{\Gamma}{G} = \inner{\Tor^n}{\ZZ^n}$; its associated Fourier
    transform is called~\cite[\S18.5]{BeerendsMorscheBergVrie2003} the \emph{discrete-time Fourier
      transform} (DTFT), with the corresponding Fourier operator
    $\Four\colon l^1(\ZZ^n) \to C(\Tor^n)$ given by
    \begin{equation}
      \label{eq:four-dtft}
      \Four s \, (\xi) = \sum_{x \in \ZZ^n} e^{i \tau x \cdot \xi} \, s(x),
    \end{equation}
    which may also be viewed as a discretized version of the Fourier
    integral~\eqref{eq:four-int}. In this case, the Riemann-Lebesgue lemma is void (as the
    torus~$\Tor^n$ is compact every continuous function vanishes at infinity). Of course one has the
    usual convolution, shift and modulation theorems. Note that convolution takes its usual form by
    writing the sequences~$s \in l^1(\ZZ^n)$ as multivariate series~$\sum_{x \in \ZZ^n} s_n x^n$.
    
    At this point it should also be clear why it makes sense---from a purely mathematical point of
    view---to distinguish the two \emph{``mirror images''} of the torus duality (confer
    Example~\ref{ex:classical-torus-group}): We obtain different Fourier
    operators~$L^1(\Tor^n) \to c_0(\ZZ^n)$ and~$l^1(\ZZ^n) \to C(\Tor^n)$, whose mapping spaces are
    obviously quite different. It is only on suitable subspaces that we may subsequently identify
    them as essentially inverse to each other
    (Example~\ref{ex:pont-doublet-inverse}\ref{itt:four-dtft}), linking continuous periodic with
    discrete aperiodic signals (see below).
  \item\label{it:four-dft} Finally, let us take the \emph{conjugate cyclic
      duality}~$\inner{\Gamma}{G} = \inner{\ZZ_N^n}{\Tor_N^n}$ from
    Example~\ref{ex:finite-group}. We recall that both~$\ZZ_N = \{ 0, \dots, N-1 \}$ and
    $\Tor_N = N^{-1} \ZZ_N$ are the cyclic group~$\ZZ/N$, but while~$\Tor_N \hookrightarrow \Tor$ is
    naturally embedded, its dual partner is canonically endowed with a
    projection~$\ZZ \twoheadrightarrow \ZZ_N$. In this case, \eqref{eq:fourier-transform} will be
    the \emph{discrete Fourier series} (DFS) given in detail by
    \begin{equation}
      \label{eq:four-dfs}
      \Four s \, (\xi) = \frac{1}{N^n} \, \sum_{x \in \Tor_N^n} e^{i \tau x \cdot \xi} \, s(x)
    \end{equation}
    where~$\xi$ ranges over~$\ZZ_N^n$. It is obvious that~\eqref{eq:four-dfs} is the uniformly
    sampled form of the Fourier coefficient~\eqref{eq:four-ser} associated with Fourier series. Note
    that the factor~$1/N^n$ arises here from discretizing the constituent integrals
    of~\eqref{eq:four-ser} via
    $\cum_0^1 \dots \, dx_i \rightsquigarrow \sum_{x_i \in \Tor_N} \dots \, N^{-1}$; this is in
    harmony with the chosen normalization of the Haar measure (see the concluding remark in
    Example~\ref{ex:finite-group}). It is well known that uniform sampling in one domain corresponds
    to periodic repetition in the other~\cite[\S7.4]{Roberts2012}, so the resulting spectrum
    under~\eqref{eq:four-ser} will be determined by its values on~$\xi \in \ZZ_N^n$. Altogether we
    obtain a transform of type~$\Four\colon L^1(\Tor_N^n) \to C_0(\ZZ_N^n)$.

    Changing to the cyclic duality~$\inner{\Tor_N^n}{\ZZ_N^n}$, we obtain now the \emph{discrete
      Fourier transform} (DFT) given by
    \begin{equation}
      \label{eq:four-dft}
      \Four s \, (\xi) = \sum_{x \in \ZZ_N^n} e^{i \tau x \cdot \xi} \, s(x),
    \end{equation}
    with~$\xi$ ranging over~$\Tor_N^n$. We may
    view~\eqref{eq:four-dft} as a sampled form of the discrete-time
    Fourier transform~\eqref{eq:four-dtft}. Restricting the latter to
    signals~$s \in L^1(\ZZ^n)$ supported
    within~$\{0, \dots, N-1\}^n \subset \ZZ^n$, the infinite
    series~\eqref{eq:four-dtft} collapses to the finite
    sum~\eqref{eq:four-dft}. The resulting
    spectrum~$\Four s \in C_0(\Tor^n)$ is subsequently sampled at
    unit-root coordinates~$\Tor_N^n \subset \Tor^n$ as these are
    sufficient to reconstruct the signal: The original signal~$s$ is
    again periodically replicated because of the uniform sampling
    of~$\Four s$, but due to the support hypothesis no aliasing occurs
    and exact reconstruction is ensured. It should be noted, however,
    that the original signal~$s$ has now been identified as
    \emph{periodic}, which is inconsistent with our prior assumption
    of finite support. The contradiction arises only from the group
    structure---on the set level, we are free to choose between
    interpreting complex tuples as representing periodic signals (as
    for the DFS) or finite signals (as for the DFT).

    Indeed, all the spaces~$L^1(\ZZ_N^n), L^1(\Tor_N^n), C_0(\ZZ_N^n), C_0(\Tor_N^n)$ are in fact
    the same plain vector space~$(\CC_N)^n$, and the transformation is the tensor
    power~$\Four_1^{\otimes n}$ of a linear map~$\Four_1\colon \CC^N \to \CC^N$. Up to scaling, the
    matrix of~$\Four_1$ with respect to the canonical basis is~\cite[Thm.~39.2]{Howell2016} the
    Vandermonde matrix generated by the $N$-th roots of unity~$\Tor_N \subset \CC$. This is the form
    commonly used~\cite[\S16.2]{BeerendsMorscheBergVrie2003} for the DFS or DFT, with plain integer
    tuples~$k, l \in \{0, \dots, N-1\}^n$ in the exponential $e^{i\tau (k \cdot l)/N}$ mentioned
    earlier~\eqref{eq:finite-duality}.
  \end{enumerate}
  For physical signals (where~$x$ is time and hence~$\xi$ is
  frequency), the characteristics of the four transform types may be
  read off from their domains and codomains: Signals on the compact
  domains~$\Tor^n, \ZZ_N^n$ are considered \emph{periodic}, those on
  the noncompact domains~$\RR^n, \ZZ^n$ accordingly
  \emph{aperiodic}. In a similar fashion, signals on the discrete
  domains~$\ZZ^n, \ZZ_N^n$ are of course called \emph{discrete}, those
  on the nondiscrete domains~$\RR^n, \Tor^n$ accordingly
  \emph{continuous}. Using this terminology, the various Fourier
  operators are classified in Figure~\ref{fig:class-four}, where the
  box around each domain signifies a suitable function space (like
  $L^1$ or $C_0$), and the labels on the arrows refer to the
  corresponding Fourier transform (using the abbreviations given in
  the text above). Similar diagrams are often found in the pertinent
  literature; see for example Table~5.3
  in~\cite[p.~396]{OppenheimWillsky2013} or Figure~8.2
  in~\cite[p.~145]{Smith1997}.

  \begin{figure}
    \centering\vspace{1.6em}
    $\xymatrix @M=0.5pc @R=1pc @C=2pc%
    { \ar@{-}[]+<3.75em,0.75em>;[dd]+<3.75em,-1em> \ar@{-}[]+<-3em,-0.75em>;[rr]+<6.75em,-0.75em> 
      & \text{\kern-1em CONTINUOUS} & \text{\kern3.5em DISCRETE}\\
      \text{APERIODIC} & *!R{\text{FI}\:\text{\Large$\circlearrowright$}\:\boxed{\RR}\quad}
      & *!L{\quad\boxed{\ZZ}} \ar@<0.75ex>^-{\text{DTFT}}[dl]\\
      \text{PERIODIC} & *!R{\boxed{\Tor}\quad} \ar@<0.75ex>^-{\text{FS}}[ur]
      & *!L{\quad\boxed{\ZZ_N}\:\text{\Large$\circlearrowleft$}\:\text{DFT}}}$
    \medskip
    \caption{Classical Fourier Operators}
    \label{fig:class-four}
  \end{figure}
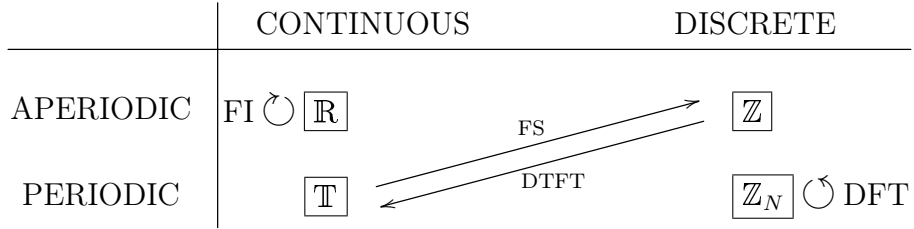

  Note that in Figure~\ref{fig:class-four} we have chosen to
  \emph{conflate DFS and DFT}, using the latter term for both. As
  mentioned earlier (Example~\ref{ex:finite-group}), their underlying
  dualities are in fact the same apart from the inessential
  normalization factor. So from a purely \emph{mathematical}
  viewpoint, no distinction is needed. (As alluded to above, one might
  even argue against finite-duration signals as inconsistent with the
  underlying group structure.) Many texts on digital signal processing
  such as~\cite{Smith1997} have therefore also chosen to neglect the
  difference. As we have seen, there are nevertheless strong
  \emph{physical} arguments in favor of upholding the distinction
  between discretized bounded but aperiodic signals (construed
  over~$\ZZ_N^n$ in our setting) and discretized periodic signals
  (correspondingly construed over~$\Tor_N^n$); some sources such
  as~\cite{OppenheimSchafer2010} or~\cite{Smith1997} follow this
  line. Indeed, the former has even chosen to adapt the normalizations
  so that the DFS formula~\cite[(8.11/12)]{OppenheimSchafer2010}
  becomes \emph{identical} with the corresponding DFT
  one~\cite[(8.65/66)]{OppenheimSchafer2010}, except for the
  truncation enforced in the latter. (These formulae also include the
  inverse transformations, which we will encounter in
  Example~\ref{ex:pont-doublet-inverse}\ref{itt:four-dft}.)
\end{myexample}

It should also be mentioned that the classical Fourier doublet~$L^1\inner{\Gamma}{G}$ for a
Pontryagin duality~$\pomega\colon \Gamma \times G \to \Tor$ can be extended via the \emph{measure
  algebra} $M(G) \supseteq L^1(G)$ consisting of all bounded regular Borel measure on~$G$. It is a
standard fact~\cite[Cor.~1.3.2]{Rudin2017} of Fourier analysis that~$M(G)$ is indeed a \emph{unital
  algebra} (actually a Banach algebra) over~$\CC$. The Fourier transform can be extended
from~$L^1(G)$ to~$M(G)$ with values in the unital algebra~$BC(\Gamma)$ of bounded and uniformly
continuous functions, and the resulting map~$\Four\colon M(G) \to BC(\Gamma)$ is again a
homomorphism that represent the natural action of~$H(\pomega)$; see~\cite[Thm.~1.3.3]{Rudin2017} and
its proof. It is usually called the \emph{Fourier-Stieltjes transformation}. Moreover, it is
straightforward~\cite[\S1.3.4]{Rudin2017} that $L^1(G)$ is a subalgebra of~$M(G)$,
while~$BC(\Gamma)$ is clearly a superalgebra of~$C_0(\Gamma)$. Taken together, this yields the
\emph{measure doublet}~$[\Four\colon M(G) \pto BC(\Gamma)]$.

\begin{proposition}
  \label{prop:measure-alg}
  Let~$G$ and~$\Gamma$ be LCA groups under Pontryagin duality
  $\pomega\colon \Gamma \times G \to \Tor$. Then the measure doublet
  $[\Four\colon M(G) \pto BC(\Gamma)]$ is an extension doublet of $L^1\inner{\Gamma}{G}$.
\end{proposition}

\begin{myexample}
  \label{ex:four-nich}
  The discrete Fourier transform generalizes to the Nicholson
  duality~$\nu\colon \hat{G} \times G \to R_*$ mentioned in
  Example~\ref{ex:nich-duality}. Under the hypotheses stipulated
  there, we have the identification~$\Hom(G, R_*) \cong R^n$, which
  then leads~\cite[(3.11)]{Nicholson1971} to define
  $\Four\colon R[G] \to R^n, s \mapsto \hat{s}$ by
  \begin{equation}
    \label{eq:four-nich}
    \hat{s}(\xi) = \sum_{x \in G} \inner{\xi}{x} \, s(x) .
  \end{equation}
  It will be noted that this specializes to the discrete Fourier
  transform~\eqref{eq:four-dft}. By the fundamental theorem of finite
  abelian groups, $G$ decomposes as a product of cyclic groups;
  thus~\eqref{eq:four-nich} may be viewed as a multivariate
  DFT. Additionally, \eqref{eq:four-nich} specializes to the Gelfand
  transform when taking~$R = \CC$.

  The Heisenberg action is defined just as in the Pontryagin case. In
  fact, the whole setting is almost subsumed by
  Theorem~\ref{thm:pont-doublet}, the only difference being the
  torus~$R_* \neq \Tor$. It is easily checked that everything
  nevertheless goes through, so that~$[\Four\colon R[G] \pto R^n]$ is
  indeed a Fourier doublet (note that Definition~\ref{def:heis-alg}
  allows for Heisenberg algebras over rings).
\end{myexample}

\begin{myremark}
  \label{rem:sympl-Four}
  Insisting on the larger category of nilquadratic rather than Heisenberg groups, one can resort to
  the \emph{symplectic Fourier transform}~\cite[Def.~6.6]{Gosson2006}, \cite[p.~7]{Folland2016}. In
  this case, one would work with a symplectic group~$P$ without a Lagrangian splitting. For example,
  in the classical case of~$P = \RR^{2n}$, the symplectic Fourier transform is given by
  \begin{equation*}
    \Four^\# s \, (x,\xi) = \int_{\RR^{2n}} \inner{x, \xi}{y, \eta}_\omega \, s(y,
    \eta) \, dy \, d\eta
  \end{equation*}
  for ``hybrid signals'' $s \in L^1(\RR^{2n})$ that depend on \emph{position}~$x \in \RR^n$ as well
  as \emph{momentum}~$\xi \in \RR^n$. The underlying \emph{symplectic duality}
  is~$\inner{x,\xi}{y,\eta}_\omega := \inner{\eta}{x}_\beta/\inner{\xi}{y}_\beta$,
  where~$\inner{\xi}{x}_\beta = e^{i\tau x \cdot \xi}$ is the standard vector duality of
  Example~\ref{ex:classical-vector-group}. This is the multiplicative symplectic form corresponding
  to the (additive) canonical symplectic form~$\Omega_G$ for~$G=\RR^n$ under the standard
  character~$\chi\colon \RR \to \Tor$. See Example~\ref{ex:heisgrp-sympl} for its relation to
  the little Heisenberg groups~$[H]_\omega$ and~$H(\beta_\omega)$.

  If~$\iota_n\colon \RR^n \to \RR^{2n}$ and~$\pi_n\colon \RR^{2n} \to \RR^n$ are, respectively, the
  \emph{standard injections and projections} of the direct sum decomposition of the phase
  space~$P = \RR^n \oplus \RR^n$, one obtains the following commutative diagram:
  \begin{equation*}
    \xymatrix @M=1pc @R=1.5pc @C=1pc%
    { L^1(\RR^n) \ar[r]^{\Four} \ar[d]_{\iota_n^*} & C_0(\RR^n)\\
      L^1(\RR^{2n}) \ar[r]^{\Four^\#} & C_0(\RR^{2n}) \ar[u]_{\pi_n^*}}
  \end{equation*}
  This allows one to recast the standard Fourier transform in terms of the symplectic
  one. Conversely, the symplectic Fourier transform may be recovered from the standard Fourier
  transform~$\Four_2$ on~$L^1(\RR^{2n})$ via~$\Four^\# = \Four_2 \circ J^*$,
  where~$J\colon \RR^{2n} \to \RR^{2n}$ is the canonical symplectic
  matrix~$\smallmat{0}{I_n}{-I_n}{0}$. Unlike its more common counterpart~$\Four_2$, the symplectic
  Fourier transform~$\Four^\#$ is involutive.

  All these ideas generalize to the setting of arbitrary LCA groups~\cite[Ex.~5.2v,
  p.~26]{Jakobsen2016}. Having a nilquadratic
  extension~$E\colon T \rightarrowtail H \twoheadrightarrow P$ with commutator
  form~$\omega := \omega_E$, one sets~$\Four^\# s(z) = \cum_P \, \omega(z, w) \, s(w) \, dw$.
  Choosing a Lagrangian splitting~$P = G \oplus \Gamma$ induces the split exact
  sequence~$G \oset{\iota}{\rightarrowtail} P \oset{\pi}{\twoheadrightarrow} \Gamma$ and a
  factorization~$\Four = \pi^* \Four^\# \iota^*$ generalizing the above diagram. Since the direct
  sum~$P$ is an LCA group~\cite[p.~362]{Moskowitz1967}, it has its own Fourier operator~$\Four_2$,
  and one may check that~$\Four^\# = \Four_2 \circ J^*$ with~$J\colon P \to P$ defined as in the
  special case above.
\end{myremark}

Classical Pontryagin duality not only provides the prototypical example of a Fourier doublet, it
also gives rise to an important class of Fourier morphisms in the following way. Recall first that
every \emph{topological automorphism} (= homeomorphism + homomorphism) of an LCA group~$G$ is
associated with a unique positive number known as the \emph{modulus} $\delta_A$ of the given
automorphism~$A$, and the association $\Aut(G) \to \RR_{>0}, A \mapsto \delta_A$ is a group
homomorphism~\cite[Prop.~17]{Nachbin1976}. The best known case is when the group is the vector
space~$\RR^n$ and the automorphism is an invertible matrix~$A \in \RR^{n \times n}$; then the
modulus is just~$\delta_A = |\det(A)|$; see Example~2 of~\cite[p.~84]{Nachbin1976}. Every
automorphism induces a contravariant action~$L^1(G) \to L^1(G)$ via the
pullback~$A^* s := s \circ A$ and a covariant action~$C_0(\Gamma) \to C_0(\Gamma)$ sending~$\sigma$
to the map~$A_* \sigma$ given by~$\xi \mapsto \sigma(A^* \xi)$. (We write the pullback of characters
in the same way as that of~$L^1(G)$ functions since they are defined analogously.)

\begin{proposition}
  \label{prop:pont-automorphism}
  Let~$G$ and~$\Gamma$ be LCA groups under Pontryagin duality
  $\pomega\colon G \times \Gamma \to \Tor$. Then every topological automorphism~$A\colon G \to G$
  gives rise to a Fourier automorphism~$(a, \alpha)$ of~$[\Four\colon L^1(G) \pto C_0(\Gamma)]$
  having signal and spectral maps
  \begin{alignat*}{2}
    a\colon& L^1(G) \to L^1(G),\quad& s &\mapsto \delta_A \: A^* s,\\
    \alpha\colon& C_0(\Gamma) \to C_0(\Gamma),\quad& \sigma &\mapsto A^{-1}_*
    \sigma,
  \end{alignat*}
  where~$\delta_A > 0$ is the modulus of the topological automorphism~$A$.
\end{proposition}
\begin{proof}
  We have to verify~$\Four(s \circ A)(\xi) = \delta_A^{-1} \, \Four s \, (\xi \circ A^{-1})$
  for~$s \in L^1(G)$ and~$\xi \in \Gamma$. According to~\eqref{eq:fourier-transform}, the left-hand
  side is given by
  \begin{align*}
    \cum_G \inner{-y}{\xi} \, s(&A(y)) \, dy = \cum_G \inner{-A(y)}{\xi \circ A^{-1}}
    \, s(A(y)) \, dy\\
    &= \delta_A^{-1} \, \cum_G \inner{y}{\xi \circ A^{-1}} s(y) \, dy 
      = \delta_A^{-1} \, \Four s \, (\xi \circ A^{-1}),
  \end{align*}
  where the second equality uses~\cite[Prop.~II.16]{Nachbin1976}. For
  checking that~$a$ is an endomorphism on~$\big(L^1(G), \star\big)$,
  one appeals to~\cite[Prop.~II.16]{Nachbin1976} again, and one sees
  immediately that~$\alpha$ is an endomorphism
  on~$\big(C_0(\Gamma), \cdot\big)$ since $A_*^{-1}$ effects a
  substitution. It is clear that both~$a$ and~$\alpha$ are bijective,
  hence~$(a, \alpha)$ is an automorphism
  of~$[\Four\colon L^1(G) \pto C_0(\Gamma)]$.
\end{proof}

Coming back to the vector space case~$G = \RR^n$ of Proposition~\ref{prop:pont-automorphism}, the
simplest example of an automorphism is the action of a nonvanishing scalar~$a \in \RR$.
Writing~$S_a$ for both induced actions on signals and spectra, we arrive at the famous
\emph{similarity theorem} $\Four S_a = |a|^{-1} \, S_{1/a} \, \Four$ of the classical Fourier
transform~\cite[p.~108]{Bracewell1986}.

\begin{remark}
  The setting of Pontryagin duality~$\pomega\colon G \times \Gamma \to \Tor$ with its Fourier
  transform~$\Four\colon L^1(G) \to C_0(\Gamma)$ also provides a sort of \emph{integral
    operator}~$\ocum\colon \big(L^1(G), \star\big) \to \CC$ acting
  as~$s \mapsto \cum_G \, s(x) \, dx$. This is a $\CC$-algebra homomorphism since we
  have~$\ocum = \evl \circ \Four$, where the \emph{evaluation}
  $\evl\colon \big(C_0(\Gamma), \cdot\big) \to \CC$ with~$\sigma \mapsto \sigma(0)$ is itself a
  homomorphism. (The associated \emph{initialization}~$\ini := 1_{C_0(\Gamma)} - \evl$ acts as a
  ``deletion operator'' when applied to~$\Gamma = \ZZ^n$.)
\end{remark}

\subsection{Fourier Inversion.}\label{sub:inversion}
In classical as well as abstract harmonic analysis, Fourier operators are \emph{always injective}:
As we shall soon show, the operator~$\Four^\land$ as well as~$\Four^\lor$ in
Theorem~\ref{thm:pont-doublet} is in fact a monomorphism. Therefore it is reasonable to try and
adapt the function spaces in some suitable way so as to obtain a \emph{bijective Fourier operator}.

\begin{definition}
  A \emph{Fourier doublet}~$\Db = [\Four\colon S \pto \Sigma]$ over a duality~$\beta$ is called
  \emph{regular} if the Fourier operator~$\Four$ is bijective. Otherwise, the doublet $\Db$ is
  called \emph{singular}.
\end{definition}

And for avoiding cumbersome terminology, we shall from now on take the
liberty of abbreviating the term ``Fourier doublet'' by just
\emph{doublet}, in particular when qualifying it by terms such as
regular/singular or slain/plain/twain. The same applies to the term
``Fourier singlet'' to be introduced later in this section (Definition~\ref{def:singlet}).

As with many other algebraic structures, a bijective Fourier operator $\Four\colon S \pto \Sigma$ is
automatically an \emph{isomorphism} in the appropriate sense: Its inverse~$\tilde\Four$ is then a
Heisenberg morphism~$\Sigma^\land \to S$ of left Heisenberg modules over~$\beta$, or equivalently a
Heisenberg morphism~$\Sigma \to S^\land$ of right Heisenberg modules over~$\beta$. Moreover,
$\tilde\Four$ also respects the respective product(s) in the case of plain/twain
algebras~$S, \Sigma$. It is then natural to denote this situation
by~$\tilde\Four\colon \Sigma \pto S$. While all this pertains to forward
operators~$\Four = \Four^\land$, analogous statements obviously hold for backward
operators~$\Four^\lor$, using respectively~$S^\lor, \Sigma^\lor$ in place
of~$S^\land, \Sigma^\land$.

Under Pontryagin duality~$\pomega\colon \Gamma \times G \to \Tor$, one
immediately obtains a regular doublet by restricting the codomain of
the Fourier operator to the so-called \emph{Fourier
  algebra}~$A(\Gamma) := \Four \, L^1(G) \le C_0(\Gamma)$ as
in~\cite[\S1.2.3]{Rudin2017}. Keeping the same notation for the
restricted operator, we have a regular
doublet~$[\Four\colon L^1(G) \pto A(\Gamma)]$ under the hypotheses of
Theorem~\ref{thm:pont-doublet}. While this is algebraically trivial,
it should be kept in mind that it is a difficult problem to find a
suitable analytic description of the Fourier group. While general
characterizations remain elusive, there are important results for
certain classes of LCA groups such
as~\cite{RudinKatznelsonKahaneHelson1959}.

There is an analogous construction for the measure algebra~$M(G)$ of
Propostion~\ref{prop:measure-alg}. Following~\cite[\S1.3.3]{Rudin2017}, we denote the image
of~$M(G)$ under the Fourier-Stieltjes transformation~$\Four$ by~$B(\Gamma) \le BC(\Gamma)$.  This
so-called the \emph{Fourier-Stieltjes algebra} is unital, and one obtains a regular
doublet~$[\Four\colon M(G) \pto B(\Gamma)]$ since~$\Four$ is
injective~\cite[Thm.~1.3.6]{Rudin2017}. In applications, this doublet is not very important since
the inclusion of distributions such as Dirac measures~$\delta_a \in M(G)$ is preferrably achieved
via tempered distributions (see Example~\ref{prop:temp-distr} below). For the theoretical development,
however, this doublet is important because of results such as Bochner's characterization of
positive-definite functions as Fourier-Stieltjes transforms of nonnegative
measures~\cite[Thm.~1.4.3]{Rudin2017}.

The Fourier algebra gives rise to just a regular plain doublet, but it may be restricted further to
obtain a \emph{regular twain doublet}. These facts appear to be well-known in analysis folklore,
though proper references are difficult to find. We follow here the hints given
in~\cite[Ex.~6.4.5]{Stade2011}.  Assume~$\pomega\colon \Gamma \times G \to \Tor$ is a Pontryagin
duality with the classical Fourier doublet~$[\Four_\pomega\colon L^1(G) \pto C_0(\Gamma)]$
over~$\pomega$ in Theorem~\ref{thm:pont-doublet}. Then we may form the function space
\begin{equation}
  \label{eq:bij-domain}
  L^{1/1}(G, \Gamma) := L^1(G) \cap \Four_\pomega^{-1} L^1(\Gamma) = \{ s \in L^1(G) \mid \hat{s}
  \in L^1(\Gamma) \}
\end{equation}
as a subspace of~$L^1(G)$. Here the space~$L^1(\Gamma)$ in~\eqref{eq:bij-domain} is of course
defined via the conjugate duality~$\pomega\trp\colon G \times \Gamma \to \Tor$. Reversing the roles
of~$G$ and~$\Gamma$ and that of~$\pomega$ and~$\pomega\trp$ in~\eqref{eq:bij-domain}, we obtain
$L^{1/1}(\Gamma, G)$ as a subspace of~$L^1(\Gamma)$, but now with the Fourier
operator~$\Four_{\pomega\trp}$ of the corresponding doublet
$[\Four_{\pomega\trp}\colon L^1(\Gamma) \pto C_0(G)]$ over~$\pomega\trp$. As we shall see presently,
the original Fourier operator restricts to a twain doublet
$[\Four_\pomega\colon L^{1/1}(G, \Gamma) \pto L^{1/1}(\Gamma, G)]$ such that~$\Four_{\pomega\trp}$
restricts to a sign inverse of~$\Four_\pomega$.  We call
\begin{equation*}
 L^{1/1}\inner{\Gamma}{G} := [\Four_\pomega\colon L^{1/1}(G, \Gamma) \pto L^{1/1}(\Gamma, G)]
\end{equation*}
the \emph{classical twain doublet} since it is cut out from the two
classical Fourier doublets~$[\Four_\pomega\colon L^1(G) \pto C_0(\Gamma)]$
and~$[\Four_{\pomega\trp}\colon L^1(\Gamma) \pto C_0(G)]$.

\begin{lemma}
  \label{lem:L11-basics}
  Let~$G$ and~$\Gamma$ be LCA groups under Pontryagin duality
  $\pomega\colon \Gamma \times G \to \Tor$. Then we
  have~$L^{1/1}(G, \Gamma) = L^1(G) \cap A(G)$ and the inclusion
  $L^{1/1}(G, \Gamma) \subseteq L^2(G)$, which is in general strict.
\end{lemma}
\begin{proof}
  Note that~$A(G) = \Four_{\pomega\trp} L^1(\Gamma)$ is the Fourier algebra, now on the signal
  side. We have~$L^{1/1}(G, \Gamma) \subseteq L^1(G) \cap \Four_{\pomega\trp} L^1(\Gamma)$ since the
  Fourier inversion theorem~\cite[(4.32)]{Folland1994} ensures
  $s = \Four_{\pomega\trp} \Par \Four_\pomega s \in \Four_{\pomega\trp} L^1(\Gamma)$ whenever
  $s \in L^{1/1}(G, \Gamma)$. For the converse, assume~$s = \Four_{\pomega\trp} \sigma \in L^1(G)$
  for $\sigma \in L^1(\Gamma)$. By virtue of the same inversion theorem (applied on the other side),
  we have $\Par \sigma = \Four_\pomega s$ and hence~$s \in \Four_\pomega^{-1} L^1(\Gamma)$. This
  establishes the other inclusion and
  thus~$L^{1/1}(G, \Gamma) = L^1(G) \cap \Four_{\pomega\trp} L^1(\Gamma)$.

  We prove now~$L^{1/1}(G, \Gamma) \subseteq L^2(G)$.  Given~$s \in L^{1/1}(G, \Gamma)$, we have
  $s \in L^1(G)$ and~$s \in \Four_{\pomega\trp} L^1(\Gamma)$ by the identity just proved. Then we
  obtain $s \in C_0(G) \subset L^\infty(G)$ by the fundamental mapping property
  $\Four_{\pomega\trp}\colon L^1(\Gamma) \to C_0(\Gamma)$ of the Fourier transform, also known as
  abstract Riemann-Lebesgue lemma~\cite[Thm.~1.2.4a]{Rudin2017}.
  From~$s \in L^1(G) \cap L^\infty(G)$ we have
  \begin{equation*}
    ||s||_2^2 = \cum_G |s|^2 \, dx \le ||s||_\infty \, \cum_G |s| \, dx
    = ||s||_\infty ||s||_1 < \infty
  \end{equation*}
  and thus~$s \in L^2(G)$. For the important special case $L^{1/1}(G, \Gamma) \subseteq L^2(G)$
  with~$G = \RR$, see also~\cite[Prop.~6.4.1]{Stade2011}.

  For seeing that the inclusion is in general strict, consider the normalized cardinal sine
  function, $\sinc x := \tfrac{\sin \pi x}{\pi x}$. We have~$\sinc \in L^2(\RR)$; in fact,
  $\cum_\RR \sinc^2 x \, dx = 1$. Moreover, $\sinc$ is also integrable, again with the
  value~$\cum_\RR \sinc x \, dx = 1$. However, it is \emph{not absolutely integrable} since
  $\cum_\RR |\sinc x| \, dx = \infty$.  Hence~$\sinc \not\in L^{1/1}(\RR) \subseteq L^1(\RR)$.
\end{proof}

This lemma implies that the
restrictions~$\Four_\pomega\colon L^{1/1}(G, \Gamma) \to L^{1/1}(\Gamma, G)$
and~$\Four_{\pomega\trp}\colon L^{1/1}(\Gamma, G) \to L^{1/1}(G, \Gamma)$ are well-defined. We shall
now prove that~$L^{1/1}\inner{\Gamma}{G}$ is a subdoublet of~$L^1\inner{\Gamma}{G}$ having the
alleged properties. Here we define the notion of a
\emph{subdoublet}~$[\Four'\colon S' \pto \Sigma']$ of a doublet $[\Four\colon S \pto \Sigma]$ by
requiring the inclusions~$i\colon S' \hookrightarrow S$
and~$\iota\colon \Sigma' \hookrightarrow \Sigma$ to constitute a Fourier morphism~$(i, \iota)$. In
other words, we have~$S' \subseteq S$ and~$\Sigma' \subseteq \Sigma$, and the Fourier operator
$\Four\colon S \to \Sigma$ restricts to the Fourier operator~$\Four'\colon S' \to \Sigma'$. 

The \emph{twain algebra structure} on~$L^{1/1}(G, \Gamma) \subset L^1(G) \cap C_0(G)$ comprises the
products~$\star$ and~$\cdot$ inherited from~$L^1(G)$ and~$C_0(G)$, respectively; for the twain
algebra~$L^{1/1}(\Gamma, G)$ the situation is analogous but now~$\cdot$ is the first and~$\star$ the
second product. Thus for~$\Four_\pomega$ to be Fourier operator, it must be a twain homomorphism
\begin{equation*}
  \Four_\pomega\colon \big( L^{1/1}(G, \Gamma), \star, \cdot \big)
  \to
  \big( L^{1/1}(\Gamma, G), \cdot, \star \big) .
\end{equation*}
that respects the Heisenberg action. However, in claiming the subdoublet
relation~$L^{1/1}\inner{G}{\Gamma} \subseteq L^1\inner{G}{\Gamma}$ we
take~$L^{1/1}\inner{\Gamma}{G}$ as a plain doublet (discarding the second products on both
sides). Thus we do \emph{not} think of~$L^1\inner{G}{\Gamma}$ as a twain doublet with trivial second
product. In other words, we think of it as a \emph{plain subdoublet} and not as a twain subdoublet
(as in computer science---prefering downcast to upcast). Let us now prove these claims.

\begin{proposition}
  \label{prop:pont-doublet}
  Let~$G$ and~$\Gamma$ be LCA groups under Pontryagin duality
  $\pomega\colon G \times \Gamma \to \Tor$.
  Then~$L^{1/1}\inner{G}{\Gamma} \subseteq L^1\inner{G}{\Gamma}$ is a
  regular twain doublet.
\end{proposition}
\begin{proof}
  In the proof of Lemma~\ref{lem:L11-basics} we have seen
  that~$\Four_{\pomega\trp} \Four_\pomega s = \Par s$ for~$s \in L^{1/1}(G, \Gamma)$. By symmetry,
  $\Four_\pomega \Four_{\pomega\trp} s = \Par \sigma$ for $\sigma \in L^{1/1}(\Gamma, G)$, which
  means~$\Four_\pomega$ has~$\Four_{\pomega\trp} \Par
  = \Par \Four_{\pomega\trp}$ as its inverse.

  Closure under the action of~$H(\beta)$ follows immediately from equivariance of~$\Four_\pomega$
  and~$\Four_{\pomega\trp}$. By showing
  $\big(L^{1/1}(G, \Gamma), \cdot\big) \subseteq \big(C_0(G), \cdot\big)$ and
  $\big(L^{1/1}(G, \Gamma), \star\big) \subseteq \big(L^1(G), \star \big)$ as Heisenberg
  subalgebras, we will at once ensure the subdoublet relation and the closure conditions for the
  twain doublet (the two other closure conditions follow by symmetry).

  We prove that~$L^{1/1}(G, \Gamma)$ is \emph{closed under pointwise multiplication}, generalizing
  the hints for~\cite[Ex.~6.4.5]{Stade2011}. Let~$s, s' \in L^{1/1}(G, \Gamma)$ be
  arbitrary. Then~$s, s' \in L^2(G)$ by Lemma~\ref{lem:L11-basics}, so that H{\"o}lder's
  inequality~\cite[Thm.~9.2]{Knapp2005b} yields $||s s'||_1 \le ||s||_2 \: ||s'||_2 < \infty$ and
  thus~$s s' \in L^1(G)$.  On the other hand, $s, s' \in \Four_{\pomega\trp} L^1(\Gamma)$ means we
  may write~$s = \Four_{\pomega\trp} \sigma$ and~$s' = \Four_{\pomega\trp} \sigma'$
  for~$\sigma, \sigma' \in L^1(\Gamma)$. Now the convolution theorem (the fact
  that~$\Four_{\pomega\trp}\colon \big(L^1(\Gamma), \star\big) \to \big(C_0(G), \cdot\big)$ is a
  homomorphism, stated in Theorem~\ref{thm:pont-doublet})
  implies~$\Four_{\pomega\trp}(\sigma \star \sigma') = (\Four_{\pomega\trp} \sigma)
  (\Four_{\pomega\trp} \sigma') = s s' \in \Four_{\pomega\trp} L^1(\Gamma)$.
  Altogether, we have
  established~$s s' \in L^{1/1}(G, \Gamma) = L^1(G) \cap \Four_{\pomega\trp} L^1(\Gamma)$.

  Finally we show that~$L^{1/1}(G, \Gamma)$ is also \emph{closed under convolution}.
  Taking~$s, s' \in L^{1/1}(G, \Gamma)$ arbitrary, it is obvious that~$s \star s' \in L^1(G)$. But
  we have also $\Four_\pomega(s \star s') = (\Four_\pomega s) (\Four_\pomega s') \in L^1(\Gamma)$ by
  applying the previous argument now
  to~$\Four_\pomega s, \Four_\pomega s' \in L^{1/1}(\Gamma, G) \subseteq L^2(\Gamma)$. Thus we
  have~$s \star s' \in L^1(G) \cap \Four_\pomega^{-1} L^1(\Gamma)$.
\end{proof}

Note that both~$\Four_{\pomega} = \Four_{\pomega}^\land$
and~$\Four_{\pomega\trp} = \Four_{\pomega\trp}^\land$ in Propostion~\ref{prop:pont-doublet} are
forward Fourier operators, which have backward analogs~$\Four_{\pomega}^\lor$
and~$\Four_{\pomega\trp}^\lor$. As it becomes tedious to track all these distinctions, we shall
content ourselves with
\begin{equation*}
  \Four^\land := \Four_{\pomega}^\land
  \quad\text{and}\quad
  \Four^\lor := \Four_{\pomega\trp}^\lor,
\end{equation*}
which we call the \emph{forward and backward Fourier operators}\footnote{Also known as the
  \emph{direct and inverse} Fourier operators in the literature.}
of~$L^{1/1}\inner{G}{\Gamma}$. These two operators are inverse to each other, as we have seen in the
proof of Propostion~\ref{prop:pont-doublet}. We transfer the notation introduced after
Remark~\ref{rem:justify-twist} to the corresponding images, $\hat{s} := \Four^\land(s)$ for the
\emph{direct transform} of a signal~$s \in S$ and~$\check{\sigma} := \FFour^\lor(\sigma)$ for the
\emph{inverse transform} of a spectrum~$\sigma \in \Sigma$; no confusion is likely to result from
this.

Under Pontryagin duality, the position group~$G$ is \emph{discrete} iff the moment group~$\Gamma$ is
\emph{compact}~\cite[Thm.~1.7.3(a)]{Rudin2017}. In this case, one can always isolate a distinguished
twain doublet within the $L^{1/1}$ twain doublet. For describing it, note first that each
position~$a \in G$ induces a function~$d_a \in L^1(G)$ with~$d_a(x) = \delta_{a,x}$ being the
Kronecker delta; we write~$\mathfrak{D}(G) \le L^1(G)$ for the $\CC$-linear span of all the~$d_a$. On
the other hand, the position~$a \in G$ induces a character~$\chi_a \in C_0(\Gamma) = C(\Gamma)$
with~$\chi_a(\xi) = \inner{\xi}{a}_\pomega$; we write~$\mathfrak{C}(\Gamma) \le C_0(\Gamma)$ for the
$\CC$-linear span of those characters~$\chi_a$.

\begin{proposition}
  \label{prop:char-doublet}
  Let~$G$ and~$\Gamma$ be LCA groups under Pontryagin duality
  $\pomega\colon G \times \Gamma \to \Tor$, with $G$ discrete.
  Then~$[\mathfrak{D}(G) \pto \mathfrak{C}(\Gamma)]$ is a regular twain subdoublet
  of~$[\Four\colon L^{1/1}(G, \Gamma) \pto L^{1/1}(\Gamma, G)]$ with~$\Four d_a = \chi_a$.
\end{proposition}
\begin{proof}
  We
  have~$\Four d_a(\xi) = \sum_{x \in G} \inner{\xi}{x}_\pomega \, d_{a,x} = \inner{\xi}{x}_\pomega$,
  hence~$\Four d_a = \chi_a$. Of course we have~$d_a \in L^1(G)$ as noted above.
  Since~$\| \chi_a \|_1 \le |\Gamma|$, we have also~$\chi_a \in L^1(\Gamma)$ and thus
  indeed~$d_a \in L^{1/1}(G, \Gamma)$. Applying Lemma~\ref{lem:L11-basics} with~$G$ and~$\Gamma$
  interchanged, we obtain~$\chi_a \in L^{1/1}(\Gamma, G)$ since
  $\chi_a = \Four d_a \in \Four L^1(G)$. Thus we have
  established~$[\mathfrak{D}(G) \pto \mathfrak{C}(\Gamma)]$ as a subobject of the $L^{1/1}$ doublet
  on the set-theoretic level. It remains only to show closure under convolution, pointwise
  multiplication and the Heisenberg action. It suffices to check this on one side,
  say~$\mathfrak{D}(G)$.

  It is easy so check that~$d_a \star d_b = d_{a+b}$ and~$d_a \cdot d_b = \delta_{a,b} \, d_a$. Thus
  we ob\-tain~$\big( \mathfrak{D}(G), \star \big) \cong \CC[G]$
  and~$\big( \mathfrak{D}(G), \cdot) \cong \CC\pathalg{G}$, which ensures closure under both product
  structures. Finally, closure under the Heisenberg action follows since~$x \act d_a = d_{a+x}$
  for~$x \in G$ and~$\xi \act d_a = \inner{\xi}{a}_\pomega \, d_a$ for~$\xi \in \Gamma$, while the
  torus acts trivially via the inclusion~$\Tor \subset \CC$.
\end{proof}

We call~$[\mathfrak{D}(G) \pto \mathfrak{C}(\Gamma)]$ the \emph{character doublet} of the Pontryagin
duality~$\pomega\colon G \times \Gamma \to \Tor$. Note that in this case the action of the discrete
group~$G$ on both signals~$s \in L^1(G)$ and spectra~$\sigma \in C_0(\Gamma) = C(\Gamma)$ is already
contained in the twain algebra structure since~$x \act s = d_x \star s$
and~$x \act \sigma = \chi_x \cdot \sigma$. Furthermore, it also encodes evaluation and the Fourier
transform since~$d_a \act s = s(a) \, d_a$
and~$\chi_a \star \sigma = \chi_a \act \Four^\lor \sigma \, (a)$.
Since~$1 = \chi_0 \in \mathfrak{C}(\Gamma) \subseteq C(\Gamma)$ is a neutral element
in~$\big( C(\Gamma), \act \big)$, such algebras are always unital
with~$\CC \hookrightarrow C(\Gamma)$ and with unital Fourier transform~$\Four(\delta_0) = 1$. We shall
see an important instance of a character doublet in
Example~\ref{ex:pont-doublet-inverse}\ref{itt:four-dtft}.

If~$G$ fails to be discrete, the~$d_a$ are in general distributions while the functions~$\chi_a$ do
not vanish at infinity. In the crucial example~$G = \RR$, we get the \emph{Dirac
  deltas}~$d_a = \delta_a \in M(\RR)$ and the \emph{oscillating exponentials}
$\chi_a = e_a \in BC(\RR)$ with~$e_a(\xi) = e^{i\tau a \xi}$. As this example suggests, the
character doublet~$[\mathfrak{D}(G) \pto \mathfrak{C}(\Gamma)]$ is no longer a twain subdoublet of
$[\Four\colon L^{1/1}(G, \Gamma) \pto L^{1/1}(\Gamma, G)]$ but just a slain subdoublet of the
measure doublet~$[\Four\colon M(G) \pto BC(\Gamma)]$. We do not pursue these matters here, but we
shall need the following application of the character algebra~$\mathfrak{C}(\Gamma)$ for describing
the \emph{Heisenberg closure}~$\bar{\Sigma}$ of a subalgebra~$\Sigma \le C_0(\Gamma)$, defined as
the smallest Heisenberg subalgebra of~$C_0(\Gamma)$ that extends~$\Sigma$. It is obtained by
tensoring~$\CC[\Sigma\Gamma] \le C_0(\Gamma)$, the complex algebra generated by the translates
of~$\Sigma$, with the character space.

\begin{proposition}
  \label{prop:heis-closure}
  Let~$G$ and~$\Gamma$ be LCA groups under Pontryagin duality
  $\pomega\colon \Gamma \times G \to \Tor$. The Heisenberg closure~$\bar{\Sigma}$ of a
  subalgebra~$\Sigma \le C_0(\Gamma)$ is isomorphic
  to~$\mathfrak{C}(\Gamma) \otimes_{\CC} \CC[\Sigma\Gamma]$.
\end{proposition}
\begin{proof}
  For the purpose of this proof, let us write the translates of a spectrum~$\sigma \in \Sigma$
  by~$\eta \in \Gamma$ as $\sigma^\eta := \eta \act \sigma$ so that every element
  of~$\CC[\Sigma\Gamma]$ can be written as a polynomial in the~$\sigma^y$. It is then clear
  that~$\mathfrak{C}(\Gamma) \, \CC[\Sigma\Gamma]$, the $\CC$-linear span of all
  products~$\chi_y \, \sigma_1^{\eta_1} \cdots \sigma_k^{\eta_k} = y \act \sigma_1^{\eta_1} \cdots
  \sigma_k^{\eta_k}$,
  is contained in the pointwise algebra~$C_0(\Gamma)$ since the latter is closed under the
  Heisenberg action. We have
  \begin{equation*}
    (\chi_y \, \sigma_1^{\eta_1} \cdots \sigma_k^{\eta_k}) \, (\chi_{\bar y} \, \bar{\sigma}^{\bar{\eta}_1}
    \cdots \bar{\sigma}_l^{\bar{\eta}_l}) = \chi_{y+\bar{y}} \, \sigma_1^{\eta_1} \cdots
      \sigma_k^{\eta_k} \, \bar{\sigma}^{\bar{\eta}_1} \cdots \bar{\sigma}_l^{\bar{\eta}_l} ,
  \end{equation*}
  which shows at once closure under the pointwise product and the isomorphism
  with~$\mathfrak{C}(\Gamma) \otimes_{\CC} \Gamma[\Sigma]$. Finally, closure under the Heisenberg
  action is clear since
  \begin{align*}
    x \act \chi_y \, \sigma_1^{\eta_1} \cdots \sigma_k^{\eta_k} &= \chi_{x+y} \,
    \sigma_1^{\eta_1} \cdots \sigma_k^{\eta_k},\\
    \xi \act \chi_y \, \sigma_1^{\eta_1} \cdots \sigma_k^{\eta_k} &= \inner{\xi}{y}^{-1} \,
  \chi_y \, \sigma_1^{\xi+\eta_1} \cdots \sigma_k^{\xi+\eta_k},
  \end{align*}
  while the torus~$\Tor \subset \CC$ acts trivially.

  We have thus verified
  that~$\mathfrak{C}(\Gamma) \, \Gamma[\Sigma] \cong \mathfrak{C}(\Gamma) \otimes_{\CC}
  \Gamma[\Sigma]$ is a Heisenberg subalgebra of~$C_0(\Gamma)$. It is clearly the smallest such since
  closure under the poitwise product and the action of~$\Gamma$ ensures inclusion of~$\CC[\Gamma]$,
  whereas closure under the action of~$G$ brings in the characters~$\chi_y$.
\end{proof}

\begin{myremark}
  Note that the algebra~$\Sigma$ may be nonuntial. Then~$\CC[\Gamma]$ and the tensor
  algebra~$\mathfrak{C}(\Gamma) \otimes_{\CC} \CC[\Gamma]$ will also be nonunital, and the latter
  does not contain~$\mathfrak{C}(\Gamma)$ as a subalgebra. We have already noted this above for the
  example~$G = \RR$ where we have the oscillating exponentials~$\chi_a = e_a \not\in C_0(\RR)$ .
\end{myremark}

Before passing to the crucial examples of the classical Fourier operators, it will be useful to
introduce one final bit of typology for ``Fourier structures'': When signal and spectral space are
isomorphic, they may be identified so that the \emph{Fourier operator is an endomorphism} in~$\Mod_K$.

\begin{definition}
  \label{def:singlet}
  A \emph{Fourier singlet} is a Fourier doublet~$[\Four\colon S \pto S]$.
\end{definition}

Of course, any singlet may still be viewed as a doublet (see
Example~\ref{ex:pont-doublet} above) by choosing \emph{not} to
identify signals and spectra even when this is possible: Even when the
spaces coincide, one may choose to make them formally distinct (say,
by introducing~$\Sigma := S \times \{ 0 \}$). Note also that regular
doublets \emph{may but need not be} viewed as singlets. The intention
of distinguishing signals from spectra may be the motivation behind
writing signals as functions of~$x$ but spectra as functions of~$\xi$.

While such a distinction may seem pedantic, it should be called to mind that the situation is very
similar for linear maps between vector spaces~$f\colon V \to W$, where the distinction is crucial
when it comes to classification: While homomorphisms are classified for equivalence by the rank,
endomorphisms are classified for similarity by the elementary divisors. In the case of Fourier
structures, the special role of singlets will become apparent when we consider \emph{adjunction} of
new elements (Remark~\ref{rem:hyperbolic-singlet}). For the moment, it suffices to point out that
the classical twain doublet~$L^{1/1}\inner{\Gamma}{G}$ for a Pontryagin
duality~$\pomega\colon \Gamma \times G \to \Tor$ may be viewed as a singlet exactly in the self-dual
case, i.e. when~$G \cong \Gamma$.

\begin{myexample}
  \label{ex:pont-doublet-inverse}
  Returning to the classical Fourier operators of Example~\ref{ex:pont-doublet}, let us review the
  inverse transformation $\Four^\lor\colon L^{1/1}(\Gamma) \to L^{1/1}(G)$, where in each
  case~$\Four^\land\colon L^{1/1}(G) \to L^{1/1}(\Gamma)$ as the restriction of the corresponding
  Fourier operator~$\Four\colon L^1(G) \to C_0(\Gamma)$ according to
  Propostion~\ref{ex:pont-doublet}. In this context, the defining formulae~\eqref{eq:four-int},
  \eqref{eq:four-ser}, \eqref{eq:four-dtft}, \eqref{eq:four-dft} for~$\Four^\land$ are referred to
  as the \emph{analysis equations}, those for~$\Four^\lor$ the \emph{synthesis equations}. We shall
  now state the latter in each of the four cases treated in Example~\ref{ex:pont-doublet}.
  \begin{enumerate}[(a)]
  \item\label{itt:four-int} For the \emph{Fourier integral}
    $\Four^\land\colon L^{1/1}(\RR) \to L^{1/1}(\RR)$, the inverse transformation is of the same
    form, except for the negative sign in the exponential so that
    \begin{equation}
      \label{eq:inv-four-int}
      \Four^\lor \sigma \, (x) = \int_{\RR^n} e^{-i \tau x \cdot \xi} \, \sigma(\xi) \, d\xi .
    \end{equation}
    At this point, it should also be mentioned that \emph{different normalizations} for the Fourier
    integral---corresponding to different normalizations of the underlying Haar measure---are in
    circulation. Some people use instead of the ``ordinary frequency'' variable~$\xi$ the ``angular
    frequency''~$\omega := \tau\xi$, so the exponential becomes~$e^{i x \cdot \omega}$ in the
    forward transformation~\eqref{eq:four-int} and~$e^{-i x \cdot \omega}$ in the above backward
    transformation~\eqref{eq:inv-four-int}, along with a scaling factor~$\tau^{-n}$ in the
    latter. Since this destroys the unitary character (when viewing it on the corresponding Hilbert
    space---see Proposition~\ref{prop:L2-doublet}), the scaling factor is evenly distributed
    as~$\tau^{-n/2}$ in both forward and backward transformation.
  \item\label{itt:four-ser} For the case of \emph{Fourier series}, the synthesis of the Fourier
    coefficients~\eqref{eq:four-ser} is just the summation that builds up their ``Fourier series'',
    namely
    \begin{equation}
      \label{eq:inv-four-ser}
      \Four^\lor \sigma \, (x) = \sum_{\xi \in \ZZ^n} e^{-i\tau x \cdot \xi} \, \sigma(\xi) .
    \end{equation}
    As with the Fourier integral, some modifications are possible. In particular, signals may be
    taken over periodic domains~$[0, T] \cong S^1$ with other \emph{periods}~$T \neq 1$. In that
    case, the exponentials in~\eqref{eq:four-ser} and~\eqref{eq:inv-four-ser}
    become~$e^{\pm i\tau x \cdot \xi/T}$ with an additional scaling factor of~$T^{-n}$ in front of
    the Fourier coefficients~\eqref{eq:four-ser}. In that case, it can be considered suitable to
    introduce the ``angular velocity''~$\omega := \tau/T$, thus simplifying the exponentials
    to~$e^{\pm i x \cdot \omega}$.
  \item\label{itt:four-dtft} The inverse of the \emph{discrete-time Fourier
      transform}~\eqref{eq:four-dtft} is
    \begin{equation}
      \label{eq:inv-four-dtft}
      \Four^\lor \sigma \, (x) = \int_{\II^n} e^{-i\tau x \cdot \xi} \, \sigma(\xi) \, d\xi ,
    \end{equation}
    which is identical to the synthesis equations~\eqref{eq:inv-four-int} of the Fourier integral
    except for the different integration domain (which is actually a parametrization
    of~$\Tor^n$). As with Fourier series, there are straightforward modifications for adopting
    different periods.

    Moreover, the spectra~$\sigma \in L^{1/1}(\Tor^n)$ may be viewed as functions on $n$ complex
    variables~$z_1, \dots, z_n \in \Tor \subset \CC$. Assuming convergence, the $n$
    circles~$\Tor = S^1$ may be expanded to annuli~$A_1, \dots, A_n$ to produce an analytic
    function~$\sigma\colon A_1 \times \cdots \times A_n \to \CC$. Under this interpretation, the
    analysis equation~\eqref{eq:four-dtft} amounts to summing the multivariate \emph{Laurent series}
    while the synthesis equation~\eqref{eq:inv-four-dtft} effects the computation of the Laurent
    coefficients, essentially via Cauchy's integral formula. In other words, the Fourier operators
    $\Four^\land$ and~$\Four^\lor$ coincide with the (bilateral) \emph{$\mathcal{Z}$-transform} and
    its inverse, possibly up to minor adoptions (as usual there is no agreement on the sign in the
    exponential).

    As~$\ZZ^n$ is discrete, there is a twain
    subdoublet~$[\mathfrak{D}(\ZZ^n) \pto \mathfrak{C}(\Tor^n)]$ of
    $[L^{1/1}(\ZZ^n, \Tor^n) \pto L^{1/1}(\Tor^n, \ZZ^n)]$ described in
    Proposition~\ref{prop:char-doublet}. In our case, we have
    \begin{equation*}
      d_a(x) = \delta_{a,x} = \delta_{a_1, x_1} \cdots \delta_{a_n, x_n},\qquad
      \chi_a(\xi) = \xi^a = \xi_1^{a_1} \cdots \xi_n^{a_n},
    \end{equation*}
    with the group algebra~$\big( \mathfrak{D}(\ZZ^n), \star \big) \cong \CC[\ZZ^n]$ consisting of
    Laurent polynomials (restrictions of the Laurent series mentioned above) and the path
    algebra~$\big( \mathfrak{C}(\Tor^n), \cdot \big) \cong \CC\pathalg{\ZZ^n}$ having integer
    $n$-tuples as orthogonal idempotents. In its abstract algebraic form, we have already seen the
    twain algebra $\CC\pathalg{\ZZ^n} \divideontimes \CC[\ZZ^n]$ in Example~\ref{ex:tor-twainalg}.
  \item\label{itt:four-dft} For the \emph{discrete Fourier series}~\eqref{eq:four-dfs} the inversion is
    \begin{equation}
      \label{eq:inv-four-dfs}
      \Four^\lor \sigma \, (x) = \sum_{\xi \in \ZZ_N^n} e^{-i \tau x \cdot \xi} \, \sigma(\xi);
    \end{equation}
    for the \emph{discrete Fourier transform} the analogous expression is
    \begin{equation}
      \label{eq:inv-four-dft}
      \Four^\lor \sigma \, (x) = \frac{1}{N^n} \, \sum_{\xi \in \Tor_N^n} e^{-i \tau x \cdot \xi} \,
      \sigma(\xi).
    \end{equation}
    Note that~\eqref{eq:inv-four-dfs} is a truncated version of~\eqref{eq:inv-four-ser}
    while~\eqref{eq:inv-four-dft} is a discretized form of~\eqref{eq:inv-four-dtft}; as in the
    forward DFS transformation~\eqref{eq:four-dfs}, the factor~$N^{-n}$ arises
    from~$\cum_0^1 \dots \, dx_i \rightsquigarrow \sum_{x_i \in \Tor_N} \dots \, N^{-1}$.  Needless
    to say there are again various conventions for the sign of the exponential and the distribution
    of the scaling factor. And one may of course adopt the viewpoint of identifying~DFS and~DFT as
    discussed at the end of Example~\ref{ex:pont-doublet}.

    The Nicholson version of the Fourier transform in Example~\ref{ex:four-nich} is also invertible.
    The fact that~$\Four\colon \big(R[G], \star\big) \to \big(R^n, \cdot \big)$ is a isomorphism was
    exactly the motivation for the conditions imposed on~$R$, as mentioned in
    Example~\ref{ex:nich-duality}. The inversion formula is analogous to~\eqref{eq:inv-four-dft}
    above, namely~\eqref{eq:four-nich} with negated sign and a factor~$|G|^{-1}$.
  \end{enumerate}
  Note that cases~\eqref{itt:four-int} and~\eqref{itt:four-dft} are self-dual while the dualities
  of~\eqref{itt:four-ser} and~\eqref{itt:four-dtft} are conjuages of each other. Thus one may
  consider the Fourier integral and the DFT as Fourier singlets. In the former case, this appears to
  be universally accepted at least in the one-dimensional case (nobody wants to distinguish
  one-dimensional column vectors from one-dimensional row vectors). But for the DFT, the singlet
  view appears to be tied up with the abstract interpretation mentioned at the end of
  Example~\ref{ex:finite-group}. Treating the DFT as a doublet corresponds to the distinction
  between ``discrete Fourier series'' and ``discrete Fourier transforms'', corresponding to the two
  realizations~$\ZZ_N$ and~$\Tor_N$ of the abstract group~$\ZZ/N$.

  (Conversely, some might even argue that Fourier series and the discrete-time Fourier transform are
  ``the same'' since inversion and parity are ``trivial'' operations; this appears to be the
  viewpoint of~\cite[p.~99]{Folland1994}, where only~\eqref{eq:four-int}, \eqref{eq:four-ser},
  \eqref{eq:four-dtft} of Example~\ref{ex:pont-doublet} are mentioned,
  or~\cite[Ex.~1.2.7]{Rudin2017}, who additionally omits the discrete Fourier transform among the
  ``classical groups of Fourier analysis''. However, most other sources in signal theory, such
  as~\cite{OppenheimWillsky2013}, \cite{Smith1997}, \cite{Roberts2012}, do give a separate treatment
  of Fourier series and the discrete-time Fourier transform.)
\end{myexample}

The $L^2$ functions form another widespread regular doublet, which enjoys great popularity due to
various beneficial properties ensuing from the Hilbert space structure of $L^2$ functions. The price
to be paid for this convenience is a certain impoverishment of the multiplicative structure as $L^2$
functions are in general neither closed under convolution nor under the pointwise product. As a
consequence, we obtain only a slain doublet that we call the \emph{square-integrable doublet} to be
denoted by~$L^2\inner{\Gamma}{G} := [\Four\colon L^2(G) \pto L^2(\Gamma)]$.

\begin{proposition}
  \label{prop:L2-doublet}
  Let~$G$ and~$\Gamma$ be LCA groups under Pontryagin duality
  $\pomega\colon G \times \Gamma \to \Tor$.  Then~$L^2\inner{\Gamma}{G}$ is a regular slain doublet
  containing $L^{1/1}\inner{\Gamma}{G}$ as subdoublet, with~$\Four\colon L^2(G) \to L^2(\Gamma)$
  unitary.
\end{proposition}
\begin{proof}
  We have proved in Lemma~\ref{lem:L11-basics} that~$L^{1/1}(G, \Gamma) \subseteq L^2(G)$ and,
  dually, $L^{1/1}(\Gamma, G) \subseteq L^2(\Gamma)$. Regularity and the required closure properites
  of~$L^2\inner{\Gamma}{G}$ are established in most texts on abstract harmonic analysis (they follow
  essentially from the $L^1$ theory by completion); see for example~\cite[(4.25)]{Folland1994}. Of
  course, the $L^2$ doublet enjoys a wealth of further properties (in particular unitarity, see
  below). It should be noted that, in the $L^2$ case, the definition~\eqref{eq:fourier-transform}
  does not generally hold \emph{on the nose}, but must be understood in the sense of a suitable
  limit. This is possible since the classical Fourier transform of Theorem~\ref{thm:pont-doublet}
  restricts to the dense subspace~$L^1(G) \cap L^2(G)$ of the Hilbert space~$L^2(G)$, and its image
  is again dense in~$L^2(G)$; see~\cite[\S1.6.1]{Rudin2017}, where the unitary nature
  of~$\Four\colon L^2(G) \to L^2(\Gamma)$ is also established.
\end{proof}

At this point we can go back to the \emph{classial Fourier operators} detailed in
Example~\ref{ex:pont-doublet}, each of which may be interpreted in their corresponding $L^2$ setting
according to Proposition~\ref{prop:L2-doublet}, and this indeed often used. As mentioned earlier,
the Fourier operator can be suitably normalized so as to become a unitary operator between the
Hilbert spaces~$L^2(G)$ and~$L^2(\Gamma)$. In the case of the DFT, it is then necessary to split the
factor~$1/N^n$ occurring in~\eqref{eq:four-dft} so that it becomes~$1/\sqrt{N^n}$ in front of both
forward and backward transformation. Note that each of the classical Fourier operators, defined
in~\eqref{eq:four-int}, \eqref{eq:four-ser}, \eqref{eq:four-dtft}, \eqref{eq:four-dft} and
reinterpreted as~$\Four\colon L^2(G) \to L^2(\Gamma)$, can be understood as taking the scalar
product of a signal~$s \in L^2(G)$ with the exponential~$e_\xi(x) := e^{i\tau x \cdot \xi}$.
Thinking of~$\Four$ intuitively as a ``matrix'' with columns indexed by positions~$x \in G$ and rows
indexed by momenta~$\xi \in \Gamma$, its $(x,\xi)$-entry would be given by the
character~$\inner{x}{\xi} = e^{i\tau x \cdot \xi}$. While this is can be made precise immediately in
the DFT case (see the remarks in Example~\ref{ex:pont-doublet}\ref{it:four-dft} for the traditional
form of the matrix), one may employ the machinery of rigged Hilbert spaces (also known as Gelfand
triples) for treating the other cases along these lines with full rigor.

\begin{myremark}
  \label{rem:pmech}
  The $L^2$ setting is also the hub of the representation theory of
  Heisenberg groups. Under the physical interpretation mentioned
  earlier (Remark~\ref{rem:physint-heisgrp}), this constitutes the
  standard approach to representing \emph{physical
    observables}. Indeed, the Heisenberg
  group~$\mathbf{H}_n^{\mathrm{pol}}$ over the standard vector
  duality~$\inner{\RR^n}{\RR_n}$ underpins kinematics, both classical
  (Hamilton's equations) and quantum (Heisenberg equation): The
  Schr\"odinger
  representation~$\rho_h\colon H_n \to \mathcal{H}\big( L^2(\RR^n)
  \big)$ yields the latter; it is parametrized by the Planck constant
  $h$ whose so-called semi-classical limit $h \to 0$ leads to
  Hamiltonian mechanics. See also~\cite{Howe1980}, \cite{Howe1980a}
  for more on the representation theory of~$H_n$.
  If one dislikes the idea of constants tending to zero (though one
  might interpret this as taking place in a hypothetical sequence of
  universes with progressively less significant quantum effects), the
  framework of Plain Mechanics (p-mechanics) offers an alternative
  viewpoint~\cite{Kisil1996}: Before specializing to any quantum or
  classical (or hyperbolic quantum) version, physical systems are
  described in the ``plain'' setting of the Heisenberg group~$H_n$,
  using~$\big(L^1(H_n), \star)$ as the algebra of observables
  (including in particular the Hamiltonian). A so-called universal
  equation rules the evolution of the system, which transforms to the
  Heisenberg equation under the representation~$\rho_h$ and to
  Hamilton's equations under~$\rho_0$. In this framework, one may
  develop corresponding brackets~\cite{Kisil2002}, notions of
  state~\cite{Kisil2003} and a detailed mechanical
  theory~\cite{Kisil2004}. In a newer
  presentation~\cite[IV.1]{Kisil2018}, p-mechanics is treated in a
  more geometric context where the elliptic / parabolic / hyperbolic
  cases correspond to different number rings (complex / dual / double
  numbers).

  It would be nice to reformulate the last results in term of suitable
  tori. For example, the classical case should correspond to the
  torus~$\Tor_\epsilon := \{ 1 + b \epsilon \mid b \in \RR \}$ of the
  dual numbers~$\RR_\epsilon := \RR[\epsilon \mid \epsilon^2 = 0]$
  with the duality~$\inner{\RR^n}{\RR_n}_\epsilon$ given
  by~$\inner{x}{\xi}_\epsilon := 1 + \epsilon (x \cdot \xi)$. Its
  $L^2$ representation (encoded in the $L^2$ Fourier doublet) should
  then yield the classical Hamilton's equations just as the standard
  duality yields the Heisenberg equation. In an even more ambitious
  enterprise, one might try to develop a general theory of ``universal
  equations'' for a given duality that yields classical and quantum
  kinematics as special cases. For this purpose, some tools developed
  in~\cite{Akbarov1995a}, \cite{Akbarov1995b} may be of help.

  It should be mentioned that the Heisenberg group also affords some other very famous
  representations, apart from the all-important Schr\"odinger representation: The \emph{phase-space
    representation} is embodied in the symplectic Fourier transform (Remark~\ref{rem:sympl-Four}).
  The \emph{theta representation} is important in algebraic geometry---it is investigated in
  Mumford's monumental opus~\cite{Mumford2007b}, which also mentions the phase-space representation
  on the way~\cite[Thm.~1.2]{Mumford2007b}. The \emph{Segal-Bargmann representation} on Fock space
  invokes a subspace of the entire functions in $n$variables~\cite[\S1.6]{Folland2016}.
\end{myremark}

\begin{myexample}
  \label{ex:L1-doublet}
  We come back to the (normalized) cardinal sine function given in the proof of
  Lemma~\ref{lem:L11-basics}. It is known~\cite[p.~106]{Bracewell1986} that its Fourier transform is
  the (normalized) \emph{rectangle function}~$\Pi := \chi_{(-1/2, 1/2)}$. The squared cardinal sine
  is also important~\cite[p.~108]{Bracewell1986}; its Fourier transform is the (normalized)
  \emph{triangle function}~$\Delta(x) := \max(1-|x|,0)$. Furthermore, in the proof of
  Lemma~\ref{lem:L11-basics} we have seen that~$\sinc \in L^2(\RR)$, thus~$\sinc^2 \in
  L^1(\RR)$.
  Since its transform~$\Delta$ is clearly in~$L^1(\RR)$ as well, this shows that in
  fact~$\sinc^2 \in L^{1/1}(\RR)$. In summary, we have the Fourier pairs
  \begin{align*}
    [\sinc \mapsto \Pi] &\in L^2\inner{\RR}{\RR} \setminus L^{1/1}\inner{\RR}{\RR},\\
    [\sinc^2 \mapsto \Delta] &\in L^{1/1}\inner{\RR}{\RR},
  \end{align*}
  both of which are crucial in signal processing applications. In particular, we
  have~$\sinc \in C_0(\RR) \setminus L^1(\RR)$ whereas
  clearly~$\Pi \in L^1(\RR) \setminus C_0(\RR)$. Thus the spaces $L^1(G)$ and $C_0(G)$ are
  generically seen to have nontrivial intersection.

  For another pair of examples, let us turn to probability. The \emph{Gaussian distribution} with
  Stigler normalization~\cite{Stigler1982} is~$g(x) := e^{-\pi x^2}$; this corresponds to a variance
  of~$1/\tau$. It is an $L^1$ function with~$||g||_1 = 1$ and the remarkable property of being a
  fixed point of the Fourier transformation.  Writing~$l(x) := e^{-|x|}$ for (twice) the
  \emph{Laplace distribution}~$f(x \,|\, 0, 1)$ and~$c(\xi) := \tfrac{2}{1+(\tau\xi)^2}$ for the
  \emph{Cauchy distribution}~$f(\xi \,|\, 0,1/\tau)$, we obtain the two Fourier pairs
  \begin{equation*}
    [g \mapsto g], [l \mapsto c] \in L^{1/1}\inner{\RR}{\RR},
  \end{equation*}
  figuring prominently in probability theory.

  In concluding, let us note that the inclusion~$L^{1/1}(G) \subset L^1(G) \cap C_0(G)$ is in
  general strict, though examples and references seem to be hard to come by.
  Taking~$G = \Gamma = \RR$ again, we have clearly
  \begin{equation*}
    s = x \mapsto \tfrac{-\Pi(x/2)}{\log(x) - \log(1-x)} \in L^1(\RR) \cap C_0(\RR) .
  \end{equation*}
  It is harder to see\footnote{We point to MathStackExchange \#67910 at
    \url{https://math.stackexchange.com/questions/67910/a-fourier-transform-of-a-continuous-l1-function}
    for a detailed estimate.} that~$|\hat{s}(\xi)| \sim \tfrac{\sin(\pi\xi)}{\pi\xi \, \log|\xi|}$
  as~$|\xi| \to \infty$ and thus~$\hat{s} \not\in L^1(\RR)$. This also shows that the Heisenberg
  twain algebra~$L^1(G) \cap C_0(G)$ does not in general map to its
  counterpart~$L^1(\Gamma) \cap C_0(\Gamma)$; it is thus of little use and not considered in Fourier
  analysis.
\end{myexample}

\begin{myexample}
  \label{ex:Lp-doublet}
  If~$\pomega\colon \Gamma \times G \to \Tor$ is again a Pontryagin duality, one may
  generalize~\cite[Prop.~4.27]{Folland1994} the square-integrable doublet by defining a Fourier
  transform from~$L^p(G)$ to~$L^q(\Gamma)$ for any~$p \in [1,2]$ with H\"older conjugate~$q$. Note
  that~$q \ge 2$ in this case, so this cannot be a singlet except for~$p=q=2$ and~$G \cong \Gamma$
  yielding of course~$L^2\inner{G}{G}$. Nevertheless, one does get a Fourier
  doublet~$[\Four\colon L^p(G) \pto L^q(\Gamma)]$ in general.

  This is a slain doublet just as in the famous~$p=q=2$ case. In fact, closure under the pointwise
  product is immediately seen to break down from examples such
  as~$s = x \mapsto x^{-1/3} \in L^{3/2}[0,1]$ with~$s^2 \not\in L^{3/2}[0,1]$. As for the
  convolution product (fixing any particular~$1<p<2$), it is known that closure in~$L^p(G)$ is
  equivalent to~$G$ being discrete, in fact even for nonabelian locally compact
  groups~\cite[Prop.~2.1]{AbtahiNasrIsfahaniRejali2013}. Furthermore,
  $\Four\colon L^p(G) \to L^q(\Gamma)$ is apparently not surjective\footnote{See MathStackExchange
    \#238692 at
    \url{https://mathoverflow.net/questions/238692/fourier-transform-surjective-on-lp-mathbbrn-for-p-in-1-2}
    for an argument.}, so this is in general only a slain singular doublet.
\end{myexample}

Let us summarize the \emph{typology of Fourier structures} as follows: Just as for morphisms in any
other category, we have isolated the two crucial \emph{properties} of Fourier operators---qua linear
maps---as being isomorphisms (regular versus singular doublets) and endomorphism (singlets versus
doublets). Furthermore, we distinguish the \emph{structures} of slain/plain/twain doublets. We
illustrate the various possibilities by listing for each case some natural example under Pontryagin
duality. (We list only doublets; as mentioned above, one may view them as singlets for self-dual
groups if desired. But this singlet/doublet distinction is only important for special purposes such
as adjunction of elements.)

\begingroup
\footnotesize
\begin{table}[h!]
  \renewcommand{\arraystretch}{1.2}
  \begin{tabular}{||l||l|l|l||}
    \hline\hline
    & Slain & Plain & Twain\\\hline
    Singular & $[L^{3/2}(G) \pto L^3(\Gamma)]$ & $[L^1(G) \pto C_0(\Gamma)]$ 
                              & $[L^{1/1}(G) \pto L^1(\Gamma) \cap C_0(\Gamma)]$\\
    Regular & $[L^2(G) \pto L^2(\Gamma)]$ & $[L^1(G) \pto A(\Gamma)]$ 
                              & $[L^{1/1}(G) \pto L^{1/1}(\Gamma)]$\\
    \hline\hline
  \end{tabular}
  \bigskip
  \caption{Examples of Fourier Structures}
  \label{tab:four-struct}
\end{table}
\endgroup

Before turning to our most important example from an alogorithmic perspective in the next Section,
let us mention a few other Fourier doublets.

\subsection{The Schwartz Class for Pontryagin Duality.}
\label{sub:schwartz-class}
Classical Fourier analysis comes into its own when introducing the space of \emph{Schwartz
  functions}~$\Schw(\RR)$ and its dual~$\Schw'(\RR)$, the space of \emph{tempered
  distributions}~\cite{Strichartz1994}. It is a remarkable fact that even this seeminlgy special
construction on real functions can be generalized to the abstract setting of Pontryagin duality,
where Schwartz functions are named \emph{Schwartz-Bruhat} functions, after Laurent Schwartz for the
classical theory~\cite{Schwartz1951} and Fran{\c c}ois Bruhat for the abstract
construction~\cite{Bruhat1961}. Following the latter, one starts from LCA groups~$G$ and~$\Gamma$
under Pontryagin duality~$\pomega\colon G \times \Gamma \to \Tor$ and defines the Schwartz-Bruhat
space~$\Schw(G)$ by appealing to the structure theory of LCA groups. We shall only
outline\footnote{A concise summary may also be found on the \textsf{nlab} page
  \url{https://ncatlab.org/nlab/show/Schwartz-Bruhat+function}.} the essential steps of the
construction~\cite[\S9]{Bruhat1961}, \cite[\S11]{Weil1964}:
\begin{enumerate}
\item\label{it:dir-lim} One starts from the important fact~\cite[p.~162]{Roeder1971} that any LCA
  group~$G$ is an inverse limit of abelian Lie groups~$(G_\alpha \,|\, \alpha \in I)$, obtained by
  dualizing~\cite[Thm.~2.5/Cor.~1]{Moskowitz1967} the representation of~$\Gamma$ as a direct limit
  of its compactly generated subgroups $(\Gamma_\alpha \,|\, \alpha \in I)$. Defining $\Schw(G)$ as
  direct limit of~$\big(\Schw(G_\alpha) \,|\, \alpha \in I\big)$, it remains to define~$\Schw(L)$
  when~$L$ is an abelian Lie group (in a different terminology: ``having no small subgroups'').
\item\label{it:lie-grp} In that case~\cite[Thm.~2.4]{Moskowitz1967}, $L$ is isomorphic
  to~$\RR^m \oplus \Tor^n \oplus D$ for a discrete group~$D$, which may again be written as a direct
  limit of its finitely generated subgroups~$D_\beta$ so that~$L$ is the direct limit of the
  elementary groups~$(\RR^m \oplus \Tor^n \oplus D_\beta \,|\, \beta \in J)$. Applying primary
  decomposition on each~$D_\beta$, one obtains the
  representation~$E_\beta = \RR^m \oplus \Tor^n \oplus \ZZ^{k(\beta)} \oplus H_\beta$ for the
  elementary groups, with~$H_\beta$ finite. One defines~$\Schw(L)$ as functions that are
  Schwartz-Bruhat on some~$E_\beta$ and vanishing elsewhere.
\item\label{it:elem} A function~$f$ on an elementary
  group~$\RR^m \oplus \Tor^n \oplus \ZZ^k \oplus H$ is called Schwartz-Bruhat
  if~$f(-, \xi, -, h)\colon \RR^m \oplus \ZZ^k \to \CC$ is a Schwartz function for
  each~$(\xi, h) \in \Tor^n \oplus H$. As usual, Schwartz functions $\RR^m \oplus \ZZ^k \to \CC$ are
  defined as smooth in the variable~$x \in \RR^m$, and with all
  derivatives~$\partial^\alpha\!f/\partial x^a \: (a \in \NN^m)$ of rapid decay. A function
  $\phi\colon \RR^m \oplus \ZZ^k \to \CC$ is said to be of rapid decay if it satisfies
  $\phi(x, \nu) = o(x^{-a} \nu^{-b})$ for all~$(a, b) \in \NN^m \times \NN^k$.
\end{enumerate}

The direct limit it Item~\eqref{it:dir-lim} above can be realized by
defining~$G \overset{\phi}{\to} \CC$ to be Schwartz-Bruhat iff it factors, for some~$\alpha$, as
$$G \overset{\pi}{\twoheadrightarrow} G_\alpha \overset{\tilde\phi}{\to} \CC$$ with~$\pi$ the
canonical projection onto~$G_\alpha = G/\!\Ann(\Gamma_\alpha)$ and~$\tilde\phi$ any
Schwartz-Bruhat function in the sense of Item~\eqref{it:lie-grp}. See~\cite[Thm.~27]{Morris1977} for
exchanging subgroups and quotients under dualization.

There are two \emph{variations} of this definition. The first is to retain item~\eqref{it:elem}
above, but top it by a more intrinsic characterization of differentiation in terms of one-parameter
subgroups~\cite{Wawrzynczyk1968}. Using some more functional analysis, these ideas can even be
generalized to non-abelian locally compact groups~\cite{Akbarov1995}. The second variation is to
bypass Lie groups altogether, defining~$\Schw(G)$ directly in terms of growth
rates~\cite{Osborne1975}.

We obtain one and the same Schwartz-Bruhat space~$\Schw(G) \subset L^1(G)$ with any of these
definitions.  It turns out that~$\Schw(G)$ is dense in $L^1(G)$ according
to~\cite[Thm.~3.3]{Wawrzynczyk1968}, and dense in~$L^2(G)$ according
to~\cite[p.~42]{Osborne1975}. Of course, this generalizes well-known facts of classical analysis;
they show that~$\Schw(G)$ is in fact ``big enough to be useful''. Let us now relate the
Schwartz-Bruhat space to the classical Fourier singlet and then review the classical Fourier
operators of Example~\ref{ex:pont-doublet}.

\begin{theorem}
  \label{thm:Schwartz-Bruhat}
  Let~$G$ and~$\Gamma$ be LCA groups under Pontryagin duality
  $\pomega\colon \Gamma \times G \to \Tor$.  Then~$[\Four\colon \Schw(G) \pto \Schw(\Gamma)]$ is a
  regular twain doublet contained in the classical twain doublet~$L^{1/1}\inner{\Gamma}{G}$.
\end{theorem}
\begin{proof}
  In the above-mentioned treatments of Schwartz-Bruhat space it is
  proved that~\cite[Thm.~3.2]{Wawrzynczyk1968} the (forward) Fourier
  operator~$\Four^\land$ of~$L^{1/1}\inner{\Gamma}{G}$ restricts to a
  map~$\Schw(G) \to \Schw(\Gamma)$. By symmetry, it is also true that
  the backward Fourier operator~$\Four^\lor$
  of~$L^{1/1}\inner{\Gamma}{G}$ restricts
  to~$\Schw(\Gamma) \to \Schw(G)$.  Since~$\Four^\land$
  and~$\Four^\lor$ are mutually inverse in the classical Fourier
  singlet~$L^{1/1}\inner{\Gamma}{G}$, the same must be true for the
  corresponding restrictions in~$\Schw(G) \pto \Schw(\Gamma)$.

  Closure of~$\Schw(G)$ under convolution is established
  in~\cite[p.~42]{Osborne1975}, where $\Schw(G)$ is (temporarily)
  denoted by~$\mathcal{C}(G)$ and the even stronger closure property
  $\mathcal{A}(G) \star \mathcal{C}(G) \subseteq \mathcal{C}(G)$ for a
  certain space~$\mathcal{A}(G) \subseteq L^\infty(G)$
  containing~$\mathcal{C}(G)$ is stated.  Since we have
  $s \cdot s' = \Four^\lor(\Four^\land s \star \Four^\land s')$
  for~$s, s' \in \Schw(G) \subseteq L^{1/1}(G)$, closure under
  convolution~$\star$ implies closure under the pointwise
  product~$\cdot$ as well.

  It is also mentioned in~\cite[p.~42]{Osborne1975} that~$\Schw(G) = \mathcal{C}(G)$ is
  translation-invariant, meaning closed under the action of~$G \subset H(\pomega)$. By symmetry, the
  same holds for the action of~$\Gamma \subset H(\pomega)$. Since the torus action
  of~$\Tor \subset \CC$ is trivial, this shows that~$\Schw(G)$ is closed under the Heisenberg
  action.
\end{proof}

We call~$\Schw \inner{\Gamma}{G} := [\Four\colon \Schw(G) \pto \Schw(\Gamma)]$ the \emph{Schwartz
  singlet} over the given Pontryagin duality~$\pomega\colon G \times \Gamma \to \Tor$. As stated
earlier, this generalizes some classical facts that we can now formulate on the basis of the forward
and backward Fourier operators~$\Four^\land$ and~$\Four^\lor$ of Example~\ref{ex:pont-doublet}.

\begin{myexample}
  Since~$\Four^\land\colon \Schw(G) \to \Schw(\Gamma)$
  and~$\Four^\lor\colon \Schw(\Gamma) \to \Schw(G)$ are restrictions
  of the corresponding Fourier operators in the classical Fourier
  singlet~$L^1\inner{\Gamma}{G}$, we need only review the signal
  space~$\Schw(G)$ and the spectral space~$\Schw(\Gamma)$ in each of
  the four cases:
  \begin{enumerate}[(a)]
  \item For the \emph{Fourier integral} (Example~\ref{ex:pont-doublet}\ref{itt:four-int})
    on~$G = \Gamma = \RR^n$, we obtain of course the classical Schwartz class~$\Schw(\RR^n)$
    mentioned in item~(3) above. So we have~$\phi \in \Schw(\RR^n)$ iff~$\phi$ is smooth and of
    rapid decay, in the sense that~$\phi(x) = o(x^{-\alpha})$ for all~$\alpha \in \NN^n$.
  \item For \emph{Fourier series} (Example~\ref{ex:pont-doublet}\ref{itt:four-ser}) we
    have~$G = \Tor^n$ and~$\Gamma = \ZZ^n$; in this case the elements of~$\Schw(\Tor^n)$ are just
    smooth functions (since~$\Tor^n$ is compact) while those of~$\Schw(\ZZ^n)$ are the
    sequences~$\eta\colon \ZZ^n \to \CC$ of rapid decay, so~$\eta(k) = o(k^{-\alpha})$
    for all~$\alpha \in \NN^n$.
  \item For the \emph{discrete-time Fourier transform}
    (Example~\ref{ex:pont-doublet}\ref{itt:four-dtft}), the situation
    is of course the same as for Fourier series, but with the roles
    of~$G$ and~$\Gamma$ interchanged.
  \item For the \emph{discrete Fourier transform}
    (Example~\ref{ex:pont-doublet}\ref{itt:four-dft}), we have the
    finite groups~$G = \Tor_N^n$ and~$\Gamma = \ZZ_N^n$, so signal and
    spectral spaces both stay the same as in the classical DFT Fourier
    singlet~$L^{1/1}\inner{\Tor_N^n}{\ZZ_N^n}$ because finite
    sequences are trivially ``smooth'' and ``rapidly decaying'', hence
    we have~$\Schw(G) \cong \Schw(\Gamma) \cong (\CC^N)^n$. The same
    is of course true for discrete Fourier series.
  \end{enumerate}
  As explained in item~(2) above, the \emph{elementary groups} are
  composed exactly of the component groups listed above. But the
  corresponding Schwartz-Bruhat functions must have the prescribed
  decay of item~(3), for all relevant variables \emph{jointly}. For
  example, consider the smooth (but non-analytic) function
  \begin{equation*}
    f\colon \CC \to \RR, z = x + iy \mapsto |\!\exp(z^4)| = e^{-x^4+6x^2y^2-y^4}.
  \end{equation*}
  Then~$f(x_0,y) = O(e^{-y^4})$ and~$f(x,y_0) = O(e^{-x^4})$ are
  in~$\Schw(\RR)$ for any fixed~$x_0, y_0 \in \RR$; but we
  have~$f(t,t) = e^{4t^2} \to \infty$ so
  that~$f \not\in \Schw(\RR^2)$. Similarly, the
  restriction~$\tilde{f}\colon \RR \times \ZZ \to \RR$ is
  in~$\Schw(\RR)$ when fixing the discrete variable and
  in~$\Schw(\ZZ)$ when fixing the continuous variable; nevertheless we
  have again~$\tilde{f} \not\in \Schw(\RR \times \ZZ)$. This shows
  that the decay properties of Schwartz-Bruhat functions on elementary
  groups, as per item~(3) above, are in general stronger than the
  corresponding decay properties for each component group separately.
\end{myexample}

Having introduced the Schwartz space~$\Schw(G)$ for a locally compact group, one may of course
define the \emph{space of tempered distributions}~$\Schw'(G)$ in the usual
way~\cite[Def.~4.1]{Wawrzynczyk1968}, and the induced Fourier transform~$\Four$ is again an
automorphism of $\CC$-vector spaces~\cite[p.~61]{Bruhat1961}. Of course, one cannot expect closure
under pointwise multiplication (as $\delta_0 \in \Schw'(\RR)$ shows) or convolution (as
$1 \in \Schw'(\RR)$ shows). In summary, one obtains the following result.

\begin{proposition}
  \label{prop:temp-distr}
  Let~$G$ and~$\Gamma$ be LCA groups under Pontryagin duality
  $\pomega\colon \Gamma \times G \to \Tor$. Then~$\Four\colon \Schw'(G) \pto \Schw'(\Gamma)$ yields
  a regular slain doublet $\Schw'\inner{\Gamma}{G}$ containing the Schwartz
  doublet~$\Schw\inner{\Gamma}{G}$.\qed
\end{proposition}

In fact, one may show~\cite[Thm.~4.3]{Wawrzynczyk1968} that~$\Schw'\inner{\Gamma}{G}$ contains the
regular doublet~$[M(G) \pto BC(\Gamma)]$ of Proposition~\ref{prop:measure-alg}. 

The above statement that~$\Schw\inner{\Gamma}{G} \le \Schw'\inner{\Gamma}{G}$ misses what is often
perceived as the most crucial property of the pointwise multiplication and convolution of tempered
distributions: While one may not apply these to \emph{two} arbitrary signals~$s, s' \in \Schw'(G)$
or spectra~$\sigma, \sigma' \in \Schw'(G)$, it is always possible to apply them to~$s \in \Schw(G)$
and~$s' \in \Schw'(G)$, or to~$\sigma \in \Schw(\Gamma)$ and~$\sigma' \in \Schw'(\Gamma)$.

We are thus led to introduce the structure of a \emph{twain module}~$(M, \star, \cdot)$ over a twain
algebra~$(A, \star, \cdot)$, defined as an abelian group~$(M, +)$ together with two scalar
multiplications~$\star\colon A \times M \to M$ and~$\cdot\colon A \times M \to M$ such
that~$(M, \star)$ is a module over~$(A, \star)$ while~$(M, \cdot)$ is a module over~$(A, \cdot)$.
Clearly, we recover modules in the usual sense (``plain modules'') if~$A$ is a plain algebra (one of
its multiplications being trivial) and naked abelian groups (``slain modules'') if~$R$ is a slain
algebra (both multiplications being trivial).

If~$A$ is moreover a Heisenberg twain algebra over~$\beta$, we obtain a \emph{Heisenberg twain
  module}~$M$ provided that~$(M, \star)$ is a recto Heisenberg module over~$\beta$ and~$(M, \cdot)$
is a verso Heisenberg modules over~$\beta$. Here the recto/verso distinction is analogous to the
case of Heisenberg twain algebras. Writing the action of both~$H(\beta)$ on~$A$ and~$M$ by
juxta\-position, $(M, \star)$ being recto thus
means~$x \act (a \star s) = (x \act a) \star s = a \star (x \act s)$ and
$\xi \act (a \star s) = (\xi \act a) \star (\xi \act s)$ for all~$a \in A, s \in M$
and~$x \in G, \xi \in \Gamma$, where $H(\beta) = TG \rtimes \Gamma$; for~$(M, \cdot)$ being verso
one just reverses the roles of scalars and operators.

Now let~$M'$ be another Heisenberg twain module for the same Heisenberg group~$H(\beta)$ such
that~$[\Four\colon M \pto M']$ is a slain doublet and let~$A'$ be another Heisenberg twain module
such that~$[\FFour\colon A \pto A']$ is a twain doublet. Then we call~$[\Four\colon M \pto M']$ a
\emph{twain doublet over} $[\FFour\colon A \pto A']$ if~$\Four (a \star s) = \FFour a \cdot \Four s$
and~$\Four (a \cdot s) = \FFour a \star \Four s$ for~$a \in A$ and~$s \in M$. By abuse of notation,
the Fourier operator~$\FFour$ on~$A$ is commonly denoted by~$\Four$ as well. We can now capture the
observation made above in the following statement.

\begin{proposition}
  \label{prop:temp-distr-twmod}
  Let~$G$ and~$\Gamma$ be LCA groups under Pontryagin duality
  $\pomega\colon \Gamma \times G \to \Tor$. Then~$\Schw'\inner{\Gamma}{G}$ is a regular
  Heisenberg twain doublet over the Schwartz doublet~$\Schw\inner{\Gamma}{G}$.\qed
\end{proposition}
\begin{proof}
  We have seen that~$\Schw\inner{\Gamma}{G}$ is a twain doublet (Theorem~\ref{thm:Schwartz-Bruhat})
  and that~$\Schw'\inner{\Gamma}{G}$ is a regular slain doublet
  (Proposition~\ref{prop:temp-distr}). The compatibility relation required for their Fourier
  operators follows from the usual duality definitions (using parentheses for the natural pairing),
  namely $\funcinner{\Four s}{\phi} = \funcinner{s}{\Four \phi}$,
  $\funcinner{a \cdot s}{\phi} = \funcinner{s}{a \cdot \phi}$,
  $\funcinner{a \star s}{\phi} = \funcinner{s}{\Par{a} \star \phi}$ for $a, \phi \in \Schw(G)$
  and~$s \in \Schw'(G)$, using also~$\Four^2 = \Par$. Similar remarks pertain to showing
  compatibility of the Heisenberg actions, defined by
  $\funcinner{x \act s}{\phi} = \funcinner{s}{(-x) \act \phi}$ and
  $\funcinner{\xi \act s}{\phi} = \funcinner{s}{\xi \act \phi}$ for~$x \in G$ and~$\xi \in \Gamma$.
\end{proof}

For applications in analysis it is very important that we can use certain relations between Fourier
and differential operators. Roughly speaking, differentiating a signal corresponds to multiplying
the spectrum by a polynomial (the \emph{symbol} of the differential operator applied to the signal),
and vice versa. It is thus crucial for us to incorporate these relations in our algebraic
framework. In doing so, we will not only capture the classical situation of Schwartz
functions~$\Schw(\RR^n)$ and tempered distributions~$\Schw'(\RR^n)$ but also their generalizations
in the \emph{Schwartz-Bruhat setting}. Since we are mainly interested in the classical case, though,
we shall only sketch the general procedure.

For a convenient and uniform treatment of differential operators on~$\Schw(G)$ and thus, by duality,
on~$\Schw'(G)$, we refer to the paper~\cite{Akbarov1995} already mentioned in
\S\ref{sub:schwartz-class}. For handling differential operators in general LCA groups, one
must admit functions that may depend on arbitrarily many variables $(x_q \mid q \in Q)$. Their
cardinality $|Q| = \dim(G)$ can be uncountable---clearly this is a theoretical seting not a
constructive framework. At any rate, differential operators are described by \emph{multi-indices},
namely functions~$\mu\colon Q \to \NN$ with~$\supp \mu = \{ q \in Q \mid \mu_q \neq 0\}$ being
finite. We follow~\cite{Akbarov1995} in denoting the set of all such multi-indices by~$\NN_Q$. Then
a differential operator has the general form
\begin{equation}
  \label{eq:diffop-akbarov}
  T = \sum_{\mu \in \NN_Q} p_\mu \der^\mu,
\end{equation}
where the $p_\mu \in C^\infty(G)$ are in the first instance arbitrary \emph{smooth functions} in the
sense of Bruhat~\cite[\S1.2]{Akbarov1995}, \cite[Def.~2]{Bruhat1961}; we shall later restrict them
to polynomials. The differential operator~\eqref{eq:diffop-akbarov} is \emph{locally of finite
  order}, meaning in a sufficiently small neighborhood of any point in~$G$, there are only finitely
many nonzero summands that contribute.

For fleshing this out in some more detail, let us briefly go through
the development of~\cite{Akbarov1995}. For any LCA group~$G$, one can
define the \emph{Lie algebra}~$\Lie(G)$ as its system of one-parameter
subgroups~\cite[\S1]{Akbarov1995}. One obtains a topological Lie
algebra~\cite[Thm.~1.1]{Akbarov1995} whose dimension coincides with
that of~$G$ in case it is finite. Defining a notion of basis suitable
for this setting~\cite[\S0]{Akbarov1995}, the vector space~$\RR^Q$ for
arbitrary index sets~$Q$ has the usual Kronecker basis, and every
other basis has the same cardinality~$|Q|$. For the Lie algebra one
gets~\cite[(1.1)]{Akbarov1995} as expected~$\Lie(G) \cong \RR^Q$, thus
allowing to fix a basis~$(e_q \mid q \in Q)$.

According to the description outlined in \S\ref{sub:schwartz-class}, the LCA group~$G$ may be
viewed as an inverse limit of Lie groups~$(G_\alpha \mid \alpha \in I)$. It is possible to realize
the latter as~$G_\alpha = G/H_\alpha$, where~$\lambda(G) := (H_\alpha \mid \alpha \in I)$ forms a
decreasing filtration of so-called good subgroups of~$G$. This complies with the terminology
of~\cite[\S1]{Bruhat1961}, where compact subgroup are called \emph{good} if their quotients are Lie
groups. Just as with the LCA group~$G$ itself, also its Lie algebra~$\Lie(G)$ is then an inverse
limit of the corresponding genuine Lie groups~$\Lie(G/H_\alpha)$; see~\cite[(1.5)]{Akbarov1995},
keeping in mind that our present setting is somewhat simpler since~$G$ is commutative and thus an LP
group~\cite[Prop.~1.10ii]{Wawrzynczyk1968}.

For defining the \emph{algebra of smooth functions}~$\E(G) \equiv C^\infty(G)$, it
suffices~\cite[\S1.2]{Akbarov1995} to define the subalgebra~$\D(G)$ of \emph{compactly supported}
ones since each function in~$\E(G)$ agrees with a function of~$\E(G)$ in a small neighborhood of any
fixed point~$x \in G$. One defines first the space~$\D(G : H_\alpha)$ of compactly supported smooth
functions invariant on a good subgroup~$H_\alpha \le G$. Functions~$\chi$ invariant on~$H_\alpha$
are in bijective correspondence with functions~$\tilde\chi$ on the Lie
group~$G_\alpha = G/H_\alpha$, so it is natural to declare~$\chi \in \D(G : H_\alpha)$
iff~$\tilde\chi \in C^\infty(G_\alpha)$. Then the representation of~$G$ as an inverse limit
over~$\lambda(G)$ translates~\cite[(1.6)]{Akbarov1995} into the direct limit
\begin{equation}
  \label{eq:smooth-bruhat}
  \D(G) = \bigcup_{\alpha \in I} \D(G : H_\alpha),
\end{equation}
so~$\psi \in \D(G)$ iff~$xH_\alpha \mapsto \psi(x)$ is a smooth Lie group
function~$G_\alpha \to \CC$ for some~$\alpha \in I$.

We explain now the \emph{action of the differential operator}~$T$ in~\eqref{eq:diffop-akbarov} on a
smooth function, which equals locally some~$\psi \in \D(G)$ around a point~$x \in G$. Each
one-parameter subgroup~$u \in \Lie(G)$ induces a derivation~$\psi \mapsto \psi'$ by the pointwise
limit~$\psi' := \lim_{t \to 0} \big(u(t) \act \phi-\phi\big)/t$, where $\act$ denotes the
Heisenberg action of~$G \le H(\pomega)$ on $\Schw(G) \le L^1(G)$. The derivations induced by the
basis vectors~$e_i$ are regarded as partial derivatives and denoted by~$\der_i$. 

By iterating and averaging over all possible differentiation orders~\cite[(2.2)]{Akbarov1995}, this
is then generalized to \emph{higher-order partial derivatives}~$\der^\alpha$ for~$\alpha \in \NN_I$.
Since~$\psi \in \D(G)$ is invariant on some~$H_\alpha$ as per~\eqref{eq:smooth-bruhat},
with~$G_\alpha = G/H_\alpha$ being a Lie algebra of finite dimension~$n$, one can find a
neighborhood~$U$ of~$x$ that is invariant under~$H_\alpha$ and a projection
map~$\rho\colon U \to \bar{U} \subseteq \RR^n$ such that~$\psi \in \D(G : H_\alpha)$
iff~$\psi = \bar{\psi} \circ \rho$ for a smooth function~$\bar\psi\colon \bar{U} \to \CC$ in the
usual sense~\cite[Lem.~2.13]{Akbarov1995}. In the neighborhood~$U$, the action of~$T$ is given in
terms of the standard partial derivatives in~$\RR^n$ by
\begin{equation*}
  T\psi = \sum_{|\mu| \le N} p_\mu \, \frac{\der^\mu \bar\psi}{\der q^\mu} \circ \rho,
\end{equation*}
where the local order~$N$ may depend on~$U$.

Let us now define the LCA version of the \emph{Weyl algebra}~$A_G(\CC)$ as the operator ring
consisting of all those~$T$ whose representation~\eqref{eq:diffop-akbarov} involves only
\emph{polynomial} coefficient functions~$p_\mu$. Here a smooth function~$G \to \CC$ is called a
polynomial~\cite[\S1]{Akbarov1994} if its restriction to each compactly generated closed
subgroup~$C$ of~$G$ can be written as a polynomial in a finite collection of real characters
on~$C$, here defined just as the characters but with the additive reals instead of~$\Tor = S^1$ as
their target group.

\begin{myremark}
  \label{rem:real-char}
  The \emph{theory of real characters} is important for LCA groups as well as more general
  topological groups~\cite{DiemWright1967} \cite{Ardanza-TrevijanoChascoDominguez2008}. They are
  also crucial for building up the Laplace transformation in the LCA setting~\cite{Mackey1948}
  \cite{Liepins1976}; see Remark~\ref{rem:laplace-transf} below. Real characters come in three
  different guises:
  \begin{itemize}
  \item As defined above, they may be taken as continuous homomorphisms~$G \to (\RR, +)$. Also
    \cite[Def.~(24.33)]{HewittRoss1994} calls such objects ``real characters'' while
    \cite[Def.~2]{Liepins1976} refers to them as ``linear functionals''.
  \item Sometimes, real characters are characterized as continuous
    homomorphisms~$G \to (\RR^+, \cdot)$; this is the stance taken in~\cite{Mackey1948}
    and~\cite[Def.~1]{Liepins1976}. Their objective is to define Laplace transforms in the LCA
    setting, where \emph{complex characters}~$G \to (\nnz{\CC}, \cdot)$ take the role of the usual
    characters~$G \to \Tor = S^1 \le \nnz{\CC}$. Following~\cite{Liepins1976} in calling the latter
    ``unitary characters'', it is clear that a complex character has a unique polar decomposition
    into a unitary character and a real character in the Mackey-Liepins sense of a continous
    homomorphism~$G \to (\nnz{\RR}, \cdot)$. It is clear that Mackey-Liepins characters and
    Hewitt-Ross characters may be identified via~$\log\colon (\RR^+, \cdot) \to (\RR, +)$.
  \item Each real character of~$G$ corresponds bijectively~\cite[Lem.~1]{Liepins1976} to a
    one-para\-meter subgroup of~$\hat{G}$; by definition, this is a continuous
    homomorphism~$(\RR, +) \to \hat{G}$.
  \end{itemize}

  The bijection between real characters and one-parameter subgroups goes as
  follows~\cite[(24.43)]{HewittRoss1994}. Given a real character~$\chi\colon G \to (\RR, +)$, the
  induced one-parameter subgroup $\hat\chi\colon (\RR, +) \to \hat{G}$ is defined by
  $\inner{\hat\chi(t)}{x} = \inner{\chi(x)}{t}$, meaning $\hat\chi(t)(x) = e^{i\tau\chi(x)t}$
  for~$t \in \RR$ and~$x \in G$. In the notation of~\cite[\S3]{Osborne1975}, this
  means~$\hat{\chi}(t) = \mathrm{Exp}(t\chi)$.

  Conversely, if~$\hat\chi\colon (\RR, +) \to \hat{G}$ is a one-para\-meter subgroup, the map
  $t \mapsto \inner{\hat\chi(t)}{x}$ for fixed~$x \in G$ yields a continuous homomorphism
  $\inner{\hat\chi}{x}\colon (\RR, +) \to \Tor$ and thus a character on~$(\RR, +)$. Since the LCA
  group~$(\RR, +)$ is self-dual~\cite[Ex.~1.2.7a]{Rudin2017}, the character~$\inner{\hat\chi}{x}$
  corresponds to a unique real number that we take to define~$\chi(x)$. By the definition of the
  correspondence~$\hat{\RR} \cong \RR$, this yields~$\inner{\hat\chi(t)}{x} = \inner{\chi(x)}{t}$,
  which establishes the claimed bijection.
\end{myremark}

In this terminology, Bruhat's definition of the \emph{Schwartz space} means that~$\psi \in \Schw(G)$
iff~$T\psi$ is bounded for each~$T \in A_G(\CC)$. It is then easy to see that the Weyl
algebra~$A_G(\CC)$ acts on~$\Schw(G)$ and thus---by duality---also on~$\Schw'(G)$.

\subsection{Classical Schwartz Functions and the Weyl Algebra.}
\label{sub:alg-schwartz-pol}
We shall from now on focus on the case most important for applications, the standard vector
duality~$\inner{\RR^n}{\RR^n}$ from Example~\ref{ex:classical-vector-group} whose Pontryagin
duality~$\pomega$ is the exponentiated inner product on~$\RR^n$. Thus setting $G = \Gamma = \RR^n$,
the \emph{classical Weyl algebra}~$A_G(\CC) = A_n(\CC)$ acts on the Schwartz class $\Schw(\RR^n)$ of
rapidly decaying functions as well as its associated space~$\Schw'(\RR^n)$ of tempered
distributions.

Let us first study the action of~$A_n(\CC)$ on Schwartz functions, which is induced by the original
Heisenberg action. Thus we are in fact confronted with two actions
\begin{equation}
  \label{eq:dbl-actn}
  A_n(\CC) \times \Schw(\RR^n) \to \Schw(\RR^n)
  \quad\text{and}\quad
  H(\pomega) \times \Schw(\RR^n) \to \Schw(\RR^n).
\end{equation}
Let~$e_{\bar\alpha}$ denote the \emph{modulation action} of~$\alpha \in \Gamma \le H(\pomega)$,
which corresponds to multiplication by~$e_{\bar\alpha}(x) := e^{i\tau x \bar\alpha}$.
If~$\bar\alpha$ is $\alpha$ times the $j$-th standard basis vector of~$\RR^n$ for some
scalar~$\alpha \in \RR$, we write~$e_j^\alpha$ in place of~$e_{\bar\alpha}$ so that
$e_{\bar\alpha} = e_1^{\alpha_1} \cdots e_n^{\alpha_n}$ for general modulations. We write
also~$t_{\bar a}$ for the \emph{translation action} of~$\bar{a} \in G \le H(\pomega)$ so
that~$t_{\bar a}s \, (x) = s(x-a)$. In analogy to the modulations, $t^a_k$ denotes~$t_{\bar a}$
with~$\bar{a}$ equal to $a$ times the $k$-th basis vector and~$a \in \RR$,
hence~$t_{\bar a} = t_1^{a_1} \cdots t_n^{a_n}$ for general translations. As a complex algebra, the
operators of~$H(\pomega) = \Tor G \rtimes \Gamma$ are generated by the~$e_j^\alpha$
and~$t_k^a$. While modulations and translations commute amongst each other, they are linked by the
crucial \emph{Heisenberg relations}~$e_j^\alpha t_k^a = \delta_{jk} \, \alpha a \, t_k^a e_j^\alpha$
harking back to~\eqref{eq:twisted-bimod}.

It is clear that the noncentral generators~$e_j := e_j^1$ and~$t_k := t_k^1$ of the Heisenberg
group~$H(\pomega)$ induce, respectively, the generators~$x_j$ and~$\der_k$ of the Weyl
algebra~$A_n(\CC)$. In the operator algebra~$\End\,\Schw(\RR^n)$, each~$x_j$ commutes with each of
the~$e_j^\alpha$ and each~$\der_k$ with each of the~$t_k^a$. To avoid awkward $i\tau$ factors (in
most places), it is customary to introduce the scaled partials~$i \tau D_k := \der_k$. Then the
Heisenberg relations correspond to the \emph{Weyl relations}~$[x_j, D_k] = \delta_{jk} \, i\tau$,
and we have the \emph{cross-relations}~$[D_k, e_j^\alpha] = \delta_{kj} \, \xi e_j^\alpha$
and~$[t_k^a, x_j] = \delta_{jk} \, a_j$.

In more detail, the said correspondence between the generators of
the $H(\pomega) = \Tor\RR^n\rtimes\RR^n$ and $A_n(\CC) = \CC\langle \der, x\rangle$ actions arises
from viewing the latter as the universal enveloping algebra of the Lie algebra of the former.  Since
we have seen in \S\ref{sub:heis-twist} that the \emph{Heisenberg
  twists}~$\hat{J}, \check{J}\colon H(\pomega) \to H(\pomega)^o$ play an important role
for~$H(\pomega)$, it is plausible to expect something similar for the induced map on~$A_n(\CC)$,
which we shall again denote by~$\hat{J}$.

On the level of Lie algebras, the induced map is the differential at the unit element, and it is
easy to see that as such, $\hat{J}$ flips the sign of momentum vectors and fixes position vectors
while~$\check{J}$ works the other way round. Hence the corresponding
anti-automorphisms~$A_n(\CC) \to A_n(\CC)$ are given
by~$\hat{J}(x^\alpha D^\beta) = (-1)^{|\beta|} D^\beta x^\alpha$
and~$\check{J}(x^\alpha D^\beta) = (-1)^{|\alpha|} D^\beta x^\alpha$, respectively. It is easy to
see that these maps are involutions (= involutive anti-automorphisms, as in
\S\ref{sub:heis-twist}), where~$\hat{J}$ is known as the \emph{(standard)
  transposition}~\cite[\S16.2]{Coutinho1995}, \cite[\S V1a]{Sabbah2007}.

One can turn both maps into honest automorphisms. Recall from~\S\ref{sub:heis-twist} the
identification~$H(\pomega) \cong H(\pomega)^o$ given
by~$c \, (x, \xi) \leftrightarrow c \, (\xi, x)$. The induced map on the enveloping algebra yields
the corresponding identification~$A_n(\CC) \cong A_n(\CC)^o$ with $x \leftrightarrow \der$. If we
combine it with~$\hat{J}\colon A_n(\CC) \to A_n(\CC)^o$ and the scaling factor~$i\tau = \der/D$, we
obtain the automorphism
\begin{equation}
  \label{eq:four-transf-weyl}
  \hat{\mathfrak{f}}\colon A_n(\CC) \to A_n(\CC)\quad\text{with}\quad
  x \mapsto D_x, D_x \mapsto -x
\end{equation}
called the \emph{Fourier transform} on the Weyl algebra~\cite[\S5.2]{Coutinho1995}, \cite[\S
V2a]{Sabbah2007}. To be precise, we should perhaps call~$\hat{\mathfrak{f}}$ the forward Fourier
transform, whereas the automorphism~$\check{\mathfrak{f}}$ induced
by~$\check{J}\colon A_n(\CC) \to A_n(\CC)^o$ would be called the backward Fourier transform---it
turns out to be~$\hat{\mathfrak{f}}^{-1}$. On basis vectors, the Fourier transforms act
via~$\hat{\mathfrak{f}}(x^\alpha D^\beta) = (-1)^{|\beta|} \, D^\alpha x^\beta$
and~$\check{\mathfrak{f}}(x^\alpha D^\beta) = (-1)^{|\alpha|} \, D^\alpha x^\beta$

It should be noted that the Fourier transform~$\hat{\mathfrak{f}}$, unlike the involutive
transposition~$\hat{J}$, has \emph{periodicity four} just as its group-theoretic parent on the
Heisenberg clock (Figure~\ref{fig:heis-clock}). It
square~$\hat{\mathfrak{f}}^2 = \hat{J} \check{J} = \check{J} \hat{J}$ is the
automorphism~$A_n(\CC) \to A_n(\CC)$ induced by~$x \mapsto -x, \der \mapsto -\der$; it is denoted by
an overbar in~\cite[\S V2b]{Sabbah2007}. Combined with the above-mentioned
identification~$A_n(\CC) \cong A_n(\CC)^o$ with~$x \leftrightarrow \der$, this yields the so-called
\emph{principal anti-automorphism}~\cite[\S I.2.4]{Bourbaki1989} of the Weyl algebra (induced by the
inversion map on the group level).

It remains to study the interaction between the Weyl algebra~$A_n(\CC)$ and the \emph{twain algebra
  structure} of~$\Schw(\RR^n)$ as well as the Fourier
operator~$\Four\colon \Schw(\RR^n) \pto \Schw(\RR^n)$. In both cases, the relevant relations follow
from the corresponding action of the Heisenberg group. Thus each~$D_k$ acts as a derivation
on~$\Schw(\RR^n)_{\tdot}$ but as a scalar on~$\Schw(\RR^n)_\star$ while each~$x_j$ is a scalar
for~$\Schw(\RR^n)_{\tdot}$ but a derivation for~$\Schw(\RR^n)_\star$; thus the recto/verso
distinction of \S\ref{sub:cat-heisalg} transfers from~$H(\pomega)$ to~$A_n(\CC)$.

While the relations between~$A_n(\CC)$ and the pointwise structure are clear, those for the
convolution follow by the Fourier transform and the well-known \emph{differentiation laws}
\begin{equation}
  \label{eq:diff-laws}
  \hat{\Four}(T \cdot s) = \hat{\mathfrak{f}}(T) \cdot \hat{\Four}(s),\quad
  \check{\Four}(T \cdot s) = \check{\mathfrak{f}}(T) \cdot \check{\Four}(s)
\end{equation}
for all~$T \in A_n(\CC)$ and~$s \in \Schw(\RR^n)$; see for example~\cite[\S6]{Bracewell1986},
\cite[\S3.3]{Strichartz1994}.

As for the Heisenberg action, one may also couch these laws in terms of left and right modules: If
the Fourier operators are considered as linear homomorphisms from a left to a right Heisenberg
module, we may additionally view these as \emph{left and right $D$-modules} (with $D = A_n(\CC)$
being the Weyl algebra). As for the Heisenberg situation, the right action then corresponds to a
left action via~$\hat{J}, \check{J}\colon A_n(\CC) \to A_n(\CC)^o$ for the Fourier
operators~$\hat{\Four}, \check{\Four}$.

Let us summarize our findings by introducing some tentative terminology. By a \emph{Weyl action}
on~$\Schw(\RR^n)$ we mean an action of~$A_n(\CC)$ satisfying the
cross-relations. Since~$\Schw(\RR^n)$ is a twain algebra, the generators~$D_k$ and~$x_j$ must
interact accordingly (derivations/scalars) with the multiplications~$\star$ and~$\cdot$; for plain
or slain algebras, these requirements would be diminished or cancelled, as the case may be. If we
have a Fourier operator between Heisenberg algbras such that the differentation
laws~\eqref{eq:diff-laws} are satisfied, we speak of a \emph{compatible} Weyl action. It is easy to
see that similar statements can be made, via the corresponding transpose operators, about the twain
module~$\Schw'(\RR^n)$ of tempered distributions. Altogether, we have then the following result.

\begin{proposition}
  \label{prop:comp-weyl}
  The Fourier singlets~$[\Four\colon \Schw(\RR^n) \pto \Schw(\RR^n)]$ as well as
  $[\Four\colon \Schw'(\RR^n) \pto \Schw'(\RR^n)]$ are endowed with compatible Weyl action.
\end{proposition}

We conjecture that Proposition~\ref{prop:comp-weyl} has a \emph{generalization} to compatible Weyl
actions of~$A_G(\CC)$ and~$A_{\hat{G}}(\CC)$ on the Schwartz-Bruhat functions
$[\Four\colon \Schw(G) \pto \Schw(\hat{G})]$ and tempered
distributions~$[\Four\colon \Schw'(G) \pto \Schw'(\hat{G})]$ for an arbitrary Pontryagin
duality~$\pomega\colon \hat{G} \times G \to \CC$. In that case, the Fourier transform of the Weyl
algebra~$\hat{\mathfrak{f}}\colon A_G(\CC) \to A_{\hat{G}}(\CC)$ would map the real
character~$\chi\colon G \to (\RR, +)$ to the derivation~$D_\chi := i\tau \, \der{\hat{\chi}}$
induced by the one-parameter group~$\hat{\chi} \in \Lie(\hat{G})$, as detailed in
Remark~\ref{rem:real-char}, while correspondingly mapping~$D_\chi$ to~$-\chi$.

\begin{remark}
  \label{rem:schwartz-integrodiff}
  Since the Schwartz class~$\Schw(\RR^n)$ and its distribution module~$\Schw'(\RR^n)$ play a
  prominent role in the theory of Fourier operators while at the same time having the structure of a
  differential algebra with derivation~$\der$, a short foray into their ``integro-differential
  structure'' is certainly not out of place. We shall see, however, that in this case we do not have
  integro-differential algebras~\cite{RosenkranzRegensburger2008a,BuchbergerRosenkranz2012} or
  integro-differential modules~\cite[Lem.~14]{RosenkranzSerwa2019}, not even differential
  Rota-Baxter algebras or modules~\cite[Ex.~3.8c]{GaoGuoRosenkranz2017}.

  Let us therefore recall the terminology of~\cite[\S4.1]{GuoRegensburgerRosenkranz2014}. Given a
  differential algebra~$(\galg, \der)$, a (linear) section of~$\der$ is called an antiderivative and
  a (linear) quasi-inverse a quasi-antiderivative. Note that antiderivatives are a special case of
  quasi-antiderivatives. What we shall need in the sequel is their converse: (linear) retractions of
  (injective but non-surjective!) derivations. Let us refer to these, tentatively, as
  \emph{retroderivatives}.

  For the sake of simpler notation, we state only the \emph{ordinary case} with
  derivations~$\der\colon \Schw(\RR) \to \Schw(\RR)$ and~$\der\colon \Schw'(\RR) \to \Schw'(\RR)$
  and (quasi)inverses to be given. It is then straightforward to set up the corresponding partial
  operators~$\der_x, \der_y, \dots$ with their (quasi)inverses as in~\cite{RosenkranzGaoGuo2019}.

  Recall that the \emph{space of bump functions}~$\D(\RR)$, meaning smooth functions of compact
  support, is a (nonunital) differential algebra that does not admit antiderivatives since its
  derivation is indeed non-surjective: Writing~$\ocum\colon \D(\RR) \to \CC$ for the definite
  integral $\ocum s := \cum_{-\infty}^\infty s(\xi) \, d\xi$, one
  has~$s \in \im(\der) \Leftrightarrow \ocum s = 0$ according
  to~\cite[Lem.~2.5.1]{Stakgold1979}. This condition also characterizes the image
  of~$\der\colon \Schw(\RR) \to \Schw(\RR)$. Indeed, if~$s = S'$ for some~$S \in \Schw(\RR)$, we
  have~$S(-\infty) = 0$ and thus
  \begin{equation}
    \label{eq:schwartz-antider}
    S(x) = \cum_{-\infty}^x s(\xi) \, d\xi,
  \end{equation}
  which for~$x \to +\infty$ implies~$\ocum s = 0$. Conversely, assuming the latter condition on~$s$
  implies for~$S$ defined by~\eqref{eq:schwartz-antider} that for any~$n > 1$ we
  have~$|S(x)| \le C |x|^{-n+1}$ where~$C = C_n$ is chosen so that~$|s(\xi)| \le C |\xi|^{-n}$,
  using the fact that~$s \in \Schw(\RR)$. This implies~$S = O(|x|^{-k})$ as~$x \to -\infty$ for
  any~$k > 0$ while~$S = O(|x|^0) = O(1)$ follows already from~$S(-\infty) = 0$ by the
  definition~\eqref{eq:schwartz-antider}. For the asymptotics as~$x \to +\infty$, we
  use~$\ocum s = 0$ to write the integral as~$S(x) = \cum_x^\infty s(\xi) \, d\xi$, and as before we
  obtain $S(x) = O(x^{-k})$ for any~$k > 0$, and this time the case~$k = 0$ follows
  from~$\ocum s = 0$.
  % Proof posted here:
  % https://math.stackexchange.com/questions/815842/integration-of-a-function-in-schwartz-space

  We can set up a retroderivative~$\vcum$ both on~$\D(\RR)$ and on~$\Schw(\RR)$.
  Following~\cite[Exc.~2.6.17]{Strichartz1994} and~\cite[\S2.5]{Stakgold1979}, we choose a bump
  function $s_0 \in \D(\RR) \subseteq \Schw(\RR)$ with~$\ocum s_0 = 1$,
  say~$s_0(x) := c \, \exp(\tfrac{1}{x^2-1}) \, H(1-|x|)$, where~$c$ is a normalization
  constant~\cite[Exc.~2.1.1]{Stakgold1979}.
  % It may be expressed in terms of elliptic functions:
  % https://math.stackexchange.com/questions/145015/evaluate-definite-integral-int-11-exp1-x2-1-dx
  Then we define~$\vcum$ on~$\D(\RR)$ and~$\Schw(\RR)$ by applying the indefinite integral
  operator~$s \mapsto S$ of~\eqref{eq:schwartz-antider} to~$s - (\ocum s) \, s_0 \in \ker \ocum =
  \im \der$ instead of~$s$. Thus we set
  \begin{equation}
    \label{eq:1}
    \vcum s := \cum_{-\infty}^x \Big( s(\xi) - (\ocum s) \, s_0(\xi) \Big) \, d\xi,
  \end{equation}
  and it is easy to see that~$\vcum \der = 1$ so that~$\vcum$ is surjective. But~$\vcum$ is not
  injective since clearly~$\ker \vcum = [s_0]$. We obtain the direct
  decompositions~$\im \der \dotplus \ker \vcum = \D(\RR)$
  and~$\im \der \dotplus \ker \vcum = \Schw(\RR)$ associated with the
  projector~$s \mapsto s - (\ocum s) \, s_0$. As we have seen, these spaces
  are~$\im \der = \ker \ocum$ and~$\ker \vcum = [s_0]$. The choice of~$s_0$ picks out a complement
  of the deficient image of the derivative, which engenders the kernel of the
  retroderivative~$\vcum$.

  It is not to be expected that~$\vcum$ be a Rota-Baxter operator, in the sense of
  satisfying~$(\vcum s_1)(\vcum s_2) = \vcum s_1 \vcum s_2 + \vcum s_2 \vcum s_1$. Indeed, the
  correction term~$(\ocum s_1) \, \vcum s_0 \vcum s_2 + (\ocum s_2) \, \vcum s_0 \vcum s_1$ is
  required (on the left-hand side), so weight terms~\cite[Def.~2.1b]{GuoRegensburgerRosenkranz2014}
  are of no avail in this case. Instances with a nonzero correction are easy to come by: For
  example, taking $s_1 = s_0$ and~$s_2 = s_0'$ leads to the correction~$\epsilon = \cum s_0^2$ with
  $|\epsilon(0)| > 0.03$.

  Let us now turn to the distribution spaces~$\D'(\RR)$ and~$\Schw'(\RR)$, where one has a
  well-known antiderivative~\cite[(5.6)]{Stakgold1979}, namely the transpose of~$-\vcum$. By abuse
  of notation, we shall also denote it by~$\vcum$. Since the derivation~$\der$ on~$\D'(\RR)$
  and~$\Schw'(\RR)$ are also defined as the tranpose of the derivation~$-\der$ on the corresponding
  primal spaces, it is clear that~$\vcum$ is an \emph{antiderivative} on the distribution
  spaces. Moreover, we see that~$\der$ is surjective and~$\vcum$ is injective for distributions,
  just as in the familiar differential Rota-Baxter and integro-differential
  settings~\cite[Ex.~3.8cd]{GaoGuoRosenkranz2017}. Indeed, the kernel of~$\der$ is easily
  seen~\cite[\S 2.5]{Stakgold1979} to be the constant
  distributions~$\RR \subset \Schw'(\RR) \subset \D'(\RR)$, and the image of~$\vcum$ the
  distributions annihilating~$s_0$. In other words, the \emph{initialization}~$1 - \vcum \der$ is
  the projector that maps a distribution~$\tilde{s}$ to the constant
  distribution~$\tilde{s}(s_0) \in \RR$.

  In view of the results for~$\D(\RR)$ and~$\Schw(\RR)$, one will anticipate that~$\vcum$
  on~$\D'(\RR)$ and~$\Schw'(\RR)$ does not satsify the Rota-Baxter axiom in the
  form~$(\vcum s)(\vcum \tilde{s}) = \vcum \tilde{s} \vcum s + \vcum s \vcum \tilde{s}$, where~$s$
  ranges over (bump or Schwartz) functions and~$\tilde{s}$ over (general or tempered) distributions,
  with the convenient notation~$\vcum$ accordingly overloaded. This axiom would make the
  corresponding distribution spaces into differential Rota-Baxter
  modules~\cite[Ex.~3.8c]{GaoGuoRosenkranz2017}, but again we have (on the left-hand side) a
  correction term $(\ocum s) \, s_0 \, \vcum \tilde{s} + \tilde{s}(\vcum s_0 \, \vcum s)$, which can
  easily be worked out from the corresponding correction on the primal spaces.
\end{remark}

Before ending this chapter, a few words seem to be in order about the \emph{Laplace transform},
where the all-important differentiation laws above continue to hold. In fact, they could be said to
really come into their own since analyticity enters the picture. Since the relevant ideas have been
developed in the general setting of Pontryagin duality, we shall now leave the specific setting of
Schwartz class on~$\RR^n$, returning to the general Schwartz-Bruhat space~$\Schw(G)$ on an LCA
group.

\begin{myremark}
  \label{rem:laplace-transf}
  Recall the \emph{complex characters} discussed in Remark~\ref{rem:real-char}. For an LCA group~$G$
  under Pontryagin duality~$\pomega\colon \hat{G} \times G \to \Tor$, these are continuous
  homomorphisms~$\zeta\colon G \to \nnz{\mathbb{C}} = \RR^+ \times \Tor$. If the space of
  characters~$G \to (\RR, +)$ is denoted by~$G^\#$ as in~\cite[Def.~2]{Liepins1976}, every complex
  character has a unique polar decomposition~$\zeta = e^\rho \xi$ for~$\rho \in G^\#$
  and~$\xi \in \hat{G}$.

  We set~$\Gamma := G^\# \times \hat{G}$ and define the \emph{extended Pontryagin duality}
  $\tilde\pomega\colon \Gamma \times G \to \nnz{\CC}$
  by~$\inner{\rho, \xi}{x}_{\tilde\pomega} := e^{\rho(x)} \, \inner{\xi}{x}_\pomega$. It is indeed a
  duality over the torus~$\nnz{\CC}$ as can check easily. Note also
  that~$\Tor G \rtimes \hat{G} \le \nnz{\CC} G \rtimes \Gamma$,
  meaning~$H(\pomega) \le H(\tilde\pomega)$, so Heisenberg actions over~$\tilde\pomega$ restrict to
  those over~$\pomega$.  The Laplace transformation of~\cite{Liepins1976} is expected to carry over
  to the setting of Heisenberg algebras in the following sense.

  Writing $\dot{L}^2(G) \subseteq L^2(G)$ for strongly $L^2$ functions~\cite[Def.~2]{Liepins1976}
  definining $C^\omega(\tilde{G})$ as the functions analytic in the sense
  of~\cite[Def.~9]{Liepins1976} on~$U \times \hat{G}$ for a convex open zero
  neighborhood~$U \subseteq G^\#$, we obtain a regular Fourier doublet
  $[\mathcal{L}\colon \dot{L}^2(G) \pto C^\omega(\tilde{G})]$ as per Theorem~7
  of~\cite{Liepins1976}. To be precise, the ``functions'' of~$C^\omega(\tilde{G})$ should be
  understood as suitably defined partial germs (direct limit indexed over
  subsets~$U \times \hat{G} \subset \Gamma$ of the form specified above), with the $L^2$ Laplace
  transformation~\cite{Mackey1948}, \cite{Liepins1976} given by
  \begin{equation*}
    \mathcal{L}s \, (\rho, \xi) = \Four(e^\rho s)(\xi) = \int_G e^{\rho(x)} \, \inner{\xi}{x}_\pomega \,
    s(x) \, dx
  \end{equation*}
  in terms of the usual $L^2$ Fourier transform (Proposition~\ref{prop:L2-doublet}). It should be
  clear, however, that the \emph{Laplace
    doublet}~$[\mathcal{L}\colon \dot{L}^2(G) \pto C^\omega(\tilde{G})]$ is defined over~$\pomega$
  rather than~$\tilde\pomega$ since the action of~$G^\# \le H(\tilde\pomega)$ is only local.

  The notion of differential operator~\eqref{eq:diffop-akbarov} sketched above is also applicable to
  functions (locally) defined on~$\Gamma = G^\# \times \hat{G}$. Indeed, a one-parameter subgroup is
  given by~$\RR \times \Gamma \to \Gamma$,
  $t \cdot (\rho, \xi) = (\rho + t \underline{\rho}, \xi + t \underline{\xi})$ for an
  arbitrary~$(\underline{\rho}, \underline{\xi}) \in \Gamma$. In contrast
  to~\cite[(1)]{Liepins1976} we write~$\hat{G}$ additively, identifying one-parameter subgroups
  on~$\hat{G}$ as in Remark~\ref{rem:real-char} with real characters~$\underline{\xi}$ so
  that~$(\xi+t\underline{\xi})(x) = e^{i\tau\underline{\xi}(x)t} \xi(x)$ for~$x \in G$
  and~$t \in \RR$. According to~\cite[(1)]{Liepins1976}, the \emph{partial derivative} of a
  spectrum~$\sigma \in C^\omega(\tilde{G})$ is then defined as
  \begin{equation*}
    \frac{\der\sigma}{\der\underline{\zeta}} \, (\zeta) = \lim_{t \to 0} \frac{\sigma(\rho
      + t \underline{\rho}, \xi + t \underline{\xi})}{t}
  \end{equation*}
  at the point~$\zeta = (\rho, \xi) \in \tilde{G} \subseteq \Gamma$ in the
  direction~$\underline{\zeta} = (\underline{\rho}, \underline{\xi})$. Writing the
  latter~$\underline{\zeta} = \underline{\rho} + i \underline{\xi}$ induces a complex structure
  on~$\Gamma$ as per~\cite[(2)]{Liepins1976}, so that the notion of analyticity means
  $\CC$-linearity along with certain $L^2$ requirements.

  The upshot is that the \emph{differentiation
    law}~$\der_{\underline{\zeta}} \mathcal{L} = \mathcal{L} \underline{\zeta}$ holds
  on~$\dot{L}^2(G)$, as reported in~\cite[Thm.~5]{Liepins1976}. (But note a typo: the function on
  the left-hand side there misses the Laplace sign.) Here the natural action
  of~$\underline{\zeta} = \underline{\rho} + i \underline{\xi} \in \Gamma$ on~$\dot{L}^2(G)$ is
  componentwise, endowing~$C^\omega(\tilde{G})$ with the structure of a~$\CC[\Gamma]$-module. Of
  course, also~$\dot{L}^2(G)$ is a $\CC[\Gamma]$-module
  via~$\underline{\zeta} \cdot s = \der_{\underline{\zeta}} s$, and
  then~$\mathcal{L}\colon \dot{L}^2(G) \to C^\omega(\tilde{G})$ is a $\CC[\Gamma]$-linear
  homomorphism.

  There are two essential instances of the Laplace doublet (each with its obvious multidimensional
  generalizations):
  \begin{itemize}
  \item Choosing~$G=\ZZ$ leads to \emph{Laurent series}~\cite[Ex.~1]{Liepins1976}. One
    obtains~$\hat{G} = \Tor$ and~$G^\# = \RR$. The Laplace transform is
    \begin{equation*}
      \mathcal{L}s \, (\rho, \xi) = \sum_{n=-\infty}^\infty s_n \, \rho^n \, e^{i\tau n\xi},
    \end{equation*}
    where we may read~$\rho \in \RR^+$ as radius and~$\xi \in \Tor \cong \RR/\ZZ$ as angle
    measure. The region of convergence (ROC) is some annulus~$\tilde{G} \subseteq U \times \RR^+$
    with~$\Tor \subseteq U$.
  \item On the other hand, taking~$G=\RR$ yields the classical \emph{bilateral Laplace
      transformation}~\cite[Ex.~2]{Liepins1976}. Its ROC is a vertical
    strip~$\tilde{G} \subseteq U \times \RR$ with horizontal span~$U \subseteq \RR$.
  \end{itemize}

  For the univariate bilateral Laplace transformation (also known as ``Fourier-Laplace
  transformation'' when applied to distributions), one has the so-called \emph{Paley-Wiener
    theorems}~\cite[7.2.1-4]{Strichartz1994} characterizing the compactly supported elements
  of~$L^2(\RR)$, $C^\infty(\RR), \D'(\RR), L^2(\RR^+)$ via growth rates of their transforms.

  Laplace transforms are of course ubiquitous in diverse applications. The bilateral version
  mentioned above is related to the more common \emph{unilateral Laplace transform}~$\mathcal{L}^+$
  by~$\mathcal{L} = \mathcal{L}^+ + \mathcal{L}^-$, where one sets~$\mathcal{L}^- := \Par
  \mathcal{L}^+ \Par$ with the reversal operator~$\Par$ of
  Theorem~\ref{thm:pont-doublet}. Conversely, one obtains the unilateral transformation from the
  bilateral one via~$\mathcal{L}^{\pm} = \mathcal{L} h_{\pm}$ where~$h_{\pm}$ is (the multiplication
  operator associated with) the characteristic function of the half-line~$\RR^{\pm}$, also known as
  the (ascending/descending) Heaviside function~\cite[\S3]{RosenkranzSerwa2019}. The presence of the
  Heavisides leads to Dirac terms upon differentation, which are in turn responsible for the
  characteristic initial values in differential relations such
  as~$\mathcal{L} \der_{\underline{\zeta}} = \underline{\zeta} \mathcal{L} - \ev_{0+}$,
  where~$\ev_{0+}$ means evaluation at~$0$ as a right-handed limit. For solving initial value
  problems, the evaluation term is crucial since it allows to incorporate the given initial
  data. Another advantage of the unilateral transformation over its bilateral sibling is that its
  ROC is usually larger. (Intuitively speaking, the bilateral transformation converges only on a
  strip but the unilateral one on a half-plane.)

  Note that the differentation law just mentioned for the unilateral Laplace transformation is the
  converse of the one mentioned above for the bilateral transformation. It would be worthwhile to
  work out suitable extended function spaces based on the tempered distribution spaces~$\Schw'(G)$
  and~$\Schw'(\hat{G})$ such that both differential laws hold for the generalized unilateral Laplace
  transform (without evalation terms), as with the compatible Weyl action of
  Proposition~\ref{prop:comp-weyl}. We expect that such function spaces would enjoy a compatible
  action of an \emph{extended Weyl algebra}, which is ``thickened'' to contain complex
  multipliers~$\underline{\zeta} = \underline{\rho} + i \underline{\xi}$ on signals and the
  corresponding differential operators~$\der_{\underline{\zeta}}$ on spectra.

  Physically speaking, this extension from Fourier to Laplace transformations may be interpreted as
  follows: The Fourier approach is based on analyzing/synthesizing signals into/from spectra that
  are essentially pure oscillations. The Laplace approach generalizes pure oscillations to
  \emph{damped oscillations}, with the possibility to control the damping factor as an additional
  parameter (via real characters).

  From the algebraic point of view, we may summarize the role of the Laplace transform as follows:
  As opposed to the Fourier transform, it appears to be less amenable to a global description such
  as we have for the Fourier transform (where we have many classical examples of Fourier doublets in
  various gradations). For the time being, we tend to see it more as a \emph{local tool} that
  facilitates the close-up study of a \emph{fixed} function by the tools of complex analysis.

  Another algebraic aspect of the Laplace transform would be interesting to pursue further in the
  present context: In the classical case of $G = \Gamma = \RR$, the functions supported on~$\RR^+$
  form an important subalgebra~$\big(L^2(\RR^+), \star\big)$ that is known to be an integral domain
  (the so-called Titchmarsh Theorem). Its fraction field is the main object of study for the
  \emph{Mikusi{\'n}ski calculus}, a kind of algebraic formulation of the Laplace
  transform~\cite{Mikusinski1959}, \cite{Yosida1984}. In particular, the Heaviside
  function~$h \equiv \{1\}$ of~\cite[\S2]{Yosida1984} has as its reciprocal~$s$ a kind of
  (bijective!)  differentiation operator. One may then adjoin further hyperfunctions such
  as~$e^{-\sqrt{s}}$ in~\cite[\S27]{Yosida1984}, which is an algebraic version of the heat
  propagator.
  
  It would be very intersting to investigate the relationship between this calculus and localized
  Fourier doublets. For example, it is clear that the Gaussian $D$-module~\eqref{eq:weyl-closure}
  developed in~\S\ref{sub:min-schwartz-algebra} below is an integral domain and so could be
  localized. For getting something like the hyperfunction~$s$, one might start
  with~$h\Schw(\RR) := \{ h_t f \mid f \in \Schw(\RR), t \in \RR\}$, where~$h_t := t \act h$ are the
  translated Heaviside functions. By a version of the Titchmarsh Theorem,
  $h\Schw(\RR) \subset \big( L^1(\RR), \star \big)$ is an integral domain and thus has a fraction
  field that might profitably be compared with the Mikusi{\'n}ski field.
\end{myremark}

\section{Constructive Fourier Analysis via Schwartz Functions}
\label{sec:constructive-schwartz}
% =====================================================================

\subsection{A Minimal Subalgebra of the Schwartz Class.}
\label{sub:min-schwartz-algebra}
We shall now focus on the Schwartz singlet~$\Schw\inner{\RR^n}{\RR^n}$ corresponding to the standard
vector duality of Example~\ref{ex:classical-vector-group} since this is the most important case for
applications of Fourier analysis to partial differential equations. We will define two
subsinglets~$\bar{\Gau}\inner{\RR}{\RR} \le \Schw(\RR^n)$, continuing at first with the base
field~$\RR$. This treatment is \emph{relatively constructive}, assuming an oracle for computations
in~$\RR$. In the subsequent \S\ref{sub:gelfond-field}, we shall build up a constructive
subfield~$\QQ^\pi$ of~$\RR$ that will allow us to build up the restricted Fourier
singlet~$\bar{\Gau}\inner{\QQ}{\QQ}$ with the constructive base field~$\QQ^\pi$ so as to allow a
purely algebraic and algorithmic description in \S\ref{sub:alg-schwartz-pol}.

For simplicity, we shall take~$n=1$. As we have seen in~\eqref{eq:fourop-mult}, this is no essential
restriction since the multivariate case can be reduced to the \emph{univariate} one. In fact, the
classical Fourier integral~\eqref{eq:four-int} exhibits this reduction, which works the same in the
classical twain singlet $L^{1/1}\inner{\RR^n}{\RR^n}$, and the latter contains the Schwartz singlet
$\Schw\inner{\RR^n}{\RR^n}$ by Proposition~\ref{thm:Schwartz-Bruhat}.

Our intention is to set up a subalgebra of~$\Schw(\RR)$ that is as simple as possible. Since we need
functions of rapid decay, polynomials will not do. The simplest choice that comes to mind is given
in terms of the \emph{Gaussian normal distribution}. Fixing mean~$\mu$ and
variance~$1 \big/ \tau\rho$, the corresponding probability density is~$\rho^{1/2} \, g_{\mu,\rho}$
with $g_{\mu,\rho}(x) := e^{-\pi\rho(x-\mu)^2}$. We prefer to avoid normalization factors that
involve~$\sqrt{\pi}$, so we shall only work with the unnormalized Gaussian distribution
functions~$g_{\mu,\rho}$, which we shall briefly call \emph{Gaussians}. Their $\CC$-linear span
yields the important algebra
\begin{equation}
  \label{eq:gaussians}
  \Gau_0(\RR) := [ \, g_{\mu, \rho} \mid (\mu, \rho) \in \RR \times \RR_{\ge 0} \, ]_\CC \le C^\infty(\RR)
\end{equation}
under pointwise multiplication. Indeed, we have the explicit product
law~$g_{\mu_1, \rho_1} g_{\mu_2, \rho_2} = c \, g_{\mu, \rho}$ for~$\rho_1 \rho_2 > 0$, where the
Gaussian parameters are~$\rho := \rho_1 + \rho_2$ and $\mu := (\rho_1 \mu_1 + \rho_2 \mu_2)/\rho$,
and where~$c := e^{-\pi\rho_{12} (\mu_1 - \mu_2)^2}$ with~$\rho_{12} := \rho_1 \rho_2 / \rho$ is a
(relative) normalization constant. Of course, the case~$\rho_1 = \rho_2 = 0$ is trivial
since~$g_{\mu,0} = 1$. In fact, this is the only element that is \emph{not} contained in the Schwartz
class~$\Schw(\RR)$ since~$1$ does not have rapid decay (well---no decay at all). Hence we shall
prefer to work with the (nonunital!) subalgebra
\begin{equation}
  \label{eq:2}
  \Gau(\RR) := [ \, g_{\mu, \rho} \mid (\mu, \rho) \in \RR \times \RR_{> 0} \, ]_\CC \le \Gau_0(\RR),
\end{equation}
whose unitalization is of course~$\Gau_0(\RR) = \RR \oplus \Gau(\RR)$ since~$g_{\mu,0} = 1$ for
all~$\mu \in \RR$. As mentioned before, we have now~$\Gau(\RR) \le \Schw(\RR)$ as a subalgebra under
the pointwise product; we call this the \emph{Gaussian algebra}.

For establishing $\CC$-linear independence of the Gaussians, one may usefully appeal to
\emph{asymptotic notions} for $x \to +\infty$; of course one could equally use $x \to -\infty$. In
the sequel, all limits and germs are for $x \to +\infty$. For any such functions
$f, g\colon \RR \to \CC$, we recall that $f$ is called negligible with respect to~$g$
when~$\lim f/g = 0$; this is either written as a relation~$f \Prec g$ following Hardy or with
Landau's little oh notation~$f = o(g)$.

While the collection of \emph{all} germs is a ring, the smooth ones clearly form a differential
ring~$R_\infty$. It is often expedient, though, to work with differential subfields of~$R_\infty$,
known as \emph{Hardy fields}~\cite[\S1]{AschenbrennerDries2005}: They come with a canonical total
order, so are ordered fields; moreover, their germs all have definite limits on the extended real
line~$\RR \cup \{\pm\infty\}$. For example, the Hardy field of logarithmic-exponential
functions~\cite[\S1]{AschenbrennerDries2005}, \cite[Ex.~1]{GrunspanHoeven2017} is certainly large
enough for our modest purposes.

Given any Hardy field~$\F$, one may want to isolate an \emph{asymptotic scale}~$\ase$
within~$\F$. Similar to~\cite[\S2.5]{GrunspanHoeven2017}, we define this to be a multiplicative
subgroup of~$\F_+ := \{ f \in \F \mid f > 0\}$ totally ordered under~$\Prec$. Then the
subfield~$\CC(\ase)$ generated by~$\ase$ is endowed with a valuation, induced from the natural
valuation~\cite[V1--3]{AschenbrennerDries2005} of~$\F$. Its valuation ring consists of all
(germs of) functions in~$\CC(\ase)$ that remain finite for~$x \to +\infty$. 

Following~\cite[\S{}III.2]{Dieudonne1973}, we consider specifically the asymptotic scale~$\ase$
consisting of the unity germ~$1$ and all germs of the form
\begin{equation*}
  f(x) = x^\alpha \log^\beta{x} e^{P(x)}
\end{equation*}
with~$\alpha, \beta \in \RR$ and~$P(x) \in \RR[\RR_{>0}]$. The representation of non-unity germs is
unique if we stipulate that~$\alpha, \beta \neq 0$. Furthermore, we agree to write the exponents in
descending order~$P(x) = a_1 x^{\gamma_1} + \cdots + a_k x^{\gamma_k}$, so
that~$\gamma_1 > \cdots > \gamma_k > 0$ and~$a_1, \dots, a_k \in \RR$. It is clear that all such
germs are positive, and they form a group under multiplication (and one may also take powers with
arbitrary real exponents). Their crucial property for asymptotic investigations, however, is that
any such germ~$f(x) \neq 1$ either tends to~$0$ or to~$\infty$. In detail, we have
\begin{equation}
  \label{eq:asrel}
  x^\alpha (\log{x})^\beta e^{P(x)} \Succ x^{\tilde\alpha} (\log{x})^{\tilde\beta} e^{\tilde P(x)}
  \quad\text{iff}\quad
  (P(x), \alpha, \beta) > (\tilde P(x), \tilde\alpha, \tilde\beta)
\end{equation}
under the lexicographic order on~$\RR[\RR_{>0}] \times \RR \times \RR$. Here the monoid
ring~$\RR[\RR_{>0}]$ is itself given the lexicographic
order~$a_1 x^{\gamma_1} + \cdots + a_k x^{\gamma_k} \Succ \tilde{a}_1 x^{\gamma_1} + \cdots +
\tilde{a}_k x^{\gamma_k}$
iff~$(a_1, \cdots, a_k) > (\tilde{a}_1, \cdots, \tilde{a}_k)$, where one pads~$P(x)$
and~$\tilde{P}(x)$ with zero coefficients~$a_j$ so that they exhibit the same exponent
sequence~$\gamma_1 > \cdots > \gamma_k > 0$.

% Inspiration for the following proof from this MSE website:
% https://math.stackexchange.com/questions/1081350/how-to-prove-the-gaussian-functions-are-linear-independent

The elements of~$\ase$ are all linearly independent over~$\CC$, so they form a $\CC$-basis for the
subspace of germs generated by~$\ase$. For seeing this, assume
\begin{equation}
  \label{eq:linrel}
  c_1 f_1 + \cdots + c_n f_n = 0
\end{equation}
is any linear relation among distinct germs~$f_1, \dots, f_n \in \ase$
with coefficients~$c_1, \dots, c_n \in \nnz{\CC}$ and length~$n >
0$. Without loss of generality, we may assume
that~$f_1 \Succ \cdots \Succ f_n$. Multiplying~\eqref{eq:linrel}
with~$c_1^{-1} f_1^{-1}$, we obtain
\begin{equation}
  \label{eq:linrelmod}
  1 + \tilde c_2 \tilde f_2 + \cdots + \tilde c_n \tilde f_n = 0
\end{equation}
with~$\tilde c_i := c_1^{-1} c_i$ and~$\tilde f_i := f_1^{-1} f_i$
for~$i = 2, \dots, n$. Note that we have again descending asmptotic
growth~$1 \Succ \tilde f_2 \Succ \cdots \Succ \tilde f_n$. By
transitivity of~$\Succ$, this implies
that~$\tilde f_2, \dots, \tilde f_n \to 0$ so
that~\eqref{eq:linrelmod} yields~$1 + 0 = 0$ upon taking the
limit~$x \to +\infty$.

It is easy to see that the induced valuation on~$\CC(\ase)$ is essentially given by the exponents
in~\eqref{eq:asrel}, except for the conventional sign (so the valuation measures decay rather than
growth at infinity). In other words, we may define\footnote{The extra factor~$\pi^{-1}$ for the
  polynomial exponent is only convenience for later purposes.}
\begin{equation}
  \label{eq:3}
  \nu(x^\alpha \log^\beta{x} e^{P(x)}) = \big( -P(x)/\pi, -\alpha, -\beta \big),
\end{equation}
for the scale functions of~$\ase$. This extends to $\CC$-linear combinations of scale
functions~$f_1, f_2 \in \ase$ via~$\nu(c_1 f_1 + c_2 f_2) = \min\{\nu(f_1), \nu(f_2)\} $ by the
ultrametric inequality. In this way, we have defined the valuation on the
subalgebra~$\CC[\ase] \subset \CC(\ase)$, which determines the general
valuation~$\nu\colon \CC(\ase) \to \RR[\RR_{>0}] \times \RR \times \RR$
by~$\nu(f_1/f_2) = \nu(f_1) - \nu(f_2)$. Clearly, $\CC[\ase]$ is the valuation ring corresponding
to~$\nu$, hence in particular an integral domain. It is moreover the group ring over the ordered
abelian group~$\G$, so the valuation~$\nu\colon \CC[\ase] \to \RR[\RR_{>0}]$ is essentially a
special case of~\cite[Exc.~11.4]{Eisenbud1995}.

For our purposes, we shall only need the \emph{exponentials} with standard polynomials in the
exponent. Since we have to ensure that the constant function~$1$ is the only germ not tending to~$0$
or to~$\infty$, we shall use the additive group~$\RR[x]^+$ of polynomials with positive degree to
introduce
\begin{equation*}
  \ase_0 := \{ e^{P(x)} \mid P(x) \in \RR[x]^+ \} \subset \ase
\end{equation*}
which is easily seen to form an asymptotic subscale with corresponding field~$\CC(\ase_0)$ and
valuation ring~$\CC[\ase_0]$. Again, the exponentials of~$\ase_0$ are then linearly independent
over~$\CC$, so the $\CC$-linear span of~$\ase_0$ is (isomorphic to) the group
algebra~$\CC[\RR[x]^+]$, a rare case of iterated monoid algebra. The valuation restricts
to~$\nu\colon \CC(\ase_0) \to \RR[x]^+$.

Focusing on the subalgebra~$\Gau(\RR) \subset \CC[\RR[x]^+]$, the valuation restricts further to a
semigroup epimorphism~$\nu\colon \nnz{\CC}\Gau_{\tdot} \twoheadrightarrow \RR_{>0} \, x^2 + \RR x$,
with $\nnz{\CC}\Gau_{\tdot} := \{ cg \mid c \in \nnz{\CC}, g \in \Gau_{\tdot} \}$ the multiplicative
subsemigroup of~$\Gau(\RR)$ generated by the
Gaussians~$\Gau_{\tdot} := \{ g_{\mu,\rho} \mid (\mu, \rho) \in \RR \times \RR_{>0} \}$. In detail,
we have~$\nu(cg_{\mu,\rho}) = \rho x^2 - 2\rho\mu x$, and its kernel by the
relation~$\sim$ with~$cg_{\mu,\rho} \sim \tilde{c} g_{\tilde\mu,\tilde\rho}$
iff~$(\mu,\rho) = (\tilde\mu, \tilde\rho)$, which implies
$\nnz{\CC}\Gau_{\tdot} / \mathord\sim \cong \Gau_{\tdot}$. Identifying~$\Gau_{\tdot}$
with~$\RR \times \RR_{>0}$ via~$g_{\mu,\rho} \leftrightarrow (\mu, \rho)$, the quotient monoid
structure is given by
\begin{equation}
  \label{eq:pw-gauss}
  (\mu_1, \rho_1) \cdot (\mu_2, \rho_2) = (\mu_1 +_\rho \mu_2, \rho_1 + \rho_2),
\end{equation}
where~$\mu_1 +_\rho \mu_2 := (\mu_1\rho_1 + \mu_2\rho_2) / (\rho_1 + \rho_2)$ denotes the
$\rho$-weighted arithmetic mean of~$\mu_1$ and~$\mu_2$ for the weight
vector~$\rho := (\rho_1, \rho_2)$. By the first isomorphism theorem (for semigroups), the
semigroup~$\Gau_{\tdot}$ is just the additive
semigroup~$\RR_{>0} \, x^2 + \RR x \cong \RR_{>0} \oplus \RR$ in disguise. In the sequel, we shall
refer to~$\Gau_{\tdot}$ as the \emph{pointwise Gaussian semigroup}.

Let us now reconstruct the full algebra structure on~$\Gau_0(\RR)$. Since this is just the
unitalization of~$\Gau(\RR)$, it suffices to build up the structure of the latter. We shall use the
strategy of group homology~\cite[\S\,VI]{HiltonStammbach1971} for this purpose, but adopted to our
present setting. Hence assume
\begin{equation}
  \label{eq:exseq-semigrp}
  \xymatrix { 0 \ar[r] & C\, \ar@{^{(}~>}[r] & H \ar[r]^-\pi & G \ar^s @/^/[l] \ar[r] & 0}
\end{equation}
is an \emph{exact sequence of semigroups} in the sense that~$\ker(\pi)$ is the congruence on~$H$
given by~$h \sim h'$ iff~$h = ch'$ for a unique~$c \in C$. Then the quotient~$h/s\pi(h) \in C$ is
well-defined for any~$h \in H$. While here we shall only need the fully abelian setting, it is
natural to allow~$H$ to be nonabelian as long we retain commutativity on the orbits (see below for
the precise set of axioms). In this way, we include central extensions of abelian groups such as the
ones used for studying Heisenberg groups in \S\heiscite{sub:nilquadratic-groups}.

Since quotients are thus well-defined, we may form the \emph{cocycle} (or ``factor set'')
\begin{equation}
  \label{eq:cocycle}
    \psi\colon G \times G \to C,\quad
    (g,g') \mapsto \frac{s(g) \, s(g')}{s(gg')},
\end{equation}
just as in the group case~\cite[Exc.~10.1]{HiltonStammbach1971}. It should be emphasized, however,
that~$C$ is \emph{not} required to be embedded in~$H$; we only require a compatible semigroup
action~$\cdot\colon C \times H \to H$, where compatibility here means that~$\cdot$ is a homomorphism
of semigroups (with~$C \times H$ being the direct product). This is what the wavy arrow
in~\eqref{eq:exseq-semigrp} is supposed to convey. Let us note that the
action~$\cdot\colon C \times H \to H$ is automatically \emph{free} in the sense that~$c \neq c'$
implies~$c \cdot h \neq c' \cdot h$ for all~$h \in H$. Equivalently, we can also stipulate freeness
while giving up uniqueness in the exactness requirement. Altogether, we have imposed
on~\eqref{eq:exseq-semigrp} the \emph{axioms}
\begin{align*}
  & c \cdot c' \cdot h = cc' \cdot h && c \neq c' \Rightarrow c \cdot h \neq c' \cdot h\\
  & (c \cdot h)(c' \cdot h') = cc' \cdot hh' && h \sim h' 
  \Leftrightarrow \exists_c c \cdot h = h'\\
  & h \sim h' \Rightarrow hh' = h'h && C, G ~\text{abelian}  
\end{align*}
of action, freeness, compatibility, exactness, orbit commutativity, and abelian flanks. It is easy
to check the calculation rules~$\tfrac{h_1}{h_1'} \tfrac{h_2}{h_2'} = \tfrac{h_1 h_2}{h_1' h_2'}$
and~$\smash{\tfrac{c \cdot h_1}{h_1'} = c \, \tfrac{h_1}{h_1'}}$ for~$c \in C$ and~$h_1 \sim h_1'$,
$h_2 \sim h_2'$. Moreover, we have the cancellation
rules~$\tfrac{h}{k} \tfrac{k}{h'} = \tfrac{h}{h'}$ as well
as~$\tfrac{h}{k} \tfrac{k}{h'} = \tfrac{h}{h'}$ for~$h, h' \sim k$. While we have not required any
of the three semigroups to be monoids or groups, it turns out that~$C$ is in fact a group.

\begin{lemma}
  \label{lem:semigrp-grp}
  Given an exact sequence as in~\eqref{eq:exseq-semigrp}, the semigroup~$C$ is a group whose
  identity element acts trivially.
\end{lemma}
\begin{proof}
  Choose any~$h \in H$. Then we have~$\tfrac{h}{h} \cdot h = h$ by definition and
  therefore~$c \, \tfrac{h}{h} \cdot h = c \cdot \tfrac{h}{h} \cdot h = c \cdot h$ for
  any~$c \in C$. This yields~$\tfrac{h}{h} \, c = c$ by freeness. Since~$c \in C$ was arbitrary,
  $\tfrac{h}{h}$ is seen to be an identity element for~$C$. But such an element is unique in any
  semigroup, so we may unambiguously write~$1 := \tfrac{h}{h}$. Thus~$C$ is a monid. It is in fact a
  group since every element~$c \in C$ can be written as~$\tfrac{h}{h'}$ for suitable~$h, h' \in H$,
  taking for example~$h := c \cdot h'$ for arbitrary~$h'$. Clearly, such an element
  has~$\tfrac{h'}{h}$ for its inverse by the calculation rules mentioned above. Moroever,
  if~$h \in H$ is arbitrary, we have~$1 \cdot h = \tfrac{h}{h} \cdot h = h$ so that~$1$ acts
  trivially as claimed.
\end{proof}

With the help of Lemma~\ref{lem:semigrp-grp}, we can derive two more useful \emph{calculation
  rules}: As usual, we have the equality condition for fractions saying
that~$\tfrac{h_1}{h_1'} = \tfrac{h_2}{h_2'}$ iff~$h_1 h_2' = h_1' h_2$ for~$h_1 \sim h_1'$
and~$h_2 \sim h_2'$, and we have the mixed associativity law~$(c \cdot h) \, h' = c \cdot hh'$ for
all~$h, h' \in H$ and~$c \in C$. The latter follows immediately from the compatibility axiom by
substituting~$c' = 1$. For showing the former, assume first the equality of fractions and
set~$c := h_1/h_1' = h_2/h_2'$ to obtain~$h_1 h_2' = (c \cdot h_1')(c^{-1} \cdot h_2) = h_1' h_2$.
Conversely, assuming this condition we
get~$h_2 = \tfrac{h_1 h_2'}{h_1' h_2} \cdot h_2 = \tfrac{h_1}{h_1'} \cdot \tfrac{h_2}{h_2'} \cdot
h_2 = \tfrac{h_1}{h_1'} \cdot h_2'$,
which implies the desired equality of fractions by the uniqueness of quotients.

Even though~$H$ is only a (nonabelian) semigroup, our axioms on~\eqref{eq:exseq-semigrp} allows us
to define the \emph{commutator}~$[,]\colon H \times H \to C$ via~$[h,k] := \tfrac{hk}{kh}$. Note
that this quotient is well-defined since we have~$hk \sim kh$ by the commutativity of~$G$. Now
assume~$h \sim h'$ and~$k \sim k'$. We claim that~$[h',k'] = [h,k]$ since this is equivalent
to~$h'k'kh = k'h'hk$, which is in turn equivalent to~$cd \cdot hkkh = cd \, khhk$ with~$c := h'/h$
and~$d := k'/k$. But the latter equality follows immediately from orbit commutativity
since~$hk = kh$. In analogy to the group case in Definition~\heiscite{def:comm-form}, we can thus
introduce the \emph{commutator form}~$\omega\colon G \times G \to C$
by~$\omega(g,g') = [s(g), s(g')]$, where the choice of the (set-theoretic) section~$s\colon G \to H$
is immaterial by what we have just shown. Moreover, $\omega$ is cleary antisymmetric and it is
linear since even~$[h_1 h_2, h] = [h_1, h] \, [h_2, h]$ for all~$h_1, h_2, h \in H$. Indeed, the
latter is true iff
\begin{equation}
  \label{eq:fraceq}
  \tfrac{h_1 h_2 h}{hh_1 h_2} = \tfrac{h_1 h}{h h_1} \, \tfrac{h_2 h}{h h_2},
\end{equation}
which we establish now. Writing~$c := \tfrac{h_1h}{hh_1}$ and~$d := \frac{h_2h}{hh_2}$ for the
right-hand quotients, we
get~$cd \cdot hh_1h_2 = (c \cdot hh_1)(d \cdot h_2) = h_1h (d \cdot h_2) = h_1h_2h$;
now~\eqref{eq:fraceq} follows as usual by the uniqueness of quotients.

Let us now state how the cocycle~\eqref{eq:cocycle} encodes the structure of the semigroup~$H$ and
its extension as algebra.

\begin{lemma}
  \label{lem:exseq-semigrp}
  For an exact sequence~\eqref{eq:exseq-semigrp} with fixed section~$s\colon G \to H$, the
  cocycle~\eqref{eq:cocycle} induces a semigroup~$C \times_\psi G$ isomorphic to~$H$.

  \smallskip

  \noindent If $C = (\nnz{K}, \cdot)$ for a commutative unital ring~$K$, the
  product of~$H$ has a unique $K$-linear extension to~$K^{(G)} \cong K^\psi[G]$.
\end{lemma}
\begin{proof}
  As one checks immediately, one obtains on the level of sets a bijection
  \begin{equation}
    \label{eq:smgrp-factorset}
    H \cong C \times G:\quad
    \begin{cases}
      h & \mapsto \big(h/s\pi(h), \pi(h)\big),\\
      c \, s(g) & \mapsfrom (c,g).
    \end{cases}
  \end{equation}
  which becomes an isomorphism of semigroups by transporting the operation from~$H$ to~$C \times G$.
  It is easy to see that the new operation~$\cdot_\psi$ is given by
  $(c, g) \cdot_\psi (c', g') = \big(\psi(g,g') \, cc', gg'\big)$, using~\eqref{eq:smgrp-factorset}
  and the compatible action~$\cdot\colon C \times H \to H$. We write~$C \times_\psi G$ for the
  resulting semigroup. Note that every element of~$h \in H$ can be written uniquely
  as~$h = c \, s(g)$ with~$c \in C$ and~$g \in G$, where~$h \leftrightarrow (c, g)$
  in~\eqref{eq:smgrp-factorset}. Identifying~$G$ via~$s$ as a subset of~$H$, we shall write this
  simply as~$h = cg$.

  Now assume~$C = \nnz{K}$ for a commutative unital ring~$K$. In that case, the unitarizations~$Kg$
  of the congruence classes~$[g]_\sim = Cg$ of~$H/{\mathord\sim}$ are free $K$-modules and direct
  components of~$K^{(G)} \cong \bigoplus_{g \in G} Kg$. Using the partition of~$H$ into its
  congruence classes, there is an injective map
  \begin{equation*}
    H = \biguplus_{g \in G} \, [g]_\sim \hooklongrightarrow  K^{(G)}
  \end{equation*}
  sending~$cg \in H$ to~$ce_g$, where~$(e_g \mid g \in G)$ is the basis
  of~$K^{(G)}$. Identifying~$H$ as a subset of~$K^{(G)}$, its product induces for any~$g, g' \in G$
  the unique $K$-linear map
  \begin{equation}
    \label{eq:mult-comp}
    m_{g,g'}\colon Ke_g \otimes Ke_{g'} \to Ke_{gg'},
  \end{equation}
  given by~$\psi(g,g') \in C$ with respect to the bases~$\{e_g \otimes e_{g'}\}$
  and~$\{e_{gg'}\}$. Since~$K^{(G)} \otimes K^{(G)}$ is the direct sum of
  all~$Ke_g \otimes Ke_{g'}$, the maps~$m_{g,g'}$ combine to a bilinear product
  map~$\cdot\colon K^{(G)} \times K^{(G)} \to K^{(G)}$. Endowed with this product, the
  algebra~$(K^{(G)}, +, \cdot)$ must now be shown isomorphic to the twisted semigroup
  algebra~$K^\psi[G]$. Writing the generators of the latter as~$\theta_g \: (g \in G)$, the
  map~$e_g \mapsto \theta_a$ is obviously a $K$-linear
  isomorphism~$\Phi\colon K^{(G)} \isomarrow K^\psi[G]$, so it remains to prove that~$\Phi$ respects
  multiplication. But this is immediate from
  \begin{align*}
    \Phi(e_g \, e_{g'}) &= \Phi( \psi(g,g') \, e_{gg'}) = \psi(g,g') \, \Phi(e_{gg'})\\
    &= \psi(g,g') \, \theta_{gg'} = \theta_g \theta_{g'} = \Phi(e_g) \, \Phi(e_{g'}),
  \end{align*}
  using~\eqref{eq:mult-comp} and the definition of multiplication in~$K^\psi[G]$.
\end{proof}

We define the \emph{category of exact sequences of semigroups}~\eqref{eq:exseq-semigrp}~$\SESsg$ by
taking as morphism~$(\chi, \Phi, \phi)\colon E_1 \to E_2$ those~$\chi \in \Hom(C_1, C_2)$,
$\Phi \in \Hom(H_1, H_2)$, $\phi \in \Hom(G_1, G_2)$ where the right square commutes and
where~$\Phi$ is equivariant over~$\chi$. In detail, we require~$\pi_2 \circ \Phi = \phi \circ \pi_1$
and~$\Phi(c \cdot h) = \chi(c) \cdot \Phi(h)$ for~$c \in C$ and~$h \in H$. It is easy to see that
this implies~$\chi(\tfrac{h}{h'}) = \tfrac{\Phi(h)}{\Phi(h')}$ for all~$h \sim h' \in H$, hence in
particular~$\chi(1) = 1$. For Heisenberg groups, one can characterize the morphisms in terms of
coboundaries; see Proposition~\heiscite{prop:ses-morph}. Here we have an analogous result for the
semigroup case.

\begin{proposition}
  \label{prop:exseq-sg}
  Let~$E$ be the exact sequence~\eqref{eq:exseq-semigrp} and~$E'$ its primed version with
  cocycles~$\psi, \psi'$ from sections~$s\colon G \to H$ and~$s'\colon G' \to H'$,
  respectively. Then $\Hom(E,E') = \{ (\chi, \Phi, \phi)\colon E \to E' \}$ is in bijective
  correspondence with
  \begin{align*}
    \{ (\chi, \phi, \zeta) \mid {} & \chi \in \Hom(C, C') \land \phi \in \Hom(G, G') \land 
    \zeta \in C^1(G, C')\\
    & \qquad \land d^2(\zeta) = \tfrac{\chi_* \psi}{\phi^* \psi'} \}
  \end{align*}
  such that~$(\chi, \Phi, \phi)$ corresponds to~$(\chi, \phi, \zeta)$ with~$\zeta(g) =
  \Phi(sg)/s'\phi(g)$ and~$\Phi(h) = \zeta(\pi h) \, \chi(h/s\pi h) \cdot s' \phi(\pi h)$.
\end{proposition}
\begin{proof}
  For fixed~$\chi \in \Hom(C, C')$ and~$\phi \in \Hom(G, G')$, we will show that the given
  relation~$\Phi \leftrightarrow \zeta$ sets up a bijective correspondence between arbitrary
  functions $\zeta\colon G \to C'$ and those functions~$\Phi\colon H \to H'$ that
  satisfy~$\pi_2 \circ \Phi = \phi \circ \pi_1$ and~$\Phi(c \cdot h) = \chi(c) \cdot \Phi(h)$ while
  not necessarily being a homomorphism (so we neither require~$(\chi, \Phi, \phi) \in \SESsg$
  nor~$d^2(\zeta) = \tfrac{\chi_* \psi}{\phi^* \psi'}$).

  First of all, it is easy to see that, for given~$\zeta\colon G \to C'$, the induced
  function~$\Phi$ does satisfy $\pi_2 \circ \Phi = \phi \circ \pi_1$ as well
  as~$\Phi(c \cdot h) = \chi(c) \cdot \Phi(h)$. Let us now check bijectivity of the
  relation~$\Phi \leftrightarrow \zeta$. Defining~$\zeta$ for a given~$\Phi$, we obtain
  \begin{align*}
    \tilde\Phi(h) &= \zeta(\pi h) \, \chi(h/s\pi h) \cdot s' \phi(\pi h) =
    \chi(h/s\pi h) \, \tfrac{\Phi(s\pi h)}{s' \phi(\pi h)} \cdot s' \phi (\pi h)\\
    &= \chi(h/s\pi h) \cdot \Phi(s\pi h) = \tfrac{\Phi(h)}{\Phi(s\pi h)} \cdot \Phi(s \pi h) = \Phi(h);
  \end{align*}
  starting from a given~$\zeta$, we get the corresponding~$\Phi$ and then
  \begin{align*}
    \tilde\zeta(g) &= \tfrac{\Phi(sg)}{s'\phi(g)} = \tfrac{\zeta(\pi s g) \, \chi(sg/s\pi s g) \cdot s'\phi(\pi sg)}{s'\phi(g)}
    = \tfrac{\zeta(g) \, \chi(sg/sg) \cdot s'\phi(g)}{s'\phi(g)} = \zeta(g) .
  \end{align*}
  Thus we have established bijectivity. Next we compute
  \begin{align*}
    \Phi(h\bar{h}) &= \zeta(g\bar{g}) \, \chi\big( h\bar{h} / s(g\bar{g}) \big) \cdot s'(\phi g \, \phi \bar{g}),\\
    \Phi(h) \, \Phi(\bar{h}) &= \zeta(g) \, \zeta(\bar{g}) \, \chi(h/sg) \, \chi(\bar{h}/s\bar{g}) \cdot s'(\phi g) \, s'(\phi \bar{g}),
  \end{align*}
  with~$g := \pi h$ and~$\bar{g} := \pi \bar{h}$, which
  implies~$\Phi\colon H \to H'$ is a homomorphism
  iff~$\chi\big( h\bar{h} / s(g\bar{g}) \big) \cdot s'(\phi g \, \phi
  \bar{g})$ equals
  $\chi(h\bar{h}/sg \, s\bar{g}) \, \delta \cdot s'(\phi g) \, s'(\phi
  \bar{g})$, where we
  set~$\delta := d^2(\zeta)(g,\bar{g}) = \zeta(g) \, \zeta(\bar{g}) /
  \zeta(g\bar{g})$. Multiplying through
  by~$\chi(sg \, s\bar{g} / h\bar{h}) \, \delta^{-1} \in C'$, this is
  equivalent
  to~$\chi_* \psi \, (g,\bar{g}) \, \delta^{-1} \cdot s'(\phi g \,
  \phi \bar{g})$ being equal to~$s'(\phi g) \, s'(\phi \bar{g})$. By
  definition of the quotients for~$E'$, this is in turn equivalent
  to~$\phi^* \psi' \, (g,\bar{g}) = \chi_* \psi \, (g,\bar{g}) \,
  \delta^{-1}$
  or~$\delta = \tfrac{\chi_* \psi}{\phi^* \psi'} \, (g,\bar{g})$, as
  was to be shown.
\end{proof}

It is clear that the setting of~$\SESsg$ comprises the common setting of \emph{central extensions of
  abelian groups}~$1 \to C \overset{\iota}{\to} H \overset{\pi}{\to} G \to 0$, where~$C$ and~$G$ are
abelian groups but where~$H$ may be any (possibly nonabelian) group such that~$C \le \zentrum(H)$.
The setting~$\SESsg$ generalizes this in two respects: We allow~$H$ and~$G$ to be semigroups (while
we have seen that~$C$ is automatically a group), and we require only a compatible free action
of~$C \times H \to H$. The latter is induced by setting~$c \cdot h := \iota(c) \, h$. Note that
compatibility (as well as orbit commutativity) follows from~$C$ being central. Writing~$\SES$ for
the category of central extensions of abelian groups as defined in
\S\heiscite{sub:nilquadratic-groups}, it is also clear that the morphisms of the latter
become morphisms in~$\SESsg$. Thus one sees that~$\SES$ is a (non-full) subcategory of~$\SESsg$, and
we shall identify the corresponding exact sequences along with their morphisms. In this sense,
Proposition~\ref{prop:exseq-sg} is in fact a special case of Proposition~\heiscite{prop:ses-morph}.

The category~$\SESsg$ arises naturally from function fields such as the ones generated by
Gaussians. The key to understanding the connection involves some \emph{valuation
  theory}~\cite[\S9]{Cohn2003}, which we briefly recall for fixing terminology. Given a field~$\K$
and an additively written totally ordered group~$\G$, a \emph{valuation} on~$\K$ with value
group~$\G$ is a map~$\nu\colon \K \to \G \cup \{\infty\}$ such that
\begin{align*}
  & \nu(\nnz{\K}) = \G\quad\text{and}\quad\nu(0) = \infty,\\
  & \nu(ff') = \nu(f) + \nu(f'),\\
  & \nu(f+f') \ge \min \{ \nu(f), \nu(f') \}
\end{align*}
In that case, $(\K, \nu)$ is called a \emph{valuated field},
$\A_\nu := \{ f \in \K \mid \nu(f) \ge 0\}$ its \emph{valuation ring},
$\mathfrak{a}_\nu := \{ f \in \K \mid \nu(f) > 0\}$ its \emph{maximal ideal},
and~$\K_\nu := \A_\nu/\mathfrak{a}_\nu$ its \emph{residue field}.

We use the notation~$\G_{\ge g} := \{ h \in \G \mid h \ge g\}$, with similar definitions
for~$\G_{>g}$, $\G_{\le g}$ and~$\G_{<g}$. For a fixed field~$\K$, the valuation rings~$\A_\nu$ are
characterized~\cite[Prop.~9.1.1]{Cohn2003} by the property that~$a \in \A_\nu$
or~$a^{-1} \in \A_\nu$ for all~$a \in \K$, so the term ``valuation ring'' may be taken to refer to
this class of subrings of~$\K$. The valuation~$\nu$ corresponding to such a subring is unique up to
equivalence~\cite[Thm.~9.1.4]{Cohn2003}.

Given an integral domain~$\A$ with field of fractions~$\K$ and a
maximal ideal~$\mathfrak{a}$, Chevalley's lemma~\cite[9.4.3]{Cohn2003}
guarantees the existence of a valuation ring~$\A_\nu$ with maximal
ideal~$\mathfrak{a}_\nu$ such that~$(\A_\nu, \mathfrak{a}_\nu)$ is a
maximal pair dominating~$(\A, \mathfrak{a})$ in the sense
that~$\A_\nu \ge \A$ and~$\mathfrak{a}_\nu \supseteq \mathfrak{a}$. We
observe that the valuation~$\nu\colon \K \to \G$ is nonnegative
on~$\A$ and write $\G_{\A} := \nu(\nnz{\A}) \subseteq \G_{\ge0}$ for
the \emph{monoid of~$\A$} as in~\cite[\S2.1]{Teissier2003}. Then we
have
\begin{equation}
  \label{eq:exseq-valring}
  \xymatrix @M=0.5pc @R=1pc @C=1.5pc { 
    0 \ar[r] & \unit{\A} \ar@{^{(}->}[r] \ar@{=}[d] & \unit{\K} \ar[r]^\nu & \G \ar[r] & 0,\\
    0 \ar[r] & \unit{\A} \ar@{^{(}~>}[r] & \nnz{\A} \ar[r] \ar@{_{(}->}[u] & \G_{\A} \ar[r]
    \ar@{_{(}->}[u] & 0, }
\end{equation}
where~$\nnz{\A}$ is in general not a group (this is why we use a wavy
arrow in the second row and view the above as a morphism
in~$\SESsg$). The lower row of~\eqref{eq:exseq-valring} encodes the
chosen valuation in the semigroup category, without recourse to the
fraction field. (Note that~$\G_{\A}$, despite begin called a semigroup
in~\cite[\S2.1]{Teissier2003}, is in fact a monoid that generates~$\G$
as a group.)

Following the terminology of~\cite[\S2.3]{Teissier2003} in algebraic geometry, we define the
\emph{associated graded algebra} of~$\A$ by
\begin{equation}
  \label{eq:assoc-gradalg}
  \gr_\nu(\A) = \bigoplus_{g \in \G_{\A}} \A_{\ge g} \big/ \A_{>g}
\end{equation}
where~$\A_{\ge g}$ for~$\{ a \in \A \mid \nu(a) \in \G_{\ge g}\}$ and
similarly for~$\A_{> g}$ etc. Note that~$(\A_{\ge g})_{g \in \G_{\A}}$
represents a filtration by ideals with the~$\A_{>g}$ serving as
``successors''~\cite[\S1]{Teissier2003}. Here~$(\G_{\A}, \le)$ need
not be a well-order, as in the well-known
case~$(\G_{\A}, \G) = (\NN, \ZZ)$ common in algebraic
geometry. Nevertheless, $\gr_\nu(\A)$ is always an integral domain, as
long as~$\A$ is.

The grading of~\eqref{eq:assoc-gradalg} can be extended to
all~$g \in \G$ by assigning a zero homogeneous component
for~$g \not\in \G_{\A}$. Since~$\mathfrak{a}$ is maximal, one may form
the \emph{residue field}~$\K_{\A} := \A/\mathfrak{a}$, and
$\gr_\nu(\A)$ is an algebra over the residue field, the latter being
the component in~\eqref{eq:assoc-gradalg} with grade~$0 \in \G_{\A}$.

Given~$a \in \nnz{A}$, there is a unique~$g \in \G_{\A}$
with~$a \in \A_{\ge g} \setminus \A_{>g}$, namely~$g =
\nu(a)$. Defining the \emph{initial form}~$\bar{a} := a + \A_{>g}$,
one thus obtains the map~$\inif_\nu\colon \nnz{\A} \to \gr_\nu(\A)$,
$a \mapsto \bar{a}$. Moreover, $\gr_\nu(\A)$ is also endowed with an
\emph{induced valuation}, denoted by~$\nu$ again and defined as
\begin{equation*}
  \bar\nu\Big(\sum_{g \in \G_\A} \bar{a}_g\Big) := \min \, \{ \nu(a_g) \mid g \in \G_\A \} ,
\end{equation*}
where one checks that the choice of representatives~$a_g$ is
irrelevant.

Being a valution ring, $\gr_\nu(\A)$ has a unique maximal
ideal~$\gr_{\nu}(\mathfrak{a})$ consisting of
all~$\bar{a} \in \gr_{\nu}(\A)$ with~$\bar{\nu}(\bar{a})>0$. Its image
under the valuation will be denoted
by~$\G_{\mathfrak{a}} := \nu(\mathfrak{a})$ and referred to as the
\emph{semigroup of~$\A$}. Clearly, we
have~$\G_{\A} = \{ 0 \} \uplus \G_{\mathfrak{a}}$, which corresponds
to the direct
decomposition~$\gr_{\nu}(\A) = \K_\A \oplus
\gr_{\nu}(\mathfrak{a})$. Thus it seems more intuitive to
view~$\gr_{\nu}(\mathfrak{a})$ as a \emph{nonunital}
$\K_\A$-subalgebra of~$\gr_{\nu}(\A)$. (This is important in analysis,
where it can be of advantage to work with nonunital algebras: We have
seen that pointwise algebras are nonunital for non-compact domain and
convolution algebras for non-discrete ones.)

As for every graded ring, one obtains the \emph{associated graded
  monoid} by merging all nonzero initial terms via
\begin{equation*}
  \gr_{\nu.}(\A) := \biguplus_{g \in \G_{\A}} \nnz{(\A_{\ge g} \big/ \A_{>g})} <
  \big( \nnz{\gr_\nu(\A)}, \cdot \big).
\end{equation*}
and the \emph{associated graded semigroup} via
\begin{equation*}
  \gr_{\nu.}(\mathfrak{a}) := \biguplus_{g \in \G_{\mathfrak{a}}} {(\A_{\ge g} \big/ \A_{>g})} < \gr_{\nu.}(\A) ,
\end{equation*}
by restricting the leading terms to the maximal
ideal~$\gr_{\nu}(\mathfrak{a}) \lhd \gr_{\nu}(\A)$. Indeed, the
identity~$1 \in \K_\A$ occurs in~$\gr_{\nu.}(\A)$ as the degree-zero
component but not in~$\gr_{\nu.}(\mathfrak{a})$, so that the former is a monoid while the latter is not.

The group~$\A^*$ acts naturally both on~$\gr_\nu(\A)$ and
on~$\gr_{\nu.}(\A)$, and this action preserves the grading. One checks
also that the initial form extends to a $\SESsg$ morphism from the
second row of the exact sequence~\eqref{eq:exseq-valring}, namely
\begin{equation}
  \label{eq:exseq-valmon}
  \xymatrix @M=0.5pc @R=1pc @C=1.5pc { 
    0 \ar[r] & \unit{\A} \ar@{^{(}~>}[r] \ar@{=}[d] & \nnz{\A} \ar[r]^\nu \ar[d]_{\inif_\nu} & 
    \G_{\A} \ar[r] \ar@{=}[d] & 0,\\
    0 \ar[r] & \unit{\A} \ar@{^{(}~>}[r] & \gr_{\nu.}(\A) \ar[r]^-{\bar\nu} & \G_{\A} \ar[r] & 0.}
\end{equation}
The lower row of this diagram can be restricted to the positive
degrees to yield
\begin{equation}
  \label{eq:exseq-valsemigrp}
  \xymatrix @M=0.5pc @R=1pc @C=1.5pc {%
    0 \ar[r] & \unit{\A} \ar@{^{(}~>}[r] & \gr_{\nu.}(\mathfrak{a})
    \ar[r]^-{\bar\nu} & \G_{\mathfrak{a}} \ar[r] & 0, }
\end{equation}
which is the exact sequence we instantiate for the Gaussians.

In detail, we shall choose the \emph{valuation ring}~$\A = \Gau_0(\RR)$ with \emph{maximal
  ideal}~$\mathfrak{a} = \Gau(\RR)$, unit group~$\unit{\A} = \nnz{\CC}$ and \emph{residue
  field}~$\K_\A \cong \CC$. The quotient field~$\K$ is a subfield of the field~$\CC(\ase_0)$
introduced earlier; it consists of all rational functions in the~$g_{\mu,\rho}$. The value
group~$\G$ may be viewed as the additive subgroup~$\RR x^2 + \RR x \le \RR[x]$ with lexicographic
order (first quadratic then linear term), with \emph{semigroup}
$\G_{\mathfrak a} = \RR_{>0} x^2 + \RR x \cong \Gau_{\tdot}$ corresponding
to~$\mathfrak{a} = \Gau(\RR)$. It is easy to see that in the present
case~$\gr_{\nu}(\A) \cong \Gau_0(\RR)$ with components~$\A_0 \cong \CC$ and
$\A_{\mu,\rho} \cong [g_{\mu,\rho}]$, so here~$\bar\nu$ essentially coincides with~$\nu$. Therefore
we obtain also~$\gr_{\nu.}(\A) \cong \nnz{\CC} \uplus \nnz{\CC} \Gau_{\tdot}$. Finally, we
have~$\gr_{\nu}(\mathfrak{a}) \cong \Gau(\RR)$
and~$\gr_{\nu.}(\mathfrak{a}) \cong \nnz{\CC} \Gau_{\tdot}$ for the nonunital counterparts, which
yields
\begin{equation}
  \label{eq:exseq-monoid}
  \xymatrix { 0 \ar[r] & \nnz{\CC} \ar@{^{(}~>}[r] & \nnz{\CC}\Gau_{\tdot} \ar[r]^-\pi 
    & \Gau_{\tdot} \ar^s @/^/[l]!<3pt,0pt> \ar[r] & 0}
\end{equation}
for the exact sequence~\eqref{eq:exseq-valsemigrp} encoding the
quotient~$\Gau_{\tdot} \cong \nnz{\CC}\Gau_{\tdot} / \nnz{\CC}$.
Here~$\pi\colon \nnz{\CC}\Gau_{\tdot} \to \Gau_{\tdot}$ is the canonical projection
$cg_{\mu,\rho} \mapsto (\mu,\rho)$ while $s\colon \Gau_{\tdot} \to \nnz{\CC}\Gau_{\tdot}$ the
\emph{set-theoretic section} of~$\pi$ given by~$(\mu,\rho) \mapsto g_{\mu,\rho}$.

Applying Lemma~\ref{lem:exseq-semigrp} to~\eqref{eq:exseq-monoid} and using the above standard
section $s\colon \Gau_{\tdot} \hookrightarrow \nnz{\CC}\Gau_{\tdot}$, we obtain
$\nnz{\CC}\Gau_{\tdot} \cong \nnz{\CC} \times_{\gamma_{\tdot}} \Gau_{\tdot}$, the \emph{point\-wise
  Gaussian
  cocycle}~$\gamma_{\tdot} := \!{\gamma_{\tdot}}(\mu_1, \rho_1; \mu_2, \rho_2) = s(\mu_1, \rho_1) \,
s(\mu_2, \rho_2) / s\big( (\mu_1, \rho_1) \cdot (\mu_2, \rho_2) \big)$ in the form
$\gamma_{\tdot} = g_{\mu_1,\rho_1} \, g_{\mu_2, \rho_2}/g_{\mu_1 +_\rho \mu_2, \rho_1+\rho_2}$ so that
\begin{equation}
  \label{eq:ptw-cocyc}
  \gamma_{\tdot} =  g_{0,\rho_1\boxedplus\rho_2}(\mu_1-\mu_2) = \exp
  \Big(\frac{-\pi\rho_1\rho_2(\mu_1-\mu_2)^2}{\rho_1+\rho_2}\Big),
\end{equation}
which is the normalization constant for the $(\mu, \rho)$ parametrization\footnote{At this point,
  the chosen parametrization may seem awkward, but it will come in handy when introducing the
  convolution structure and Fourier operator.} with product law~\eqref{eq:pw-gauss}. Note
that~$\gamma_{\tdot} = g_{\mu_1,\rho_1\boxedplus\rho_2}(\mu_2) =
g_{\mu_2,\rho_1\boxedplus\rho_2}(\mu_1)$

We obtain from Lemma~\ref{lem:exseq-semigrp} also a characterization of the twisted semigroup
algebra~$\CC^{\gamma_{\tdot}}[\Gau_{\tdot}]$, namely as the complex vector space~$\CC^{(\Gau_{\tdot})}$ with
multiplication induced by~$\nnz{\CC}\Gau_{\tdot}$. We obtain~$\big(\Gau(\RR), \cdot\big) \cong \CC^{\gamma_{\tdot}}[\Gau_{\tdot}]$, so
the Gaussians form a basis over~$\CC$ with
\begin{equation}
  \label{eq:gauss-ptw-prod}
  g_{\mu_1, \rho_1} \, g_{\mu_2, \rho_2} =
  \gamma_{\tdot} \, g_{\mu,\rho}
  \qquad\text{with}\qquad
  \mu = \mu_1 +_\rho \mu_2, \quad\rho = \rho_1+\rho_2
\end{equation}
and~\eqref{eq:ptw-cocyc} for the multiplication law.

It is well-known that the Gaussians are also closed under \emph{convolution} (another reason for
discarding the constant function~$g_{\mu,0} = 1$). Using normalized Gaussians, the result is again a
normalized Gaussian whose mean and variance are given by, respectively, adding the original means
and variances. The explicit expression is
\begin{equation}
  \label{eq:conv-gauss}
  g_{\mu_1, \rho_1} \star g_{\mu_2, \rho_2} = \gamma_\star \, g_{\mu,\rho}
  \qquad\text{with}\qquad
  \mu = \mu_1 + \mu_2, \quad \rho = \rho_1 \boxedplus \rho_2
\end{equation}
where~$\gamma_\star := \gamma_\star(\mu_1, \rho_1; \mu_2, \rho_2)$ is the convolutive Gaussian
cocycle to be introduced below in~\eqref{eq:convol-cocyc}.

In analogy to the pointwise structure, we may also introduce the
semigroup~$\Gau_\star = (\RR, +) \oplus (\RR_{>0}, \boxedplus)$ where~$(\RR_{>0}, \boxedplus)$ is
the \emph{harmonic semigroup}. In other words, the product law of~$\Gau_\star$ is given by
\begin{equation*}
  (\mu_1, \rho_1) \star (\mu_2, \rho_2) = (\mu_1 + \mu_2, \rho_1 \boxedplus \rho_2) ,
\end{equation*}
and we obtain
\begin{equation}
  \xymatrix { 0 \ar[r] & \nnz{\CC} \ar@{^{(}~>}[r] & \nnz{\CC}\Gau_\star \ar[r]^-\pi 
    & \Gau_\star \ar^s @/^/[l]!<3pt,0pt> \ar[r] & 0}
\end{equation}
as the convolutive analog of~\eqref{eq:exseq-monoid}. Here~$\nnz{\CC}\Gau_\star$ is the same
as~$\nnz{\CC}\Gau_{\tdot}$ as a \emph{set} but endowed with the convolution~$\star$ inherited
from~$\Schw(\RR)$. Moreover, the maps~$\pi$ and~$s$ are also the same as in~\eqref{eq:exseq-monoid},
but they are now homomorphisms with respect to the convolution rather than pointwise product. In
analogy to the pointwise structure, we obtain once more a twisted semigroup
algebra~$\big(\Gau(\RR), \star\big) \cong \CC^{\gamma_\star}[\Gau_\star]$ with corresponding
\emph{convolutive Gaussian cocycle}
\begin{equation}
  \label{eq:convol-cocyc}
  \gamma_\star = \tfrac{1}{\sqrt{\rho_1+\rho_2}\rule{0pt}{7.5pt}}
\end{equation}
already used in~\eqref{eq:conv-gauss} above.

Altogether,
$\big(\Gau(\RR), \star, \cdot\big) \cong \CC^{\gamma_\star}[\Gau_\star] \divideontimes
\CC^{\gamma_{\tdot}}[\Gau_{\tdot}]$ is thus seen to be a (nonunital) \emph{twain subalgebra} of the
Schwartz class~$\big(\Schw(\RR), \star, \cdot\big)$. To get a Fourier singlet, we need closure under
the Heisenberg action. As~$\Gau(\RR)$ is closed under translation, we have
$\CC[\Gau(\RR) \RR] = \Gau(\RR) \le C_0(\RR)$, and Proposition~\ref{prop:heis-closure} yields the
\emph{Gaussian closure}
\begin{equation}
  \label{eq:gaussian-closure}
  \overline{\Gau(\RR)} = \mathfrak{C}(\RR) \, \Gau(\RR) \cong \mathfrak{C}(\RR) \otimes_{\CC} \Gau(\RR),
\end{equation}
where~$\mathfrak{C}(\RR)$ is the $\CC$-linear span of the \emph{oscillating exponentials}~$e_\alpha$
with frequency~$\alpha \in \RR$, as mentioned before Proposition~\ref{prop:heis-closure}. The tensor
product structure of~\eqref{eq:gaussian-closure} shows that~$\overline{\Gau(\RR)}$ has the complex
basis $\big(e_\alpha \, g_{\mu,\rho} \mid \alpha \in \RR, (\mu, \rho) \in \RR \times
\RR_{>0}\big)$.

Proposition~\ref{prop:heis-closure} ensures on general grounds that~$\overline{\Gau(\RR)}$ is a
Heisenberg (plain) subalgebra of~$\Schw(\RR)$ under the product law
\begin{equation}
  \footnotesize
  \label{eq:ext-pw-gauss}
  \left\{
  \begin{aligned}
    & e_{\alpha_1} g_{\mu_1, \rho_2} \, \cdot \, e_{\alpha_2} g_{\mu_2, \rho_2}
    = \bar{\gamma}_{\tdot} \, e_\alpha \, g_{\mu, \rho},\\
    & (\rho_1, \mu_1, \alpha_1) \cdot (\rho_2, \mu_2, \alpha_2)
    = (\rho_1 + \rho_2, \mu_1 +_\rho \mu_2, \alpha_1 + \alpha_2) =: (\rho, \mu, \alpha)
  \end{aligned}
  \right.
\end{equation}
based on~\eqref{eq:gauss-ptw-prod} and with the same cocycle~$\bar{\gamma}_{\tdot} = \gamma_{\tdot}$
as in~\eqref{eq:ptw-cocyc} because the product is direct.

While it is not clear \emph{a priori} that it is in fact a \emph{Heisenberg twain
  subalgebra}~$\big( \overline{\Gau(\RR)}, \star, \cdot \big) \le (\Schw(\RR), \star, \cdot)$, this
may be seen from the convolution law
\begin{equation}
  \footnotesize
  \label{eq:ext-conv-gauss}
  \left\{
  \begin{aligned}
    & e_{\alpha_1} g_{\mu_1, \rho_2} \, \star \, e_{\alpha_2} g_{\mu_2, \rho_2} =
    \bar{\gamma}_\star \, e_\alpha \, g_{\mu, \rho},\\
    & (\rho_1, \mu_1, \alpha_1) \star (\rho_2, \mu_2, \alpha_2)
    = (\rho_1 \boxedplus \rho_2, \mu_1 + \mu_2, \alpha_2 +_\rho \alpha_1) =: (\rho, \mu, \alpha)
  \end{aligned}
  \right.
\end{equation}
based on~\eqref{eq:conv-gauss}, but with the normalization constant given by the \emph{extended
  convolutive Gaussian cocycle}
\begin{align}
  \label{eq:ext-convol-cocyc}
  &\bar{\gamma}_\star := \gamma_\star \, \inner{\alpha_1-\alpha_2\,}{\,\mu_1 -_\rho \mu_2}
  \, g_{0,1/(\rho_1+\rho_2)}(\alpha_1-\alpha_2)\\
  &= \tfrac{1}{\sqrt{\rho_1+\rho_2}\rule{0pt}{7.5pt}} \, 
  \exp \Big( i\tau(\alpha_1-\alpha_2)(\mu_1\rho_1-\mu_2\rho_2)/(\rho_1+\rho_2) 
  - \tfrac{\pi (\alpha_1-\alpha_2)^2}{\rho_1+\rho_2} \Big) \notag
\end{align}
as may be seen via the relation~%
$g_{\mu,\rho} \, e_\alpha = \exp \, (i\tau\alpha\mu-\pi\alpha^2/\rho) \, g_{\mu+i\alpha/\rho,\rho}$,
where on the right-hand side we have brusquely usurped the Gaussian notation with a complex-valued
``mean''. For the record, let us also state the explicit \emph{Heisenberg action}
\begin{equation}
  \label{eq:normal-heisact}
  x \act e_\alpha \, g_{\mu,\rho} = \exp \, (-i\tau\alpha x) \, e_\alpha \, g_{\mu+x,\rho},\qquad
  \xi \act e_\alpha \, g_{\mu,\rho} = e_{\alpha+\xi} \, g_{\mu,\rho}
\end{equation}
for~$x \in \RR$ and~$\xi \in \RR \cong \RR^*$. 

We have already seen (Example~\ref{ex:L1-doublet}) that~$g = g_{0,1}$ is a fixed point of the
\emph{Fourier transform}~$\Four\colon L^{1/1}(\RR) \pto L^{1/1}(\RR)$ or the corresponding
restriction~$\Four\colon \Schw(\RR) \pto \Schw(\RR)$. Continuing from the singlet perspective
(i.e.\@ viewing~$\Four$ as an \emph{endo}morphism on the Schwartz class), the fixed point
result~$\Four g_{0,1} = g_{0,1}$ generalizes to
\begin{equation}
  \label{eq:Four-gaussian}
  \Four (e_\alpha g_{\mu,\rho}) = c \, e_\mu \, g_{-\alpha,1/\rho}
  \quad\text{with}\quad
  c := \inner{\alpha}{\mu} / \sqrt{\rho} = \rho^{-1/2} \, e^{i\tau\alpha\mu},
\end{equation}
as one confirms by applying the similarity theorem (mentioned after
Proposition~\ref{prop:pont-automorphism}) to~$g_{0,\rho} = S_{\sqrt{\rho}} \, g_{0,1}$ and the
Heisenberg relations~\eqref{eq:cofourop} to $g_{\mu,\rho} = \mu \act g_{0,\rho}$
and~$e_\alpha \, g_{\mu,\rho} = \alpha \act g_{\mu,\rho}$.

Thus the Fourier transform on~$\Schw(\RR)$ restricts to an endomorphism on~$\overline{\Gau(\RR)}$,
we obtain a twain subsinglet, which is clearly regular since~\eqref{eq:Four-gaussian} is
invertible. Due to the prominent role played by the Gaussian normal distribution, we call the
resulting singlet the \emph{Gaussian singlet}~$\bar{\Gau}\inner{\RR}{\RR}$.

\begin{proposition}
  \label{prop:gauss-singlet}
  The Gaussian
  singlet~$\bar{\Gau}\inner{\RR}{\RR} = [\overline{\Gau(\RR)} \pto \overline{\Gau(\RR)}]$ is a
  regular twain subsinglet of~$\Schw\inner{\RR}{\RR} = [\Schw(\RR) \pto \Schw(\RR)]$.\hfill\qed
\end{proposition}

It may be interesting to view also the Gaussian closure~$\overline{\Gau(\RR)}$ with its Fourier
transform~$\Four\colon \overline{\Gau(\RR)} \pto \overline{\Gau(\RR)}$ from the viewpoint of
\emph{exact sequences} in~$\SESsg$. It is easy to see that the exact
sequence~\eqref{eq:exseq-monoid} extends to
\begin{equation}
  \label{eq:exseq-monoid-four}
  \xymatrix { 0 \ar[r] & \nnz{\CC} \ar@{=}[d] \ar@{^{(}~>}[r] &
    \nnz{\CC}\bar{\Gau}_\star \ar[r]^-\pi \ar[d]_\Four
    & \bar{\Gau}_\star \ar[r] \ar[d]_{\four} & 0\\
  0 \ar[r] & \nnz{\CC} \ar@{^{(}~>}[r] & \nnz{\CC}\bar{\Gau}_{\tdot} \ar[r]^-\pi 
    & \bar{\Gau}_{\tdot} \ar[r] & 0}
\end{equation}
where~$\bar{\Gau}_{\tdot} := \Gau_{\tdot} \times \mathfrak{C}(\RR)$ is a direct product of monoids
while~$\nnz{\CC}\bar{\Gau}_{\tdot}$ is semidirect (with the ``same'' cocycle as
in~\eqref{eq:exseq-monoid} since the direct product does not contribute to cocycles). The elements
of~$\bar{\Gau}$ may be taken as triples~$(\rho,\mu,\alpha) \in \Gau_{\tdot} \times \RR$, on which
the Fourier reflex operates by $\four(\rho,\mu,\alpha) = (1/\rho, -\alpha, \mu)$. Thus we
have~$\four = i \times j$ with the
isomorphism~$i\colon (\RR_{>0}, \boxedplus) \isomarrow (\RR_{>0}, +)$ and the tilt
map~$j\colon \RR^2 \to \RR^2$ discussed in \S\ref{sub:heis-twist}. The latter makes sense
because the parameters $(\mu,\alpha) \in \RR^2 = G \oplus \Gamma$ may be seen as position/momentum
pairs with~$G$ acting naturally on~$\mu$ and~$\Gamma$ on~$\alpha$.

By Proposition~\ref{prop:exseq-sg}, the semigroup homomorphism~$\four$ determines~$\Four$ up to the
$1$-chain~$\zeta \in C^1(\bar{\Gau}_\star, \nnz{\CC})$ given by the constant~$c$
of~\eqref{eq:Four-gaussian} as $\zeta(\rho, \mu, \alpha) = \inner{\alpha}{\mu}/\sqrt{\rho}$. By a
straightforward calculation using the convolution law~\eqref{eq:ext-conv-gauss} as well as the
(extended) cocycles~\eqref{eq:ext-convol-cocyc} and~\eqref{eq:ptw-cocyc}, one may verify that
\begin{equation*}
  d^2\zeta(\rho_1\mu_1\alpha_1; \rho_2\mu_2\alpha_2)
  = \zeta(\rho_1\mu_1\alpha_1) \, \zeta(\rho_2\mu_2\alpha_2)
  / \zeta(\rho_1\mu_1\alpha_1 \star \rho_2\mu_2\alpha_2)
\end{equation*}
is given
by~$\bar{\gamma}_\star \,/\, \four^{\!*} \bar{\gamma}_{\tdot} = \inner{\rho_1'\alpha_1}{\mu_1}
\inner{\rho_2'\alpha_2}{\mu_2} / \inner{\alpha_1}{\rho_2'\mu_2} \inner{\alpha_2}{\rho_1'\mu_1}$ as
required in Proposition~\ref{prop:exseq-sg}.

\subsection{The Gelfond Field for Coefficients.}
\label{sub:gelfond-field}
As we have seen in \S\ref{sub:min-schwartz-algebra}, the normal Fourier
singlet~$\bar{\Gau}\inner{\RR}{\RR}$ leads to rational expressions in~$e^{\tau\xi}$, where~$\xi$ is itself a
rational expression in the parameters. In the next subsection we shall build up an algorithmic
subdomain of~$\bar{\Gau}\inner{\RR}{\RR}$ generated by allowing only rational values for the parameters. We are
thus led to consider~$\QQ\big( e^{\tau\xi} \mid \xi \in \QQ(i) \big)$ as coefficient field. In this
small subsection we will collate various number-theoretic facts about this field, which we define now
in the equivalent form
\begin{equation}
  \label{eq:Qpi}
  \QQ^\pi := \QQ\big( e^{\pi\xi} \mid \xi \in \QQ(i) \big).
\end{equation}
In other words, $\QQ^\pi$ is generated by all powers of Gelfond's constant~$e^\pi$ having Gaussian
rationals as exponents, and we shall thus refer to~$\QQ^\pi$ as the \emph{Gelfond field}. We analyze
now the algebraic structure of this field, giving special emphasis to the important fact that it is
well suited for the algorithmic treatment.

The crucial observation is that~$\QQ^\pi$ is built up from two rather dissimilar subfields, which we
shall write as
\begin{equation}
  \label{eq:Qtrab}
  \QQtr := \QQ(e^{\pi\alpha} \mid \alpha \in \QQ)
  \qquad\text{and}\qquad
  \QQab := \QQ(e^{i\pi\beta} \mid \beta \in \QQ) .
\end{equation}
Obviously, the Gelfond field is the compositum~$\QQ^\pi = \QQtr \mathbin{.} \QQab$. Its second
factor~$\QQab$ is the \emph{maximal abelian extension}~\cite[\S8.1(j)]{KatoKurokawaSaito2011} of the
rational field, obtained by adjoining all roots of unity to~$\QQ$. In other words, we have the
direct limit of cyclotomic fields
\begin{equation}
  \label{eq:maxabext}
  \QQab = \bigcup_{n \in \NN_{>0}} \QQ(\zeta_n),
\end{equation}
where~$\zeta_n := e^{i\tau/n}$ is the $n$-th standard primitive root
of unity. While~$\QQab$ is thus an algebraic extension field (in fact,
a Galois extension), its companion~$\QQtr$ is a \emph{transcendental
  extension} of~$\QQ$, for which it is easy to provide a natural
transcendece basis.

\begin{lemma}
  \label{lem:QQtr}
  The field~$\QQtr$ is transcendental over~$\QQ$ with transcendence
  basis~$\{ e^\pi \}$. It may be realized as the fraction field of the
  group ring~$\QQ[\eta_\alpha \mid \alpha \in \QQ] \hookrightarrow \QQtr$.
\end{lemma}
\begin{proof}
  For showing that~$\{ e^\pi \}$ is a transcendence basis, we must
  show that~$e^\pi$ is transcendental and that~$\QQtr$ is algebraic
  over~$K := \QQ(e^\pi)$. The transcendence of Gelfond's
  constant~$e^\pi$ is well-known~\cite[\S2.1]{Baker1975}, so we need
  only prove the second claim. It suffices to check that every
  generator~$e^{\pi\alpha} \: (\alpha \in \QQ)$ of~$\QQtr$ is
  algebraic over~$K$. Writing~$\alpha = m/n$ with~$m \in \ZZ$
  and~$n \in \NN_{>0}$, we note first that~$\eta := e^{\pi/n}$ is
  algebraic over~$K$ since it is annihilated
  by~$x^n - e^\pi \in K[x]$. But then~$e^{\pi\alpha} = \eta^m$ is
  algebraic over~$K$ as well since algebraic elements are closed under
  field operations.

  Now let~$R := \QQ[\eta_\alpha \mid \alpha \in \QQ]$ be the group
  ring having~$\QQ$ both as a coefficient field and as the underlying
  (additive) group. We show that the $\QQ$-linear
  map~$\iota\colon R \to \QQtr, \eta_\alpha \mapsto e^{\pi\alpha}$ is
  injective. Hence assume
  \begin{equation*}
    \lambda_0 \, e^{\pi\alpha_0} + \cdots + \lambda_n \, e^{\pi\alpha_n} = 0
  \end{equation*}
  for coefficients~$\lambda_0, \dots, \lambda_n \in \nnz{\QQ}$ and
  exponents~$\alpha_0, \cdots, \alpha_n \in \QQ$. We write the latter
  with common denominator~$N \in \NN_{>0}$ and with the
  numerators~$k_0 < \dots < k_n$ so that we
  have~$\lambda_0 \, \eta^{k_0} + \cdots + \lambda_n \, \eta^{k_n} =
  0$ for $\eta := e^{\pi/N}$. Multiplying this equation by~$\eta^{-k}$
  for~$k := \min(k_1, \dots, k_n)$, we may ensure
  that~$0 = k_0 < \cdots < k_n$. But then~$\eta$ is annihilated by
  $\lambda_0 + \lambda_1 x + \cdots + \lambda_n x^n \in \QQ[x]$, so
  that~$\eta$ and hence~$e^\pi = \eta^N$ is algebraic, contradicting
  Gelfond's result. We conclude that~$\iota$ is indeed injective,
  so~$R$ is an embedded $\QQ$-subalgebra of the field~$\QQtr$ and thus
  in particular an integral domain. So we may form the fraction
  field~$\bar{R}$ of~$R$, and we obtain the extended
  map~$\bar{\iota}\colon \bar{R} \to \QQtr$
  since~$0 \not\in \iota(\nnz{R})$;
  confer~\cite[Thm.~III.4.5]{Hungerford1980}. Being a homomorphism of
  fields, $\bar{\iota}$ is clearly injective as well. But it is also
  surjective since we have~$e^{\pi\alpha} = \bar{\iota}(\eta_\alpha)$
  for each generator of the
  field~$\QQtr = \QQ(e^{\pi\alpha} \mid \alpha \in \QQ)$.
\end{proof}

The \emph{group ring} mentioned in the above lemma (as well as its
image in~$\QQtr$) shall be written as~$\QQ[\QQ, +]$. The lemma also
points the way to generalizing this kind of extension: For arbitrary
fields~$K, L$ of characteristic zero, one may introduce the group
ring~$K[L, +]$ with coefficients in~$K$ over the additive
group~$(L, +)$. Since the latter is always torsionfree and
cancellative, one may infer~\cite[Prop.~4.20(b)]{BrunsGubeladze2009}
that the group ring~$K[L, +]$ is an integral domain. Choosing in
particular~$K = L$, we can introduce~$\Ktr$ as the fraction field
of~$K[K, +]$, an intrinsic \emph{transcendental extension} by a
``saturated'' set of exponential-like generators. The special case
of~$K = \QQ$ is recovered via the identification mentioned in
Lemma~\ref{lem:QQtr}. In this case, $\QQtr$ is an ordered subfield
of~$\RR$ and hence formally real.

\begin{figure}
  \centering\vspace{1.6em}
  $\xymatrix @M=0.5pc @R=1pc @C=2pc {
    & \QQ^\pi \ar@{-}_0[dl] \ar@{-}^1[ddr]\\
    \QQtr \ar@{-}_0[dd]\\
    && \QQab \ar@{-}^0[ddl]\\
    \QQ(e^\pi) \ar@{-}_1[dr]\\
    & \QQ}$
  \medskip
  \caption{Subfields of the Gelfond Field}
  \label{fig:field-extns}
\end{figure}
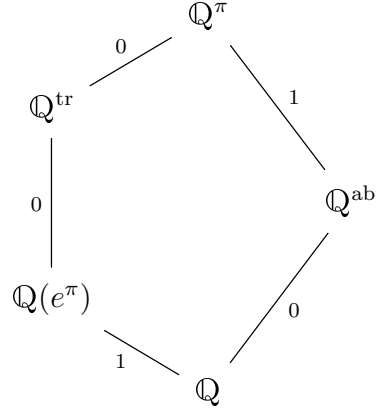

See Figure~\ref{fig:field-extns} for an \emph{extension diagram} of
the fields~$\QQtr$ and~$\QQab$. The edges of this diagram are labelled
by transcendence degrees (zero means algebraic extension): On the left
branch we have $\trdeg\big(\QQ(e^\pi)/\QQ\big) = 1$ by Gelfond's
result, $\trdeg\big(\QQtr/\QQ(e^\pi)\big) = 0$ by the proof of
Lemma~\ref{lem:QQtr}, and $\trdeg(\QQ^\pi/\QQtr) = 0$ since~$\QQ^\pi$
arises from~$\QQtr$ by adjoing the algebraic elements~$\QQab$. For the
right branch, we have~$\trdeg(\QQab/\QQ) = 0$ since the maximal
abelian extension~\eqref{eq:maxabext} is algebraic, and this forces
$\trdeg(\QQ^\pi/\QQab) = 1$ because the transcendence degrees must add
up to~$1$. With this background knowledge, we can now establish the
relationship between the subfields~$\QQtr$ and~$\QQab$ of the Gelfond
field~$\QQ^\pi$, namely that there is essentially no interaction
between them.

\begin{lemma}
  \label{lem:QQtr-QQab}
  The $\QQ$-algebras~$\QQ[\QQ, +]$ and~$\QQab$ are linearly disjoint.
\end{lemma}
\begin{proof}
  According to~\cite[Prop.~11.6.1]{Cohn2003}, it suffices to prove
  that the
  canonical~$\QQ$-basis~$\{ e^{\pi\alpha} \mid \alpha \in \QQ \}$
  of~$\QQ[\QQ, +]$ is also linearly independent over~$\QQab$. Assuming
  otherwise, we may proceed as in the proof of Lemma~\ref{lem:QQtr} to
  obtain a relation of the
  form~$\lambda_0 + \lambda_1 \, \eta + \cdots + \lambda_n \, \eta^n =
  0$ with~$\eta = e^{\pi/N}$ and~$N \in \NN_{>0}$, but with
  coefficients~$\lambda_0, \dots, \lambda_n \in \QQab$. This implies
  that~$\eta$ and hence~$e^\pi = \eta^N$ is algebraic
  over~$\QQab$. Since~$\QQ^\pi$ has~$\{ e^\pi \}$ for a transcendence
  basis, the whole Gelfond field~$\QQ^\pi$ is then algebraic
  over~$\QQab$, contradicting the extension relations established
  above (Figure~\ref{fig:field-extns}).
\end{proof}

From the results of the preceding lemmata, we can now provide a fairly
explicit description of the Gelfond field in terms of a tensor
product.

\begin{proposition}
  \label{prop:gelfond-fracfield}
  Up to isomorphism, the Gelfond field $\QQ^\pi$ is given as the
  fraction field of~$\QQab[\QQ, +]$.
\end{proposition}
\begin{proof}
  We have~$\QQab[\QQ, +] \cong \QQ[\QQ, +] \otimes \QQab$
  as~$\QQ$-algebras, and this tensor product is in turn
  isomorphic~\cite[Prop.~5.4.2]{Cohn2003} to the $\QQ$-algebra~$A$
  generated by~$\QQ[\QQ, +]$ and~$\QQab$. Hence it suffices to show
  that~$\QQ^\pi$ is the fraction field~$\bar{A}$ of~$A$. It is clear
  that~$\bar{A} \le \QQ^\pi$ since~$\QQ[\QQ, +], \QQab \le
  \QQ^\pi$. For the converse, it suffices to check that each
  generator~$e^{\pi\xi} \: \big(\xi \in \QQ(i) \big)$ is contained
  in~$A$, which is obvious since~$\xi = \alpha + \beta i$
  with~$\alpha, \beta \in \QQ$ so
  that~$e^{\pi\xi} = e^{\pi\alpha} e^{i\pi\beta}$ with
  $e^{\pi\alpha} \in \QQ[\QQ, +]$ and $e^{i\pi\beta} \in \QQab$.
\end{proof}

For obtaining a fully algorithmic description, we have to provide
canonical forms. This is possible by using \emph{direct limits} with
respect to inclusions, similar to the maximal abelian
extension~\eqref{eq:maxabext}. In the latter case, basic
theory~\cite[p.~187]{AlacaWilliams2003} suggests restricting to the
cyclotomic fields~$\QQ(\zeta_n)$ with~$n \not\equiv 2 \pmod 4$ for
avoiding duplication, and then $\QQ(\zeta_n) \le \QQ(\zeta_{n'})$
holds true iff $n | n'$. In the case of the transcendental extension,
we have
\begin{equation}
  \label{eq:trext}
  \QQtr = \bigcup_{n \in \NN_{>0}} \QQ(e^{\pi/n})
\end{equation}
on the set-theoretic level, which may again be interpreted as a direct
limit as follows.

\begin{lemma}
  \label{lem:transc-incl}
  We have a direct limit~\eqref{eq:trext} with respect to the inclusions
  $\QQ(e^{\pi/n}) \le \QQ(e^{\pi/n'})$, which hold true iff~$n | n'$. In fact, for any algebraic
  subextension~$\QQ \le K \le \QQ^\pi$, one has~$K(e^{\pi/n}) \le K(e^{\pi/n'})$ iff~$n | n'$.
\end{lemma}
\begin{proof}
  The implication from right to left is clear. Hence let us assume $e^{\pi/n} \in K(e^{\pi/n'})$ to
  show~$n | n'$. Then we have~$e^{\pi/n} = r(e^{\pi/n'})$ for some rational
  function~$r(x) = p(x)/q(x) \in K(x)$ with~$p(x)$ and~$q(x)$ coprime.
  Setting~$\eta := e^{\pi/nn'}$, this is equivalent to~$q(\eta^n) \, \eta^{n'} = p(\eta^n)$.
  Since~$\eta$ is transcendental over~$K$, the evaluation
  homomorphism~$K(x) \twoheadrightarrow K(\eta)$, $x \mapsto \eta$ is injective so that we have
  also~$q(x^n) \, x^{n'} = p(x^n)$. Taking degrees, we obtain~$n \, \deg(q) + n' = n \, \deg(p)$ and
  hence~$n | n'$.

  For showing that the corresponding direct limit is indeed the same as the union~\eqref{eq:trext},
  it suffices to verify the following universal property: Given a family of
  homomorphisms~$\big( f_n\colon \QQ(e^{\pi/n}) \to K \big)_{n \in \NN}$ to an arbitrary field~$K$
  with the coherence constraints~$f_{n'} |_{\QQ(e^{\pi/n})} = f_n$ whenever~$n | n'$, there is a
  unique homomorphism~$\lambda\colon \QQtr \to K$ such that~$f_n = \lambda|_{\QQ(e^{\pi/n})}$. The
  latter condition determines~$\lambda$ on the subalgebra~$\QQ[\QQ, +] \le \QQtr$
  as~$\lambda(\eta^{\pi\alpha}) = f_n(\eta^{m/n})$, where we have set~$\eta := e^{\pi/n}$ and
  where~$\alpha = m/n \in \QQ$ is written in terms of~$m \in \ZZ$ and~$n \in \NN_{>0}$. This is
  well-defined since~$m/n = m'/n'$ implies
  \begin{equation*}
    f_n(\eta^{m/n}) = f_{nn'}(\eta^{mn'/nn'}) = f_{nn'}(\eta^{m'n/nn'}) =
    f_{n'}(\eta^{m'/n'})
  \end{equation*}
  by the coherence constraints, hence also~$\lambda(\eta^{m/n}) = \lambda(\eta^{m'/n'})$ as
  required. But then~$\lambda$ is also determined on~$\QQtr$ since this is the fraction field
  of~$\QQ[\QQ, +]$ by Lemma~\ref{lem:QQtr} so that~$\lambda(p/q) = \lambda(p)/\lambda(q)$ for
  any fraction~$p/q \in \QQtr$ with~$p, q \in \QQ[\QQ, +]$ and~$q \neq 0$. Thus we have established
  existence and uniqueness of the homorphism~$\lambda\colon \QQtr \to K$.
\end{proof}

We may now put together the two direct limits into a single one, which correspondingly has the form
\begin{equation}
  \label{eq:gelfond-dirlim}
  \QQ^\pi = \bigcup_{m \in \NN_{>0}} \bigcup_{n \in \NN_{>0}} \QQ(e^{i\pi/m}, e^{\pi/n})
\end{equation}
with the respect to the expected natural inclusion maps.

\begin{proposition}
  \label{prop:gelfond-dirlim}
  We have the direct limit~\eqref{eq:gelfond-dirlim} with respect to the
  inclusions~$\QQ(e^{i\pi/m}, e^{\pi/n}) \le \QQ(e^{i\pi/m'}, e^{\pi/n'})$, which hold true iff~$m |
  m'$ and~$n | n'$.
\end{proposition}
\begin{proof}
  The implication from right to left is again clear, so assume the
  inclusion~$\QQ(e^{i\pi/m}, e^{\pi/n}) \le \QQ(e^{i\pi/m'}, e^{\pi/n'})$. Since linear disjointness
  is preserved under subalgebras~\cite[\S11.6]{Cohn2003}, we see that the group
  ring~$\QQ[e^{\pi/n}] \le \QQ[\QQ, +]$ and the field~$\QQ(e^{i\pi/m}) \le \QQab$ are also linearly
  disjoint. As in Proposition~\ref{prop:gelfond-fracfield}, it follows
  that~$\QQ(e^{i\pi/m'}, e^{\pi/n'})$ is the fraction field of the
  $\QQ$-algebra~$A := \QQ(e^{i\pi/m'})[e^{\pi/n'}]$, which clearly has
  $\QQ(e^{i\pi/m'})$-basis~$\{ e^{(k/n')\pi} \mid k \in \ZZ\}$.

  Now consider~$e^{i\pi/m} \in \QQ(e^{i\pi/m'}, e^{\pi/n'})$.
  Any~$e^{i\pi/m} \not\in \QQ(e^{i\pi/m'})$ would be a nonconstant element of the transcendental
  extension~$K(e^{\pi/n'})$ of the field~$K := \QQ(e^{i\pi/m'})$, which
  implies~\cite[p.~217]{Waerden2003a} that~$e^{i\pi/m}$ is trans\-cendental over~$K$ and hence
  over~$\QQ$; but~$e^{i\pi/m}$ is in fact algebraic. Hence we see
  that~$e^{i\pi/m} \in \QQ(e^{i\pi/m'})$, which forces~$m | m'$ by the observations
  before~\eqref{eq:trext}. We have also~$e^{\pi/n} \in \QQ(e^{i\pi/m'}, e^{\pi/n'})$, which
  implies~$K(e^{\pi/n}) \le K(e^{\pi/n'})$ for~$K := \QQ(e^{i\pi/m'})$. By
  Lemma~\ref{lem:transc-incl}, this yields~$n | n'$.

  The proof of the direct limit statement~\eqref{eq:gelfond-dirlim}
  proceeds as for Lemma~\ref{lem:transc-incl}. Hence
  let~$\big( f_{m,n}\colon \QQ(e^{i\pi/m}, e^{\pi/n}) \to K \big)_{m,
    n \in \NN}$ be any family of homomorphisms to some field~$K$,
  satisfying the corresponding coherence
  constraints~$f_{m',n'} |_{\QQ(e^{i\pi/m}, e^{\pi/n})} = f_{m,n}$ in
  case of~$m|m'$ and~$n|n'$. We must show that there is a unique
  homomorphism~$\lambda\colon \QQ^\pi \to K$ such
  that~$f_{m,n} = \lambda |_{\QQ(e^{i\pi/m}, e^{\pi/n})}$. By its
  definition~\eqref{eq:Qpi}, the Gelfond field~$\QQ^\pi$ is the
  rational function field in the transcendental
  generators~$e^{\pi\xi}$ with $\xi = \alpha + \beta i$
  and~$\alpha, \beta \in \QQ$. Therefore any~$a \in \QQ^\pi$ can be
  written as~$a = r(e^{\pi\xi_1}, \dots, e^{\pi\xi_N})$ for a rational
  function~$r \in \QQ(x_1, \dots, x_N)$
  and~$\xi_j = \tfrac{k_j}{n_j} + \tfrac{l_j}{m_j} \, i \; (j = 1,
  \dots, N)$; in this case we set
  \begin{equation*}
    \lambda(a) := r\big(f_{m_1,n_1}(e^{\pi\xi_1}), \dots, f_{m_N,
      n_N}(e^{\pi\xi_N})\big) .
  \end{equation*}
  Note that~$\lambda$ is well-defined thanks to the coherence
  constraints on the~$f_{m,n}$. By its construction, it is also clear
  that~$\lambda$ satisfies the required restriction
  property~$\lambda |_{\QQ(e^{i\pi/m}, e^{\pi/n})} =
  f_{m,n}$. Moreover, $\lambda$ is uniquely determined by this
  condition since~$\QQ^\pi$ is covered (set-theoretically) by the
  component fields~$\QQ(e^{i\pi/m}, e^{\pi/n})$. We have thus
  established the universal property, and~\eqref{eq:gelfond-dirlim} is
  indeed a direct limit.
\end{proof}

The \emph{practical significance} of
Proposition~\ref{prop:gelfond-dirlim} is the following. Any term~$T$
denoting an element in~$\QQ^\pi$ can be located in some component
field~$K_{m,n} := \QQ(e^{i\pi/m}, e^{\pi/n})$, and all possible
choices of~$m$ are multiples of each other; likewise for the choices
of~$n$. Choosing the minimal~$m_0$ and~$n_0$, we can thus rewrite the
given term in the form of a rational function~$T'$ in
only~$e^{i\pi/m_0}$ and~$e^{\pi/n_0}$. The
transformation~$T \mapsto T'$ is then clearly a \emph{canonical
  simplifier}~\cite{BuchbergerLoos1982}, provided we have a canonical
simplifier for the fields~$K_{m,n}$.

Canonical simplifiers on~$K_{m,n}$ are readily available: As we have
seen in the proof of Proposition~\ref{prop:gelfond-dirlim}, each
field~$K_{m,n}$ is the fraction field of the Laurent polynomial
ring~$K_m[\eta, \eta^{-1}]$ in the indeterminate~$\eta := e^{\pi/n}$
over the coefficient field~$K_m = \QQ(\zeta)$,
where~$\zeta := e^{i\pi/m}$ is the $m$-th standard primitive root of
unity. Arithmetic is~$K_m$ is straightforward since it is an algebraic
extension of~$\QQ$ with basis~$\{ 1, \zeta, \dots, \zeta^{d-1} \}$ and
dimension given by the Euler totient function as~$d :=
\Phi(n)$. Canonical forms in the Laurent polynomial
ring~$K_m[\eta, \eta^{-1}]$ can be achieved e.g.\@ by
expanding/reducing fractions (multiplying by a coefficient from~$K_m$
and a suitable power of~$\eta$) such that numerator and denomoninator
are polynomials in~$\eta$ with minimial degree, and such that the
denominator is monic. Computing field operations with any
representatives and subsequent reduction to canonical form
establishes~\cite[p.~13]{BuchbergerLoos1982} that~$K_{m,n}$ and thus
also~$\QQ^\pi$ is an \emph{effective field}.

As we shall see below (Lemma~\ref{lem:rational-gauss-gen}), the Gelfond field~$\QQ^\pi$ is indeed
sufficient for the basic operations of the normal Fourier singlet~$\bar{\Gau}\inner{\RR}{\RR}$. For
the additional action of the Weyl algebra, however, we shall need the \emph{extended Gelfond
  field}~$\QQ^\pi(\pi)$ in~\eqref{eq:gaussian-weyl-closure} since powers of~$\pi$ are cropping up
when differentiating gaussians. Hence we need to ensure that we may still treat~$\pi$ as a
transcendental indeterminate when computing over~$\QQ^\pi$. Fortunately, this is the case.

\begin{proposition}
  \label{prop:gelf-with-pi}
  The number~$\pi$ is transcendental over~$\QQ^\pi$.
\end{proposition}
\begin{proof}
  Suppose~$\pi$ is algebraic over~$\QQ^\pi$. Since~$\{ e^\pi \}$ is a
  transcendence basis of~$\QQ^\pi$ by
  Proposition~\ref{prop:gelfond-dirlim}, this implies that the
  set~$\{ \pi, e^\pi \}$ is algebraically dependent. But this
  contradicts a well-known result by
  Nesterenko~\cite[\S1.5.7]{ManinPanchishkin2006}.
\end{proof}

\subsection{The Rational Fourier Singlet of Gaussians.}
\label{sub:rational-gauss}
In order to isolate a computable Fourier singlet from the uncountable Schwartz
singlet~$\Schw\inner{\RR}{\RR}$, we shall need two restrictions:
\begin{itemize}
\item We restrict the Heisenberg action~$H_\RR \times \Schw(\RR) \to \Schw(\RR)$ from the original
  Heisenberg group~$H_\RR := \Tor \RR \rtimes \RR$ to the
  subgroup~$H_\QQ := \Tor_\QQ \QQ \rtimes \QQ$, where~$\Tor_\QQ \cong \QQ/\ZZ$ is the torsion
  subgroup of~$\Tor$. Note that~$\Tor_\QQ$ consists of all roots of unity so that the extension
  field of~$\QQ$ generated by~$\Tor_\QQ$ is just the maximal abelian extension~$\QQab \le \QQ^\pi$
  considered in \S\ref{sub:gelfond-field}. As it is clear that~$H_\QQ$ is a Heisenberg
  group in the sense of Definition~\ref{def:assoc-heisgrp}, we may refer to it as the
  (one-dimensional) \emph{rational Heisenberg group}.
\item The scaling parameter~$\rho$ of the Gaussians~$g_{\mu,\rho}$ is also restriced to rational
  numbers. It is then clear that the action of the rational Heisenberg group~$H_\QQ$ will only
  produce $\QQ^\pi$-linear combinations of~$e_\alpha \, g_{\mu,\rho}$ with the
  parameters~$\alpha, \mu, \rho$ all rational. Referring to the multiplication laws for~$\star$
  and~$\odot$ in \S\eqref{sub:min-schwartz-algebra}, one checks immediately that they form
  a twain subalgebra of the Gaussian closure~$\overline{\Gau(\RR)}$, which we call the
  \emph{rational Gaussian closure}~$\overline{\Gau(\QQ)}$.
\end{itemize}

It is furthermore clear that~$[\Four\colon \overline{\Gau(\QQ)} \pto \overline{\Gau(\QQ)}]$ is a
Fourier subsinglet
of~$[\Four\colon \overline{\Gau(\RR)} \pto \overline{\Gau(\RR)}] = \bar{\Gau}\inner{\RR}{\RR}$ since
the Fourier operator~$\Four$ in~\eqref{eq:Four-gaussian} creates only rational parameters from the
given rational parameters. We shall therefore refer
to~$[\overline{\Gau(\QQ)} \pto \overline{\Gau(\QQ)}]$ as the \emph{rational Gaussian
  singlet}~$\bar{\Gau}\inner{\QQ}{\QQ}$. It should be emphasized that~$\bar{\Gau}\inner{\QQ}{\QQ}$
is a \emph{computable Fourier singlet} since the coefficient field~$\QQ^\pi$ is computable and all
operations (especially convolution, pointwise product, Heisenberg action, Fourier operator) are
algorithmic.

Recall that the \emph{unit Gaussian} in Stigler normalization is denoted by
$g := g_{0,1} \in \Schw(\RR)$. We claim that the rational Gaussian
singlet~$\bar{\Gau}\inner{\QQ}{\QQ}$ may also be characterized as the Fourier singlet
\emph{generated} by~$g$ in the Schwartz singlet~$\Schw\inner{\RR}{\RR}$, of course over the rational
Heisenberg action~$\beta_\QQ\colon H_\QQ \times \Schw(\RR) \to \Schw(\RR)$.

\begin{lemma}
  \label{lem:rational-gauss-gen}
  The rational Gaussian closure~$\bar{\Gau}\inner{\QQ}{\QQ}$ is the smallest Fourier subsinglet
  in~$\Schw\inner{\RR}{\RR} \in \Fou{\beta_\QQ}$ over the field~$\QQ^\pi$ that contains the unit
  Gaussian~$g$.
\end{lemma}
\begin{proof}
  For the moment, let us write~$\Schw^g \le \Schw\inner{\RR}{\RR}$ for the smallest Fourier singlet
  containing~$g$, meaning the intersection of all such singlets. Since the Fourier operator $\Four$
  of~$\Schw\inner{\RR}{\RR}$ fixes~$g$, we have
  $\Schw^g = [(g)_{\Schw(\RR)} \pto (g)_{\Schw(\RR)}]$, where~$(g)_{\Schw(\RR)}$ is the Heisenberg
  twain algebra generated by~$g$ in~$\Schw(\RR)$ over~$H_\QQ$. It is clear
  that~$(g)_{\Schw(\RR)} \le \overline{\Gau(\QQ)}$, so it remains to prove the converse
  inclusion. But this follows from $e_\alpha \, g_{\mu,\rho} = \alpha \act \mu \act g_\rho$,
  where~$g_\rho := g_{0,\rho}$ may be obtained from~$g$ via convolution and pointwise multiplication
  in the following two equivalent ways: Writing~$\rho = n/m \in \QQ$, one has
  \begin{equation}
    \label{eq:conv-ptw-prop}
    (g^n)^{\star m} = (m n^{m-1})^{-1/2} \, g_{n/m}
    \quad\text{and}\quad
    (g^{\star m})^n = m^{-n/2} \, g_{n/m}
  \end{equation}
  as one may easily check. In other words, the pointwise and convolutive powers coincide within a
  numerical factor---a property that we shall use below.\footnote{This holds for plain
    Gaussians~$g_{\mu,\rho}$, but not for linear combinations: Taking $f := g_{1,0}-g_{3,0}$, say,
    leads to~$h := (f^2)^{\star 2} - (f^{\star 2})^2 \neq 0$ with~$|h_{2,2}(1/2)|>0.19$.}
\end{proof}

The rational Gaussian closure $\bar{\Gau}\inner{\QQ}{\QQ}$ is in fact a regular \emph{twain
  subsinglet} of~$\Schw\inner{\RR}{\RR} \in \Fou{\beta_\QQ}$, and it turns out that it can be
described as a \emph{quotient of the free Heisenberg twain algebra}~$\freetwn(S)$ introduced
in~Proposition\eqref{prop:free-heis-twn}. As usual, we write~$\freetwn(u)$ for the case of a
singleton~$S = \{ u \}$.

Quotients of a Heisenberg twain algebra~$A$ must be taken with respect to \emph{twain Heisenberg
  ideals}, meaning twain ideals (ideals simultaeously with respect to convolution and pointwise
product) that are closed under the Heisenberg action. Given any subset~$S \subseteq A$, the twain
Heisenberg ideal generated by~$S$ is of course the smallest Heisenberg twain ideal (= intersection
of all Heisenberg twain ideals) containing~$S$; we denote it by~$\langle S \rangle$.

Using this setup, we can define the corresponding relations as the Heisenberg twain
ideal~$\langle\mathcal{R}_u\rangle \subset \freetwn(u)$ generated by
$$\mathcal{R}_u = \mathcal{R}_u(\Gamma,\star) \cup \mathcal{R}_u(G,\cdot) \cup
\mathcal{R}_u(\star,\cdot)$$
using the three relation sets to be given shortly. Anticipating the desired isomorphism, we
introduce the abbreviation~$u_\rho := \sqrt{m n^{m-1}} \, (u^n)^{\star m}$ for~$\rho = n/m \in \QQ$.
Then the relations~$\mathcal{R}_u(\Gamma,\star)$ are gleaned from the multiplication law of~$\star$
given in \S\eqref{sub:min-schwartz-algebra}, appearing here as
\begin{equation}
  \label{eq:conv-rel}
  \begin{aligned}
    \mathcal{R}_u(\Gamma,\star) &:= \{ (\alpha \act u_{\rho}) \star (\alpha' \act u_{\rho'}) -
    c_{\alpha,\alpha',\rho,\rho'} \, \big( \tfrac{\alpha\rho'+\alpha'\rho}{\rho+\rho'} \act
    u_{\rho\boxedplus\rho'} \big) \}\\
    &\quad\text{with}\quad c_{\alpha,\alpha',\rho,\rho'} :=
    e^{-\pi(\alpha-\alpha')^2/(\rho+\rho')}/\sqrt{\rho+\rho'},
  \end{aligned}
\end{equation}
where\footnote{Though~$G = \Gamma = \QQ$ as sets, we write~$\alpha \in \Gamma$ and~$\mu \in G$ to
  identify $\mu = 1(\mu, 0)$ and~$\alpha = 1(0, \alpha)$, respectively, via the standard embeddings
  in $H_{\beta_\QQ}$.} the parameters~$\alpha,\alpha',\rho,\rho'$ range
over~$\Gamma = \QQ \le TG \rtimes \Gamma = H_{\beta_\QQ}$. Similarly, the
relations~$\mathcal{R}_u(G,\cdot)$, extracted from the multiplication law of~$\cdot$, manifest
themselves as
\begin{equation}
  \label{eq:ptw-rel}
  \begin{aligned}
    \mathcal{R}_u(G,\cdot) &:= \{ (\mu \act u_{\rho}) \cdot (\mu' \act u_{\rho'}) -
    c_{\mu,\mu',\rho,\rho'} \, \big( \tfrac{\mu\rho+\mu'\rho'}{\rho+\rho'} \act
    u_{\rho+\rho'} \big) \}\\
    &\quad\text{with}\quad c_{\mu,\mu',\rho,\rho'} :=
    e^{-\pi(\rho\boxedplus\rho')(\mu-\mu')^2}.
  \end{aligned}
\end{equation}
We notice the intriguing symmetry between~\eqref{eq:conv-rel} and~\eqref{eq:ptw-rel}, ultimately due
to the convolution theorem (but note the distinct relative position of the primes in the two
Heisenberg actors). The last relation set~$\mathcal{R}_u(\star,\cdot)$ comes from the
commutation~\eqref{eq:conv-ptw-prop} of convolution and pointwise powers, yielding
\begin{equation}
  \label{eq:conv-ptw-rel}
  \mathcal{R}_u(\star,\cdot) := \{ (u^{\star m})^n - \sqrt{n^{m-1}/m^{n-1}} \, (u^n)^{\star m} \mid
  n \in \ZZ, m \in \ZZ_{>0} \} .
\end{equation}
Now we can state the promised isomorphism that characterizes~$\bar{\Gau}\inner{\QQ}{\QQ}$ explicitly
as a quotient of the free twain Heisenberg algebra.

\begin{theorem}
  \label{thm:gaussians-via-free-structure}
  We have~$\bar{\Gau}\inner{\QQ}{\QQ} \cong \freetwn(u)/\langle\mathcal{R}_u\rangle$ via the
  Heisenberg isomorphism induced by~$g \leftrightarrow u$.
\end{theorem}
\begin{proof}
  It is sufficient to show
  that~$\{ \alpha \act \mu \act u_\rho \mid (\mu, \alpha) \in G \times \Gamma, \rho \in \QQ \}$
  forms a system of mutually incongruent representatives whose classes are a $\QQ^\pi$-linear basis
  of~$\freetwn(u)/\langle\mathcal{R}_u\rangle$. For then
  $e_\alpha \, g_{\mu,\rho} \leftrightarrow \alpha \act \mu \act u_\rho$ will deliver the desired
  isomorphism. Indeed, it is clearly a $\QQ^\pi$-linear isomorphism, it respects the Heisenberg
  action since~$e_\alpha \, g_{\mu,\rho} = (\mu, \alpha) \act g_\rho$
  in~$\bar{\Gau}\inner{\QQ}{\QQ}$, and it is also a twain homomorphism: We have factored out the
  corresponding relations~\eqref{eq:conv-rel} and~\eqref{eq:ptw-rel}, which yield the required
  multiplication laws for~$\star$ and~$\cdot$ when combined with the scalar/operator properties
  (see Definition~\ref{def:heis-alg}) valid in~$\freetwn(u)$.

  Oriented from left to right, we view the relations~$\langle\mathcal{R}_u\rangle$ as a term
  rewriting system over the signature~$[{\TwAlgH{\beta_\QQ}}_u]$ introduced in the alternative proof
  of Proposition~\ref{prop:free-heis-twn}. Due to space limitations, we can only sketch the rest of
  the proof. We consider~$\star$ and~$\cdot$ modulo AC (associativity+commutativity), identifying
  unary products with their arguments and nullary ones with~$1 \in \QQ^\pi$. Moreover, we identify
  the torus elements~$c \in \Tor_\QQ < H_{\beta_\QQ}$ with the field scalars~$c \in \QQ^\pi$, so the
  Heisenberg actors are essentially given by~$(\mu, \alpha) \in \QQ \times \QQ$. Note that this
  rewrite system is ground (with $u$ being a constant since it is not subject to substitution). It
  is terminating as one can see by taking the lexicographic path
  order~\cite[\S5.4.2]{BaaderNipkow1998} with~$u > \cdot > \star > \act$ on monomials (i.e.\@ terms
  not containing field scalars or $+$). As usual, this extends to a Noetherian quasi-order on
  $\QQ^\pi$-linear combinations of monomials; confer Theorem~5.12 in~\cite{BeckerWeispfenning1993}.

  The formal arguments in the relations~$\mathcal{R}_u(\Gamma,\star)$, $\mathcal{R}_u(G,\cdot)$,
  $\mathcal{R}_u(\star,\cdot)$ of~\eqref{eq:conv-rel}, \eqref{eq:ptw-rel}, \eqref{eq:conv-ptw-rel}
  indicate the joint occurrence of the corresponding symbols in the redexes; this reveals that there
  is no overlap between the rewrite rules. Hence the rewrite system is confluent, and we conclude
  that every element of~$\freetwn(u)$ has a unique normal
  form~\cite[Lem.~2.1.8]{BaaderNipkow1998}. Since all
  elements~$\alpha \act \mu \act u_\rho \in \freetwn(u)$ are irreducible, they must be mutually
  incongruent~\cite[Thm.~2.1.9]{BaaderNipkow1998}. Any $\QQ^\pi$-linear combination of them is also
  irreducible and hence clearly nonzero; this means that the~$\alpha \act \mu \act u_\rho$ are
  $\QQ^\pi$-linearly independent as elements of~$\freetwn(u)/\langle\mathcal{R}_u\rangle$. It
  remains to show that every term in~$\freetwn(u)$ reduces to a $\QQ^\pi$-linear combination of
  irreducible normal forms, which then implies that the $\alpha \act \mu \act u_\rho$ actually
  form a $\QQ^\pi$-basis.

  For seeing that every normal form~$\phi$ of~$\freetwn(u)$ is in the span of the normal
  forms~$\alpha \act \mu \act u_\rho$, it suffices to expand~$\phi$ into a $\QQ^\pi$-linear
  combination of monomials and then reduce each of them to normal form. We may view the monomial as
  a tree of alternating $[\star]$~and $[\cdot]$~nodes decorated by Heisenberg actors (identifying
  the case of no actors with the action of~$1 \in H_{\beta_\QQ}$). All leaves are occurrences
  of~$u$. Let us speak of a $[\star]$~tree if the root is a $[\star]$~node (and hence every even
  tree level consists of $[\star]$~nodes) and let us speak of a $[\cdot]$~tree otherwise. Starting
  from the leaves, the relations~$\mathcal{R}_u(\Gamma,\star)$ and~$\mathcal{R}_u(G,\cdot)$ together
  with the scalar/operator properties from Definition~\ref{def:heis-alg} serve to eliminate
  Heisenberg actors from internal nodes: The $u_{\rho+\rho'}, u_{\rho\boxedplus\rho'}$ in the
  corresponding right-hand monomials are trees without internal Heisenberg actors, so these
  monomials have Heisenberg actors only in their roots. After exhausting these rewrite steps, we are
  left with~$\alpha \act \mu \act U$, where~$U$ is either of the form~$u_{n/m} = (u^n)^{\star m}$ or
  of the form~$(u^{\star m})^n$. In the former case, we are done; in the latter case, we
  apply~$\mathcal{R}_u(\star,\cdot)$ for reducing to the former case.
\end{proof}

The above proof is probably not as concise as it should be. It would be preferrable to use something
like Gr\"obner(-Shirshov) bases, extended to the case of ``twain polynomials'', for dealing with the
\emph{twain polynomial ideal}~$\mathcal{R}_u$. This would be interesting to develop in future
work. It might also be worthwhile to consider enhanced rewrite approaches that incorporate Gr\"obner
bases such as \cite{BachmairGanzinger1994}.

The algebraic description of $\overline{\Gau(\RR)}$ or~$\overline{\Gau(\QQ)}$ may appear
insignificant at first, but it should be pointed out that in some sense the Gaussians contain the
whole essence of classical Fourier analysis: They are known~\cite[Thm.~2.2]{Thaller2000} to be
\emph{dense in~$L^2(\RR)$}. In the context of our present approach, we view the Gaussian
singlets~$\overline{\Gau(\RR)}$ and~$\overline{\Gau(\QQ)}$ as a base for bootstrapping
\emph{algebraic hierarchies} of more involved Fourier singlets amenable to Symbolic Computation.

\subsection{Holonomic Fourier Extensions.}
\label{sub:hol-four-extns}
A distribution~$s \in \D(\RR^n)$ is called \emph{holonomic} if it is defined through a maximally
overdetermined set of polynomial PDEs. More precisely, writing
$I_s = \{ T \in A_n(\CC) \mid Ts = 0 \}$ for its annihilation ideal, one requires the quotient
module~$A_n(\CC)/I_s$ to be holonomic~\cite[Def.~7.2.1]{Bjork1979},
\cite[\S2.4]{Zeilberger1990}. The collection of all holonomic distributions is known as the
\emph{Bernstein class}~$\B'(\RR^n)$. It is clear~\cite[Prop.~2.2]{Bjork1979} that~$\B'(\RR^n)$ is
then an $A_n(\CC)$-submodule\footnote{Writing~$D := A_n(\CC)$, this is usually called a
  ``$D$-module''. We refrain here from this terminology so as to avoid confusion with the space of
  bump functions~$\D(\RR^n)$.} of~$\D(\RR^n)$.

It is an important fact~\cite[Prop.~2.3]{Bjork1979} that a tempered distribution
$s \in \Schw(\RR^n)$ belongs to~$\B'(\RR^n)$ iff~$\Four s \in \Schw(\RR^n)$ does. This means the
Fourier operator~$\Four$ is a $\CC$-linear automorphism on the $A_n(\CC)$-module
$\Hol'(\RR^n) := \B'(\RR^n) \cap \Schw'(\RR^n)$. The Dirac
distributions~$\delta_\xi \: (\xi \in \RR^n)$ are clearly contained in~$\Hol'(\RR^n)$, and so are
their Fourier transforms $\chi_\xi(x) = e^{i\tau \xi \cdot x}$. Since~$\D'(\RR^n)$ is a module
over~$C^\infty(\RR^n)$, the product of any holonomic distribution~$s \in \Hol'(\RR^n)$
with~$\chi_\xi \in C^\infty(\RR^n)$ is well-defined in~$\D'(\RR^n)$, and since both~$s$
and~$\chi_\xi$ are holonomic, so is their product~\cite[Prop.~3.2]{Zeilberger1990}. This implies
that~$\Hol'(\RR^n)$ is closed under pointwise multiplication by~$\chi_\xi$, which is identical to
the action of the Heisenberg operator~$\xi \in H(\pomega)$. By applying the Fourier operator
$\Four\colon \Hol'(\RR^n) \to \Hol'(\RR^n)$, we obtain also closure under Heisenberg scalars while
the torus action~$\Tor \le \nnz{\CC}$ is anyway trivial. Thus~$\Hol'(\RR^n)$ is a Heisenberg
submodule of~$\Schw'(\RR^n)$, and it inherits the compatible Weyl action. We summarize these facts
on the space~$\Hol'(\RR^n)$ of \emph{holonmic dis\-tributions} as follows.

\begin{proposition}
  \label{prop:hol-distr}
  The holonomic distributions form a regular slain sub\-singlet $\Hol'\inner{\RR^n}{\RR^n}$ of
  $\Schw'\inner{\RR^n}{\RR^n}$ with compatible Weyl actions.
\end{proposition}

Writing~$\Hol(\RR^n) := \B'(\RR^n) \cap \Schw(\RR^n)$ for the \emph{holonomic Schwartz class}, the
Fourier operator~$\Four\colon \Hol'(\RR^n) \to \Hol'(\RR^n)$ clearly restricts further to a
$\CC$-linear automorphism~$\Hol(\RR^n) \to \Hol(\RR^n)$. As the intersection of Heisenberg
modules~$\Hol'(\RR^n)$ and~$\Schw(\RR^n)$, the holonomic Schwartz class~$\Hol(\RR^n)$ is clearly a
Heisenberg module itself. But it is moroever a Heisenberg twain algebra since pointwise and
convolution products preserve holonomicity~\cite[Prop.~3.2, 3.2$^\ast$]{Zeilberger1990}. Altogether
we obtain the following facts.

\begin{proposition}
  \label{prop:hol-Schwartz}
  The holonomic Schwartz class~$\Hol\inner{\RR^n}{\RR^n}$ forms a regular twain subsinglet of
  $\Schw\inner{\RR^n}{\RR^n}$ with compatible Weyl action.
\end{proposition}

Holonomic distributions and Schwartz functions are very suitable for \emph{Symbolic Computation}
since one can use normal forms for deciding equalities. To be precise, one does not have a canonical
simplifier but a a \emph{normal simplifier} in the sense of~\cite[\S1]{BuchbergerLoos1982}, so
deciding equalities is reduced to zero recognition; see~\cite[\S4.1]{Zeilberger1990} and Algorithm~Z
of \cite[B.2]{FlajoletSedgewick2009}. The crucial tool for computing the Fourier operator~$\Four$
symbolically is the compatible Weyl action: Applying the \emph{differentiation
  laws}~\eqref{eq:diff-laws} allows us to extract the defining PDEs of~$\Four s$ from those of a
holonomic distribution or Schwartz function, as also pointed out in the paragraph
preceding~\cite[Prop.~3.2$^\ast$]{Zeilberger1990}. Sums, pointwise and convolution products are all
effectively computable~\cite[Thm.~6.6]{OakuShirakiTakayama2003}. When dealing with holomorphic
functions (which form a valuation ring), one may also take advantage of the valuation techniques
from \S\ref{sub:min-schwartz-algebra}.

\begin{myexample}
  As a warmup, let us rederive the Fourier transform~\eqref{eq:Four-gaussian} of the modulated
  Gaussians~$e_\alpha g_{\mu,\rho} \in \overline{\Gau(\RR)}$ from the holonomic perspective. By
  differentiating the function once, we immediately obtain its
  annihilator~$T_{\alpha,\mu,\rho} = \der + \tau\rho \, x - \tau(\rho\mu + i \alpha)\in
  A_1(\CC)$.
  Note that~$T := T_{\alpha,\mu,\rho}$ is uniquely determined (if chosen monic, for any ordering
  with~$\der > x$), and that~$\alpha,\mu,\rho$ can then be read off via
  \begin{equation}
    \label{eq:param-from-annih}
    \rho = \tfrac{1}{\tau} \, [x] T,\quad
    \mu = - \tfrac{1}{\tau\rho} \, \Re \, [1] T,\quad
    \alpha = -\tfrac{1}{\tau} \, \Im \, [1] T,
  \end{equation}
  where~$[\der^j x^k]$ denotes the coefficient of~$\der^j x^k$. Applying the Fourier
  transform~$\mathfrak{f}$ of~\eqref{eq:four-transf-weyl}, we see
  that~$\Four(e_\alpha g_{\mu,\rho})$ is annihilated by
  \begin{equation*}
    i \rho^{-1} \, \widehat{T_{\alpha,\mu,\rho}} = \der + \tfrac{\tau}{\rho} \, x +
    \tfrac{\tau}{\rho} \, (\alpha - i\rho\mu),
  \end{equation*}
  which via~\eqref{eq:param-from-annih} yields
  immediately~$\Four(e_\alpha g_{\mu,\rho}) = e_{\hat{\alpha}} g_{\hat\mu,\hat\rho}$
  with~$\hat{\rho} = 1/\rho$, $\hat{\alpha} = \mu$, $\hat{\mu} = -\alpha$ as
  in~\eqref{eq:Four-gaussian}. Thus we see that the holonomic approach recovers the Fourier
  transform up to the overall constant denoted by~$c$ in~\eqref{eq:Four-gaussian}.

  The constant~$c = c_{\alpha,\mu,\rho}$ may be determined as follows. By our choice of
  normalization for the Gaussians~$g_{\mu,\rho}$, the $L^2$ norm square of~$e_\alpha g_{\mu,\rho}$
  is given by~$1/\sqrt{2\rho}$. Since~$\Four$ preserves the $L^2$ norm square, we obtain
  $|c_{\alpha,\mu,\rho}|^2 = \sqrt{\hat{\rho}/\rho} = 1/\rho$, and it remains only to determine the
  phase dependence~$\arg c_{\alpha,\mu,\rho}$. Note first
  that~$\Four g_{0,\rho} = c_{0,0,\rho} \, g_{0,1/\rho}$ implies~$\arg c_{0,0,\rho} = 1$ since the
  Fourier transform of a real-valued even signal is again real valued and
  even~\cite[\S2]{Bracewell1986}. Together with~$|c| = 1/\sqrt{\rho}$ this fixes the
  value~$c_{0,0,\rho} = 1/\sqrt{\rho}$. Applying the twist axiom~\eqref{eq:twisted-bimod}
  yields~$e_\alpha g_{\mu,\rho} = (0,\alpha) \act (\mu, 0) \act g_{0,\rho} = \inner{\alpha}{\mu}
  \, (\mu, \alpha) \act g_{0,\rho}$
  whose Fourier transform
  is~$c_{\alpha,\mu,\rho} \, e_\mu g_{-\alpha,1/\rho} = \inner{\alpha}{\mu} \, (\mu, -\alpha) \act
  \Four g_{0,\rho}$
  by the equivariance laws~\eqref{eq:fourop-par-id}. Using
  again~$\Four g_{0,\rho} = c_{0,0,\rho} \, g_{0,1/\rho}$ and the Heisenberg action on
  spectra~\eqref{eq:pont-actions-Gamma}, we get
  $c_{\alpha,\mu,\rho} \, e_\mu g_{-\alpha,1/\rho} = \inner{\alpha}{\mu} \, c_{0,0,\rho} \, e_\mu
  g_{-\alpha,1/\rho}$
  and thus
  $c_{\alpha,\mu,\rho} = \inner{\alpha}{\mu} \, c_{0,0,\rho} = \inner{\alpha}{\mu}/\sqrt{\rho}$ as
  in~\eqref{eq:Four-gaussian}.
\end{myexample}

\begin{myexample}
  \label{ex:quartic}
  Consider the \emph{quartic Gaussian}~$g(x) := e^{-x^4} \in \Hol(\RR)$, which is relevant in
  various applications ranging from computer graphics, neutral networks and data
  interpolation~\cite{Boyd2014} to energy correlations in random
  matrices~\cite{GorskaPenson2013}. Obviously, $g$ is annihilated by the Weyl
  operator~$T = \der + 4x^3 \in A_1(\CC)$ whose Fourier transform is given
  by~$(i\tau)^3 \, \hat{T} = 4\der^3 - \tau^4 x$. Therefore the Fourier transform~$\hat{g}(x)$
  of~$g(x)$ satisfies the third-order equation $4 \hat{g}'''(x) - \tau^4x \, \hat{g}(x) = 0$. For
  fixing the solution, we need three initial values~$\hat{g}^{(j)}(0)$ for~$j = 0, 1, 2$. Via the
  differentiation laws~\eqref{eq:diff-laws}, we have $\hat{g}^{(j)}(0) = (i\tau)^j \, \mu_j$ with
  the moments~$\mu_j = \cum_{-\infty}^\infty \, g(x) \, x^j \, dx$. The general solution of the
  third-order equation is hypergeometric and can be obtained by standard symbolic software packages.
  Using \mma, one finds with integration constants named~$c_0, c_1, c_2 \in \CC$ that
  \begin{equation}
    \label{eq:gen-hypergeom}
    \hat{g}(x) = c_0 \, \Phi_{1/2,3/4}(x) + c_1\zeta x\, \Phi_{3/4,5/4}(x) 
    + c_2\zeta^2 x^2\, \Phi_{5/4,3/2}(x)
  \end{equation}
  where~$\zeta := \tau\zeta_8/4$ with~$\zeta_8 = e^{i\tau/8}$ the primitive eighth root of unity,
  and where $\Phi_{a,b}(z) := \pFq{0}{2}{}{a,b}{(\tau z/4)^4}$ denotes a (generalized)
  hypergeometric function with lower parameters~$a,b \in \CC \setminus \ZZ^-$. Note
  that~$\Phi_{a,b}$ is an entire function of the complex variable~$z$.

  The moments are~$(\mu_0, \mu_1, \mu_2) = \big(2\Gamma(5/4), 0, \tfrac{1}{2}\Gamma(3/4)\big)$, the
  derivatives of~\eqref{eq:gen-hypergeom}
  by~$\big(\hat{g}_0, \hat{g}'(0), \hat{g}''(0)\big) = (c_0, \zeta c_1, 2\zeta^2 c_2)$, hence the
  constants are given by~$(c_0, c_1, c_2) = \big(2 \Gamma(5/4), 0, 4i \, \Gamma(3/4)\big)$.
  Substitung these in~\eqref{eq:gen-hypergeom} and
  using~$\Gamma(3/4) \, \Gamma(5/4) = \tfrac{\pi}{2\sqrt{2}}$, this yields the representation
  $\hat{g}(x) = \tfrac{\pi}{\Gamma(3/4)\sqrt{2}} \, \Phi_{1/2,3/4}(x) - \Gamma(3/4) \, \pi^2 \, x^2
  \, \Phi_{5/4,3/2}(x)$.
  Applying Legendre's duplication~$\Gamma(3/4) \Gamma(1/4) = \pi \sqrt{2}$ leads to
  \begin{equation}
    \label{eq:quartic-gaussian-four}
    \hat{g}(x) = \tfrac{\gamma}{2} \, \Phi_{1/2,3/4}(x) - \tfrac{\pi^3\sqrt{2}}{\gamma} 
    \, \Phi_{5/4,3/2}(x) \, x^2,
  \end{equation}
  where have set~$\gamma := \Gamma(1/4)$. Adapting to different normalizations,
  \eqref{eq:quartic-gaussian-four} agrees with~\cite[(7)]{Boyd2014} as~$\hat{g}(k/\tau)$ and
  with~\cite[(6)]{GorskaPenson2013} in the form~$\hat{g}(x\sqrt{2}/\tau)$.

  It is known that~$\gamma$ is transcendental. But more is true: Since by Nesterenko's results one
  knows~\cite[\S1.5.7]{ManinPanchishkin2006} that~$\{\pi, e^\pi, \gamma\}$ is algebraically
  independent over~$\QQ$, we can conclude that~$\gamma$ is in fact \emph{transcendental
    over~$\QQ^\pi(\pi)$}. In \S\ref{sub:gelfond-field} we have seen how one may compute in
  the Gelfond field~$\QQ^\pi$. From~\eqref{eq:quartic-gaussian-four} we see that it is sufficient to
  pass to the bivariate rational function field~$\QQ^\pi(\pi, \gamma)$ for working with the Fourier
  transform~$\hat{g}(x)$.
\end{myexample}

In \S\ref{sub:rational-gauss} have seen that the unit Gaussians~$g = g_{0,1}$ plays a
central role in describing the rational Gaussian closure~$\overline{\Gau(\QQ)}$. It has the
convenient property of being a fixed point (i.e.\@ an eigenfunction with eigenvalue~$1$) of the
Fourier operator~$\Four\colon \Schw(\RR) \to \Schw(\RR)$. This is generalized by the \emph{Hermite
  polynomials}, which are eigenvectors for the four possible eigenvalues~$\pm 1, \pm i$ of~$\Four$.

For doing so, we extend the Gaussian closure~$\overline{\Gau(\RR)} \subset \Schw(\RR)$ to a
Heisenberg twain subalgebra of~$\Schw(\RR)$ just large enough for accommodating the action of the
Weyl algebra. Clearly, such an algebra is given by
\begin{equation}
  \label{eq:weyl-closure}
  \overbrk{\Gau(\RR)} := \CC[x] \otimes_{\CC} \overline{\Gau(\RR)},
\end{equation}
where the tensor product creates a nonunital $\CC$-algebra from the unital~$\CC[x]$ and the
nonunital~$\overline{\Gau(\RR)}$. We refer to~$\overbrk{\Gau(\RR)}$ as the \emph{Gaussian
  $D$-module}.

Of course, we may also define its computable core, the \emph{rational Gaussian
  $D$-module}~$\overbrk{\Gau(\QQ)}$, by making the obvious replacements in~\eqref{eq:weyl-closure},
thus defining
\begin{equation}
  \label{eq:gaussian-weyl-closure}
  \overbrk{\Gau(\QQ)} := \tilde{\QQ}[x] \otimes_{\tilde{\QQ}} \overline{\Gau(\QQ)},
\end{equation}
where~$\tilde{\QQ} := \QQ^\pi(\pi)$ is the extended Gelfond field introduced before
Proposition~\ref{prop:gelf-with-pi}. In the sequel, we shall stick to~$\Gau(\RR)$ for ease of
presentation. Its (pointwise) algebra structure is clear from the defintion~\eqref{eq:weyl-closure},
but for its convolution product and Fourier operator are best understood via Hermite
functions~\cite[\S7.1]{Stade2011}, so let us now study these in some more detail.

\begin{myexample}
  \label{ex:Hermite-functions}
  As in~\cite[Exc.~7.1.3]{Stade2011}, we use the so-called physicists' Hermite polynomials
  $H_n(x) := (2x-\tfrac{d}{dx})^n(1) = (-1)^n e^{x^2} \tfrac{d^n}{dx^n} e^{-x^2}$ and the unit
  Gaussian~$g(x) = e^{-\pi x^2}$ to define the scaled \emph{Hermite functions}
  $\eta_n(x) = c_n \, g(x) \, H_n(\tau^{1/2} x)$ with normalization
  constant~$1/c_n^2 = 2^n n! \sqrt{\pi}$ chosen to suit our conventions in this paper. Compared to
  the standard Hermite functions~$\phi_n$ of~\cite[\S7.2.1]{Thaller2000}, which are orthonormal
  in~$L^2(\RR)$, we have~$\eta_n(x) = \phi_n(\tau^{1/2} x)$.
  
  % THIS APPEARS TO BE TOO CUMBERSOME IN THIS EXAMPLE:
  % We can now proceed as in Example~\ref{ex:quartic}.  Thus we compute
  % %
  % \begin{equation*}
  %   (\mu_0[\eta_n], \mu_1[\eta_n]) = \frac{\sqrt{n!}}{2^k k!} \times \begin{cases}
  %      (0, \pi^{-3/4} \, ) & \text{for odd $n=2k+1$}\\
  %      (\pi^{-1/4}, 0) & \text{for even $n=2k$}
  %   \end{cases}
  % \end{equation*}
  % %
  % for the first two moments. Since as always~$(i\tau)^j \mu_j[\eta_n] = \hat{\eta}_n^{(j)}(0)$,
  % this
  % yields the initial values~$\hat{\eta}_n(0) = \mu_0[\eta_n]$
  % and~$\hat{\eta}_n'(0) = i\tau\mu_1[\eta_n]$. Hence~$\eta_n(0)$ is given by $0$ for odd~$n$ and
  % by $(-1)^k \, \pi^{-1/4} \tfrac{\sqrt{n!}}{2^k k!}$ for even~$n=2k$ while $\eta_n'(0)$ is given
  % by $0$ for even~$n$ and by~$(-1)^k \, \pi^{1/4} \tfrac{\sqrt{n!}}{2^{k-1}k!}$ for odd~$n=2k+1$.

  Since~$\Four^4 = 1$, it is clear that~$\Four$ has eigenvalues~$\pm 1, \pm i$ and that~$\Four$ is
  an algebraic operator (an algebraic element of the linear operator algebra according
  to~\cite[Thm.~1.4.1]{PrzeworskaRolewicz1988}. The corresponding eigenfunctions are the~$\eta_n$
  whose scaling was specifically chosen to ensure~$\Four \eta_n = i^n \, \eta_n$. This is a standard
  result of Fourier theory. For example, the proof in~\cite[\S7.6]{Strichartz1994}, with slightly
  different normalization, is essentially algebraic: It involves only the ladder operators
  associated to the \emph{harmonic oscillator}~$H = D^2 + x^2$, whose eigenvalues determine the
  annihilators~$T_n = D^2 + x^2 - \tfrac{2n+1}{\tau} \in A_1(\CC)$ of the corresponding~$\eta_n$.
  Except for the different sign convention for~$\Four$, the Fourier
  pairs~$[\eta_n \mapsto i^n \eta_n]$ coincide with~\cite[Exc.~8.25]{Bracewell1986}.

  As an \emph{algebraic operator} with four simple eigenvalues~$\pm 1, \pm i$, the
  space~$\overbrk{\Gau(\RR)}$ splits~\cite[Cor.~1.4.1]{PrzeworskaRolewicz1988} into the four
  associated eigenspaces $E_k = [\eta_n \mid n \equiv k \pmod 4]$ with the corresponding
  projectors~$P_k$. Over the larger space, the finite linear combinations inherent in the spans and
  projectors are relaxed to series subject to suitable growth
  constraints~\cite[\S7.6]{Strichartz1994}. When the eigenvalues are $n$-th roots of unity
  (in~\cite[Exc.~1.4.1]{PrzeworskaRolewicz1988} the corresponding operators are named ``involutions
  of order $n$''), the projectors associated to~$A$ may be expressed as averages over the distinct
  iterates~$A^k$ weighted by the eigenvalues~$\lambda_k$. The simplest example is the
  symmetrizer~$P_1 = \Par = (1+A)/2$ along with the corresponding antisymmetrizer~$P_{-1} = (1-A)/2$
  in the case~$n=2$. In the present setting with~$n=4$, we have a fourfold
  sum~\cite[(1.4.27)]{PrzeworskaRolewicz1988}; from a purely analytic viewpoint such a decomposition
  is of course ``not considered very interesting''~\cite[p.~132]{Strichartz1994}.

  What is more important, at least from an algebraic perspective, is the fact that there are
  explicit formulae, such as~(13.9) and~Exc.~13.1.6 in~\cite{ArfkenWeberHarris2013}, for
  \emph{changing between the monomial to the Hermite basis} in~$\CC[x]$, and hence beween the
  corresponding bases of~$\overbrk{\Gau(\RR)}$.
\end{myexample}

For verifying that~$\overbrk{\Gau(\RR)}$ is closed under~$\Four$ and also computable (on the
Gaussian $D$-submodule, that is), we compute the action of~$\Four$ on the standard
basis~$x^k e_\alpha \, g_{\mu,\rho}$. Let us first restrict to the case~$\alpha=0$. Using
\eqref{eq:four-transf-weyl} and~\eqref{eq:Four-gaussian}, we
have~$\Four(x^k g_{\mu,\rho}) = (i\tau)^{-k} \tfrac{d^k}{dx^k} \, \rho^{-1/2} \, e_\mu(x) \,
g_{0,1/\rho}(x)$.
At this point, we need the repeated derivatives of Gaussians, which may be expressed in terms of
Hermite polynomials as
\begin{equation}
  \label{eq:gaussian-derivative}
  g_{\mu,\rho}^{(j)}(x) = (-1)^j \, (\pi\rho)^{j/2} \, H_j\big( (\pi\rho)^{1/2}
  (x-\mu) \big) \, g_{\mu,\rho}(x),
\end{equation}
as one can immediately check by induction on~$k$. Using the Leibniz rule for repeated
differentiation, \eqref{eq:diff-laws}, and~$e_\mu^{(j)} = (i\tau\mu)^j \, e_\mu$ yields
\begin{equation}
  \label{eq:four-monom-gauss-raw}
  \Four(x^k g_{\mu,\rho}) = \tfrac{\mu^k}{\sqrt{\rho}} \, e_\mu(x) \, g_{0,1/\rho}(x) 
  \sum_{j=0}^k \tbinom{k}{j} \big( \tfrac{-1}{4\pi\rho\mu^2} \big)^{j/2} H_j(\sqrt{\pi/\rho} \, x) .
\end{equation}
This can be simplified by using the so-called Appell identity~\cite[\S4.2.1]{Roman2005}, which in
our settings is given by
\begin{equation*}
  H_k \big( \tfrac{y+s}{(2\nu)^{1/2}} \big) = \big( \tfrac{\nu}{2} \big)^{-k/2} \sum_{j=0}^k
  \tbinom{k}{j} \big( \tfrac{\nu}{2} \big)^{j/2} H_j \big( \tfrac{y}{(2\nu)^{1/2}} \big) \, s^{k-j} .
\end{equation*}
Substituting~$\nu = -1\,/\,2\mu^2\pi\rho$, $y = ix\,/\rho\mu$, $s=1$ in this,
\eqref{eq:four-monom-gauss-raw} becomes
\begin{equation}
  \label{eq:four-monom-gauss-intermed}
  \Four(x^k g_{\mu,\rho}) = \tfrac{1}{\sqrt{\rho}} \, \big( \tfrac{i}{2\sqrt{\pi\rho}} \big)^k
  H_k \big( \sqrt{\pi/\rho} \, (x-i\mu\rho) \big) \, e_\mu(x) \, g_{0,1/\rho}(x) .
\end{equation}
Before passing on to the general case, it is expedient to introduce the Hermite polynomials with
variance~$\nu$ as in~\cite[\S4.2.1]{Roman2005} by
\begin{equation}
  \label{eq:hermitepoly-var}
  H_k^{[\nu]}(x) := \big( \tfrac{\nu}{2} \big)^{k/2} \, H_k(x/\sqrt{2\nu}) .
\end{equation}
Using these, \eqref{eq:four-monom-gauss-intermed} is given by
\begin{equation}
  \label{eq:four-monom-gauss-final}
  \Four(x^k g_{\mu,\rho}) = \tfrac{(i/\rho)^k}{\sqrt{\rho}} \, H_k^{[\rho/\tau]}(x-i\mu\rho) \,
  e_\mu(x) \, g_{0,1/\rho}(x) ,
\end{equation}
and now the unmodulated special case immediately generalizes via~\eqref{eq:fourop-par-id} to
\begin{equation}
  \label{eq:four-monom-osc-gauss}
  \Four(x^k e_\alpha \, g_{\mu,\rho}) = \tfrac{\inner{\alpha}{\mu}}{\sqrt{\rho}} \, (i/\rho)^k \,
  H_k^{[\rho/\tau]}(x+\alpha-i\mu\rho) \, e_\mu(x) \, g_{-\alpha,1/\rho}(x) ,
\end{equation}
which at the same time generalizes~\eqref{eq:Four-gaussian}.

Closure of~$\overbrk{\Gau(\RR)}$ under $\star$ follows
from~$s_1 \star s_2 = \Four^{-1}(\Four s_1 \cdot \Four s_2)$ for
all~$s_1, s_2 \in \overbrk{\Gau(\RR)}$, where we may use~\eqref{eq:four-monom-osc-gauss} for
computing~$\Four$ as well as $\Four^{-1} = \Four \circ \Par$ in the~$x^k e_\alpha \, g_{\mu,\rho}$
basis. However, this may not be the most promising route for obtaining an explicit form of the
\emph{convolution law} since products of translated Hermite polynomials with different variances
tend to be cumbersome.

Substituting directly into the definition of convolution, one gets the
product~$p(x) := x^{k_1} e_{\alpha_1} \, g_{\mu_1,\rho_1} \star x^{k_2} e_{\alpha_2} \,
g_{\mu_2,\rho_2}$ as
\begin{equation}
  \label{eq:conv-law-start}
  p(x) = \inner{\alpha_1}{x} \int_{-\infty}^\infty \inner{\alpha_2-\alpha_1}{y} \, (x-y)^{k_1} \, y^{k_2}
  g_{x-\mu_1,\rho_1}(y) \, g_{\mu_2,\rho_2}(y) \, dy ,
\end{equation}
using~$g_{\mu_1,\rho_1}(x-y) = g_{x-\mu_1,\rho_1}(y)$ and the definition of oscillating
exponentials. We observe that~\eqref{eq:conv-law-start} has the form of a Fourier transform with
respect to~$y$, and we may apply~\eqref{eq:gauss-ptw-prod} for multipliying the two Gaussians in the
integrand to obtain a Gaussian with paramters~$\hat{\rho} := \rho_1 + \rho_2$
and~$\hat{\mu}_x := \tfrac{\rho_1x-\rho_1\mu_1+\rho_2\mu_2}{\rho_1+\rho_2}$.
Thus we can simplify~\eqref{eq:conv-law-start} to
\begin{equation}
  \label{eq:conv-law-mid}
  p(x) = e_{\alpha_1}(x) \, g_{\mu_1+\mu_2,\rho_1 \boxedplus \rho_2}(x) \, \Four_y \big( (x-y)^{k_1}
  y^{k_2} g_{\hat{\mu}_x,\hat\rho}(y) \big) (\Delta\alpha) ,
\end{equation}
with~$\Delta\alpha := \alpha_2-\alpha_1$, and we can apply~\eqref{eq:four-monom-gauss-final} to
compute the Fourier transform in~\eqref{eq:conv-law-mid} as
\begin{align*}
  & \quad \Four_y \big( \dots \big) (\Delta\alpha) = 
    \sum_{j=0}^{k_1} \tbinom{k_1}{j} \, (-1)^{k_1+j} \, x^j \, \Four_y \big( y^{k-j}
    g_{\hat{\mu}_x,\hat\rho}(y) \big) (\Delta\alpha)\\
  & = \tfrac{(i/\rho)^k \,
    g_{0,1/\hat\rho}(\Delta\alpha)}{\hat\rho^{1/2}} \, \inner{\Delta\alpha}{\hat{\mu}_x}
    \sum_{j=0}^{k_1} \tbinom{k_1}{j} \, (-1)^{k_1} (i \hat\rho x)^j
    H_{k-j}^{[\hat\rho/\tau]} (\Delta\alpha-i\hat{\mu}_x\hat\rho) ,
\end{align*}
where~$k := k_1+k_2$. We have~%
$\inner{\Delta\alpha}{\hat{\mu}_x} \, e_{\alpha_1}(x) =
\inner{\alpha_1-\alpha_2\,}{\,\mu_1 -_\rho \mu_2} \, e_{\alpha}(x)$
with the same compound frequency~$\alpha := \alpha_2 +_\rho \alpha_1$
as for the convolution product~\eqref{eq:ext-conv-gauss} without
monomials. We thus obtain
\begin{equation*}
  p(x) = (-1)^{k_1} \, (i/\rho)^k \, \bar{\gamma}_\star \, \Phi_{k_1,k_2}^{[\rho/\tau]}(i \hat\rho x,
  \Delta\alpha - i\hat{\mu}_x\hat\rho) \, e_\alpha g_{\mu,\rho},
\end{equation*}
where~$\bar{\gamma}_\star$ is the extended convolutive Gaussian
cocycle~\eqref{eq:ext-convol-cocyc} and where
\begin{equation}
  \label{eq:conv-pol}
  \Phi_{m,n}^{[\nu]}(\xi, \eta) := \sum_{j=0}^{m} \tbinom{m}{j} \, \xi^j \, H_{m+n-j}^{[\nu]} (\eta) 
\end{equation}
is a certain polynomial in~$(\xi,\eta)$ of bidegree~$(m,n)$.

For identifying this polynomial, we apply the methods of the umbral
calculus~\cite[\S2.2]{Roman2005}, viewing differential operators as power series in the formal
variable~$t$, for which we may substitute~$\der_\xi$ or~$\der_\eta$ at will. Since the Hermite
polynomials~$H^{[\nu]}$ form an Appell sequence~\cite[\S4.2]{Roman2005} with generator
$g(t) = e^{\nu t^2/2}$, we have the relation
$H_{m+n-j}^{[\nu]}(\eta) = e^{-\nu \der_\eta^2/2} \, \eta^{m+n-j}$ so that~\eqref{eq:conv-pol}
immediately yields the nice representation
\begin{equation}
  \label{eq:phi-rep}
  \Phi_{m,n}^\nu(\xi) = e^{-\nu \der_\eta^2/2} \, (\xi+\eta)^m \, \eta^n .
\end{equation}
Computationally speaking, the task of determining~$\Phi_{m,n}^{[\nu]}$ is in a sense already
achieved at this point: Expanding the exponential series in~$\der_\eta$ up to order~$m+n$, one can
determine every polynomial~$\Phi_{m,n}^{[\nu]}(\xi,\eta)$ explicitly for any specific values
of~$m,n$ and~$\nu$. We will provide a concrete practical calculation scheme below. But for
theoretical purposes, however, it is important to identify the precise ``umbral species''
of~$\Phi_{m,n}^{[\nu]}$.

\begin{proposition}
  \label{prop:sheff-op}
  The polynomials~$\Phi_{m,n}^{[\nu]}(\xi,\eta) \in \QQ(\eta)[\xi]$ form an Appell sequence with
  Sheffer
  operator~$g(t)^{-1} = e^{\eta t-\nu t^2/2} \, H_n^{[\nu]}(\eta-\nu t)$.
\end{proposition}
\begin{proof}
  Differentiating~\eqref{eq:phi-rep} with respect to~$\xi$ and with~$\eta=\eta_0$, $n=n_0$ fixed
  leads to $\der_\xi \, \Phi_{m,n_0}^{[\nu]}(\eta_0) = m \, \Phi_{m-1,n_0}^{[\nu]}(\eta_0)$.
  Regarded as sequence in the polynomial ring~$\QQ(\eta)[\xi]$, this exhibits~$\Phi_{m,n_0}(\eta_0)$
  as a Sheffer sequence for~$\big(t, g(t)\big)$ according to~\cite[Thm.~2.3.7]{Roman2005}. Since
  Appell sequences are by definition~\cite[\S2.3]{Roman2005} Sheffer sequences for~$f(t) = t$, this
  establishes the first claim. (We shall henceforth drop the subscript~$0$ indicating the fixed
  values.)

  For identifying the Appell generator~$g(t)$ or, equivalently~\cite[(3.5.1)]{Roman2005}, the
  induced Sheffer operator~$g(t)^{-1}$, we use the fact~\cite[\S3.5]{Roman2005} that the latter maps
  the standard monomials~$\xi^m \in \QQ(\eta)[\xi]$ to the Appell poly\-nomials
  $\Phi_{m,n}^{[\nu]}(\xi,\eta) \in \QQ(\eta)[\xi]$. From~\eqref{eq:conv-pol} we have
  \begin{equation*}
    \Phi_{m,n}^{[\nu]}(\xi, \eta) = \sum_{j=0}^{\infty} \tfrac{1}{j!} \, H_{j+n}^{[\nu]}
    (\eta) \, m\ffact{j} \, \xi^{m-j} ,
  \end{equation*}
  which, together with~$m\ffact{j} \, \xi^{m-j} = t^j \xi^m$ for~$t = \der_\xi$, yields the Sheffer
  operator in the form~%
  $g(t)^{-1} = \sum_j H_{j+n}^{[\nu]}(\eta) \, \tfrac{t^j}{j!} \in \QQ(\eta)[[t]]$. Applying an
  index shift and the relation~$\der_t^n \, \tfrac{t^j}{j!} = \tfrac{t^{j-n}}{(j-n)!}$, this means
  \begin{equation*}
    g(t)^{-1} = \der_t^n \, \sum_{j=0}^\infty H_j^{[\nu]}(\eta) \, \tfrac{t^j}{j!} = \der_t^n \, 
    e^{\eta t-\nu t^2/2},
  \end{equation*}
  where we have used the Hermite generating function~\cite[\S4.2]{Roman2005} for eliminating the
  sum. For finishing the proof, one applies induction to show
  $\der_t^n \, e^{\eta t-\nu t^2/2} = e^{\eta t-\nu t^2/2} \, H_n^{[\nu]}(\eta-\nu t)$. The base case
  follows from $H_0^{[\eta]} = 1$, the induction step from the recurrence
  relation~\cite[(4.2.2)]{Roman2005}.
\end{proof}

We have \emph{determined} the Appell generator~$g(t)$ in so far as we know its
reciprocal~$g(t)^{-1}$ from Lemma~\ref{prop:sheff-op} and can then apply the usual recursive formula
for finding the coefficients of~$g(t)$. For actually~\emph{computing}~$g(t)$, we would prefer an
explicit representation. To this end, we shall employ the so-called \emph{generalized Feldheim
  identity}~\cite[(13.47)]{ArfkenWeberHarris2013} for computing powers of Hermite polynomials, which
we quote here for easier reference with the following notational conventions: For a
multi-index~$\lambda \in \NN^{m-1}$ we set~$|\lambda|_i := \lambda_1 + \cdots + \lambda_{i-1}$
and~$|\lambda| := |\lambda|_m$. Note that this implies~$|\lambda|_1 = 0$.

\begin{lemma}
  \label{lem:feldheim}
  For~$k>0$, the $k$-th Hermite power is given by
  \begin{equation*}
    H_n^{[\nu]}(\zeta)^k = \!\! \sum_{\lambda \in \Lambda(n,k)} a_\lambda \, H_{nk-2|\lambda|}^{[\nu]}(\zeta)
  \end{equation*}
  with coefficients
  \begin{equation*}
    a_\lambda = \prod_{i=1}^{k-1} \binom{n}{\lambda_i} \binom{ni-2|\lambda|_i}{\lambda_i} \,
    \nu^{\lambda_i} \lambda_i!
  \end{equation*}
  and~%
  $\Lambda(n,k) = \{ \lambda \in \NN^{k-1} \mid \smash{\underset{i=1,\dots,k-1\,}{\forall}} \, 
  0 \le \lambda_i \le \min(n, ni-2|\lambda|_i)\}$ as summation range.
\end{lemma}

We can now present the \emph{Appell generator}~$g(t)$ in a fairly explicit representation. Note that
it does contain Hermite reciprocals, but unlike the raw expression of~$g(t)$, they do not involve
the formal parameter~$t$.

\begin{proposition}
  The Appell generator~$g(t)$ of~$\Phi_{m,n}^{[\nu]}(\xi,\eta) \in \QQ(\eta)[\xi]$, as treated in
  Lemma~\ref{prop:sheff-op}, is given by
  \begin{align*}
    g(t) = \sum_{l=0}^\infty \sum_{m=0}^l \sum_{k=0}^m & \sum_{\lambda \in \Lambda(n,k)}
    \tbinom{l}{m} \tbinom{m+1}{k+1} \, (-1)^{k+l} \: \nu^m \, (nk-2|\lambda|)\ffact{m} \;
    a_\lambda \, \times {}\\
    & \tfrac{H_{nk-m-2|\lambda|}^{[\nu]}(\eta) \, H_{l-m}^{[-\nu]}(\eta)}{%
      H_n^{[\nu]}(\eta)^{k+1}}  \, \tfrac{t^l}{l!}
  \end{align*}
  where~$\Lambda(n,k)$ and~$a_\lambda$ are as in Lemma~\ref{lem:feldheim}.
\end{proposition}
\begin{proof}
  From Lemma~\ref{prop:sheff-op}, we
  have~$g(t) = e^{-\eta t+\nu t^2/2} \, H_n^{[\nu]}(\eta-\nu t)^{-1}$, so the first task it so find
  the reciprocal of~$\smash{H_n^{[\nu]}(\eta-\nu t)}$. Since we are dealing with formal power
  series, it is sufficient to determine its Taylor coefficients. To this end, we utilize the nice
  generic formula~\cite{Leslie1991} for iterated derivatives of the reciprocal of a function that we
  may take to be a power series~$h \in K[[t]]$.  Writing~$\der = \der_t$, this formula reads
  \begin{equation}
    \label{eq:5}
    \der^m(1/h) = \sum_{k=0}^m \tbinom{m+1}{k+1} \, (-1)^k \: \der^m(h^k)/h^{k+1} ,
  \end{equation}
  and we apply it to~$h = H_n^{[\nu]}(\eta-\nu t)$. We substitute~$\zeta = \eta-\nu t$ in
  Lemma~\ref{lem:feldheim}, whence we
  obtain~$\der^m h^k = \sum_\lambda a_\lambda \, (-\nu)^m \, (nk-2|\lambda|)\ffact{m} \; \tilde{h}$
  with\linebreak%
  $\tilde{h} := \smash{H_{nk-m-2|\lambda|}^{[\nu]}(\eta-\nu t)}$, by the usual differentation rule
  for Appell sequences~\cite[Thm.~2.5.6]{Roman2005} since $H_n^{[\nu]}$ is indeed
  Appell~\cite[\S4.2.1]{Roman2005}. Thus we have for~$h(t)^{-1}$ the Taylor series
  \begin{equation*}
    \sum_{m=0}^\infty
    \sum_{k=0}^m \sum_{\lambda \in \Lambda(n,k)} \tbinom{m+1}{k+1} \, (-1)^{k+m} \: \nu^m \,
    (nk-2|\lambda|)\ffact{m} \; a_\lambda \,
    \tfrac{H_{nk-m-2|\lambda|}^{[\nu]}(\eta)}{H_n^{[\nu]}(\eta)^{k+1}} \, \tfrac{t^m}{m!} .
  \end{equation*}
  Now we need only convolve this with the Taylor series of
  \begin{equation*}
    e^{-\eta t+\nu t^2/2} = \sum_{m=0}^\infty H_m^{[-\nu]}(-\eta) \, \tfrac{t^m}{m!}
  \end{equation*}
  embodying the generating function of Hermite polynomials~\cite[\S4.2.1]{Roman2005}, to end up with
  the formula given in the proposition (after a small sign simplification).
\end{proof}

As mentioned above, an explicit expression for the Appell generator~$g(t)$ is certainly not
necessary for ``computing'' the bivariate polynomials $\Phi_{m,n}^{[\nu]}(\xi, \eta)$. One
straightforward method is the following \emph{recursive scheme}. We note
that~$\Phi_{m,0}^{[\nu]}(\xi, \eta) = H_m^{[\nu]}(\xi+\eta)$
and~$\Phi_{0,n}^{[\nu]}(\xi, \eta) = H_n^{[\nu]}(\eta)$, latter relation following
from~\eqref{eq:conv-pol} immediately, the former by the Appell
identity~\cite[Thm.~2.5.8]{Roman2005}. Furthermore, we have the recurrence
$\Phi_{m+1,n}^{[\nu]}(\xi, \eta) = \xi \, \Phi_{m,n}^{[\nu]}(\xi, \eta) - \Phi_{m,n+1}^{[\nu]}(\xi,
\eta)$, which can be obtained as a simple consequence of~\eqref{eq:phi-rep}. Thus one may compute
all polynomials~$\Phi_{m,n}^{[\nu]}$ with~$(m,n) \in \{ 0, \dots, M \}^2$ for some fixed~$M > 0$
starting with the boundary values for~$(0,n)$ and~$(m,0)$ at the axes, then proceeding in diagonals
with the stencil~$(m,n) \leftarrow \{(m-1,n),(m-1,n+1)\}$.

\IncMargin{1em}
\begin{algorithm}
\SetKwData{Left}{left}\SetKwData{This}{this}\SetKwData{Up}{up}
\SetKwFunction{Union}{Union}\SetKwFunction{FindCompress}{FindCompress}
\SetKwInOut{Input}{Input}\SetKwInOut{Output}{Output}
\Input{$M \in \ZZ_{>0}, \nu \in \RR$}
\Output{$\big(\Phi_{m,n}^{[\nu]}\big)_{m,n = 0, \dots, M}$}
\BlankLine
 \For{$m\leftarrow 0$ \KwTo $2M$}{
   $\Phi_{m,0}^{[\nu]} \leftarrow H_m^{[\nu]}(\xi+\eta)$}
\For{$n\leftarrow 1$ \KwTo $2M$}{
  $\Phi_{0,n}^{[\nu]} \leftarrow H_n^{[\nu]}(\eta)$}
\For{$d\leftarrow 2$ \KwTo $2M$}{
  \For{$j\leftarrow 1$ \KwTo $d-1$}{
    $\Phi_{j,d-j}^{[\nu]} \leftarrow \xi \, \Phi_{j-1,d-j}^{[\nu]} + \Phi_{j-1,d-j+1}^{[\nu]}$}}
\caption{Calculation of~$\Phi_{m,n}^{[\nu]}$}\label{alg:phimn}
\end{algorithm}\DecMargin{1em}

\subsection{The Hyperbolic Fourier Doublet.}
\label{sub:hyp-four-doublet}
Let us briefly look at a last example exhibiting explicit Fourier transforms. It is well-known that
the \emph{hyperbolic secant} is another fixed point of~$\Four\colon \Schw(\RR) \pto \Schw(\RR)$. In
our normalization, $\Four s = s$ for the Schwartz function~$s(t) := \sech(\pi t)$; note
that~$\cum_{-\infty}^{\infty} s(t) \, dt = 1$. We are interested in the (plain) Fourier
doublet~$[S \pto \Sigma]$ generated by~$s$ in~$[\Four\colon \Schw(\RR) \pto \Schw(\RR)]$ \emph{qua}
Fourier doublet. We shall here restrict ourselves to the trivial Heisenberg group~$H(\beta) = 0$,
which means in effect that we admit only~$s$ without adjoining its translates (see the remark
below).

In the case of hyperbolic---as opposed to trigonometric---functions, it is possible to restrict
ourselves to real-valued functions. Indeed, the function~$s$ and all its powers are even and
real-valued, hence so are their Fourier transforms (obtained via~$\Four$ or~$\Four^{-1}$ without
difference). It is therefore suitable, in the scope of this subsection, to reserve the
notation~$\Schw(\RR)$ to \emph{real-valued Schwartz functions} and to describe the desired Fourier
doublets~$[S \pto \Sigma]$ in terms of real-valued function
algebras~$(S, \star) \le \Schw\big(\RR, \star\big)$
and~$(\Sigma, \cdot) \le \big(\Schw(\RR), \cdot\big)$. One may of course subsequently employ
complexification if complex-valued functions are desired.

We shall use the generator~$s$ on the spectral side. In other words, we define the \emph{spectral
  space} as~$\Sigma = \RR[s]_+$, the nonunital algebra of polynomials having positive
degree. This time we shall follow the expedient custom of writing signals as functions in~$x$ and
spectra in~$\xi$. Thus we have~$s(\xi) \in \Sigma$ as generator.

For writing out its Fourier transform, it is useful to introduce a certain \emph{polynomial
  sequence}~$\sigma_n$ that somehow plays the role of the Hermite polynomials in the Gaussian
case. Using binomial coefficients or falling factorials, we can write them as
\begin{equation*}
  \sigma_n(t) := (2i)^k \, \tbinom{(n-1)/2 - it}{n} = \tfrac{(2i)^n}{n!} \,
  (\tfrac{n-1}{2} - it)^{\underline{n}}
\end{equation*}
for~$n \ge 0$. It is easy to see that~$n! \, \sigma_n(t)$ is given
by~$\prod_{j=1}^k \big(4t^2 + (2j+1)\big)$ for~$n=2k$ and by~$2^n \, t \prod_{j=1}^k (t^2+j^2)$
for~$n=2k+1$, so these are in fact real polynomials (which are odd/even exactly when $n$ is
odd/even, just as with the Hermite polynomials).

The Fourier transform of arbitrary positive powers of the hyperbolic secant is a polynomial multiple
of either the hyperbolic secant or \emph{hyperbolic cosecant}, according as the exponent is odd or
even, respectively. Setting~$c(t) := \csch(\pi t)$, we have for~$k \ge 0$ the Fourier transforms
\begin{equation}
  \label{eq:four-hyp}
  \left\{
  \begin{aligned}
    \Four\big(s(\xi)^{2k+1}\big) &= \sigma_{2k}(x) \, s(x) \in \RR[x^2] \, s(x),\\
    \Four\big(s(\xi)^{2k+2}\big) &= \sigma_{2k+1}(x) \, c(x) \in \RR[x^2] \, x \, c(x),
  \end{aligned}
  \right.
\end{equation}
which are all real-valued even functions, as expected.

It is thus clear that the image of the spectral space under~$\Four$, and under~$\Four^{-1}$ alike,
is the $\RR[x^2]$-submodule~$S$ generated by~$s, xc \in \Schw(\RR)$. While this yields the signal
space as a real vector space, we must yet identify the convolution product on it. To this end, note
that $S = \RR[x^2] \, s \oplus \RR[x^2]x \, c$, so every signal may uniquely be written in the
form~$u(x) = p(x^2) \, s(x) + q(x^2) \, x \, c(x)$ with polynomials~$p(t), q(t) \in
\RR[t]$. Incidentally, this shows that~$S \cong \RR[x]$ as real vector spaces (formally
setting~$s=c=1$). Since the~$\sigma_n(t)$ are a basis of~$\RR[t]$, there are unique
coefficients~$c_0, c_1, c_2, \ldots \in \RR$, essentially given by the Stirling partition
numbers~\cite[(4.1.3)]{Roman2005}, such that
\begin{equation*}
  u(x) = \sum_{k=0}^{\deg(p)} c_{2k} \, \sigma_{2k}(x) \, s(x)
  + \sum_{k=0}^{\deg(q)} c_{2k+1} \, \sigma_{2k+1}(x) \, c(x) .
\end{equation*}
Since convolution is bilinear on~$\Schw(\RR)$, it suffices to compute all products
between~$\sigma_{2k}(x) \, s(x) = \Four\big(s(\xi)^{2k+1}\big)$
and~$\sigma_{2k+1}(x) \, c(x) = \Four\big(s(\xi)^{2k+2}\big)$; these are given by
\begin{equation*}
  \Four\big(s(\xi)^m\big) \star \Four\big(s(\xi)^n\big) = \Four\big(s(\xi)^{m+n}\big)
  \in S \quad \text{by~\eqref{eq:four-hyp}}
\end{equation*}
since~$\Four$ is an isomorphism between~$\big(\Schw(\RR), \star\big)$
and~$\big(\Schw(\RR), \cdot\big)$.

\begin{myremark}
  Admitting the full Heisenberg group~$H(\pomega) = \Tor \RR \rtimes \RR$ would commit us to adding
  all translates~$s_b(t) := s(t-b)$ for~$b \in \RR$, which could be restricted for reasons of
  constructivity to rational values of~$b$ as in the case of the rational Gaussian singlet in
  \S\ref{sub:rational-gauss}.

  For seeing that the functions~$s_b$ are algebraically independent over~$\RR$, assume a polynomial
  relation~$r(t) = \sum_{\alpha \in \NN^m} c_\alpha s_{b_1}(t)^{\alpha_1} \cdots
  s_{b_m}(t)^{\alpha_m}$ of minimal total degree, for translates by~$b_1, \dots, b_m \in \RR$ and
  coefficients $c_\alpha \in \RR$. Taking the limit~$|t| \to \infty$ shows that we must
  have~$a_0 = 0$ since all powers of~$s_b$ vanish at~$\infty$. But then we may factor the relation
  as~$s_b(t) \tilde{r}(t) = 0$ for some translate~$s_b$ and a similar relation~$\tilde{r}(t)$ of
  smaller total degree. This is a contradiction since we may divide by $s_b > 0$. Since the~$s_b$
  are algebraically independent, the Heisenberg $H(\pomega)$-algebra generated by~$s$ may be viewed
  as the tensor product of~$\mathfrak{C}(\RR)$ with the algebra~$\RR[s_b \mid b \in \RR]_+$ of
  multivariate polynomials of positive degree. Here~$\mathfrak{C}(\RR)$ is the $\CC$-linear span of
  the \emph{oscillating exponentials}~$e_\alpha$, as we have also used it in
  \S\ref{sub:min-schwartz-algebra} for the Gaussians.

  The problem is that taking Fourier transforms of products such as~$s_b s_{b'}$, or equivalently
  convolving~$s_b$ and~$s_{b'}$, leads to factors involving beta functions whose algebraic treatment
  might be cumbersome. Further investigation on this topic would be needed.
\end{myremark}

We can nevertheless carry the Fourier doublet~$[S \pto \Sigma]$ one step further, achieving closure
under the \emph{action of the Weyl algebra}. To this end, we go back to our usual interpretation
of~$\Schw(\RR)$ as the algebra of complex-valued Schwartz functions on which~$A_1(\RR)$ acts
naturally. Defining the hyperbolic tangent in the form~$h(t) := \tanh(\pi t)$, we
have~$s' = - \pi \, hs$ and~$h' = \pi \, s^2$ so that~$\CC[h][s]_+$ is a differential subalgebra
of~$\Schw(\RR)$. Note that~$\CC[h][s]_+$ consists of all polynomials in~$t$ and~$s$ that have
positive degree in~$s$, but we prefer to see this as enlarging our (complexified) previous spectral
space~$\CC \Sigma = \CC[s]_+$ by extending the scalar ring from~$\CC$ to~$\CC[h]$.
Since~$h^2 = 1 - s^2$, we need only expand~$\CC \Sigma$ by the functions~$hs^n$.

It is easy to find their \emph{Fourier transforms}. Using~$(s^n)' = -n\pi \, hs^n$
and~\eqref{eq:four-transf-weyl}, we find
\begin{equation*}
  \Four^\lor(hs^n) = - \tfrac{2i}{2k+1} \, x \, \Four(s^n),
\end{equation*}
where~\eqref{eq:four-hyp} may be employed for computing the Fourier transform of~$s^n$. Note that
this time we must explicitly use the backward Fourier transform~$\Four^\lor$ since the
functions~$hs^n$ are odd so that their images under~$\Four^\land$ differ from those
under~$\Four^\lor$ by a sign. By the same token, the corresponding (forward or backward) Fourier
transforms are imaginary- rather than real-valuede functions. On the signal side, the (complexified)
space~$\CC S = \CC[x^2] \, s \oplus \CC[x^2]x \, c$ gains the ``missing''
components~$\CC[x^2]x \, s \oplus \CC[x^2]_+ \, c$.

For \emph{obtaining closure under the Weyl action}, the space~$\CC[h][s]_+$ must be extended by the
polynomials~$\CC[x]$; thus we set~$\Sigma' := \CC[x, h][s]_+$ for the enlarged spectral space. Its
image under~$\Four^\lor$ or, equivalently, under~$\Four^\land$ is then enlarged
from~$\CC[x] \, s \oplus \CC[x] \, xc$ by adding in all derivatives. We call
it~$S' \subset \Schw(\RR)$, but unfortunately its explicit characterization appears to be cumbersome
and shall not be given here. We prefer to proceed in a more roundabout way, enlarging~$S'$ by
meromorphic functions outside of~$\Schw(\RR)$.

Iterated derivatives of~$s$ and~$c$ may be computed using so-called \emph{derivative
  polynomials}~\cite{Boyadzhiev2007}. We shall not need their explicit form here; it suffices to
know that the $m$-th derivative of~$s$ is given by~$s$ times a certain $m$-th degree polynomial
in~$h$ while the corresponding derivative of~$c$ is~$c$ times an $m$-th degree polynomial
in~$h^{-1}$. Clearly, all derivatives of~$xc$ are then given by similar
expressions. Writing~$\mathcal{H} := \CC[h,h^{-1}][x]$ for the new coefficient ring, we may thus
define~$S'' := \mathcal{H} s \oplus \mathcal{H} c$. It is easy to see that~$S''$ is a
$D$-module (for~$D = A_1(\CC)$, that is) containing~$S'$. Provided the elements
of~$S'$ are identified in the form~$w \cdot s$ or~$w \cdot xc$ for a Weyl actor~$w \in A_1(\CC)$, is
is straightforward to compute their Fourier transforms via~\eqref{eq:four-transf-weyl}
and~\eqref{eq:four-hyp}. Let us summarize our results.

\begin{proposition}
  The Fourier doublets~$[S \pto \Sigma] \le [S' \pto \Sigma']$ are regular subdoublets
  of~$[\Schw(\RR) \pto \Schw(\RR)]$ qua plain doublet. Moreover, the compatible Weyl action on the
  latter restricts to~$[S' \pto \Sigma']$.
\end{proposition}

\begin{myremark}
  \label{rem:hyperbolic-singlet}
  It would be significantly more ambititious to study the Fourier \emph{singlet} generated by~$s$
  in~$\Schw(\RR)$, even for the trivial Heisenberg group as used above. We shall leave this as a
  challenge for future investigations.

  We might consider~$\A := \CC[x][c, s \mid c^2 s^2 - c^2 - s^2]$, which is an algebra and a
  $D$-module. Then we might adjoin~$h$ and~$h^{-1}$ to~$\A$, subject to the further
  relations~$h^2 = 1 - s^2$, $h^{-2} = 1 + c^2$ and~$h^{-1}- h = sc$. Computations in Mathematica
  suggest that most (or all?) Fourier transforms of such functions can be determined explicitly. But
  certain simplifications such as~$\tanh(\pi t/2) = h^{-1}(t) - c(t)$ and
  $\tfrac{1}{2} \, \sech^2(\pi t/2) = s(t) \, c(t)^2 - c(t)^2 + s(t)$ may be needed. Moroever, one
  may need larger function spaces (perhaps not quite distributions) for justifying~$\Four(h) = ic$
  and~$\Four(h^{-1}) = i h^{-1}$, but maybe one should not adjoin these functions themselves (only
  their products with ``moderating'' functions such as $s$).
\end{myremark}

\section*{Acknowledgments}

We are very grateful to Heinrich Rolletschek for his helpful and
friendly advice on number-theoretic issues.

%\bibliographystyle{plain}
%\bibliography{../AFS}

\end{document}